\RequirePackage{fix-cm}
\documentclass{svjour3}                     
\smartqed  

\usepackage{subfig}
\usepackage{subfloat}
\usepackage{graphicx}

\usepackage{algorithm}
\usepackage{algorithmicx}
\usepackage{algpseudocode}
\usepackage{amsfonts,amsmath,yfonts}
\usepackage{amssymb}
\usepackage{color}
\usepackage{multirow}
\usepackage{hyperref}
\usepackage{wrapfig}

\newcommand{\field}[1]{\mathbb{#1}}

\newtheorem{observation}[theorem]{Observation}

\begin{document}

\title{Numerical analysis of shell-based geometric image inpainting algorithms and their semi-implicit extension\thanks{CBS acknowledges support from the Leverhulme Trust project ’Breaking the non-convexity barrier’, the EPSRC grant ’EP/M00483X/1’, EPSRC centre ’EP/N014588/1’, the Cantab Capital Institute for the Mathematics of Information, the CHiPS (Horizon 2020 RISE project grant), the Global Alliance project 'Statistical and Mathematical Theory of Imaging' and the Alan Turing Institute.
RH and TH acknowledge support from the EPSRC Cambridge Centre for Analysis and the Cambridge Trust.}
}
\titlerunning{Analysis of Shell-Based image inpainting}


\author{L. Robert Hocking, Thomas Holding, and Carola-Bibiane Sch\"onlieb}


\institute{L. Robert Hocking \at
              Department for Applied Mathematics and Theoretical Physics, University of Cambridge, Cambridge
CB3 0WA, UK \\
              Tel.: +44 1223 764284\\
              \email{rob.l.hocking@gmail.com}           
           \and
           Thomas Holding \at
              Mathematics Institute, Zeeman Building, University of Warwick, Coventry CV4 7AL \\
              Tel.: +44 (0)24 765 73495\\
              \email{t.holding@warwick.ac.uk} 
              \and
              Carola-Bibiane Sch\"onlieb \at
              Department for Applied Mathematics and Theoretical Physics, University of Cambridge, Cambridge
CB3 0WA, UK \\
              Tel.: +44 1223 764251\\
              \email{cbs31@cam.ac.uk} 
}

\date{\today}

\maketitle

\begin{abstract}
In this paper we study a class of fast geometric image inpainting methods based on the idea of filling the inpainting domain in successive shells from its boundary inwards. Image pixels are filled by assigning them a color equal to a weighted average of their already filled neighbors. However, there is flexibility in terms of the order in which pixels are filled, the weights used for averaging, and the neighborhood that is averaged over. Varying these degrees of freedom leads to different algorithms, and indeed the literature contains several methods falling into this general class.  All of them are very fast, but at the same time all of them leave undesirable artifacts such as ``kinking'' (bending) or blurring of extrapolated isophotes.  Our objective in this paper is to build a theoretical model, based on a continuum limit and a connection to stopped random walks, in order to understand why these artifacts occur and what, if anything, can be done about them.  At the same time, we consider a semi-implicit extension in which pixels in a given shell are solved for simultaneously by solving a linear system. We prove (within the continuum limit) that this extension is able to completely eliminate kinking artifacts, which we also prove must always be present in the direct method.  Although our analysis makes the strong assumption of a square inpainting domain, it makes weak smoothness assumptions and is thus applicable to the low regularity inherent in images.
\keywords{image processing, image inpainting, partial differential equations, stopped random walks, numerical analysis, convergence rates}
\subclass{68U10, 65M12, 65M15, 65F10, 60G50, 60G40, 35F15, 60G42, 65M15, 35Q68}

\end{abstract}

\section{Introduction}

Image inpainting refers to the filling in of a region in an image where information is missing or needs to be replaced (due to, for example, an occluding object), in such a way that the result looks plausible to the human eye.  The region to be filled is subsequently referred to as the inpainting domain.  Since the seminal work of Bertalmio et al. \cite{bertalmio2000image}, image inpainting has become increasingly important, with applications ranging from removing an undesirable or occluding object from a photograph, to painting out a wire in an action sequence, to 3D conversion of film \cite{Guidefill}, as well as 3D TV and novel viewpoint construction \cite{OtherPipelineDoneWell,DepthGuidedInpaintingAlgorithmForFreeViewPointVideo,DepthAidedInpainingForNovelViewSynthesis} in more recent years.  See  \cite{ImageInpaintingOverview} for a recent survey of the field.

Image inpainting methods can loosely be categorized as {\em exemplar-based} and {\em geometric}.  Exemplar-based methods generally operate by copying pieces of the undamaged portion of the image into the inpainting domain, in such a way as to make the result appear seamless.  Examples include \cite{Criminisi04regionfilling}, \cite{Spacetime}, \cite{Arias2011}.  Exemplar-based methods also operate behind the scenes in Photoshop's famous {\em Content-Aware Fill} tool.  These methods excel at interpolating texture across large gaps, but may produce artifacts in structure dominated images with strong edges, or be needlessly expensive if the inpainting domain is thin.

\begin{figure}
\centering
\begin{tabular}{cccc}
\subfloat[Original image.]{\includegraphics[width=.21\linewidth]{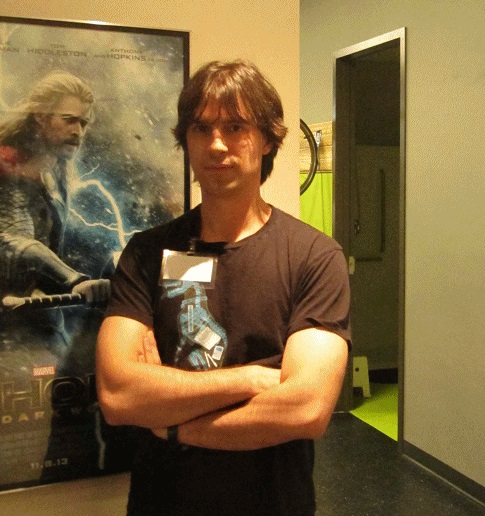}} & 
\subfloat[After filling 7 shells.]{\includegraphics[width=.21\linewidth]{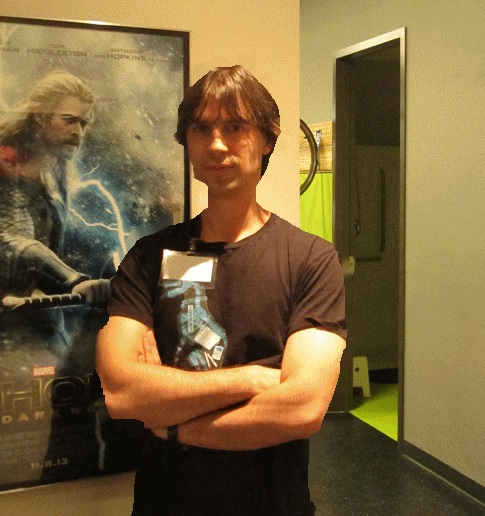}} &
\subfloat[After filling 20 shells.]{\includegraphics[width=.21\linewidth]{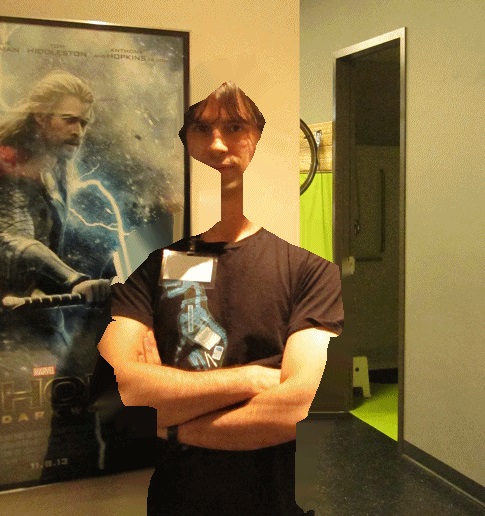}} &
\subfloat[After filling 31 shells.]{\includegraphics[width=.21\linewidth]{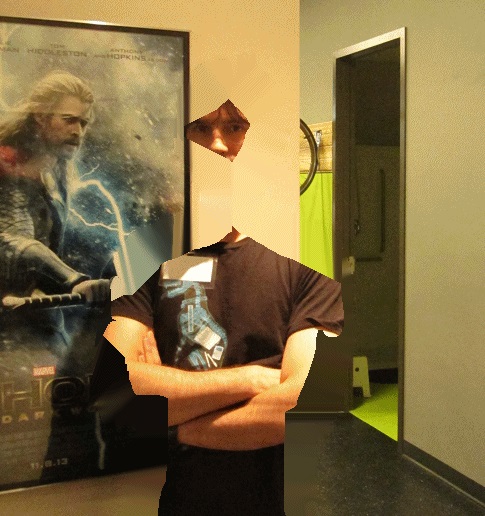}} \\
\end{tabular}
\caption{{\bf Shell-based inpainting: }  Here we illustrate the shell-based inpainting of an image including an undesirable human to be removed.  In (a), we see the original image, including a human that is gradually eroded in (b)-(d) as we fill more shells.  In this case the inpainting method is Guidefill \cite{Guidefill}, and the application is disocclusion for 3D conversion, which means that the human does not need to be removed entirely.  See \cite{Guidefill} or \cite[Chapter 9]{CarolaBook} for more details on this application.  }
\label{fig:erode}
\end{figure}

Geometric inpainting methods attempt to smoothly extend image structure into the inpainting domain, typically using partial differential equations (PDEs) or variational principles - see \cite{CarolaBook} for a comprehensive survey.  Examples based on PDEs include the seminal work of Bertalmio et al. \cite{bertalmio2000image}, methods based on anisotropic diffusion such as \cite{CurvatureSmoothingPDEs}, while examples based on variational principles include TV, TV-H$^{-1}$, Mumford-Shah, Cahn-Hilliard inpainting \cite{chan2005variational,Burger2009}, Euler's Elastica \cite{masnou1998level,chan2002euler}, as well as the joint interpolation of image values and a guiding vector field in Ballester et al. \cite{JointGreyVector}.  These approaches are typically iterative and convergence is often slow (due to either the high order of the underlying PDE model or the difficult nature of the variational problem).  On the other hand, Telea's Algorithm \cite{Telea2004}, coherence transport \cite{Marz2007,Marz2011}, and our previous work Guidefill \cite{Guidefill} are based on the simple idea of filling the inpainting domain in successive shells from the boundary inward, setting the color of pixels on the current boundary equal to a weighted average of their already filled neighbors - Figure \ref{fig:erode} illustrates this process.  In general, ``neighbors'' may include ``ghost pixels'' lying between pixel centers and defined using bilinear interpolation - this concept was introduced in \cite{Guidefill}.  By choosing the weights, fill order, and averaging neighborhood carefully, image isophotes may be extrapolated into the inpainting domain.  Compared to their iterative counterparts, these methods have the advantage of being very fast.  However, at the same time all of them create, to a greater or lesser degree, some potentially disturbing visual artifacts.  For example, extrapolated isophotes may ``kink'' (change direction), blur, or end abruptly.  Although generally these problems have been gradually reduced over time as newer and better algorithms have been proposed, they are still present to some degree even in the state of the art.  

The present paper has two goals.  Firstly, to analyze these algorithms as a class in order to understand the origins of these artifacts and what, if anything, can be done about them.  Secondly, to propose a simple extension in which the pixels in a given shell are solved simultaneously as a linear system, and to analyze its benefits.  We will refer to the original methods as ``direct'' and the extended methods as ``semi-implicit'' - see Figure \ref{fig:directVsExtension} for an illustration of the difference between them.  To our knowledge, this extension has never before been proposed in the literature.  We also propose a simple iterative method for solving the linear system arising in the semi-implicit extension, show that it is equivalent to damped-Jacobi or successive over-relaxation (SOR), and analyze its convergence rates when applied to the linear system arising in semi-implicit Guidefill.  Successive over-relaxation is shown to converge extremely quickly, provided unknowns are ordered appropriately.

\begin{figure}
\centering
\begin{tabular}{cc}
\subfloat[{\bf The direct method:} the color of a given pixel (highlighted in red) on the current inpainting domain boundary (blue) is computed as a weighted sum of its already known neighbors in the filled portion of the image (pale yellow).  Pixels included in the sum are highlighted in green.]{\includegraphics[width=.45\linewidth]{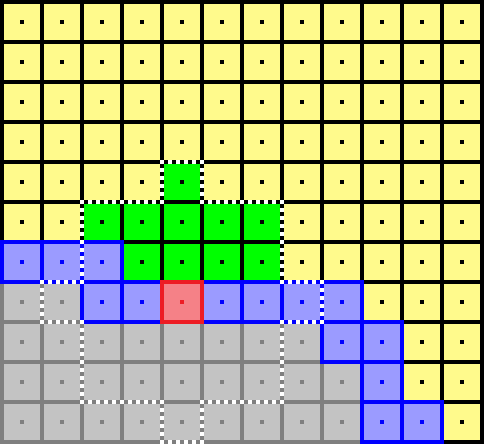}} & 
\subfloat[{\bf The semi-implicit extension: } the color of a given pixel (highlighted in red) on the current inpainting domain boundary (blue) is given implicitly as a weighted sum of its already known neighbors in the filled portion of the image (pale yellow), as well as unknown neighboring pixels on the current boundary.  Pixels included in the sum are highlighted in green.]{\includegraphics[width=.45\linewidth]{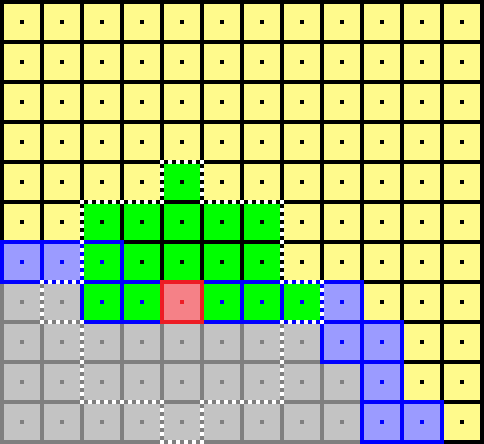}} \\
\end{tabular}
\caption{{\bf The direct method and its semi-implicit extension: } In this illustration, the filled portion of the image is highlighted in pale yellow, and the current inpainting domain is highlighted in both grey and blue, with the former denoting pixels in the interior of the inpainting domain and the latter pixels on its current boundary.  The boundary is the ``shell'' that is currently being filled -  note that the filled (known) portion of the image consists not only of the original undamaged region, but also of previously filled shells.  At this stage, the grey and blue pixels are both unknown.  In the direct method (a), the color of a given pixel (highlighted in red) is computed directly as a weighted average of pixels (in green) within a given neighboorhood (outlined in white) that are already known.  In the semi-implicit extension, the sum also includes pixels on the current boundary of the inpainting domain, but not pixels in its interior.  This results in linear system that needs to be solved, but this can be done relatively cheaply (see Section \ref{sec:guidefillGS}) and has benefits in terms of artifact reduction (see Section \ref{sec:kinking}).  }
\label{fig:directVsExtension}
\end{figure}

\begin{remark} \label{rem:Russ}
The original motivation for this paper comes from one of the author's experiences working in industry for the 3D conversion company Gener8 (see \cite{Guidefill} or \cite[Chapter 9]{CarolaBook} for a description of 3D conversion).  The programmers there had implemented a series of shell-based inpainting methods of the form discussed above (all essentially variants of Telea's algorithm \cite{Telea2004}, which they were unaware of) in order to solve an inpainting problem arising in their 3D conversion pipeline.  This type of algorithm was attractive to them because it is fast and simple, while generally adequate for their application (where inpainting domains are typically very thin).  However, they were puzzled by some of the artifacts they were seeing, most notably kinking artifacts.  After a literature review it became clear that the existing theory \cite{Marz2007,Marz2011,Marz2015} was enough to explain some, but not all, of what they observed.  This paper is an attempt to fill the gap.
\end{remark}

\subsection{Organization}  The remainder of this paper is organized as follows.  First, in Section \ref{sec:shellBased}, we present the general form of the class of inpainting methods under consideration, including pseudocode for both the direct method and an implementation of our proposed semi-implicit extension.  We also cover in detail known artifacts created by the direct form of the method, how they have been reduced over time, related work and our contributions.  Next, in Section \ref{sec:review} we review the main variants of Algorithm 1 present in the literature.  Section \ref{sec:fiction} describes an equivalence principle between weighted sums of ghost pixels and a sum over real pixels with different weights.  This principle will be applied repeatedly throughout our theoretical analysis.  Section \ref{sec:semiImplicit} describes our semi-implicit extension of Algorithm 1, including a description of an implementation of two alternatives for solving the resulting linear system, namely damped Jacobi and successive over-relaxation (SOR).  We also discuss in some detail the semi-implicit extension of Guidefill \cite{Guidefill}, a special case of Algorithm 1.  Section \ref{sec:analysis} contains our core analysis, and is divided into a number of subsections.  The first of these, Section \ref{sec:symmetry}, describes the simplifying assumptions that we make throughout.  Next, Section \ref{sec:guidefillGS} presents a theoretical analysis of the convergence rates of damped Jacobi and SOR for solving the linear system arising in semi-implicit Guidefill, and compares with real experiments.  In Section \ref{sec:continuumLimit1} we describe the connection between the class of algorithms under consideration and stopped random walks, and then prove convergence to our continuum limit.  Finally, Section \ref{sec:marzLimit} discusses the relationship between our continuum limit and the one proposed by Bornemann and M\"arz \cite{Marz2007}, and proves convergence to it as well.  In Section \ref{sec:consequences}, we apply our analysis to gain a theoretical understanding of kinking and blur artifacts.  Section \ref{sec:numerics} validates our analysis with numerical experiments, and finally in Section \ref{sec:conclusion} we draw some conclusions and sketch some directions for potential future research.  We also have a number of appendices giving technical details of certain proofs and some additional numerical experiments.

\subsection{Notation}

\begin{itemize}
\item $h = \mbox{ the width of one pixel}$.
\item $\field{Z}^2_h := \{ (nh,mh) : (n,m) \in \field{Z}^2 \}.$
\item Given ${\bf v} \in \field{R}^2$, we denote by $L_{\bf v}:=\{ \lambda {\bf v} : \lambda \in \field{R} \}$ the line through the origin parallel to ${\bf v}$.
\item Given ${\bf x} \in \field{R}^2$, we denote by $\theta({\bf x}) \in [0,\pi)$ the counter-clockwise angle $L_{\bf x}$ makes with the $x$-axis.  $\theta({\bf x})$ can also be thought of as the counterclockwise angle ${\bf x}$ makes with the $x$-axis, modulo $\pi$.
\item Given ${\bf v} \in \field{R}^2$, we denote by ${\bf v}^{\perp}$ the counterclockwise rotation of ${\bf v}$ by $90^{\circ}$.
\item $\Omega = [a,b] \times [c,d]$ and $\Omega_h = \Omega \cap \field{Z}^2_h$ are the continuous and discrete image domains.
\item $D_h = D^{(0)}_h \subset \Omega_h$ is the (initial) discrete inpainting domain.
\item $D^{(k)}_h \subseteq D^{(0)}_h$ is the discrete inpainting domain on step $k$ of the algorithm. 
\item $D \subset \Omega :=  \{ {\bf x} \in \Omega : \exists {\bf y} \in D_h \mbox{ s.t. } \|{\bf y}-{\bf x}\|_{\infty} < h \}$ is the continuous inpainting domain.
\item $D^{(k)}$ is the continuous inpainting domain on step $k$ of the algorithm, defined in the same way as $D$.
\item $u_0 : \Omega_h \backslash D_h \rightarrow \field{R}^d$ is the given (discrete) image.  In an abuse of notation, we also use $u_0$ to denote an assumed underlying continuous image $u_0 : \Omega \backslash D \rightarrow \field{R}^d$.
\item $u_h : \Omega_h \rightarrow \field{R}^d$ is the inpainted completion of $u_0$.
\item ${\bf g}({\bf x})$ is the guidance vector field used by coherence transport and Guidefill.
\item $B_{\epsilon}({\bf x})$ the solid ball of radius $\epsilon$ centered at ${\bf x}$.  
\item $A_{\epsilon,h}({\bf x}) \subset B_{\epsilon}({\bf x})$ denotes a generic discrete (but not necessarily lattice aligned) neighborhood of radius $\epsilon$ surrounding the pixel ${\bf x}$ and used for inpainting.
\item $\mbox{Supp}(A_{\epsilon,h}({\bf x})) \subset \field{Z}_h^2$ denotes the set of real pixels needed to define $A_{\epsilon,h}({\bf x})$ biased on bilinear interpolation.
\item $B_{\epsilon,h}({\bf x}) = \{ {\bf y} \in \Omega_h : \|{\bf x}-{\bf y}\| \leq \epsilon \}$, the discrete ball of radius $\epsilon$ centerd at ${\bf x}$ and the choice of $A_{\epsilon,h}({\bf x})$ used by coherence transport.
\item $r = \epsilon/h$ the radius of $B_{\epsilon,h}({\bf x})$ measured in pixels.
\item $\tilde{B}_{\epsilon,h}({\bf x}) = R(B_{\epsilon,h}({\bf x}))$, where $R$ is the rotation matrix taking $(0,1)$ to ${\bf g}({\bf x})$, the choice of $A_{\epsilon,h}({\bf x})$ used by Guidefill.
\item $\mathcal N({\bf x})=\{ {\bf x}+{\bf y} : {\bf y} \in \{-h,0,h\} \times \{-h,0,h\}\}$ is the 9-point neighborhood of ${\bf x}$.
\item Given $A_h \subset \field{Z}^2_h$, we define the dilation of $A_h$ by
$$D_h(A_h) = \cup_{{\bf x} \in A_h} \mathcal N({\bf x}).$$
If $h=1$ we write $D$ instead of $D_1$.
\item Given $A_h \subset \field{Z}^2_h$, we define the discrete (inner) boundary of $A_h$ by $$\partial A_h := \{ {\bf x} \in A_h : \mathcal N({\bf x}) \cap \field{Z}^2_h \backslash A_h \neq \emptyset \}.$$  For convenience we typically drop the word ``inner'' and refer to $\partial A_h$ as just the boundary of $A_h$.
\item $O$ denotes the zero matrix.
\item Given $c \in \field{R}$, we define $\{ y \leq c \} := \{ (x,y) \in \field{R}^2 : y \leq c \}.$
\item Given $x$, $y \in \field{R}$, we define $x \wedge y := \min(x,y)$.
\end{itemize}

\section{Shell-based geometric inpainting} \label{sec:shellBased}

\begin{algorithm}
\caption{Shell Based Geometric Inpainting}
\begin{algorithmic} \label{alg:general}
\State $u_h$ = damaged image, initialized to $0$ on inpainting domain.
\State $\Omega = [a,b] \times [c,d]$ = continuous image domain.
\State $\Omega_h = \Omega \cap \field{Z}^2_h$ = discrete image domain.
\State $D^{(0)}_h$ = initial inpainting domain.
\State $\partial D^{(0)}_h $= initial inpainting domain boundary.
\State {\color{blue} semiImplicit = false, unless we use the semi implicit extension (Section \ref{sec:semiImplicit})}.
\For{$k=0,\ldots$}
\If{$D^{(k)}_h = \emptyset$}
\State break
\EndIf
\State $\partial_{\mbox{ready}} D^{(k)}_h = \{ {\bf x} \in \partial D^{(k)}_h : \mbox{ready}({\bf x})\}$
\State $u_h$ = \Call{FillBoundary}{$u_h$, $D^{(k)}_h$, $\partial_{\mbox{ready}} D^{(k)}_h$}
\State $D^{(k+1)}_h = D^{(k)}_h \backslash \partial_{\mbox{ready}} D^{(k)}_h$
\vskip 2mm \color{blue}
\If{semiImplicit}
\State $u^{(0)}_h = u_h$
\For{$n=1,2,\dots$(until convergence)}
\State $u^{(n)}_h$ = \Call{FillBoundary}{$u^{(n-1)}_h$, $D^{(k+1)}_h$, $\partial_{\mbox{ready}} D^{(k)}_h$}
\EndFor
\EndIf \color{black}
\vskip 2mm
\State $\partial D^{(k+1)}_h = \{ {\bf x} \in \partial D^{(k+1)}_h : \mathcal N({\bf x}) \cap (\Omega_h \backslash D^{(k+1)}_h ) \neq \emptyset \}.$
\EndFor
\vskip 3mm
\Function{$u_h$ = FillBoundary}{$u_h$, $D_h,\partial D_h$}
\For{${\bf x} \in \partial D_h$}
\State compute $A_{\epsilon,h}({\bf x})$ = neighborhood of ${\bf x}$.
\State compute non-negative weights $w_{\epsilon}({\bf x},{\bf y}) \geq 0$ for $A_{\epsilon,h}({\bf x})$.
\If{ready({\bf x})}
\State 
\vspace{-5mm}
\begin{equation} \label{eqn:update}
u_h({\bf x})  = \frac{\sum_{{\bf y} \in (A_{\epsilon,h}({\bf x}) \backslash \{ {\bf x}\}) \cap (\Omega \backslash D)}w_{\epsilon}({\bf x},{\bf y})u_h({\bf y})  }{\sum_{{\bf y} \in (A_{\epsilon,h}({\bf x}) \backslash \{ {\bf x}\}) \cap (\Omega \backslash D) }w_{\epsilon}({\bf x},{\bf y})} \phantom{} 
\end{equation}
\vspace{-1mm}
\EndIf
\EndFor
\EndFunction
\end{algorithmic}
See \eqref{eqn:ready}  for a definition of the $\mbox{ready}$ function for Guidefill.  Coherence transport and Guidefill use the neighborhoods $A_{\epsilon,h}({\bf x}) = B_{\epsilon,h}({\bf x})$, $A_{\epsilon,h}({\bf x}) = \tilde{B}_{\epsilon,h}({\bf x})$ respectively - see Figure \ref{fig:ballRotate}.  They also both use the same weights \eqref{eqn:weight}.  Blue text is only relevant for the semi-implicit extension we introduce in Section \ref{sec:semiImplicit}.
\end{algorithm}

\begin{figure}
\centering \includegraphics[width=0.4\textwidth]{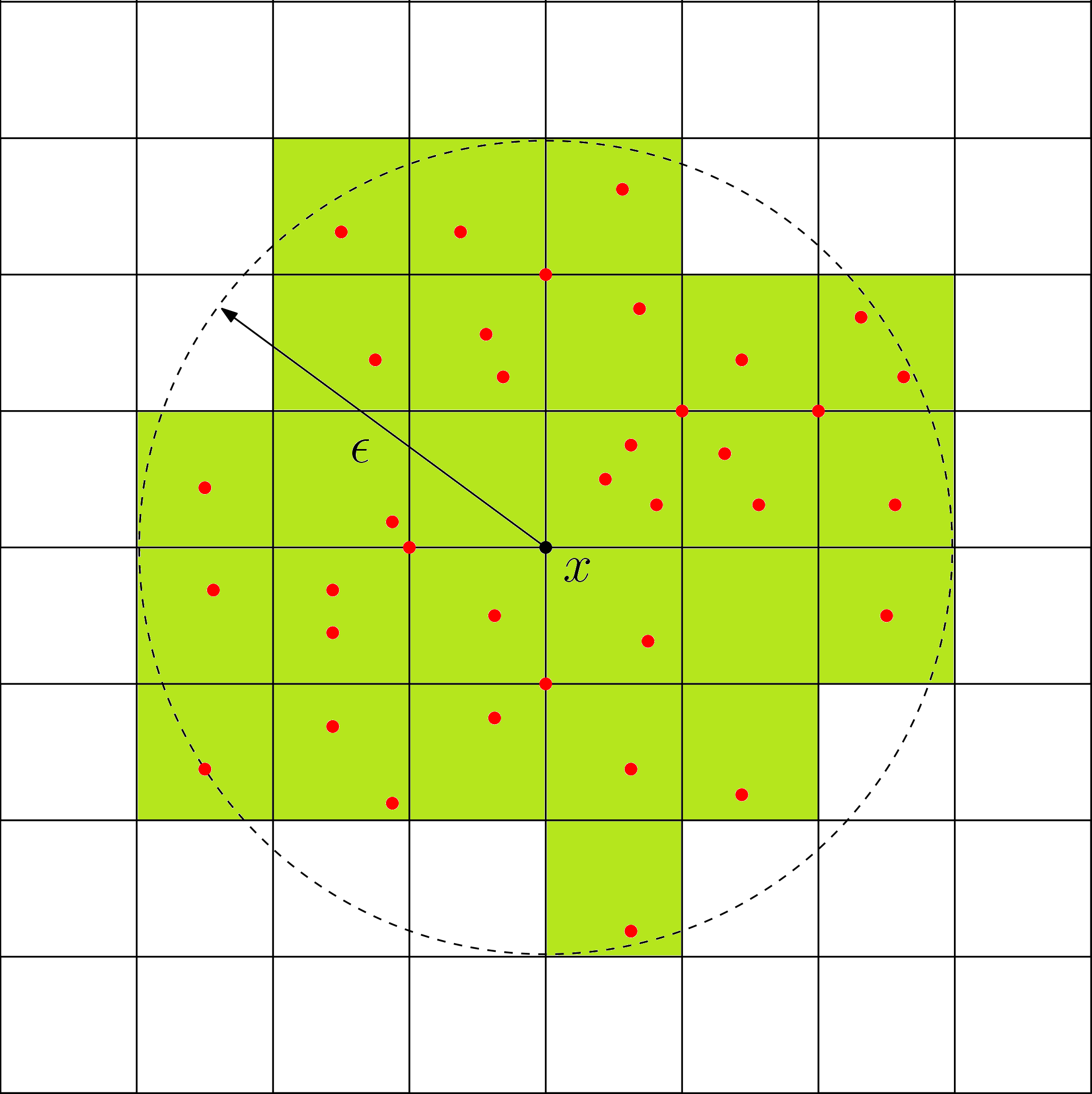}
\caption{{\bf Illustration of a generic set $A_{\epsilon,h}({\bf x})$ containing ghost pixels:}  In this illustration the overlaid grid is the lattice $\field{Z}^2_h$ with pixel centers at its vertices.  The elements of a generic $A_{\epsilon,h}({\bf x})$ are represented as red dots - they do not need to occupy pixel centers, but they must all lie within distance $\epsilon$ of ${\bf x}$.  Ghost pixels are defined based on bilinear interpolation of their real pixel neighbors.  Here we have highlighted in green the squares whose vertices are the real pixels needed to define $A_{\epsilon,h}({\bf x})$.  We call this set of real pixels $\mbox{Supp}(A_{\epsilon,h}({\bf x}))$.  Note that while $A_{\epsilon,h}({\bf x}) \subset B_{\epsilon}({\bf x})$, this inclusion is not in general true of $\mbox{Supp}(A_{\epsilon,h}({\bf x}))$.}
\label{fig:ghosts}
\end{figure}

In this paper we are interested in a simple class of shell-based geometric inpainting methods and their proposed semi-implicit extension.  These methods fill the inpainting domain in successive shells from the boundary inwards, as illustrated in Figure \ref{fig:erode}.  In the direct form of the method, the color of a given pixel ${\bf x}$ due to be filled is computed as a weighted average of its already filled neighbors within a discrete neighborhood $A_{\epsilon,h}({\bf x}) \subset B_{\epsilon}({\bf x})$.  In the semi-implicit extension, this sum also includes unknown pixels within the current shell, resulting in a linear system (Figure \ref{fig:directVsExtension}).  We will cover the resulting linear system in detail in Section \ref{sec:semiImplicit}, where we also propose an iterative method for its solution.  The direct method as well as this proposed iterative solution to the semi-implicit extension are illustrated in Algorithm 1 with pseudo code (the blue code indicates the parts relevant to the semi-implicit extension).  The neighborhood $A_{\epsilon,h}({\bf x})$ need not be axis aligned and may contain ``ghost pixels'' lying between pixel centers - see Figure \ref{fig:ghosts} for an illustration.  Ghost pixels were introduced in \cite{Guidefill} (where they were shown to be beneficial for reducing ``kinking artifacts'' - see Figure \ref{fig:taleOfTwo}), and the color of a given ghost pixel is defined as the bilinear interpolation of its four real pixel neighbors, but is undefined if one or more of them has not yet been assigned a color.  We denote by $\mbox{Supp}(A_{\epsilon,h}({\bf x})) \subset \field{Z}_h^2$ the set of real pixels needed to define $A_{\epsilon,h}({\bf x})$ in this way.  Here $h$ and $\epsilon$ denote respectively the width of one pixel and the radius of the bounding disk $B_{\epsilon}({\bf x}) \supset A_{\epsilon,h}({\bf x})$.  The averaging weights $w_{\epsilon}$ are non-negative and are allowed to depend on ${\bf x}$, but must scale proportionally with the size of the neighborhood $A_{\epsilon,h}({\bf x})$, like so:
\begin{equation} \label{eqn:scalinglaw}
w_{\epsilon}({\bf x},{\bf y})=\hat{w}\left({\bf x}, \frac{{\bf y}-{\bf x}}{\epsilon}\right)
\end{equation}
for some function $\hat{w}(\cdot,\cdot) : \Omega \times B_1({\bf 0}) \rightarrow [0,\infty]$.  Note that we will sometimes write $w_r$ or $w_1$ in place of $w_{\epsilon}$ - in this case we mean \eqref{eqn:scalinglaw} with $\epsilon$ replaced by $r$ or $1$ in the denominator on the right hand side.    As the algorithm proceeds, the inpainting domain shrinks, generating a sequence of inpainting domains $D_h = D^{(0)}_h \supset D^{(1)}_h \supset \ldots \supset D^{(K)}_h = \emptyset$.  We will assume the non-degeneracy condition
\begin{equation} \label{eqn:technical}
\sum_{{\bf y} \in A_{\epsilon,h}({\bf x}) \cap (\Omega \backslash D^{(k)})}w_{\epsilon}({\bf x},{\bf y}) > 0
\end{equation}
\noindent holds at all times, this ensures that the average \eqref{eqn:update} in Algorithm 1 is always well defined.  One trivial way of ensuring this is by having strictly positive weights $\hat{w}$, which all the methods considered do (see Section \ref{sec:review}).  At iteration $k$, only pixels belonging to the current boundary $\partial D^{(k)}_h$ are filled, but moreover we fill only a subset $\partial_{\mbox{ready}} D^{(k)}_h \subseteq \partial D^{(k)}_h$ of pixels deemed to be ``ready'' (In Section \ref{sec:review} we will review the main methods in the literature and give their ``ready'' functions).  The main inpainting methods in the literature of the general form given by Algorithm 1 are (in chronological order) Telea's Algorithm \cite{Telea2004}, coherence transport \cite{Marz2007}, coherence transport with adapted distance functions \cite{Marz2011}, and our previous work Guidefill \cite{Guidefill}.  As we will see, these methods essentially differ only in the choice of weights \eqref{eqn:scalinglaw}, the choice of fill order as dictated by the ``ready'' function, and the choice of neighborhood $A_{\epsilon,h}({\bf x})$.  We will review these methods in Section \ref{sec:review}.  

\vskip 2mm

\begin{remark} \label{remark:genericSinglePass}
It is worth mentioning that this class of algorithms is nearly exactly the same as the ``generic single-pass algorithm'' first systematically studied by Bornemann and M\"arz in \cite{Marz2007}.  The two main differences are 
\begin{enumerate}
\item They assume $A_{\epsilon,h}({\bf x})=B_{\epsilon,h}({\bf x})$, while we allow for greater flexibility.  
\item They consider only the direct form of Algorithm 1, not the semi-implicit extension.
\end{enumerate}
Beyond this, they also phrase things in terms of a pre-determined fill order, rather than a ``ready'' function, but the former may easily be seen, mathematically at least, to be a special case of the latter (Section \ref{sec:review}).
\end{remark}
\vskip 2mm 
\subsection{Advantages of Algorithm 1}  The main appeal of Algorithm 1 is its simplicity and parallelizability, which enable it to run very fast.   A second advantage, first noted by Bornemann and M\"arz \cite{Marz2007}, is the stability property
\begin{equation} \label{eqn:stability}
\min_{{\bf y} \in B}u_0({\bf y}) \leq u_h({\bf x}) \leq \max_{{\bf y} \in B} u_0({\bf y}) \qquad \mbox{ for all }{\bf x} \in D_h,
\end{equation}
(which holds channelwise) where $u_0 : \Omega_h \backslash D_h \rightarrow \field{R}^d$ is the given image and $B$ is the band of width $\epsilon$ pixels surrounding $\partial D_h$.  This property holds because we have chosen non-negative weights summing to one.

\begin{remark} \label{remark:telea}
Although we have presented Telea's algorithm \cite{Telea2004} as an example of Algorithm 1, this is not strictly true as its update formula \eqref{eqn:Telea} (see Section \ref{sec:review}) contains a gradient term that, after it has been approximated with finite differences, effectively violates the rule of non-negative weights summing to one.  This means that Telea's algorithm does not satisfy the stability property \eqref{eqn:stability}.  See Figure \ref{fig:stability}.
\end{remark}

\subsection{Disadvantages and artifacts}  The main disadvantage of Algorithm 1 is obviously loss of texture, and in cases where texture is important, these methods should not be used.   However, beyond loss of texture, inpainting method of the general form given in Algorithm 1 can also introduce a host of of other artifacts, which we list below.
\begin{itemize}
\item {\bf``kinking'' of isophotes}  where extrapolated isophotes change direction at the boundary of the inpainting domain - see Figure \ref{fig:taleOfTwo} and Figure \ref{fig:artifacts2}.
\item {\bf ``blurring'' of isophotes}  where edges that are sharp in the undamaged region may blur when extended into the inpainting domain - see Figure \ref{fig:taleOfTwo} and Figure \ref{fig:artifacts2}.
\item {\bf``cutting off'' of isophotes} where isophotes extrapolated into the inpainting domain end abruptly - see Figure \ref{fig:cutoff}.
\item {\bf formation of shocks}  where sharp discontinuities may form in the inpainting domain - see Figure \ref{fig:cutoff}.
\item {\bf bright or dark spots} that are only a problem if the stability condition \eqref{eqn:stability} is violated, as it is for Telea's algorithm.  See Figure \ref{fig:stability} and Figure \ref{fig:artifacts2}.
\end{itemize}

\begin{figure}
\centering
\begin{tabular}{cc}
\subfloat[Coherence transport with default onion shell ordering.  Isophotes are cut off.]{\includegraphics[width=.47\linewidth]{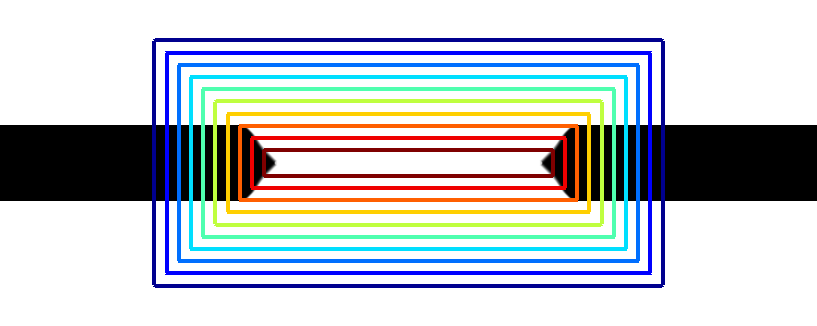}} & 
\subfloat[Guidefill with smart pixel ordering is able to make a successful connection (M\"arz's adapted distance functions \cite{Marz2011} would also do the job).]{\includegraphics[width=.47\linewidth]{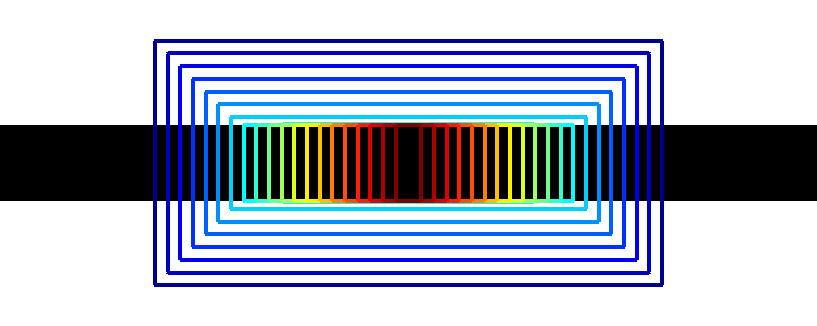}} \\
\end{tabular}
\begin{tabular}{ccc}
\subfloat[Guidefill's smart pixel ordering is not able to prevent a shock in this case because of incompatible boundary conditions.]{\includegraphics[width=.47\linewidth]{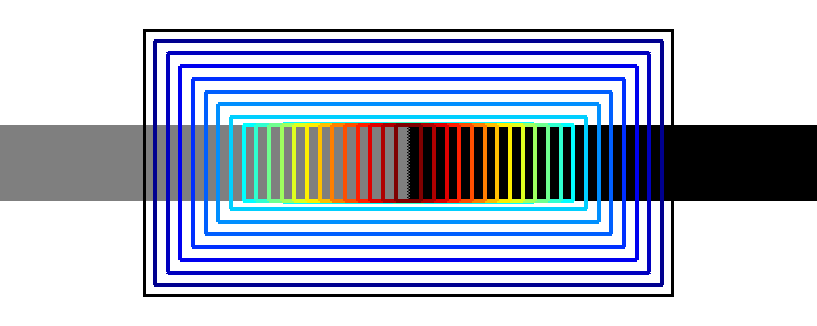}} &
\subfloat[Guidefill with onion shell ordering results in a cut off picture frame that ends abruptly.]{\includegraphics[width=.21\linewidth]{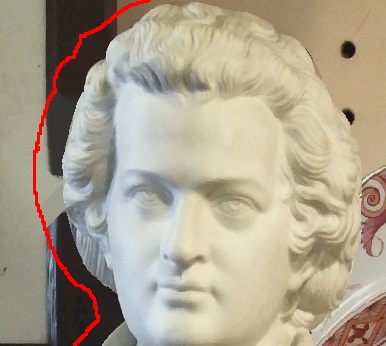}} &
\subfloat[Guidefill with smart ordering connects the picture frame, but creates a shock.]{\includegraphics[width=.21\linewidth]{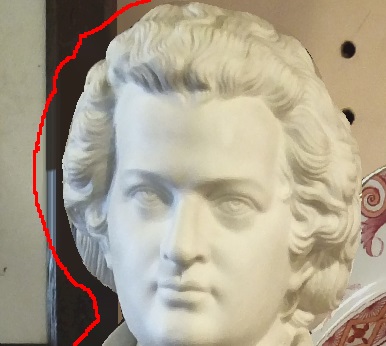}} \\
\end{tabular}
\caption{{\bf Cut off isophotes and shocks:} Because Algorithm 1 fills the inpainting domain from many directions at once, ``cut off isophotes'' or shocks can sometimes be formed.  In (a), this is due to the (superimposed) fill order, which is the default onion shell ordering and a bad choice in this case.  In (b), we have a chosen a new fill order better adapted to the image and the problem is solved in this case.  However, the shock in (c) is due to incompatible boundary conditions and it is unlikely any special fill order could solve the problem.  If (c) seems a little contrived, consider the ``real life'' examples (d)-(e).  In (d), we inpainted using Guidefill with the default onion shell ordering, resulting in the picture frame being cut off.  In (e) we used Guidefill's build in smart order, which successfully completes the picture frame, but creates a shock in the middle.  This shock is due to incompatible lighting conditions at either end of the inpainting domain, which is outlined in red.  }
\label{fig:cutoff}
\end{figure}

\begin{figure}
\centering
\begin{tabular}{ccc}
\subfloat[Inpainting problem with $D_h$ colored yellow and outlined in black.]{\includegraphics[width=.3\linewidth]{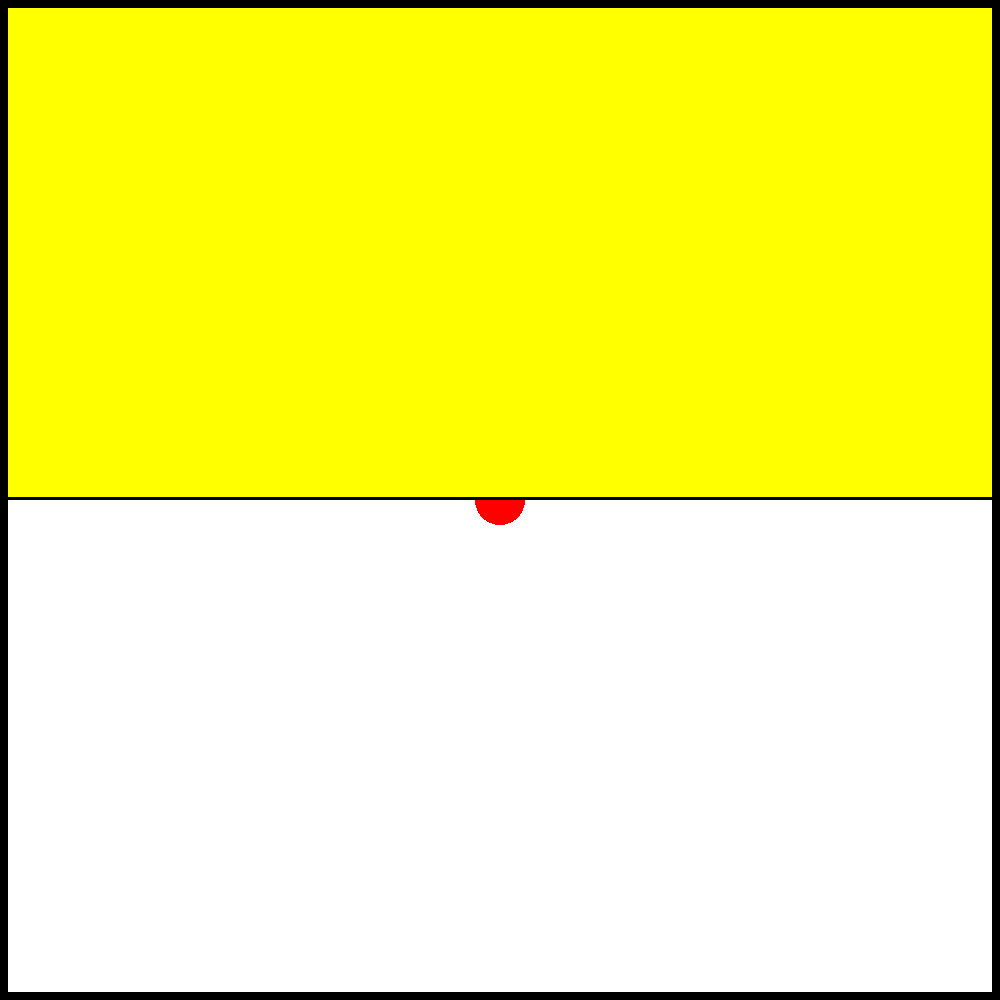}} & 
\subfloat[Superposition of multiple inpaintings of (a) using coherence transport with guidance direction ${\bf g}$ sweeping out an arc from $1^{\circ}$ up to $179^{\circ}$.  In this case $\epsilon = 3$px and $B_{\epsilon,h}({\bf 0})$ is superimposed.]{\includegraphics[width=.3\linewidth]{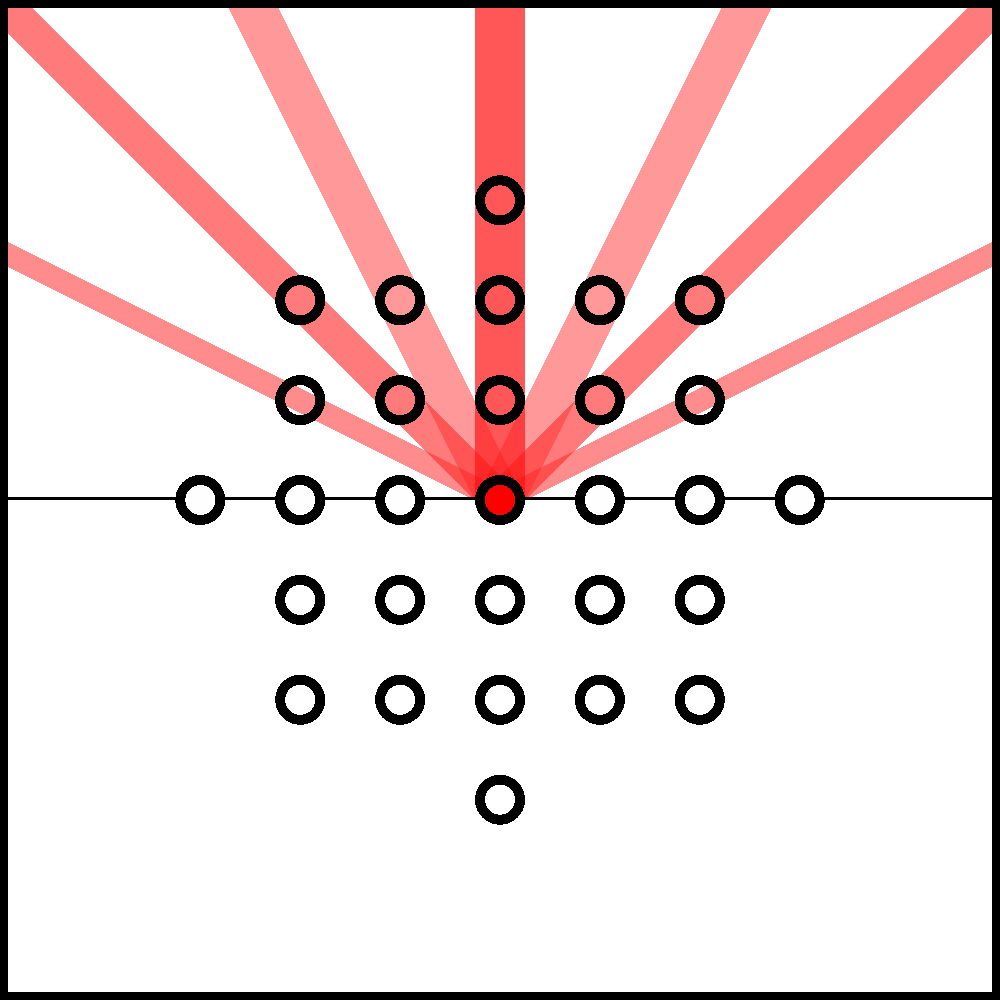}} &
\subfloat[Now Guidefill is used instead of coherence transport, but all parameters including $\epsilon=3$px and $\mu = 100$ are kept the same.  This time we superimpose the dilated ball $D_h(B_{\epsilon,h}({\bf 0}))$.  New points are shown in grey.]{\includegraphics[width=.3\linewidth]{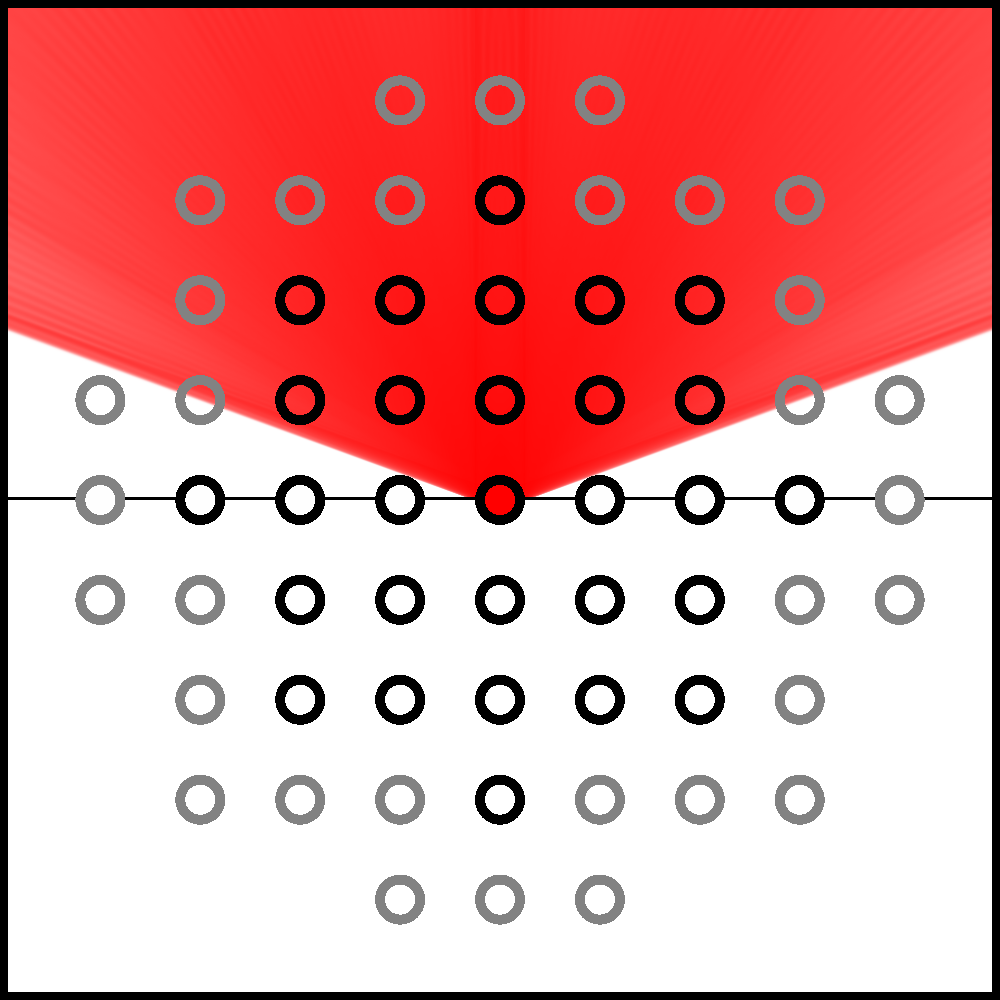}} \\
\end{tabular}
\caption{{\bf Special directions:}  For a given guidance direction ${\bf g}=(\cos\theta,\sin\theta)$, coherence transport \cite{Marz2007,Marz2011} can successfully extrapolate isophotes parallel to ${\bf g}$ only if ${\bf g}=\lambda {\bf v}$, for some ${\bf v} \in B_{\epsilon,h}({\bf 0})$.  This is illustrated in (b), where have solved the inpainting problem posed in (a) multiple times using coherence transport with $\epsilon = 3$px with a sequence of guidance directions ${\bf g}_k =(\cos\theta_k,\sin\theta_k)$ (with $\theta_k$ ranging from $1^{\circ}$ up to $179^{\circ}$ in increments of one degree), combined the results (specifically, we took a weighted average of all 179 solutions in which the weight of a given pixel decreases exponentially as its color moves away from pure red), and superimposed $B_{\epsilon,h}({\bf 0})$ (the parameter $\mu$ in \eqref{eqn:weight} is $\mu = 100$).  Instead of a smoothly varying line sweeping through the upper half plane and filling it with red, we see a superposition of finitely many lines, each passing through some ${\bf v} \in B_{\epsilon,h}({\bf 0})$.  When we repeat the experiment in (c) using Guidefill \cite{Guidefill}, we see that it is not free of problems either.  In this case Guidefill can extrapolate along ${\bf g}=(\cos\theta,\sin\theta)$ so long as $0<\theta_c \leq \theta \leq \pi-\theta_c < \pi$, where $\theta_c$ is a critical angle, and we get a red cone bounded on either side by $\theta_c$.  Here we have superimposed the dilated ball $D_h(B_{\epsilon,h}({\bf 0}))$, and it is evident that $\theta_c$ is in some way related to this dilation - this will be explained in Section \ref{sec:kinkingAndConvex}.}
\label{fig:specialDir}
\end{figure}

\begin{figure}
\centering
\begin{tabular}{cccc}
\subfloat[Inpainting problem with  $\theta = 63^{\circ}$.]{\includegraphics[width=.2\linewidth]{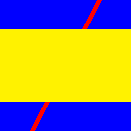}} & 
\subfloat[Telea's algorithm.]{\includegraphics[width=.2\linewidth]{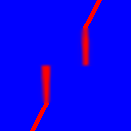}} &
\subfloat[coherence transport.]{\includegraphics[width=.2\linewidth]{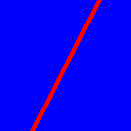}} &
\subfloat[Guidefill.]{\includegraphics[width=.2\linewidth]{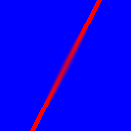}} \\
\subfloat[Inpainting problem with  $\theta = 73^{\circ}$.]{\includegraphics[width=.2\linewidth]{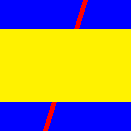}} & 
\subfloat[Telea's algorithm.]{\includegraphics[width=.2\linewidth]{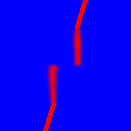}} &
\subfloat[coherence transport.]{\includegraphics[width=.2\linewidth]{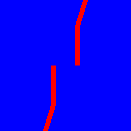}} &
\subfloat[Guidefill.]{\includegraphics[width=.2\linewidth]{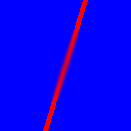}} \\
\end{tabular}
\begin{tabular}{cc}
\subfloat[Midpoint cross-sections for $\theta = 63^{\circ}$.]{\includegraphics[width=.45\linewidth]{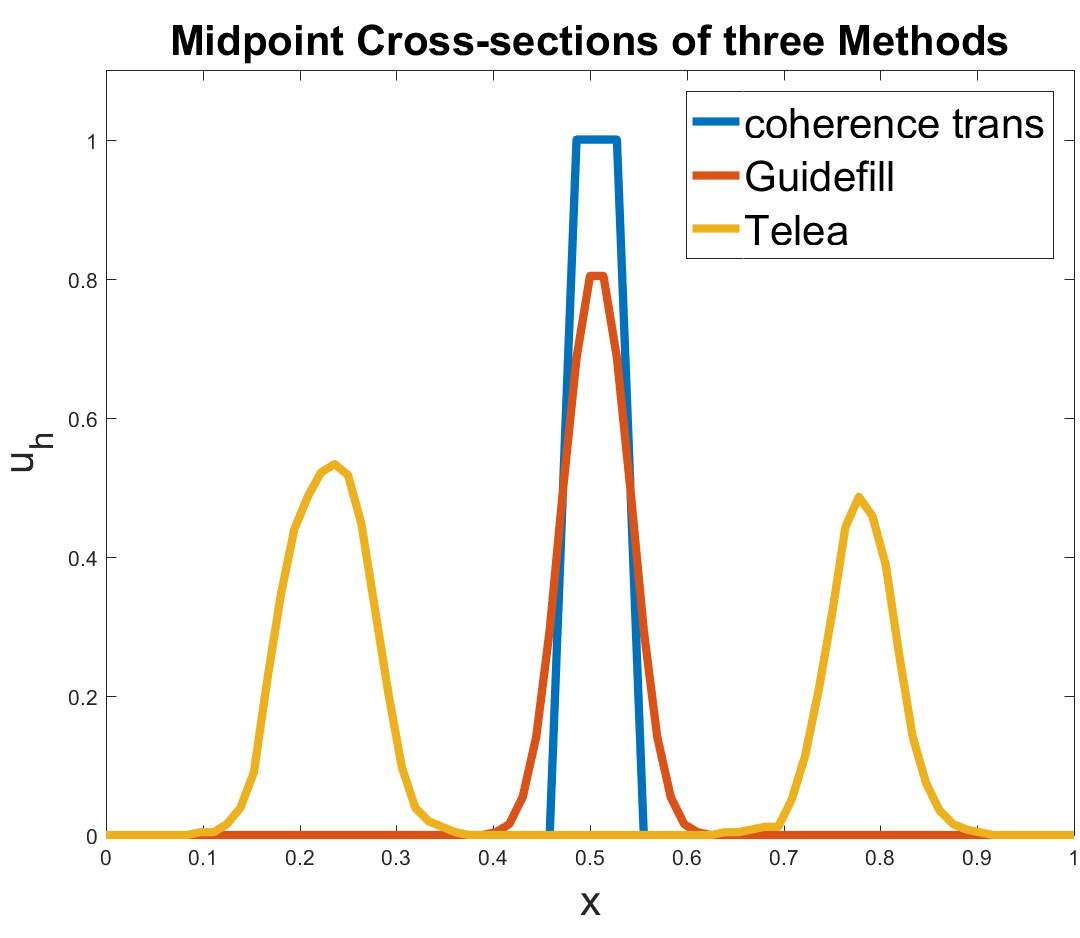}} &
\subfloat[Midpoint cross-sections for $\theta = 73^{\circ}$.]{\includegraphics[width=.45\linewidth]{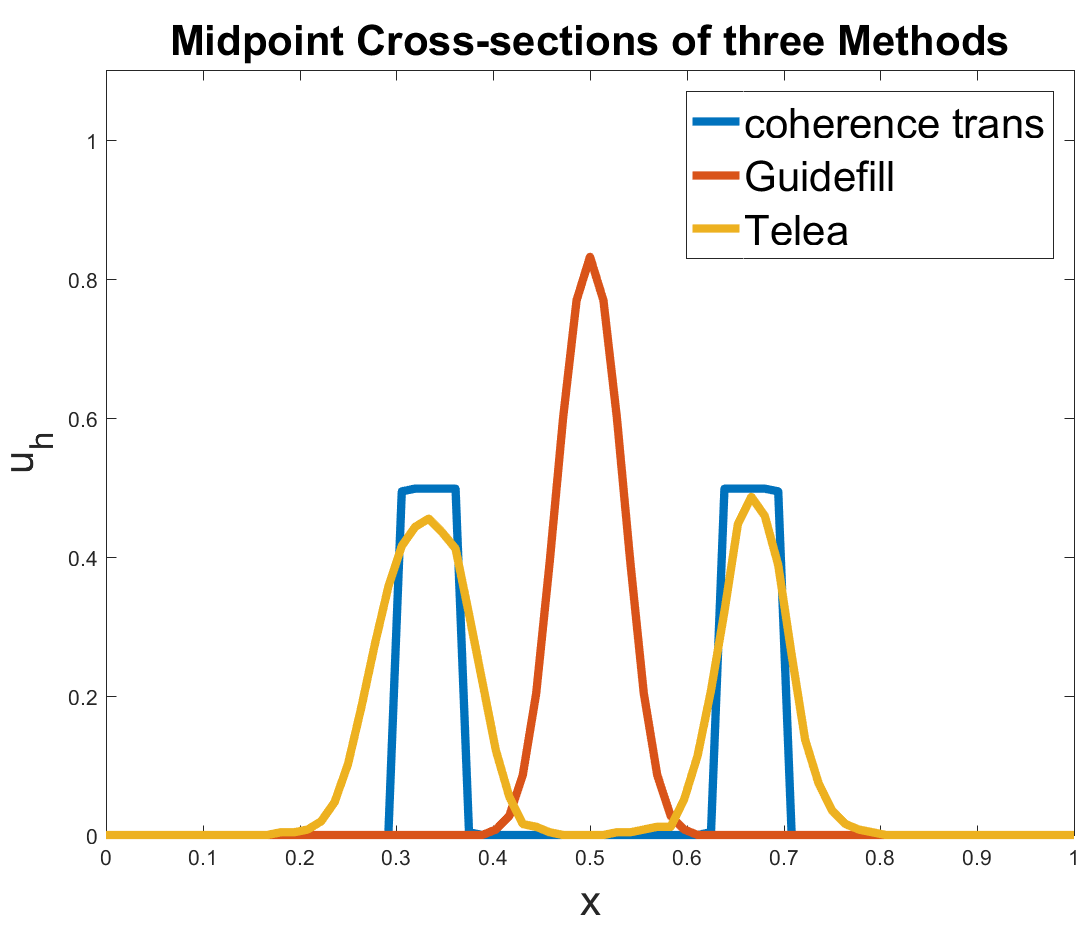}} \\
\end{tabular}
\caption{{\bf A tale of two inpainting problems:}  In (a)-(d), a line making an angle of $\theta = 63^{\circ}$ with the horizontol is inpainted using each of Telea's algorithm \cite{Telea2004}, coherence transport, \cite{Marz2007,Marz2011}, and Guidefill \cite{Guidefill} (the inpainting domain is shown in yellow).  In this case the radius of $A_{\epsilon,h}({\bf x})$ is $\epsilon = 3$px, and since $63^{\circ} \approx \arctan(2) \approx 63.44^{\circ}$ is close to one of the ``special directions'' in which coherence transport can extend isophotes successfully for this value of $\epsilon$ (see Figure \ref{fig:specialDir}), both coherence transport and Guidefill make a successful connection.  In (e)-(h) we change the angle of the line slightly to $\theta = 73^{\circ}$.  This isn't one of coherence transport's admissable directions for $\epsilon = 3$px, so it fails to make the connection, while Guidefill continues to have no problems, at the expense of some blurring.  Telea's algorithm, on the other hand, propagates in the direction of the normal to the inpainting domain boundary regardless of the undamaged image content, and thus fails to make the connection in both cases while also introducing significant blur.  In (i)-(j), we examine horizontal cross sections (of the red channel) of all three methods at the midpoint of the inpainting domain.  Here, a disadvantage of Guidefill in terms of blur becomes more apparent - coherence transport by contrast produces a much sharper result.  The reasons for this are explored in Section \ref{sec:blur}.
}
\label{fig:taleOfTwo}
\end{figure}

\begin{figure}
\centering
\begin{tabular}{ccc}
\subfloat[Inpainting problem (inpainting domain in yellow).]{\includegraphics[width=.2\linewidth]{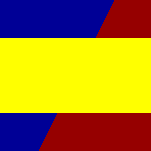}} & 
\subfloat[Inpainting with Telea's algorithm.]{\includegraphics[width=.2\linewidth]{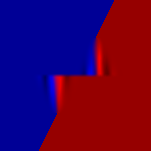}} &
\subfloat[Inpainting with coherence transport.]{\includegraphics[width=.2\linewidth]{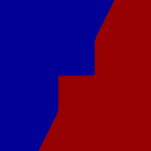}} \\
\end{tabular}
\begin{tabular}{cc}
\subfloat[Red channel cross-section for Telea's algorithm.]{\includegraphics[width=.45\linewidth]{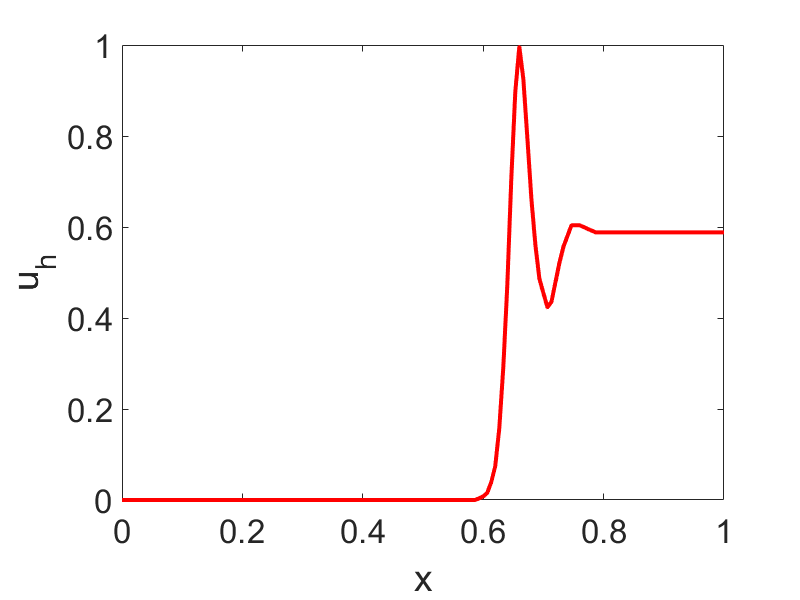}} &
\subfloat[Red channel cross-section for coherence transport.]{\includegraphics[width=.45\linewidth]{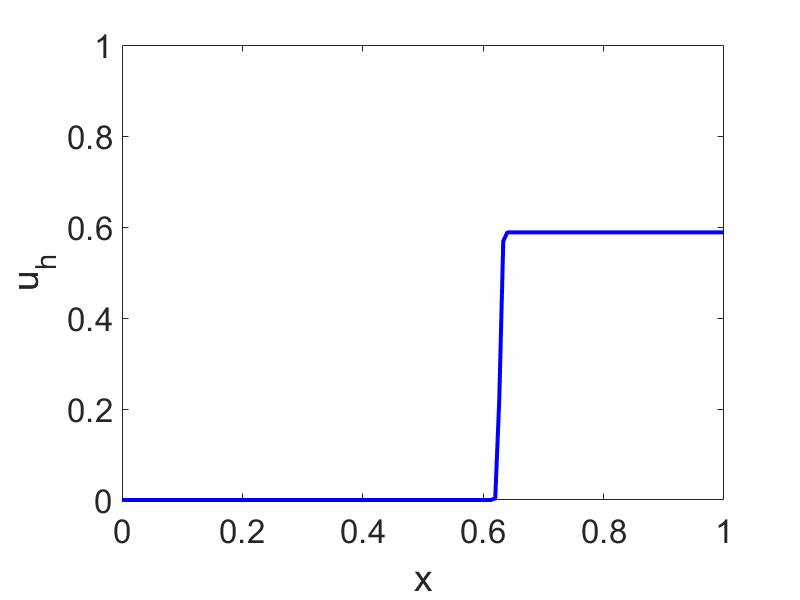}} \\
\end{tabular}
\caption{{\bf Bright spots in Telea's algorithm:}  In this example we consider the inpainting problem shown in (a) consisting of a line separating a region of dark blue from a region of dark red.  We inpaint both using Telea's algorithm (b) and coherence transport (c).  Coherence transport obeys the stability property \eqref{eqn:stability} and hence the brightness of the inpainted solution remains bounded above by the brightness on the exterior of the inpainting domain.  This is not true of Telea's algorithm, which exhibits bright spots outside the the original color range.  These were not visible in Figure \ref{fig:taleOfTwo}, because the brightness of each color channel was already saturated, and Telea's algorithm uses clamping to prevent the solution from going outside the admissible color range.  This is further illustrated in (d)-(e), where we plot horizontal cross sections of the red channel of each inpainted solution.  These slices are located ten rows of pixels above the midpoint of the inpainting domain.
}
\label{fig:stability}
\end{figure}

\begin{figure}
\centering
\begin{tabular}{ccc}
\subfloat[Damaged image with inpainting domain in red.]{\includegraphics[width=.3\linewidth]{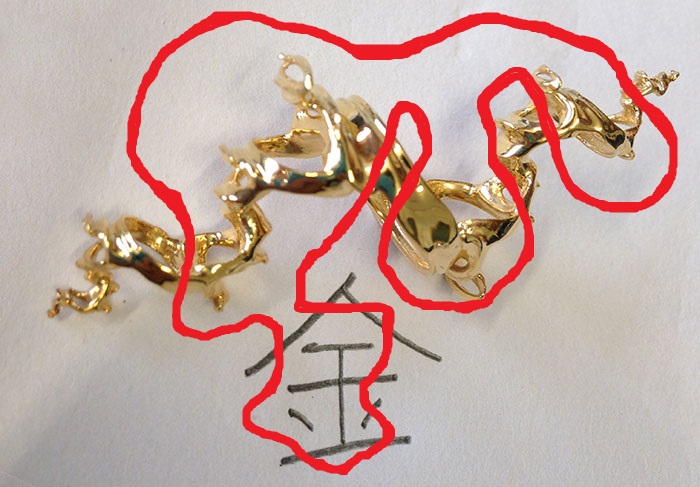}} & 
\subfloat[The result of Telea's algorithm \cite{Telea2004}.]{\includegraphics[width=.3\linewidth]{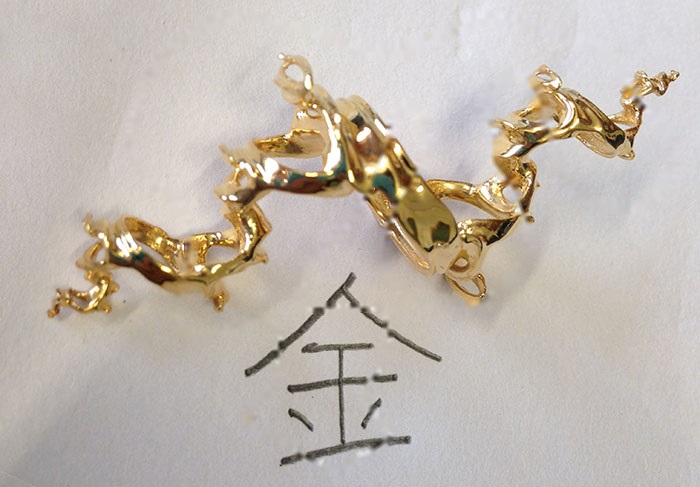}} &
\subfloat[Closeup of (b).  Note the bright spots and disconnected isophotes.]{\includegraphics[width=.21\linewidth]{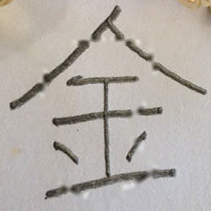}} \\
\end{tabular}
\begin{tabular}{ccc}
\subfloat[Further closeup of (b), with blurry, disconnected isophotes circled in red and a bright spot circled in blue.]{\includegraphics[width=.24\linewidth]{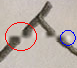}} &
\subfloat[The result of coherence transport \cite{Marz2007}.]{\includegraphics[width=.3\linewidth]{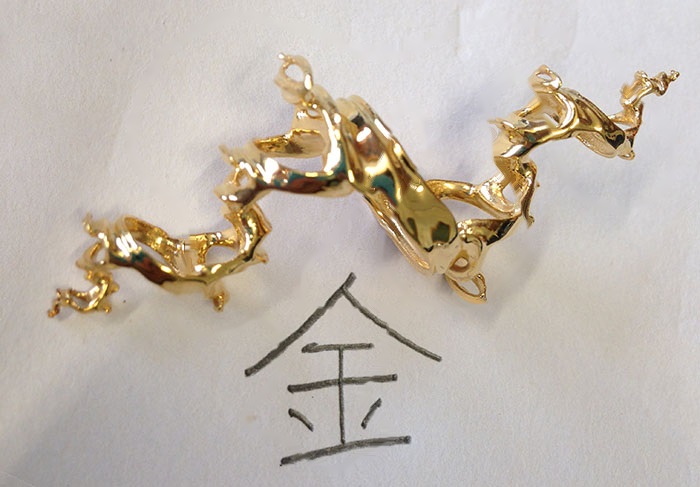}} &
\subfloat[Closeup of (d) - note the better reconstruction of 金.]{\includegraphics[width=.21\linewidth]{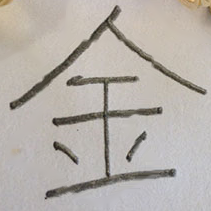}} \\
\end{tabular}
\caption{{\bf Blurring, kinking, and bright spots with Telea's algorithm:}  Even for problems such as (a) where the inpainting domain is very thin, Telea's algorithm (b)-(d) still creates strong blurring artifacts and fails to connect isophotes effectively.  Also, due to the presence of the gradient term in \eqref{eqn:Telea}, Telea's algorithm violates the stability condition \eqref{eqn:stability} and as a result can ``overshoot'' when filling pixels close to edges in the filled area, where the (numerical) gradient changes rapidly.  This leads to the bright spots near the reconstructed character in (c)-(d).  In this case coherence transport (e)-(f) is a much better choice. }
\label{fig:artifacts2}
\end{figure}


\subsection{Related work (artifact reduction)}  
Broadly speaking, there has been incremental progress as follows:  Telea's algorithm \cite{Telea2004}, the earliest variant to appear in the literature, suffers from strong artifacts of every type.  In particular, the weights make no attempt to take into account the orientation of undamaged isophotes in $\Omega_h \backslash D_h$, and the result shows strong kinking artifacts (see Figure \ref{fig:taleOfTwo}).  Bornemann and M\"arz  identified and sought to address this problem with coherence transport \cite{Marz2007}, which proposed carefully chosen weights that are proven (in a high resolution and vanishing viscosity limit) to extend isophotes in any desired guidance direction ${\bf g}$ not parallel to the inpainting domain boundary.  This was combined with a method aimed at robustly measuring the orientation of isophotes at the boundary, so that a suitable ${\bf g}$ allowing for a seamless transition could be found.  The problem of ``kinking'' ostensibly resolved, in a follow up work M\"arz proposed coherence transport with adapted distance functions \cite{Marz2011} designed to minimize the problem of ``cut off'' isophotes and shocks.  This was accomplished by recognizing that artifacts such as the incomplete line in Figure \ref{fig:cutoff}(a) are often the byproduct of a suboptimal fill order such as the one superimposed (in this case the default onion shell ordering).  The situation can often be corrected as in Figure \ref{fig:cutoff}(b), by using an ordering better adapted to the image such as the one illustrated there.  Rather than filling pixels in an order proportional to their distance from the boundary, i.e. having the $\mbox{ready}$ function in Algorithm 1 always return ``true'', M\"arz proposed a number of ways of generating improved orderings based on non-Euclidean distance from boundary maps.  At the same time, recognizing that the presence of shocks was related to the ``stopping set'' \cite{Marz2011} of the distance map, M\"arz was able to exert some measure of control over those as well, if not prevent them entirely.  Guidefill \cite{Guidefill} brought the focus back to the reduction of kinking artifacts, by noting that coherence transport is actually only able to propagate along a given guidance vector ${\bf g}$ if it points in one of a finite set of special directions - see Figure \ref{fig:specialDir}(b).  Whereas previous improvements to Algorithm 1 had focused first on improving the choice of weights, then the fill order (equivalently the choice of $\mbox{ready}$ function), Guidefill proposed for the first time to change the averaging neighborhood $A_{\epsilon,h}({\bf x})$, which until now had always been the discrete ball $B_{\epsilon,h}({\bf x})$ (Figure \ref{fig:ballRotate}(a)).  Specifically, it proposed to replace $A_{\epsilon,h}({\bf x})=B_{\epsilon,h}({\bf x})$ with $A_{\epsilon,h}({\bf x})=\tilde{B}_{\epsilon,h}({\bf x})$, where $\tilde{B}_{\epsilon,h}({\bf x})$ is the {\em rotated} discrete ball shown in Figure \ref{fig:ballRotate}(b), aligned with the guidance direction ${\bf g}$.  Since $A_{\epsilon,h}({\bf x})$ is in this case no longer axis aligned, it contains what the authors called ``ghost pixels'' lying between pixel centers, which they defined based on bilinear interpolation.  This small change enabled Guidefill to propagate along most guidance directions, but it too has problems when the angle between ${\bf g}$ and the boundary to the inpainting domain is too shallow - see Figure \ref{fig:specialDir}(c).  However, Guidefill pays a price for its reduction in kinking artifacts in the form of an increase in blur artifacts.  See Figure \ref{fig:taleOfTwo}, where coherence transport produces a sharp extension of image isophotes, albeit possibly in the wrong direction, whereas Guidefill extrapolates in the right direction, but the extrapolation suffers from blur.  Guidefill also proposed its own ``smart order'' computed on the fly as an alternative to M\"arz's adapted distance functions, but this does not have any serious advantage in terms of the quality of the results.  Either approach will do for preventing ``cut off'' isophotes.

\subsection{Related theoretical work}  The direct form of Algorithm 1 has been studied from a theoretical point of view by proposing two distinct continuum limits.  The first of these is the high resolution and vanishing viscosity limit proposed by Bornemann and M\"arz, in which $h \rightarrow 0 $ and then $\epsilon \rightarrow 0$ \cite{Marz2007}.  The second is a fixed-ratio limit proposed in our previous work \cite{Guidefill} in which $(h,\epsilon) \rightarrow (0,0)$ along the line $\epsilon = r h$.  The non-negative integer $r$ is simply the radius of $A_{\epsilon,h}({\bf x})$ measured in pixels and will play a key role in our analysis.  Although both are perfectly valid mathematically, we will argue (see Section \ref{sec:marzLimit} and in particular Remark \ref{rem:marzLimit}) that this second limit gives a better indication of the behaviour of Algorithm 1 in practice.  This is supported by numerical evidence both in \cite{Guidefill} and in this paper.  There has also been significant work in studying the well-posedness of the high resolution and vanishing viscosity limit of Algorithm 1, both in \cite{Marz2007} and especially in \cite{Marz2015}.  See Figure \ref{fig:limits} for an illustration of these two separate limits.

\begin{figure}
\centering
\begin{tabular}{cccc}
\subfloat[$\epsilon=1$, $h=\frac{1}{4}$.]{\includegraphics[width=.22\linewidth]{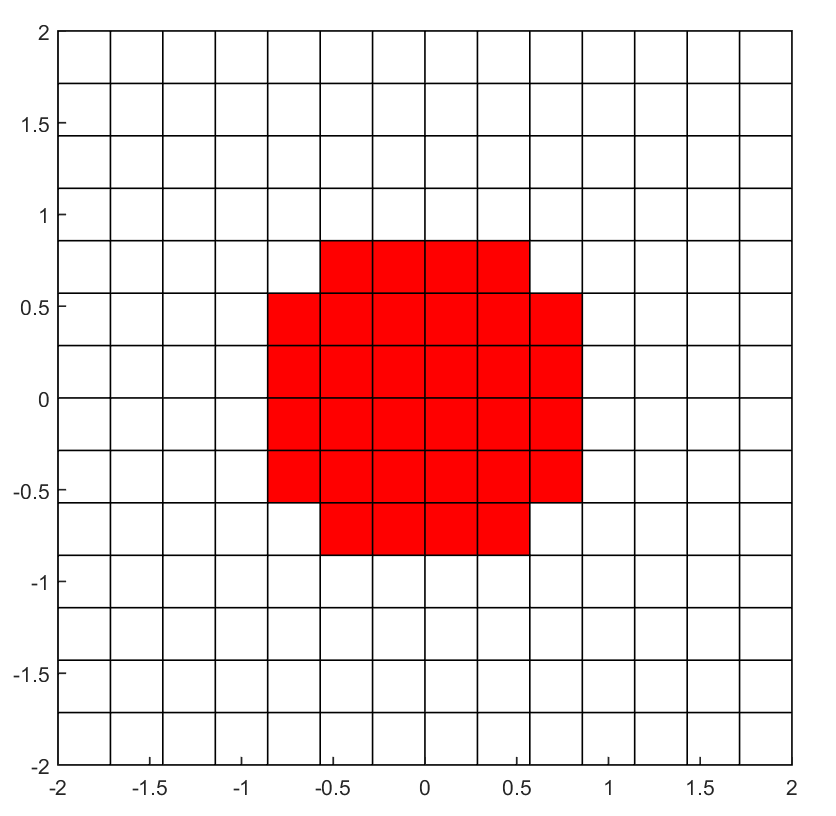}} & 
\subfloat[$\epsilon=1$, $h=\frac{1}{8}$.]{\includegraphics[width=.22\linewidth]{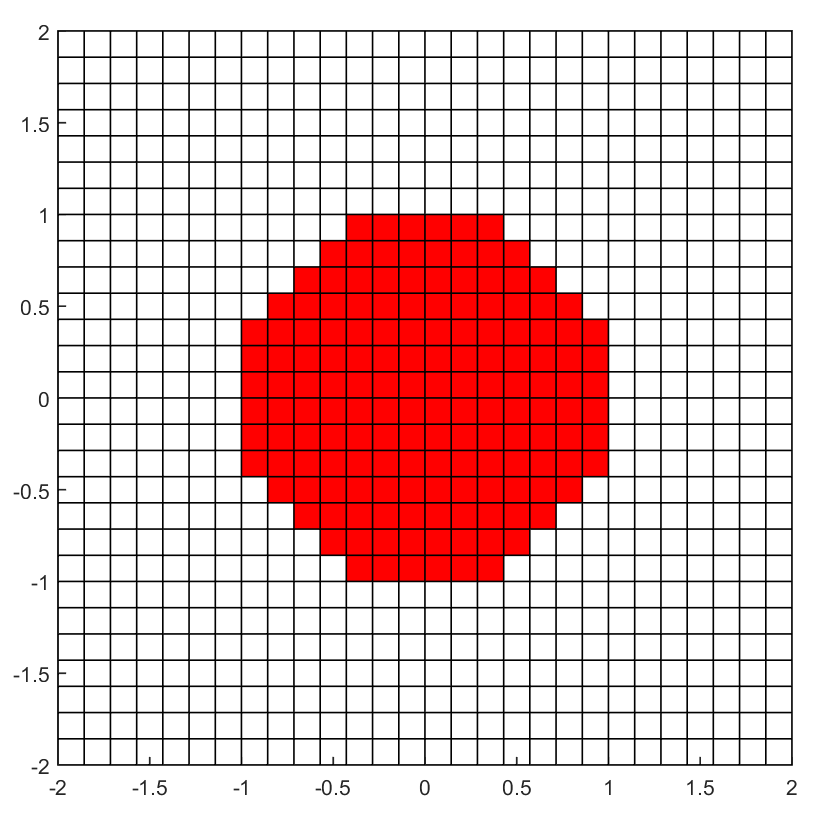}} &
\subfloat[$\epsilon=1$, $h=\frac{1}{16}$.]{\includegraphics[width=.22\linewidth]{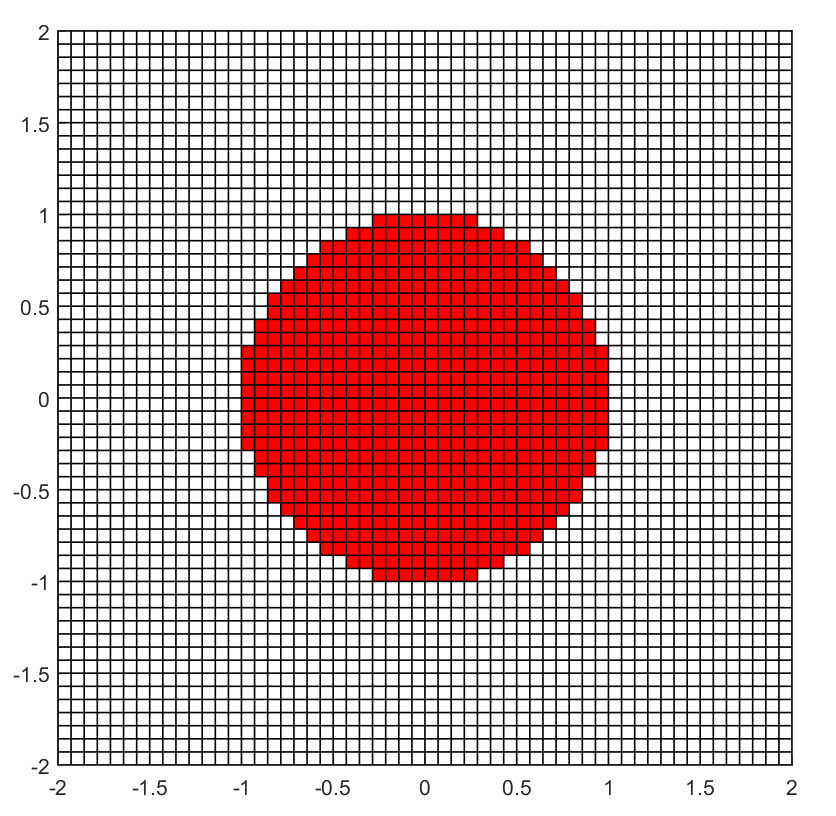}} &
\subfloat[$\epsilon=1$, $h=\frac{1}{32}$.]{\includegraphics[width=.22\linewidth]{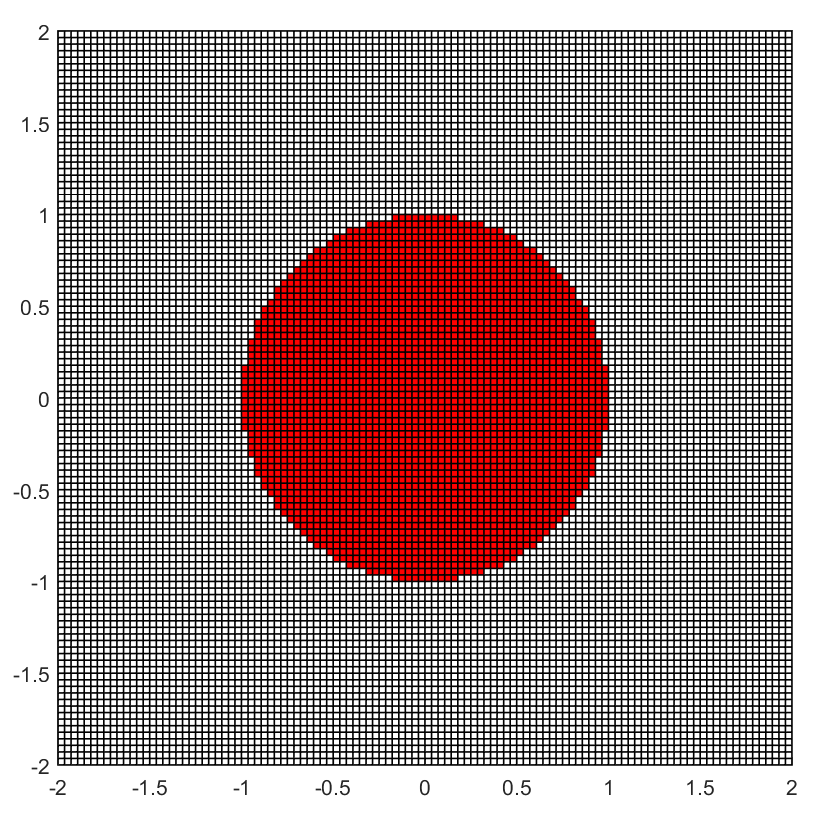}}\\
\subfloat[$\epsilon=1$, $h=0$.]{\includegraphics[width=.22\linewidth]{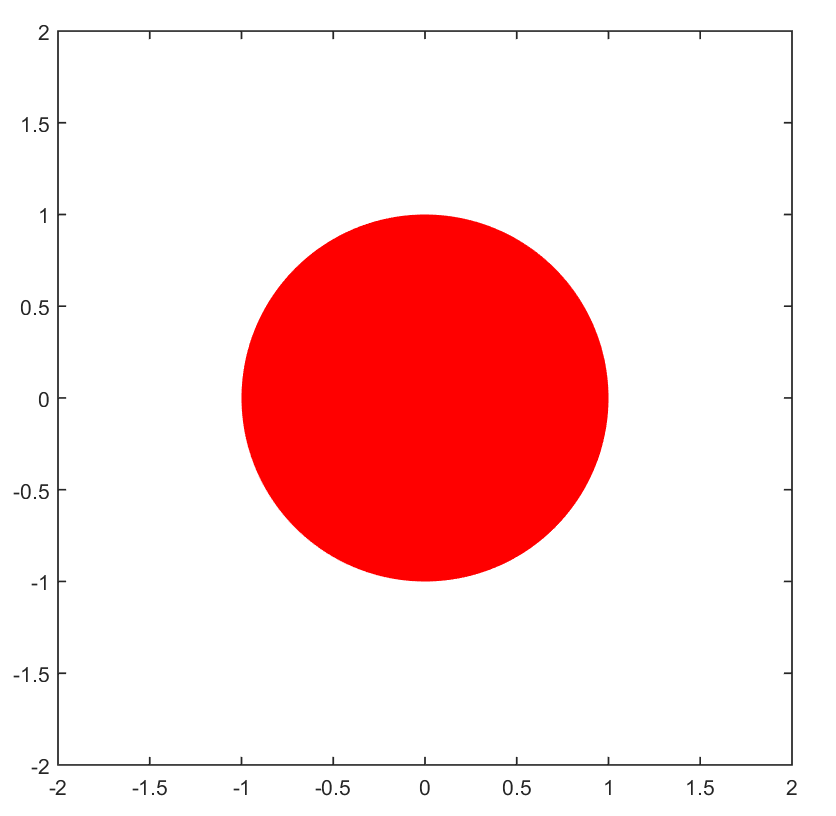}} & 
\subfloat[$\epsilon=\frac{1}{2}$, $h=0$.]{\includegraphics[width=.22\linewidth]{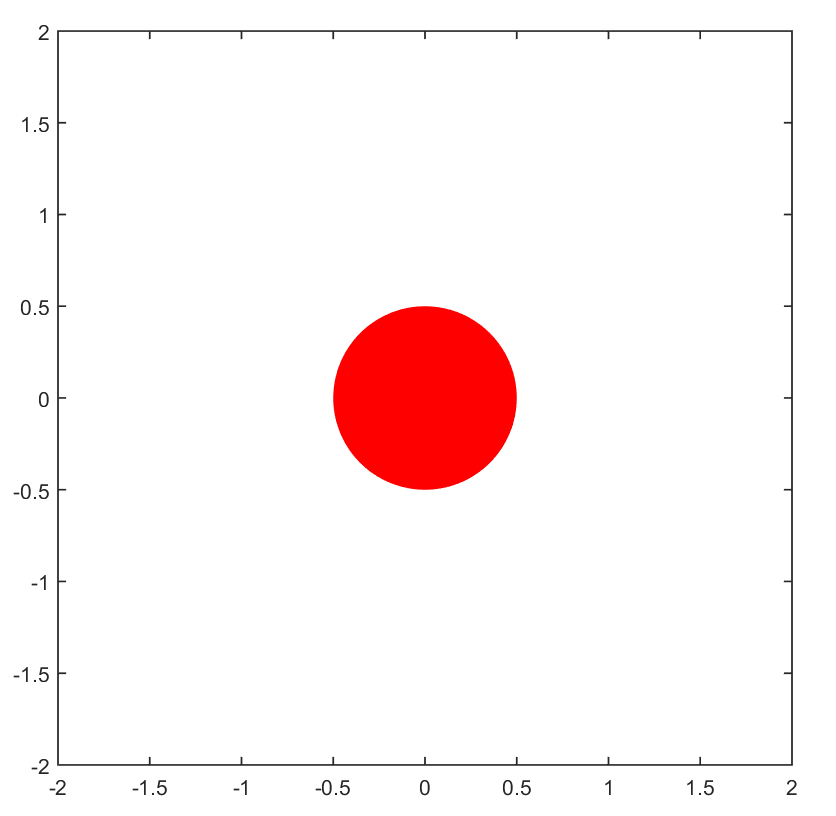}} &
\subfloat[$\epsilon=\frac{1}{4}$, $h=0$.]{\includegraphics[width=.22\linewidth]{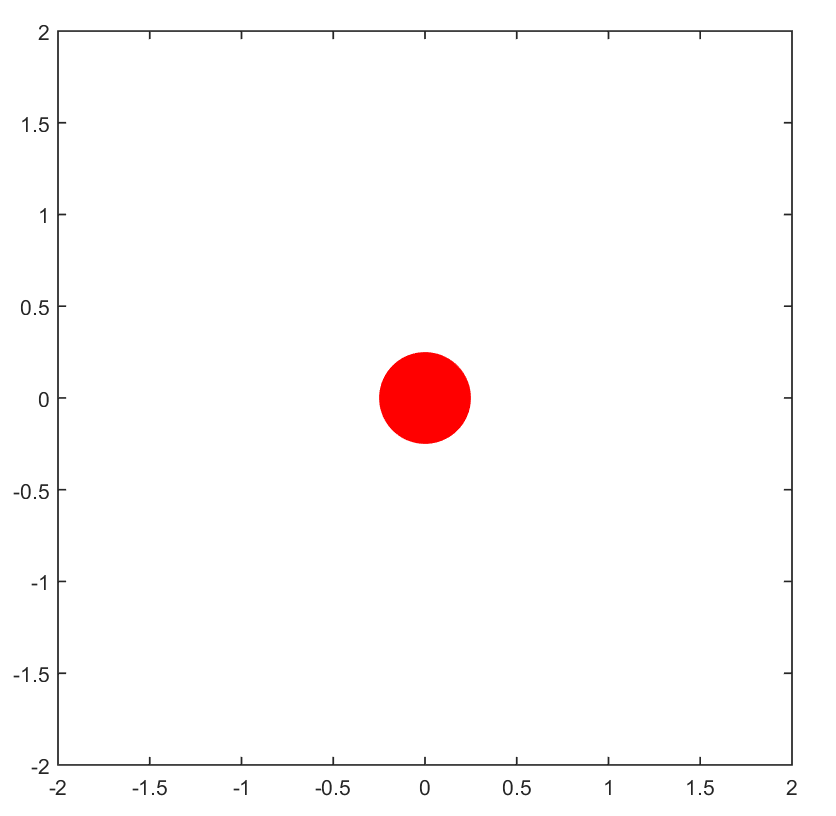}} &
\subfloat[$\epsilon=\frac{1}{8}$, $h=0$.]{\includegraphics[width=.22\linewidth]{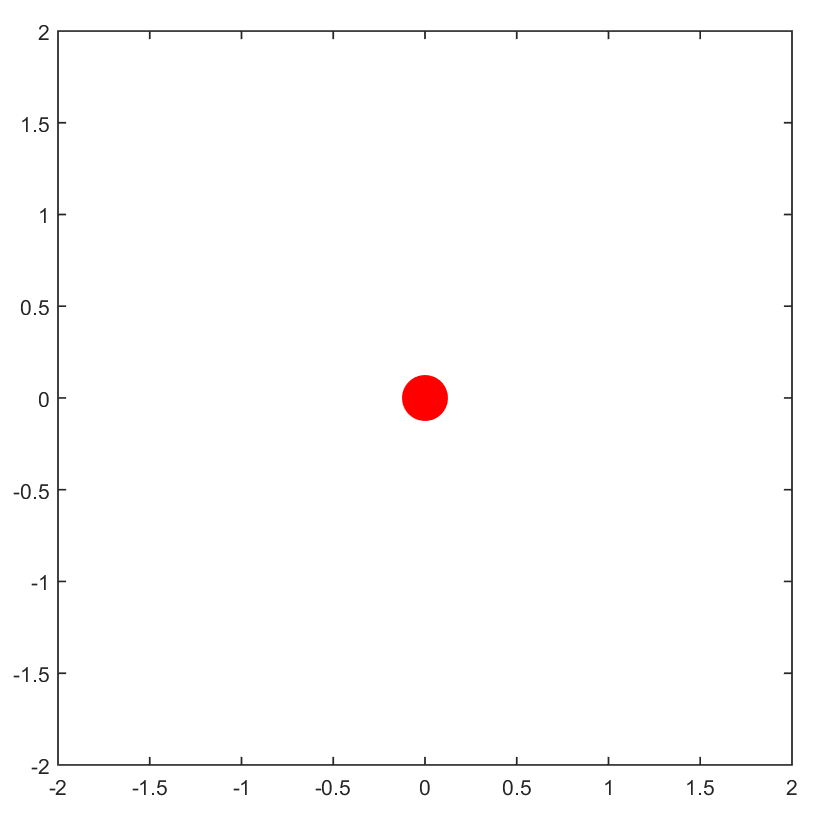}}\\
\subfloat[$\epsilon=1$, $h=\frac{1}{4}$.]{\includegraphics[width=.22\linewidth]{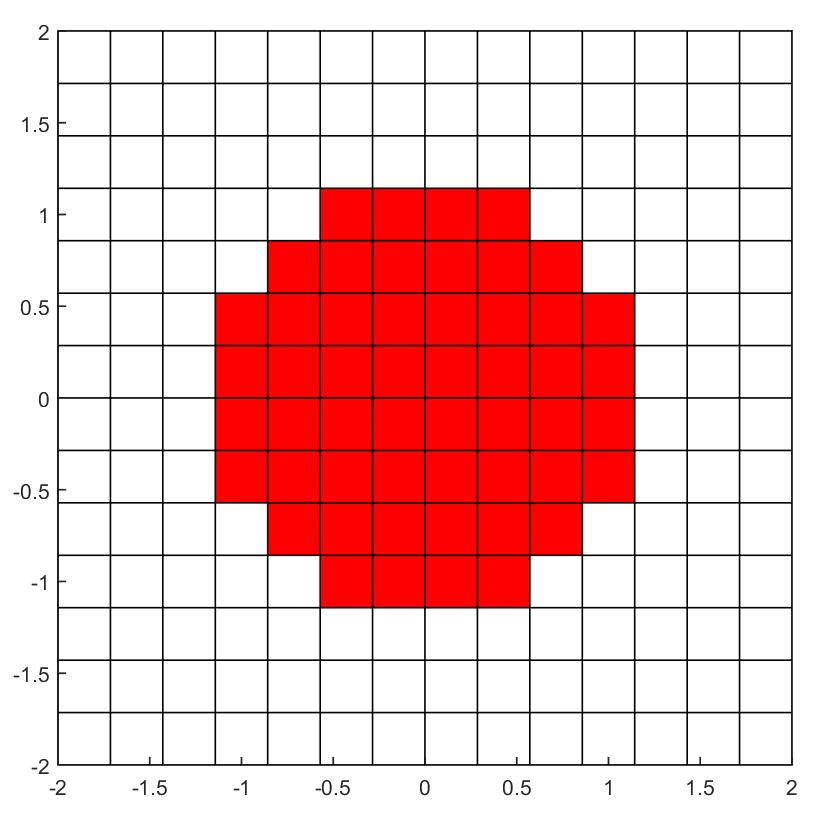}} & 
\subfloat[$\epsilon=\frac{1}{2}$, $h=\frac{1}{8}$.]{\includegraphics[width=.22\linewidth]{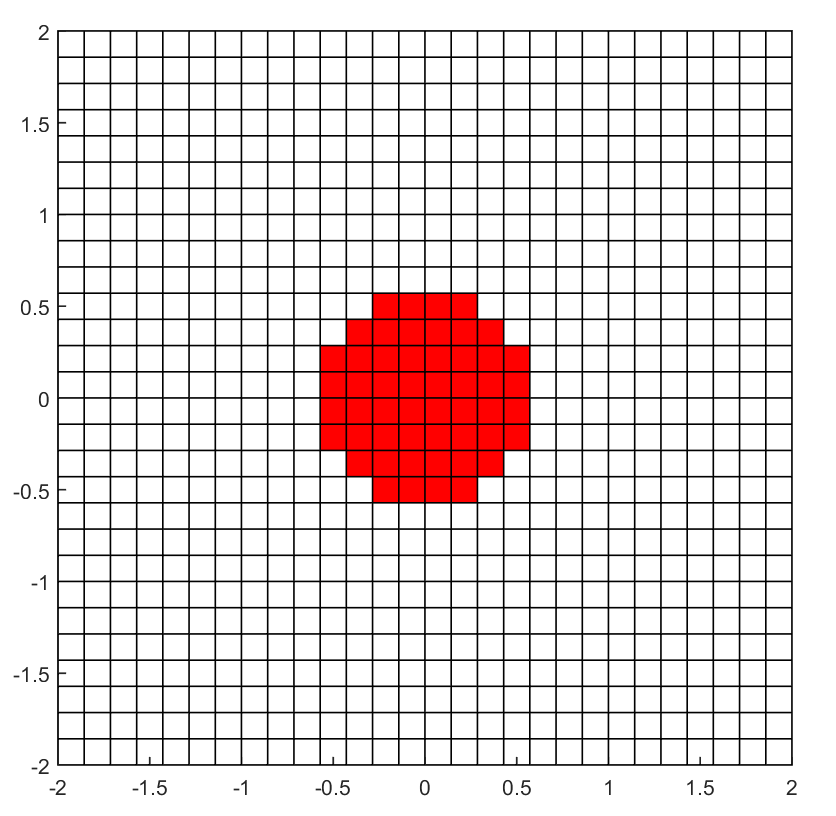}} &
\subfloat[$\epsilon=\frac{1}{4}$, $h=\frac{1}{16}$.]{\includegraphics[width=.22\linewidth]{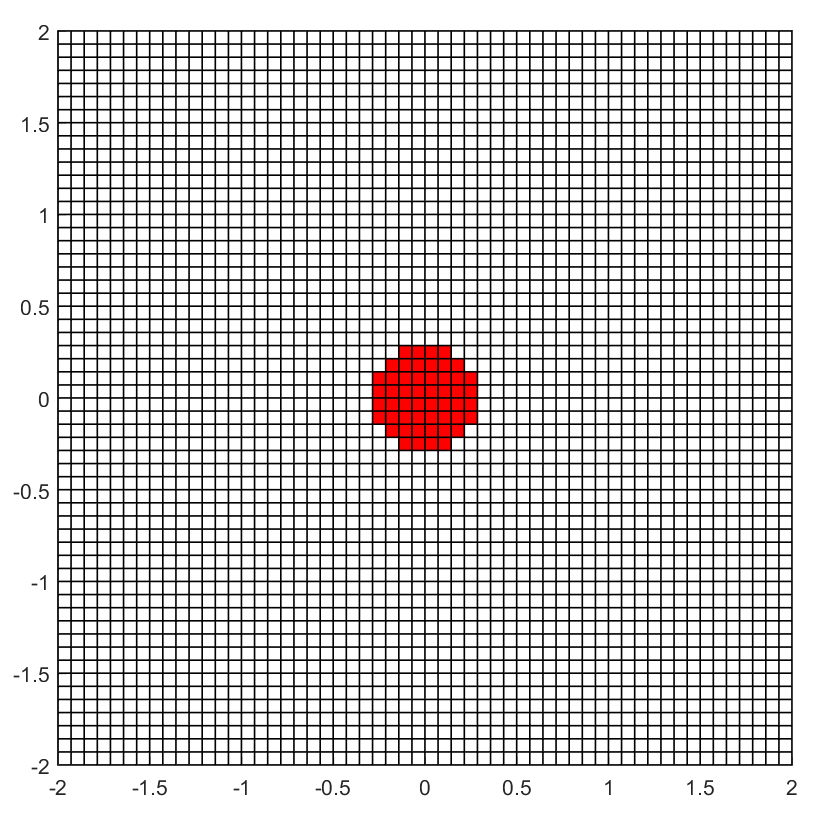}} &
\subfloat[$\epsilon=\frac{1}{8}$, $h=\frac{1}{32}$.]{\includegraphics[width=.22\linewidth]{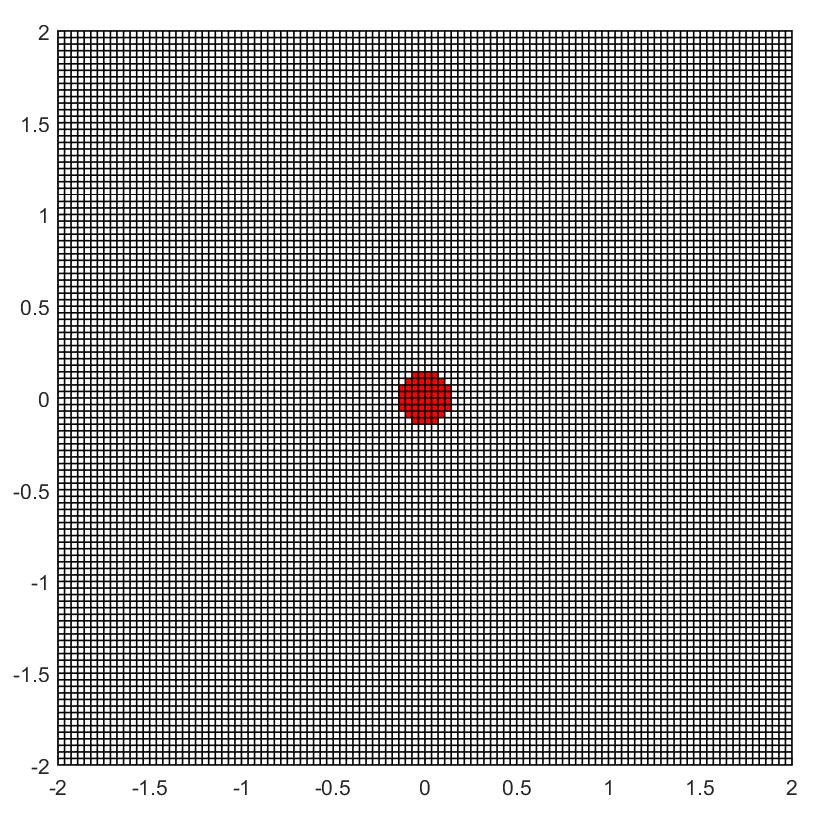}}\\
\end{tabular}
\caption{{\bf Two distinct continuum limits:}  In this paper we study two separate continuum limits of Algorithm 1, illustrated here for $A_{\epsilon,h}({\bf x})=B_{\epsilon,h}({\bf x})$.  The first, illustrated in (a)-(h), is the high-resolution vanishing viscosity double limit proposed by Bornemann and M\"arz \cite{Marz2007}, in which $h \rightarrow 0 $ (a)-(d) and then $\epsilon \rightarrow 0$ (e)-(h).  The second is the fixed-ratio limit single limit $(\epsilon,h) \rightarrow (0,0)$ with $r=\frac{\epsilon}{h}$ fixed proposed in our previous work \cite{Guidefill}, illustrated in (i)-(l) for $r=4$.  We will prove that while they are both valid limits of Algorithm 1, they predict very different behaviour.  Moreover, we will see that the predictions of the latter correspond much more closely than the former to the actual behaviour of Algorithm 1, especially when $r$ is small (which it typically is in applications - \cite{Marz2007} recommends $r$ between $3$ and $5$, for example).  See Theorems \ref{thm:convergence} and \ref{thm:convergence2}, as well as Sections \ref{sec:kink3main} and \ref{sec:MarzIsFar}.}
\label{fig:limits}
\end{figure}

\vskip 2mm

\noindent {\bf Motivation for ghost pixels.}  In Section \ref{sec:fiction} we will prove that any weighted sum over a set $A_{\epsilon,h}({\bf x})$ of ghost pixels is equivalent to a sum over the real pixels in $\mbox{Supp}(A_{\epsilon,h}({\bf x}))$ with equivalent weights.  While this makes ghost pixels in some sense redundant, they are useful concept.  Specifically, in Theorem \ref{thm:convergence} we will prove that the fixed ratio continuum limit described above and illustrated in Figure \ref{fig:limits} is a partial differential equation with coefficients that depend continuously on the weights $w_{\epsilon}$ and on the elements of $A_{\epsilon,h}({\bf x})$.  It will be desirable to control this limit  by making suitable choices for $w_{\epsilon}$ and $A_{\epsilon,h}({\bf x})$, and this is easier if the elements of the latter may be varied continuously.

\subsection{Contributions}  Our contributions are both theoretical and practical, aimed at a deeper understanding of the mathematical properties of both the direct and semi-implicit versions of Algorithm 1, and through that understanding, a better grasp of the underlying causes of the artifacts described above and what, if anything, can be done about them.  Our main targets are ``kinking'' and ``blurring'' artifacts, the others having already been thoroughly analyzed in \cite{Marz2007,Marz2015} and well understood.  Firstly, we establish convergence in $L^p$ for $p \in [1,\infty]$ of $u_h$ produced by either form of Algorithm 1 to the fixed-ratio continuum limit described above and illustrated in Figure \ref{fig:limits}, and obtain tight bounds on the rate of convergence, albeit under strong simplifying assumptions including a rectangular inpainting domain.  This is a step overlooked in the analysis of Bornemann and M\"arz \cite{Marz2007}, who take for granted that $u_h$ convergences as $h \rightarrow 0$, and perform all of their analysis in the continuum (they are, however, able to tackle more general inpainting domains as a result).  Moreover, unlike the analysis in \cite{Marz2007} and in our previous work \cite{Guidefill}, which both made strong smoothness assumptions about the given image $u_0$, our analysis makes very weak assumptions regarding the regularity of $u_0$.  We are, in particular, able to prove convergence when $u_0$ has jump discontinuities, is nowhere differentiable, or both.  Second, we analyze mathematically the effect of ``ghost pixels'' in $A_{\epsilon,h}({\bf x})$ lying between pixel centers on the resulting continuum limit.  This is something we did not do in our earlier work \cite{Guidefill}, but which is required to understand Guidefill's relative success in avoiding ``kinking artifacts''.  Third, we aim to understand the phenomena of ``kinking'' from broader perspective - we prove that not only Guidefill, but {\em any} variant of the direct version of Algorithm 1 {\em must} create ``kinking'' artifacts when the angle between the guidance direction ${\bf g}$ and the tangent to the inpainting domain boundary becomes smaller than a universal threshold dependent only on the radius $r$ (at least, under the same assumptions under which our limit was derived).  However, fourth, we also propose a semi-implicit version of Algorithm 1, in which pixels on the current inpainting boundary are inpainted simultaneously by solving a linear system.  As far as we know this idea has not yet been proposed in the literature.  We prove that the previous result is not true for the semi-implicit form of Algorithm 1, and prove that this extension applied to Guidefill ``kinks'' only if ${\bf g}$ is {\em exactly} parallel to $\partial D_h$.
Fifth, we tackle the problem of blur.  In particular, by considering an {\em asymptotic} limit (distinct from our continuum limit) where $h$ is small but non-zero, we are able to make quantitative predictions regarding the severity of blur artifacts.  Although a proof currently alludes us, this asymptotic limit is supported by strong numerical evidence.  Sixth, we fill a gap in the literature by proving convergence is also possible to the original limit proposed by Bornemann and M\"arz \cite{Marz2007}, but with a nuance:  to guarantee convergence, it is not enough that $h \rightarrow 0$ and $\epsilon \rightarrow 0$, what is required is that this occur while the ratio $r = \epsilon/h$ simultaneously diverges to infinity.  This observation helps us to put some mathematical rigor into our claim that our fixed ratio limit gives a better indication of the behaviour of Algorithm 1 in practice (where one typically fixes $r$ as some small non-negative integer).

\section{Review of main methods} \label{sec:review}

Here we briefly review the main inpainting methods of the general form sketched in Algorithm 1.

\vskip 2mm
\noindent {\bf Telea's algorithm.}  The earliest algorithm (to our knowledge) appearing in the literature and of the form sketched in Algorithm 1, Telea's algorithm \cite{Telea2004} is also the only such algorithm to use a different formula for $u_h({\bf x})$ than the expression \eqref{eqn:update} appearing in Algorithm 1 (see Remark \ref{remark:telea}).  Instead of computing $u_h({\bf x})$ a weighted average of $u_h({\bf y})$ evaluated at nearby already filled pixels ${\bf y}$, it takes a weighted average of the {\em predictions} that each of these pixels makes, based on linear extrapolation, for $u_h({\bf x})$.  That is,
\begin{equation} \label{eqn:Telea}	
u_h({\bf x}) = \frac{\sum_{{\bf y} \in B_{\epsilon,h}({\bf x}) \cap (\Omega_h \backslash D^{(k)}_h )} w_{\epsilon}({\bf x},{\bf y}) (u_h({\bf y})+\nabla_h u_h({\bf y}) \cdot ({\bf x}-{\bf y}))  }{\sum_{{\bf y} \in B_{\epsilon,h}({\bf x})\cap (\Omega_h \backslash D^{(k)}_h )} w_{\epsilon}({\bf x},{\bf y})},
\end{equation}
where $\nabla_h u_h({\bf y})$ denotes the centered difference approximation to the gradient of $u_h$ at ${\bf y}$, that is
$$\nabla_h u_h({\bf y}) :=  \frac{1}{2}\left(u_h({\bf y}+e_1)-u_h({\bf y}-e_1),u_h({\bf y}+e_2)-u_h({\bf y}-e_2)\right).$$
As we have already noted in Remark \ref{remark:telea}, this approach has a disadvantage in that it results in the loss of the stability property \eqref{eqn:stability}.  Moreover, the ``predictions'' 
$$u_{\mbox{predicted}}({\bf x}) := u_h({\bf y})+\nabla_h u_h({\bf y}) \cdot ({\bf x}-{\bf y})$$
can become highly inaccurate when ${\bf y}$ is on an edge $\Omega_h \backslash D^{(k)}_h$, leading to significant over or undershooting, visible as bright or dark spots as in Figure \ref{fig:stability} and Figure \ref{fig:artifacts2}.  Perhaps in recognition of this, the gradient term was dropped from \eqref{eqn:Telea} in all subsequent algorithms.  The weights in this case are 
$$w_{\epsilon}({\bf x},{\bf y}) = \mbox{dir}({\bf x},{\bf y}) \cdot \mbox{dst}({\bf x},{\bf y}) \cdot \mbox{lev}({\bf x},{\bf y}),$$
where
$$\mbox{dir}({\bf x},{\bf y}) := \frac{{\bf x}-{\bf y}}{\|{\bf x}-{\bf y}\|} \cdot N({\bf x}), \hskip 1mm 
\mbox{dst}({\bf x},{\bf y}) := \frac{d_0^2}{\|{\bf x}-{\bf y}\|^2}, \hskip 1mm 
\mbox{lev}({\bf x},{\bf y}) :=\frac{T_0}{1+|T({\bf y})-T({\bf x})|},$$
and $T({\bf x})$ denotes the Euclidean distance from ${\bf x}$ to the (original) boundary of the inpainting domain, and $N({\bf x}) = \nabla_h T({\bf x})$ (estimated based on central differences).  $T$ is precomputed using the fast marching method.  Telea's algorithm uses the default onion shell ordering, that is ``$\mbox{ready}({\bf x}) \equiv \mbox{true}$''.

\vskip 2mm

\noindent {\bf Coherence transport.}  Coherence transport \cite{Marz2007} improves upon Telea's algorithm by adapting the weights in order to encourage extrapolation of isophotes in the direction of their tangent.  This is done by calculating a ``local coherence direction'' ${\bf g}({\bf x})$ in terms of a modified structure tensor.  Coherence transport calculates the color of a given pixel to be filled using the formula \eqref{eqn:update} in Algorithm 1 with weights
\begin{equation} \label{eqn:weight}
w_{\epsilon}({\bf x},{\bf y}) = \frac{1}{\|{\bf y}-{\bf x}\|}\exp\left(-\frac{\mu^2}{2\epsilon^2}({\bf g}^{\perp}({\bf x}) \cdot ({\bf y}-{\bf x}))^2\right),
\end{equation}		
and with $A_{\epsilon,h}({\bf x}) = B_{\epsilon,h}({\bf x})$ - see Figure \ref{fig:ballRotate}(a) and Figure \ref{fig:ballRotate}(c).  Like Telea's algorithm, coherence transport uses the default onion shell ordering, that is ``$\mbox{ready}({\bf x}) \equiv \mbox{true}$''.

\vskip 2mm

\noindent {\bf Coherence transport with adapted distance functions.}  In a subsequent work \cite{Marz2011}, M\"arz made improvements to coherence transport by replacing the default onion shell ordering with one based on a variety of non-Euclidean distance functions.  One such distance function defines an ``active boundary'' $\Gamma_h \subseteq \partial D_h$ defined by
$$\Gamma_h := \{ \partial D_h : \langle{\bf g}({\bf x}), {\bf N}({\bf x}) \rangle^2 > \gamma \}$$
where $\gamma > 0$ is a small constant.  The non-Euclidean distance to boundary $T^*_h$ is then computed as the Euclidean distance to the active boundary.  The algorithm is modified so that at any given iteration, only a subset of boundary pixels are filled - namely those minimizing $T^*_h$.  That is
$$\mbox{ready}({\bf x}) = \mbox{true} \Leftrightarrow {\bf x} \in \operatorname{argmin}_{{\bf y} \in \partial D_h} T^*_h({\bf y}).$$
This adaptation leads to improvements in the long range extrapolation of isophotes, as in Figure \ref{fig:cutoff}.

\begin{figure}
\centering
\begin{tabular}{ccc}
\subfloat[$A_{\epsilon,h}({\bf x})=B_{\epsilon,h}({\bf x})$.]{\includegraphics[width=.28\linewidth]{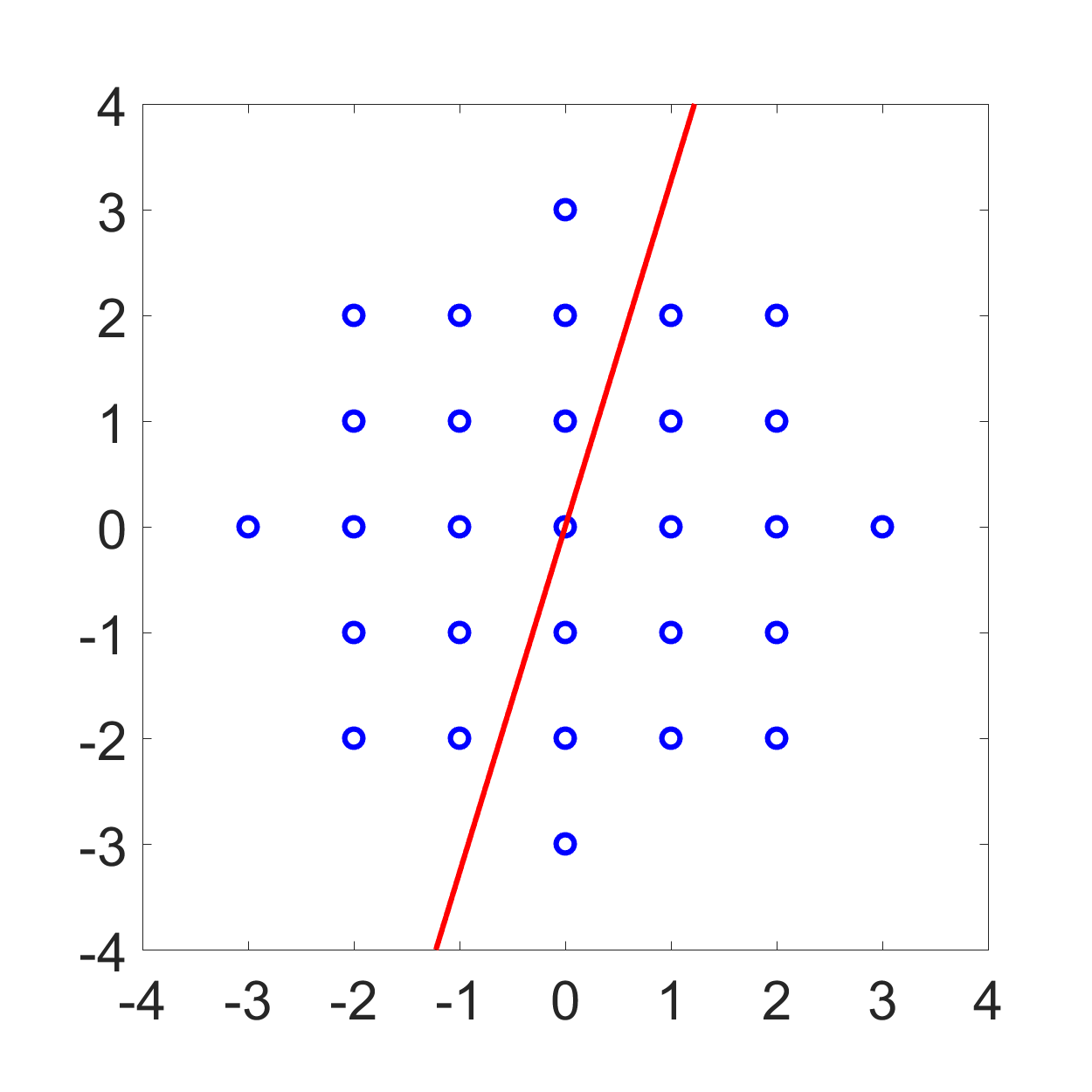}} & 
\subfloat[$A_{\epsilon,h}({\bf x})= \tilde{B}_{\epsilon,h}({\bf x})$.]{\includegraphics[width=.28\linewidth]{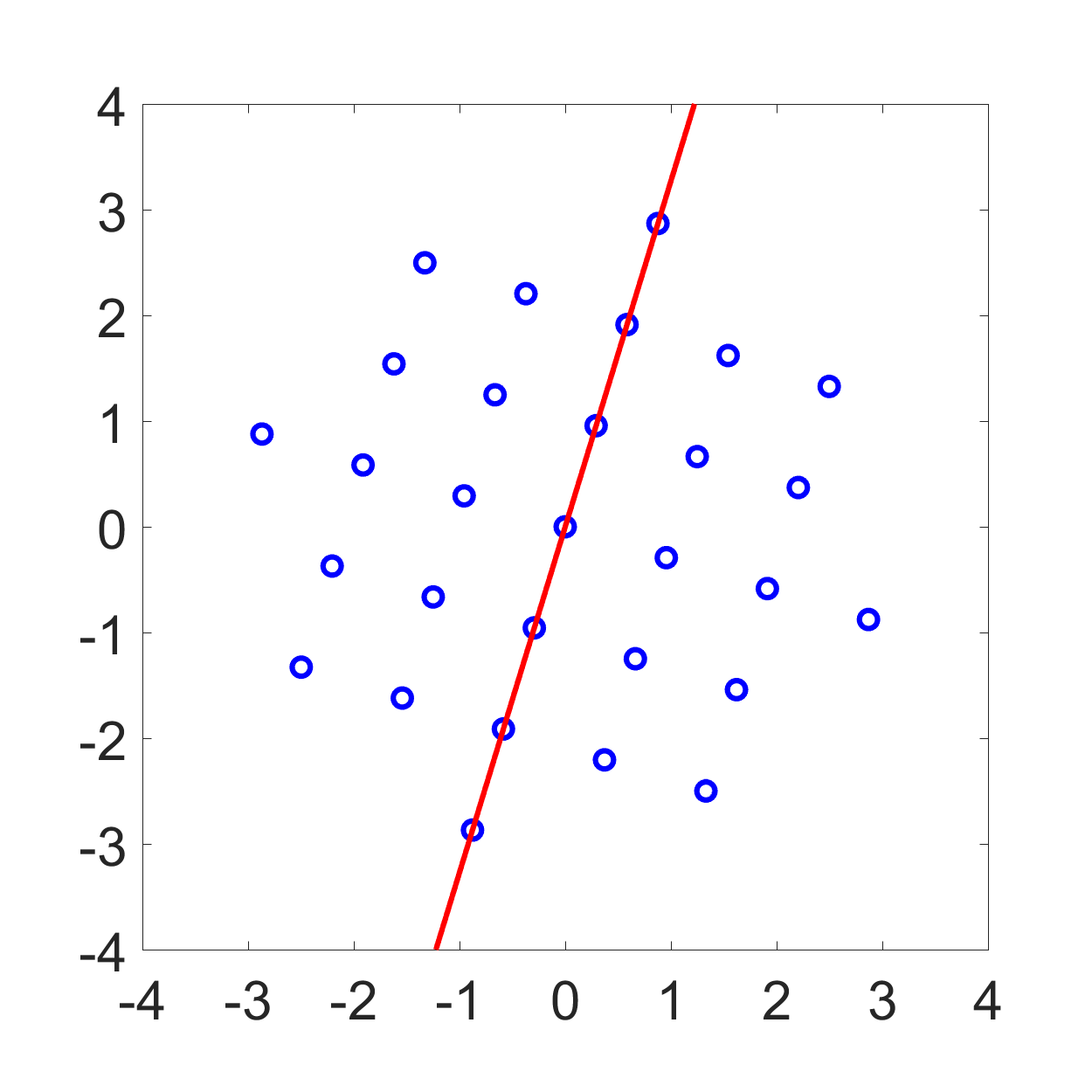}} &
\subfloat[Illustration of the (normalized) weights \eqref{eqn:weight} for $\mu = 10$.]{\includegraphics[width=.33\linewidth]{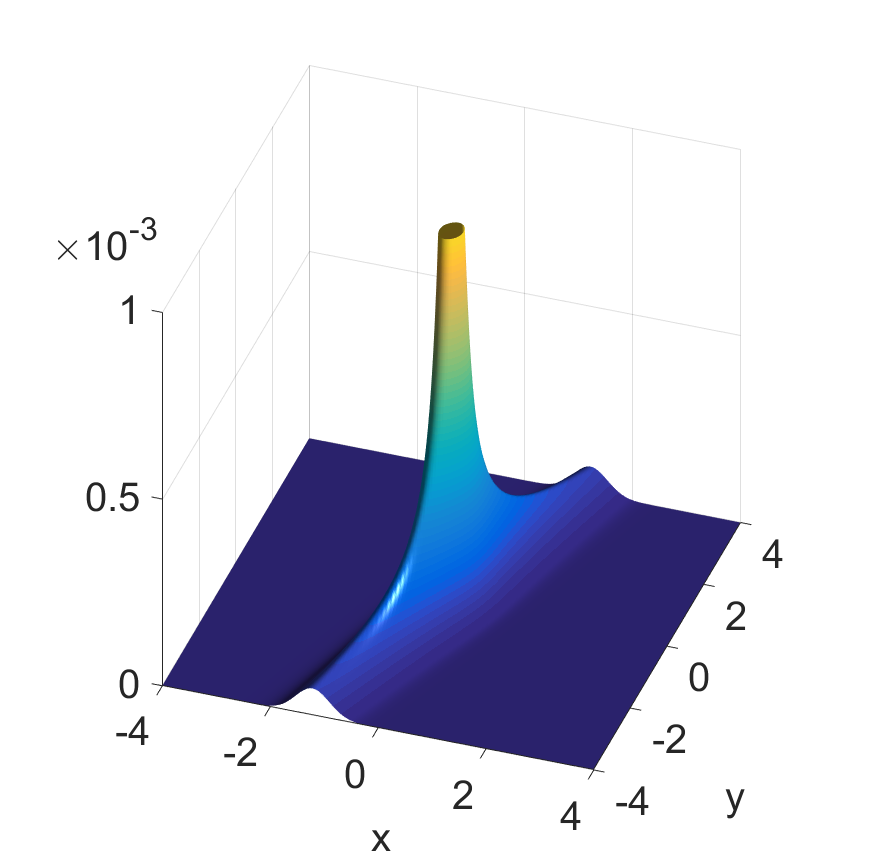}}\\
\end{tabular}
\caption{{\bf Neighborhoods and weights for coherence transport and Guidefill:} Here we illustrate the neighborhoods $A_{\epsilon,h}({\bf x})$ and weights \eqref{eqn:weight} used by coherence transport and Guidefill.  In each case $\epsilon = 3$px and ${\bf g}({\bf x})=(\cos73^{\circ},\sin73^{\circ})$.  Coherence transport (a) uses the lattice-aligned discrete ball $A_{\epsilon,h}({\bf x})=B_{\epsilon,h}({\bf x})$, while Guidefill (b) uses the rotated discrete ball $A_{\epsilon,h}({\bf x})=\tilde{B}_{\epsilon,h}({\bf x})$.  The ball $\tilde{B}_{\epsilon,h}({\bf x})$ is rotated so that it is aligned with the line $L$ (shown in red) passing through ${\bf x}$ parallel to ${\bf g}({\bf x})$.  In general $\tilde{B}_{\epsilon,h}({\bf x})$ contains ``ghost pixels'' lying between pixel centers, which are defined using bilinear interpolation of their ``real'' pixel neighbors.  Both use the same weights \eqref{eqn:weight} illustrated in (c).  The parameter $\mu$ controls the extent to which the weights are biased in favor of points lying on or close to the line $L$.}
\label{fig:ballRotate}
\end{figure}

\vskip 2mm

\noindent {\bf Guidefill.}  Guidefill \cite{Guidefill} is a recent inpainting algorithm designed to address, among other things, the kinking issues in Figure \ref{fig:specialDir}(b) and Figure \ref{fig:taleOfTwo}.  While coherence transport is able to extrapolate along guidance direction ${\bf g}({\bf x})$ only if ${\bf g}({\bf x}) = \lambda ({\bf v}-{\bf x})$ for some ${\bf v} \in B_{\epsilon,h}({\bf x})$ (see Figure \ref{fig:specialDir}(b)), Guidefill replaces the lattice aligned discrete ball $B_{\epsilon,h}({\bf x})$ with the {\em rotated discrete ball} $\tilde{B}_{\epsilon,h}({\bf x})$ aligned with the local transport direction ${\bf g}({\bf x})$, so that ${\bf g}({\bf x}) = \lambda ({\bf v}-{\bf x})$ for some ${\bf v} \in \tilde{B}_{\epsilon,h}({\bf x})$ is {\em always} true.   The rotated ball $\tilde{B}_{\epsilon,h}({\bf x})$ contains ``ghost pixels'' lying between pixel centers which are defined using bilinear interpolation.  See Section \ref{sec:fiction} for a deeper discussion of ghost pixels, as well as Figure \ref{fig:ballRotate}(a)-(b) for an illustration of $B_{\epsilon,h}({\bf x})$ and $\tilde{B}_{\epsilon,h}({\bf x})$.

Guidefill uses the same weights \eqref{eqn:weight} as coherence transport (illustrated in Figure \ref{fig:ballRotate}(c)) and similarly to the latter's extension \cite{Marz2011}, it has a way of automatically determining a good fill order.  Unlike coherence transport which computes ${\bf g}({\bf x})$ concurrently with inpainting, Guidefill computes a guide field ${\bf g}({\bf x}) : D_h \rightarrow \field{R}^2$ prior to inpainting.  The guide field is computed based on splines which the user may adjust in order to influence the results.  It is used to automatically compute a good fill order by computing for each ${\bf x} \in \partial D_h$ a confidence $C({\bf x}) \in [0,1]$ inspired by Criminisi et al. \cite{Criminisi04regionfilling} and given by
\begin{equation} \label{eqn:confidence}
C({\bf x})=\frac{\sum_{{\bf y} \in \tilde{B}_{\epsilon,h}({\bf x}) \cap (\Omega \backslash D^{(k)}  )}w_{\epsilon}({\bf x},{\bf y})}{\sum_{{\bf y} \in \tilde{B}_{\epsilon,h}({\bf x}) }w_{\epsilon}({\bf x},{\bf y})},
\end{equation}
and then only filling those pixels for which $C({\bf x}) > c$, where $ c \in (0,1)$ is a small constant.  That is 
\begin{equation} \label{eqn:ready}
\mbox{ ready}({\bf x}) = 1( C({\bf x}) > c)
\end{equation}
Guidefill was designed for use as part of a 3D conversion pipeline, and as such makes use of a set $B_h$ of ``bystander pixels'' which are neither inpainted nor may be used for inpainting.  However, this is not relevant to our current investigation and we will assume $B_h = \emptyset$ throughout.  As shown in Figure \ref{fig:specialDir}(c) - Guidefill is able to largely, but not completely, eliminate kinking artifacts.  It was in the hope of overcoming this that we designed the semi-implicit version of Algorithm 1 discussed in Section \ref{sec:semiImplicit}.

\begin{remark} \label{rem:punctured}

Note that we have deliberately excluded the point ${\bf x}$ from the update formula \eqref{eqn:update} in Algorithm 1, even if the set $A_{\epsilon,h}({\bf x})$ contains ${\bf x}$.  This is {\em not} done in any of the methods \cite{Telea2004,Marz2007,Marz2011,Guidefill} we have just discussed, but it makes no difference to them or any other variant of the direct form of Algorithm 1, because the subroutine FillRow only involves sums taken over $A_{\epsilon,h}({\bf x}) \cap (\Omega \backslash D^{(k)})$, which never contains ${\bf x}$.  However, the semi-implicit extension of Algorithm 1 expresses $u_h({\bf x})$ as a sum of $u_h({\bf y})$ over a set of points that might include ${\bf x}$.  This creates problems with weights such as \eqref{eqn:weight} for which $w_{\epsilon}({\bf x},{\bf x})=\infty$.  See Appendix \ref{app:punctured} for further details.

\end{remark}

\section{Ghost pixels and equivalent weights} \label{sec:fiction}  Because ghost pixels are defined using bilinear interpolation, any sum over a finite set of ghost pixels $A({\bf x})$ can be converted into a sum over an equivalent set of real pixels with equivalent weights\footnote{note that here we mean a general family of finite sets $A({\bf x}) \in \field{R}^2$ and general weights $w({\bf x},{\bf y})$.  We do not mean the specific family of sets $A_{\epsilon,h}({\bf x})$ or the specific weights $w_{\epsilon}({\bf x},{\bf y})$, which have special properties. }, that is
$$\sum_{{\bf y} \in A({\bf x})} w({\bf x},{\bf y}) u_h({\bf y}) = \sum_{{\bf y} \in \mbox{Supp}(A({\bf x}))} \tilde{w}({\bf x},{\bf y}) u_h({\bf y})$$
where $\mbox{Supp}(A({\bf x}))$ denotes the set of real pixels needed to define $u_h({\bf y})$ for each ${\bf y} \in A({\bf x})$ and $\tilde{w}$ denotes a set of equivalent weights.  This works because each $u_h({\bf y})$ is itself a weighted sum of the form
$$u_h({\bf y}) = \sum_{{\bf z} \in \field{Z}_h^2}\Lambda_{{\bf z},h}({\bf y})u_h({\bf z}),$$
where $\{ \Lambda_{{\bf z},h} \}_{{\bf z} \in \field{Z}^2_h}$ denote the basis functions of bilinear interpolation associated with the lattice $\field{Z}^2_h$.  This is illustrated in Figure \ref{fig:equivWeights}(a)-(b), where we show a heat map of the weights \eqref{eqn:weight} over the set $\tilde{B}_{\epsilon,h}({\bf x}) \backslash \{ {\bf x}\}$ for $\mu = 50$ and $\epsilon=3$px, as well as a similar heat map of the modified weights over $\mbox{Supp}(\tilde{B}_{\epsilon,h}({\bf x})\backslash \{ {\bf x}\}) \subseteq D_h(B_{\epsilon,h}({\bf x}))$.  Note that even though $\tilde{B}_{\epsilon,h}({\bf x})\backslash \{ {\bf x}\}$ does not contain the point ${\bf x}$, the support of this set does.  This will be important in Section \ref{sec:semiImplicit}.
Here we briefly list some properties of equivalent weights, including an explicit formula.  A proof is sketched, but details are deferred to Appendix \ref{app:ghostPixels}.
\begin{figure}
\centering
\begin{tabular}{ccc}
\subfloat[Heatmap of weights \eqref{eqn:weight} over $\tilde{B}_{\epsilon,h}({\bf x}) \backslash \{ {\bf x} \}$.]{\includegraphics[width=.3\linewidth]{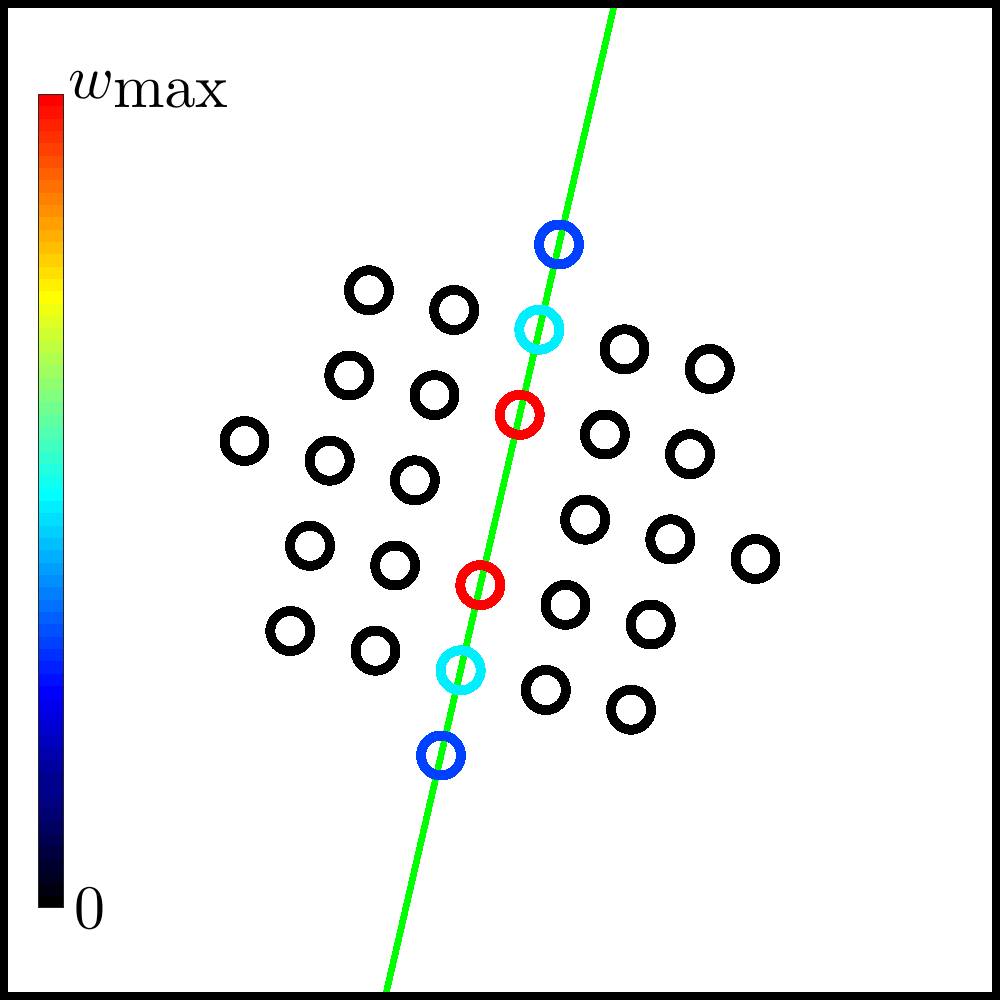}} & 
\subfloat[Heatmap of equivalent weights over $\mbox{Supp}(\tilde{B}_{\epsilon,h}({\bf x})\backslash \{ {\bf x} \})$.]{\includegraphics[width=.3\linewidth]{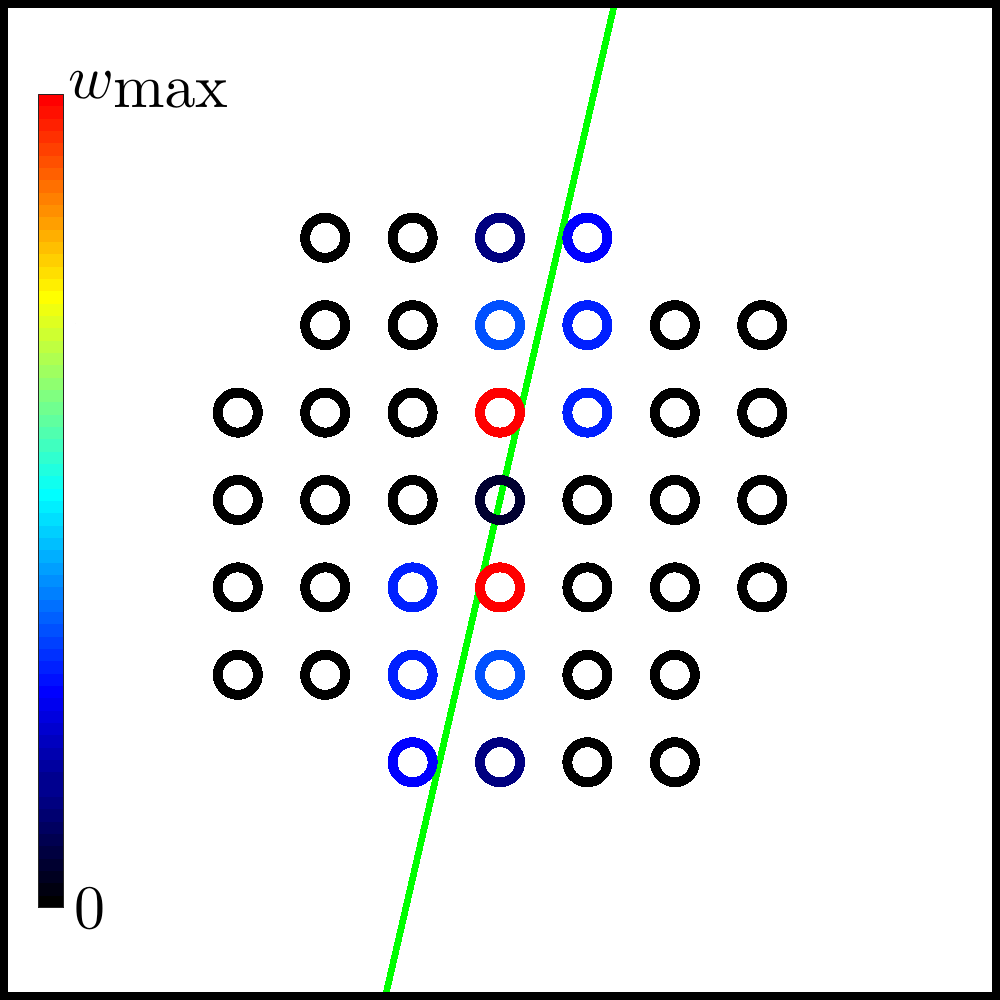}} &
\subfloat[Heatmap of equivalent weights over $D_h(B_{\epsilon,h}({\bf x}))$.]{\includegraphics[width=.3\linewidth ]{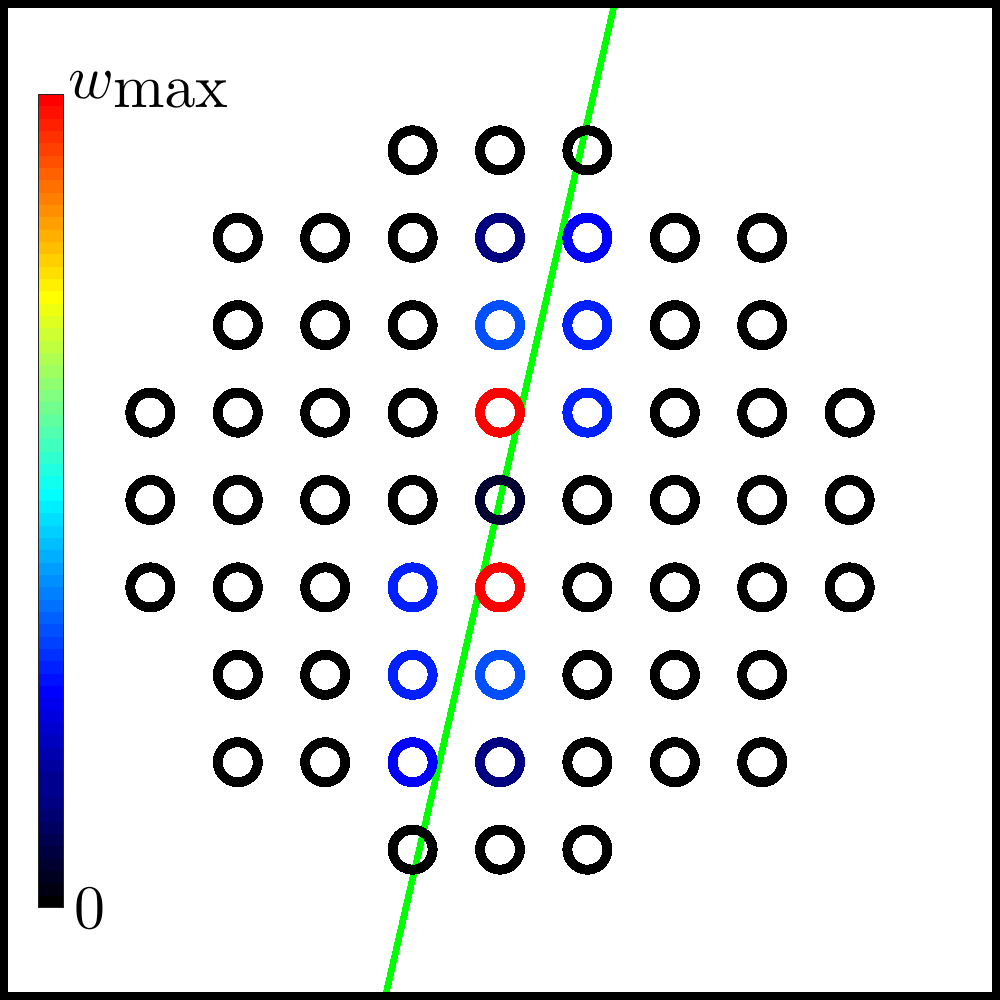}}\\
\end{tabular}
\caption{{\bf Ghost pixels and equivalent weights:}  Because ghost pixels are defined using bilinear interpolation, any weighted sum over a set of ghost pixels $A_{\epsilon,h}({\bf x})$ 
is equivalent to a sum with equivalent weights over real pixels in the set $\mbox{Supp}(A_{\epsilon,h}({\bf x}))$, defined as the set of real pixels needed to define each ghost pixel ${\bf y}$ in $A_{\epsilon,h}({\bf x})$. We illustrate this in (a)-(c) using Guidefill with $\epsilon = 3$px, ${\bf g}=(\cos77^{\circ},\sin77^{\circ})$, and $\mu = 100$. In (a), the (normalized) weights \eqref{eqn:weight} are visualized as a heat map over $\tilde{B}_{\epsilon,h}({\bf x}) \backslash \{ {\bf x} \}$. In (b), we show the equivalent weights over $\mbox{Supp}(\tilde{B}_{\epsilon,h}({\bf x}) \backslash \{ {\bf x} \} ) \subseteq D_h(B_{\epsilon,h}({\bf x}))$ (this containment comes from \eqref{eqn:ballContain} in Remark \ref{rem:equivWeights2}).  Note that even though $\tilde{B}_{\epsilon,h}({\bf x}) \backslash \{ {\bf x} \}$ does not contain the point ${\bf x}$, the support of this set does.  In (c), we visualize the equivalent weights over the set $D_h(B_{\epsilon,h}({\bf x}))$, which is strictly larger than $\mbox{Supp}(\tilde{B}_{\epsilon,h}({\bf x}) \backslash \{ {\bf x} \} )$.  For reference, we include the line parallel to ${\bf g}$ in green.}
\label{fig:equivWeights}
\end{figure}





\vskip 1mm
{\bf Properties of equivalent weights}  Properties 1-3 deal with a general finite set $A({\bf x})$ and general weights $w({\bf x},{\bf y})$, while properties 4-6 deal with the specific set $A_{\epsilon,h}({\bf x}) \subset B_{\epsilon}({\bf x})$ and the specific weights $w_{\epsilon}({\bf x},{\bf y})$ obeying \eqref{eqn:technical}.
\begin{enumerate}
\item Explicit formula:
\begin{equation} \label{eqn:explicitFormula}
\tilde{w}({\bf x},{\bf z}) = \sum_{{\bf y} \in A({\bf x})} \Lambda_{{\bf y},h}({\bf z})w({\bf x},{\bf y})
\end{equation}
\item Preservation of total mass:
\begin{equation} \label{eqn:totalMass}
\sum_{{\bf y} \in A({\bf x})} w({\bf x},{\bf y}) = \sum_{{\bf y} \in \mbox{Supp}(A({\bf x}))} \tilde{w}({\bf x},{\bf y}).
\end{equation}
\item Preservation of center of mass (or first moment):
\begin{equation} \label{eqn:centerMass}
\sum_{{\bf y} \in A({\bf x})} w({\bf x},{\bf y}){\bf y} = \sum_{{\bf y} \in \mbox{Supp}(A({\bf x}))} \tilde{w}({\bf x},{\bf y}){\bf y}.
\end{equation}
\item Inheritance of non-negativity:
\begin{equation} \label{eqn:nonNegativeInheritance}
\tilde{w}_{\epsilon}({\bf x},{\bf z}) \geq  0 \quad \mbox{ for all } {\bf z} \in \mbox{Supp}(A_{\epsilon,h}({\bf x})).
\end{equation}
\item Inheritance of non-degeneracy condition \eqref{eqn:technical}:
\begin{equation} \label{eqn:technical_equiv}
\sum_{{\bf y} \in \mbox{Supp}(A_{\epsilon,h}({\bf x}) \cap (\Omega \backslash D^{(k)}))}\tilde{w}_{\epsilon}({\bf x},{\bf y}) > 0.
\end{equation}
\item Universal support:
For any $n \in \field{Z}$, we have
\begin{equation} \label{eqn:universalSupport}
\mbox{Supp}(A_{\epsilon,h}({\bf x}) \cap \{ y \leq nh\} ) \subseteq D_h(B_{\epsilon,h}({\bf x})) \cap \{ y \leq nh\}  \subseteq B_{\epsilon+2h,h}({\bf x}) \cap \{ y \leq nh\}.
\end{equation}
where $\{y \leq nh \} := \{ (x,y) \in \field{R}^2 : y \leq nh\}$, and where $D_h$ is the dilation operator defined in our section on notation.
\end{enumerate}
\vskip 1mm
\begin{proof}
Most of these properties are either obvious or are derived based on a simple exercise in changing the order of nested finite sums.  Properties \eqref{eqn:totalMass} and \eqref{eqn:centerMass} are slightly more interesting - they follow from the fact that the bilinear interpolant of a polynomial of degree at most one is just the polynomial again.  Note that an analogous formula for preservation of the second moment does {\em not} hold, because a quadratic function and its bilinear interpolant are not the same thing.  The last identity is based on an explicit formula for the support of a point.  Details are provided in Appendix \ref{app:ghostPixels}. 
\qed
\end{proof}

\vskip 2mm

\begin{remark} \label{rem:equivWeights}
Although we have explicit formula \eqref{eqn:explicitFormula} for the equivalent weights which will occasionally be useful, most of the time it is more fruitful to think about them in the following way:  To compute $\tilde{w}_{\epsilon}({\bf x},{\bf y})$ for some real pixel ${\bf y}$, loop over the ghost pixels ${\bf z}$ such that ${\bf y} \in \mbox{Supp}({\bf z})$.  Then, each such ${\bf z}$ redistributes to $\tilde{w}_{\epsilon}({\bf x},{\bf y})$ a fraction of its weight $w_{\epsilon}({\bf x},{\bf z})$ equal to the proportion of $u_h({\bf y})$ that went into $u_h({\bf z})$.  
\end{remark}

\vskip 2mm

\begin{remark} \label{rem:equivWeights2}
An obvious corollary of the universal support property \eqref{eqn:universalSupport} is that we also have the containment
\begin{equation} \label{eqn:ballContain}
\mbox{Supp}(A_{\epsilon,h}({\bf x})  ) \subseteq D_h(B_{\epsilon,h}({\bf x}))  \subseteq B_{\epsilon+2h,h}({\bf x}).
\end{equation}
Figure \ref{fig:equivWeights} illustrates an example where this containment is strict, and in fact it is not hard to show that this holds in general.  However, \eqref{eqn:universalSupport} is tight enough for our purposes in this paper.
\end{remark}

\section{Semi-implicit extension of Algorithm 1} \label{sec:semiImplicit}

Here we present a semi-implicit extension of Algorithm 1, to our knowledge not previously proposed in the literature, in which instead of computing $u_h({\bf x})$ for each ${\bf x} \in \partial_{\mbox{ready}} D^{(k)}_h$ independently, we solve for them simultaneously by solving a linear system.  We call our method semi-implicit in order to distinguish it from fully implicit methods in which the entire inpainting domain $\{u_h({\bf x}) : {\bf x} \in D_h\}$ is solved simultaneously, as is typically the case for most inpainting methods based on PDEs or variational principles, e.g. \cite{bertalmio2000image,chan2002euler,Burger2009}.  Specifically, we solve
\begin{equation} \label{eqn:Linearboundary}
\mathcal L{\bf u} = {\bf f} \quad \mbox{ where } {\bf u} = \{ u_h({\bf x}) : {\bf x} \in \partial_{\mbox{ready}} D^{(k)}_h \}.
\end{equation}
and ${\bf f}$ is a vector of length $|\partial_{\mbox{ready}} D^{(k)}_h|$.  The explicit entries of $\mathcal L$ are written in terms of the equivalent weights  $\tilde{w}_{\epsilon}$ introduced in Section \ref{sec:fiction}.  Defining
$$S^{(k)}_{\epsilon,h}({\bf x}) := \mbox{Supp}(A_{\epsilon,h}({\bf x})) \cap \partial_{\mbox{ready}} D^{(k)}_h,$$  
it follows that $\mathcal L$ couples each ${\bf x} \in  \partial_{\mbox{ready}} D^{(k)}_h$ to its immediate neighbors in $S^{(k)}_{\epsilon,h}({\bf x})$.  In particular, we have
\begin{equation} \label{eqn:Matrix}
(\mathcal L{\bf u})({\bf x}) =\left(1-\frac{\tilde{w}_{\epsilon}({\bf x},{\bf x})}{W}\right)u_h({\bf x}) - \sum_{{\bf y} \in S^{(k)}_{\epsilon,h}({\bf x}) \backslash \{ {\bf x} \}}\frac{\tilde{w}_{\epsilon}({\bf x},{\bf y})}{W}u_h({\bf y}),
\end{equation}
where $W$ is the total mass and can be computed in one of two ways, using either the original weights $w_{\epsilon}$ or the equivalent weights $\tilde{w}_{\epsilon}$, exploiting preservation of mass \eqref{eqn:totalMass}:
\begin{eqnarray}
W&:=&\sum_{{\bf y} \in S^{(k)}_{\epsilon,h}({\bf x}) \cup \left(\mbox{Supp}(A_{\epsilon,h}({\bf x})) \cap (\Omega_h \backslash D^{(k)}_h)\right) }\tilde{w}_{\epsilon}({\bf x},{\bf y}) \label{eqn:1stway} \\ 
&=&\sum_{{\bf y} \in (A_{\epsilon,h}({\bf x}) \backslash \{ {\bf x}\}) \cap (\Omega \backslash D^{(k)}) } w_{\epsilon}({\bf x},{\bf y}). \label{eqn:2ndway}
\end{eqnarray}
Generally, \eqref{eqn:2ndway} is more convenient to work with than \eqref{eqn:1stway}, but \eqref{eqn:1stway} combined with the inherited non-degeneracy condition \eqref{eqn:technical_equiv} tells us that
$$\sum_{{\bf y} \in S^{(k)}_{\epsilon,h}({\bf x})} \tilde{w}_{\epsilon}({\bf x},{\bf y}) < W,$$
because the latter implies that a non-zero proportion of the total weight goes into the known pixels in $\mbox{Supp}(A_{\epsilon,h}({\bf x})) \cap (\Omega \backslash D^{(k)}_h)$ rather than the unknown pixels in $S^{(k)}_{\epsilon,h}({\bf x})$.  From this it immediately follows that $\mathcal L$ is strictly diagonally dominant - a property we will use later.  To compute ${\bf f}$, we do not need the concept of equivalent weights.  We have
\begin{equation} \label{eqn:f}
{\bf f}({\bf x}) = \sum_{{\bf y} \in (A_{\epsilon,h}({\bf x}) \backslash \{ {\bf x} \}) \cap ( \Omega \backslash D^{(k)})}\frac{w_{\epsilon}({\bf x},{\bf y})}{W}u_h({\bf y}).
\end{equation}

\subsection{Solving the linear system}  Designing maximally efficient methods for solving \eqref{eqn:Linearboundary} is beyond the scope of this paper - our main purpose lies in understanding the effect of this extension on the continuum limit that we will derive later.  Therefore, in this paper we consider only two very simple methods:  damped Jacobi and SOR (successive over-relaxation).  These are natural choices for a number of reasons.  First, since $\mathcal L$ is strictly diagonally dominant, these methods are both guaranteed to converge \cite[Theorem 3.10, pg. 79]{varga1962matrix}, at least in the case $\omega=1$, where they reduce to Jacobi and Gauss-Seidel.  Second, at least for the semi-implicit extension of Guidefill, the performance of SOR is already satisfactory, (see Section \ref{sec:guidefillGS}, Proposition \ref{prop:rates}).  Third, both methods can be implemented with minimal changes to the direct form of Algorithm 1.  In fact, changing the variable ``semiImplicit'' from ``false'' to ``true'' in Algorithm 1 and executing the ``FillBoundary'' subroutine in parallel is equivalent to solving \eqref{eqn:Linearboundary} using damped Jacobi, with underrelaxation parameter
\begin{equation} \label{eqn:specialomega}
\omega^* = \left(1 - \frac{\tilde{w}_{\epsilon}({\bf x},{\bf x})}{W}\right) \leq 1.
\end{equation}
Similarly, executing the ``FillBoundary'' subroutine sequentially results in SOR with the same underrelaxation parameter - see Proposition \ref{prop:secretlyDampedJacobi} in Appendix \ref{app:secret} for a proof.  Note that $\omega^*$ is typically very close $1$, so these methods are very similar to plain Jacobi and Gauss-Seidel.  
\begin{remark}
The reason Algorithm 1 with ``semiImplicit'' set to ``true'' results in damped Jacobi/SOR, rather than plain Jacobi/Gauss-Seidel, is because even though the update formula \eqref{eqn:update} for the $n$th iterate $u_h^{(n)}({\bf x})$ is expressed as a sum of the $(n-1)$st iterate evaluated at {\em ghost pixels} that do not include ${\bf x}$, some of those ghost pixels may indirectly depend on $u^{(n-1)}_h({\bf x})$ because they are defined using bilinear interpolation.  The result is that the $n$th iterate $u_h^{(n)}({\bf x})$ depends on $u_h^{(n-1)}({\bf x})$, which is true of damped Jacobi/SOR but not of Jacobi/Gauss-Seidel.  
\end{remark}

\begin{figure}
\centering
\begin{tabular}{cccc}
\subfloat[Original inpainting problem, inpainting domain in yellow.]{\includegraphics[width=.22\linewidth]{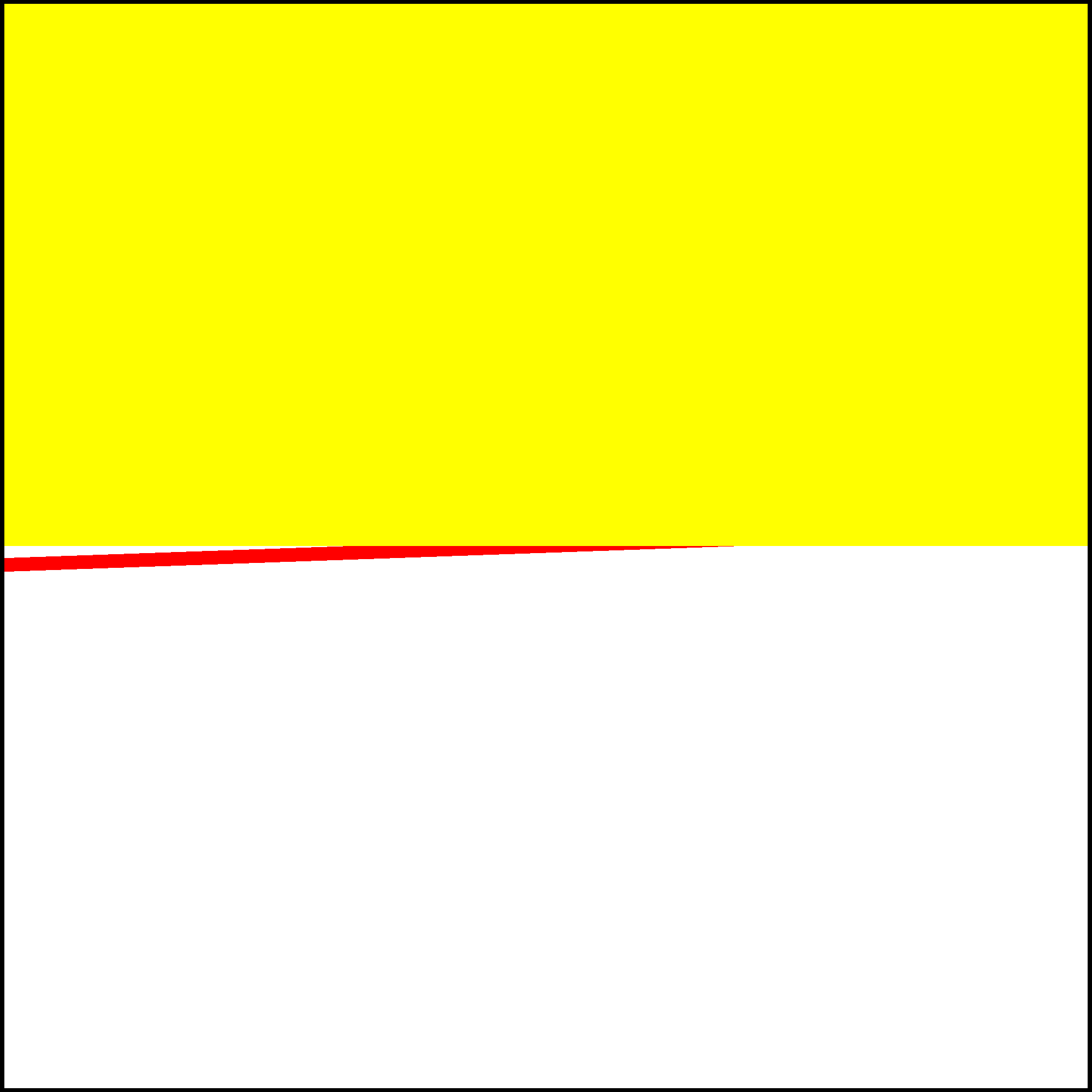}} & 
\subfloat[Inpainting with Guidefill.]{\includegraphics[width=.22\linewidth]{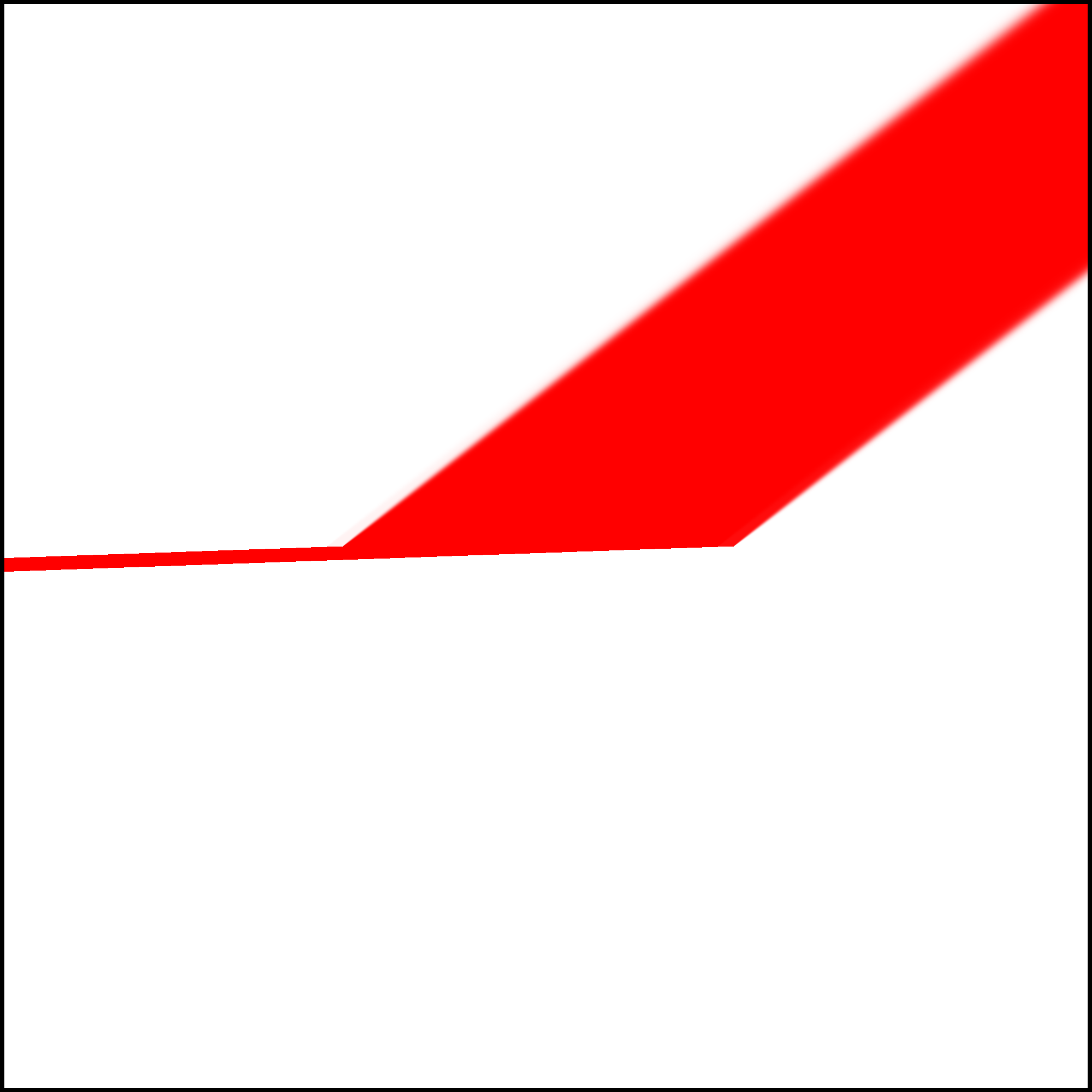}} &
\subfloat[Inpainting with semi-implicit Guidefill, 5 SOR iterations per shell.]{\includegraphics[width=.22\linewidth ]{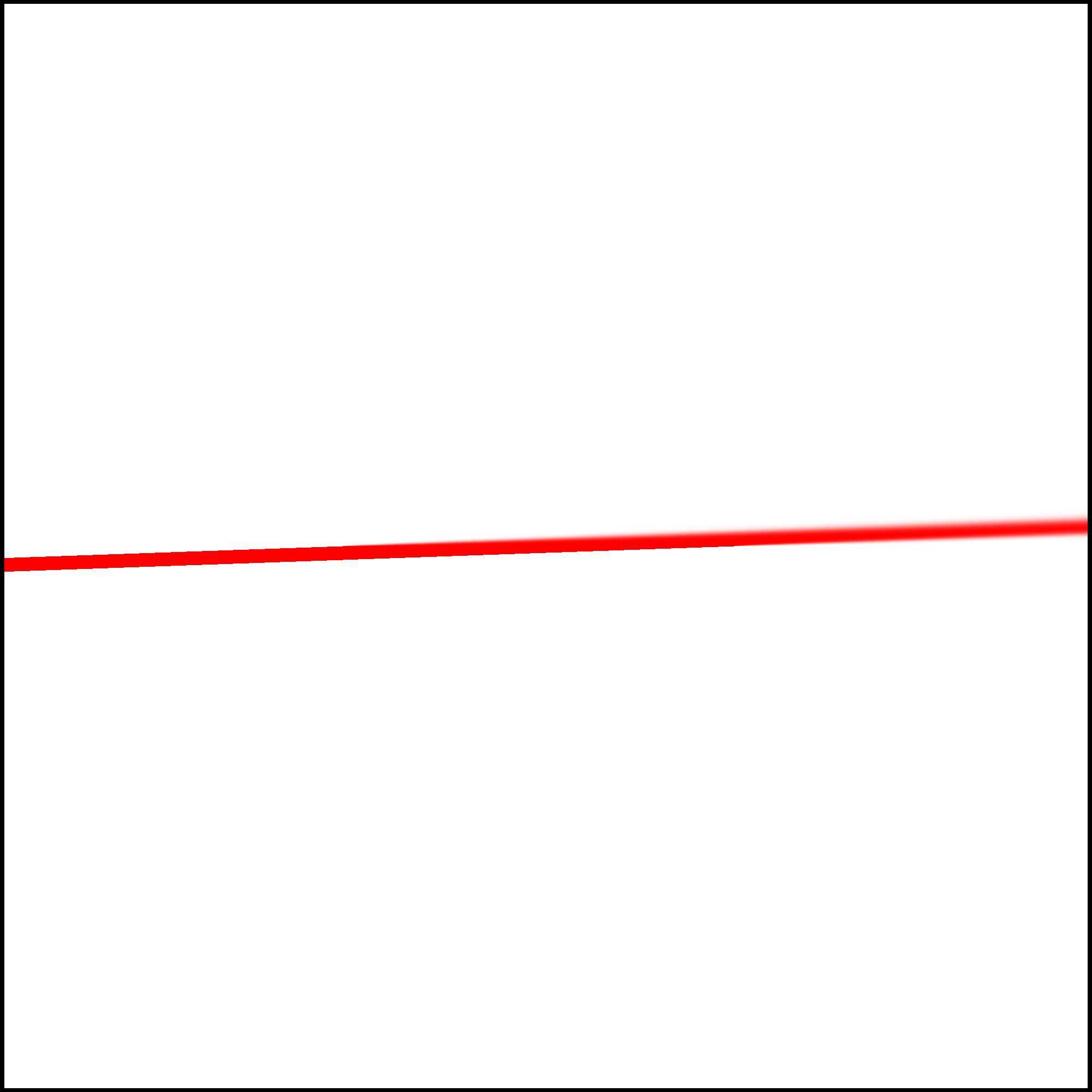}} &
\subfloat[Inpainting with semi-implicit Guidefill, 50 damped Jacobi iterations per shell.]{\includegraphics[width=.22\linewidth ]{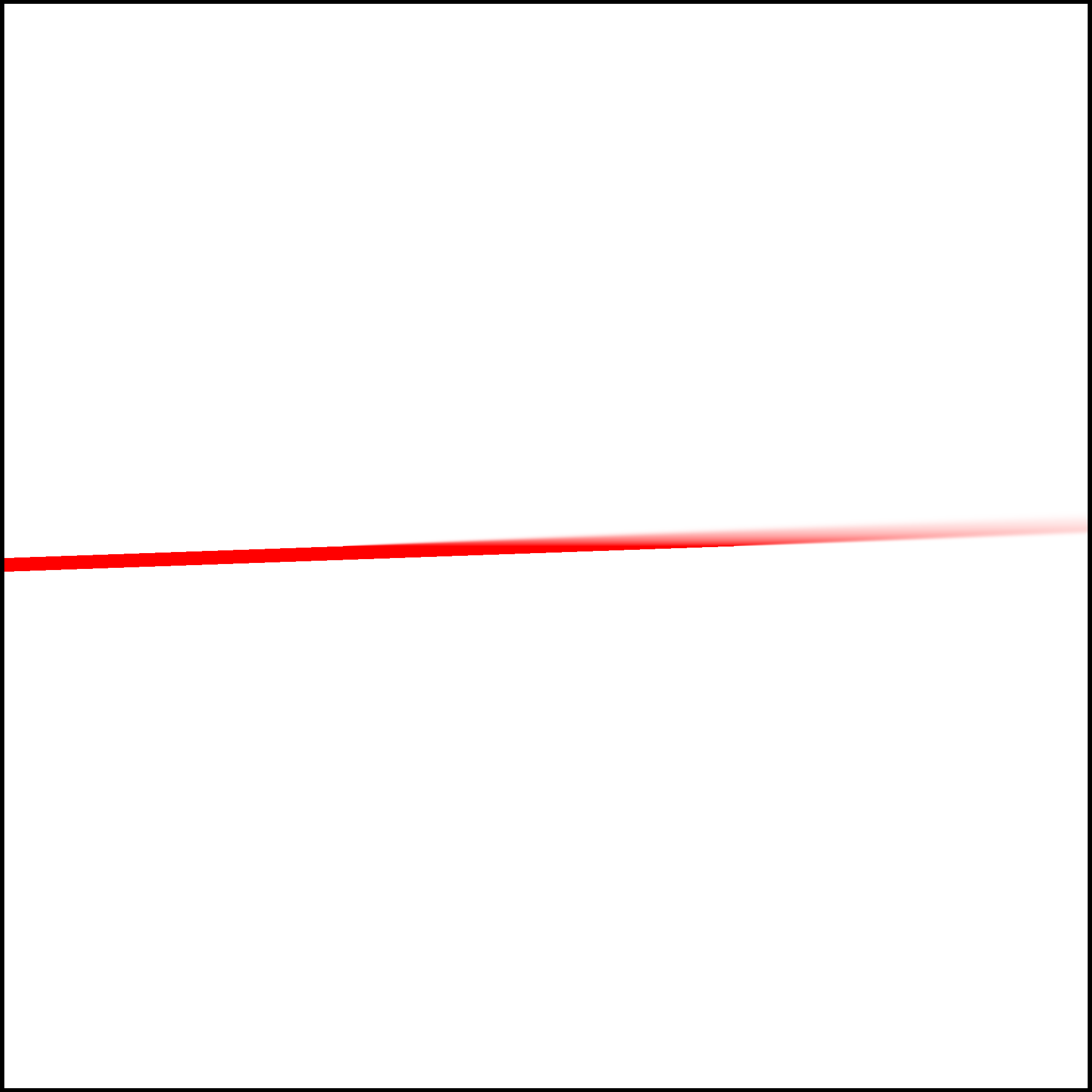}} \\
\end{tabular}
\caption{{\bf Semi-Implicit Guidefill and Shallow Inpainting Directions:}  In (a), a line makes a very shallow angle of just $2^{\circ}$ with the inpainting domain, shown in yellow.  This line is then inpainted using first Guidefill (b) and then the semi-implicit extension thereof (c).  The latter uses 5 iterations of SOR per shell to solve the linear system \eqref{eqn:Linearboundary} arising in every shell (in this case, the original image is $2000\times2000$px, so there are $1000$ shells).  Visually identical results can be obtained using damped Jacobi, but more than 100 iterations per shell are required - see Proposition \ref{prop:rates}.  Both methods use relaxation parameter $\omega = \omega^*$ \eqref{eqn:specialomega} and are given ${\bf g}=(\cos 2^{\circ},\sin 2^{\circ})$, $\mu = 100$, $\epsilon=3$px, and use the default onion shell ordering (smart-order is turned off).  Guidefill kinks while its extension does not.  In (d), we see the result of failing to solve the  linear system \eqref{eqn:Linearboundary} to sufficient accuracy by applying too few iterations of damped Jacobi.  In this case only 50 iterations per shell are used and the extrapolated line gradually fades away. }
\label{fig:shallowAngles}
\end{figure}

\subsection{Semi-implicit Guidefill}  In this paper we particularly interested in the extension of Guidefill since, as illustrated in Figure \ref{fig:shallowAngles} and as we will prove in Section \ref{sec:kink3main}, it is able to overcome the issue of kinking that we saw with Guidefill for shallow angles  in Figure \ref{fig:specialDir} (and see again in \ref{fig:shallowAngles}(b)).  In Section \ref{sec:guidefillGS} we will analyze the convergence of damped Jacobi and SOR for solving the linear system \eqref{eqn:Linearboundary} arising in this method - see Proposition \ref{prop:rates} in particular.  Special attention will be paid to the parameter value $\omega = \omega^*$ \eqref{eqn:specialomega} that naturally arises from the implementation we have proposed in Algorithm 1 (blue text).  First, however, we note that beyond changing the boolean variable ``semiImplicit'' to ``true'' in Algorithm 1, this extension performs optimally only if we also change the definition of the ``ready'' function.  If we didn't do this, whenever a very shallow line such as the one in Figure \ref{fig:shallowAngles} was encountered, the ready function would tell us not to try to inpaint it.  This is because \eqref{eqn:ready} takes into account only information about pixels that have already been filled, but now $u_h({\bf x})$ is constructed using information from its neighbors in $\partial D^{(k)}_h$ that are being filled at the same time.  The modified ``ready'' function must reflect this.  The idea is to allow $\mbox{ready}({\bf x})$ to depend on pixels belonging to the current shell, but only if those pixels are also ``ready'' to be filled.  Thus, just like the colors of pixels in the current shell are now coupled together, the binary values $\{ \mbox{ready}({\bf x}) : {\bf x} \in \partial D^{(k)}_h\}$ are coupled as well.

First, let us fix some notation.  We define $\tilde{D}_h^{(k+1)} = D^{(k)}_h \backslash \partial D^{(k)}_h$, that is, the inpainting domain as it would be on the next step if all of $\partial D^{(k)}_h$ were filled (clearly $D^{(k)}_h \supseteq \tilde{D}^{(k)}_h$, as we never fill {\em more} than $\partial D^{(k)}_h$ on iteration $k$).  Then we denote the continuous version of $\tilde{D}^{(k+1)}_h$ by $\tilde{D}^{(k+1)}$, defined in the usual way as in the notational section.  Next, defining a pixel that has already been filled to be ``ready'' by default, and a ghost pixel to be ``ready'' if and only if the real pixels needed to define it are all  ``ready'', we write the modified confidence $C^*({\bf x})$ in terms of $\mbox{ready}({\bf y})$ for ${\bf y}$ neighboring ${\bf x}$ as
\begin{eqnarray*} \label{eqn:confidence2}
C^*({\bf x})&=&\frac{\sum_{{\bf y} \in \tilde{B}_{\epsilon,h}({\bf x}) \cap (\Omega \backslash \tilde{D}^{(k+1)}  )}w_{\epsilon}({\bf x},{\bf y})\mbox{ready}({\bf y})}{\sum_{{\bf y} \in \tilde{B}_{\epsilon,h}({\bf x}) }w_{\epsilon}({\bf x},{\bf y})} \\
& = &  C({\bf x}) + \frac{\sum_{{\bf y} \in \tilde{B}_{\epsilon,h}({\bf x}) \cap (D^{(k)} \backslash \tilde{D}^{(k+1)}  )}w_{\epsilon}({\bf x},{\bf y})\mbox{ready}({\bf y})}{\sum_{{\bf y} \in \tilde{B}_{\epsilon,h}({\bf x}) }w_{\epsilon}({\bf x},{\bf y})} \\
&= & C({\bf x}) +\Delta C({\bf x},{\bf r})
\end{eqnarray*}
where $C({\bf x})$ is defined by \eqref{eqn:confidence} and ${\bf r} : \partial D^{(k)}_h \rightarrow \{0,1\}$ is defined by 
$${\bf r}({\bf x}) = \mbox{ready}({\bf x}).$$
Finally the ``ready'' function at ${\bf x}$ is coupled to that of its neighbors by
\begin{equation} \label{eqn:ready2}
\mbox{ ready}({\bf x}) = 1( C({\bf x})+\Delta C({\bf x},{\bf r}) > c)
\end{equation}
Let $\partial^0_{\mbox{ready}} D^{(k)}_h$, $\partial_{\mbox{ready}} D^{(k)}_h$ denote the portion of $\partial D^{(k)}_h$ that is ``ready'' to be filled, according to standard and semi-implicit Guidefill respectively.  Then, since $\Delta C({\bf x},{\bf r}) \geq 0$, we clearly have
$$\partial^0_{\mbox{ready}} D^{(k)}_h \subseteq \partial_{\mbox{ready}} D^{(k)}_h \subseteq \partial D^{(k)}_h.$$
In other words, semi-implicit Guidefill will always agree with Guidefill that a pixel is ready to be filled, but may decide that other pixels Guidefill believes are not ready to be filled are actually ready.  In Algorithm 2 we propose a simple, iterative, and parallel algorithm to solve for $\partial_{\mbox{ready}} D^{(k)}_h$, at least approximately.  

\begin{algorithm}
\caption{Semi-Implicit Guidefill's ready function}
\begin{algorithmic} \label{alg:ready}
\State ${\bf r} : \partial D^{(k)}_h \rightarrow \{0,1\}$, initialized to $0$ everywhere. 
\State $\partial_{\mbox{ready}} D^{(k)}_h = \{ {\bf x} \in \partial D^{(k)}_h : r({\bf x}) = 1\}$.
\State $\mbox{maxIt}$ : maximum number of iterations.
\State $k=0$.
\While{$|\partial_{\mbox{ready}} D^{(k)}_h|$ keeps growing and $k<\mbox{maxIt}$}
\For{${\bf x} \in \partial D^{(k)}_h$}
\If{$C({\bf x}) + \Delta C({\bf x},{\bf r}) > c$}
\State ${\bf r}({\bf x}) = 1.$
\EndIf
\EndFor
\State $\partial_{\mbox{ready}} D^{(k)}_h = {\bf r}^{-1}(\{1\})$.
\State $k = k+1$
\EndWhile
\end{algorithmic}

\end{algorithm}

Note that in iteration $0$, we have $\Delta C({\bf x},{\bf r}) \equiv 0$, so that $C^*({\bf x})=C({\bf x})$ and hence after one iteration $\partial_{\mbox{ready}} D^{(k)}_h=\partial^0_{\mbox{ready}} D^{(k)}_h$ (that is, it is the same as the set of ready pixels determined by standard, non-implicit Guidefill).  In subsequent iterations $\partial_{\mbox{ready}} D^{(k)}_h$ can only grow, but it must stop growing within finitely many iterations, even if we set $\mbox{maxIt}=\infty$, since $|\partial D^{(k)}_h| < \infty$.  We will not attempt to prove that the $\partial_{\mbox{ready}} D^{(k)}_h$ output by Algorithm 2 is equal the $\partial_{\mbox{ready}} D^{(k)}_h$ defined by \eqref{eqn:ready2}.  For the sake of space, in the present work we do not analyze the benefits of this modified ``ready'' function - this will be the subject of future work.

\section{Analysis} \label{sec:analysis}

This section contains our core analysis, firstly of the convergence properties of semi-implicit Guidefill, then of the convergence of both Algorithm 1 and its semi-implicit extension to a continuum limit as $(h,\epsilon) \rightarrow (0,0)$ along the ray $\epsilon = r h$, then finally how under certain additional hypotheses we can also converge to the original high resolution and vanishing viscosity limit proposed by Bornemann and M\"arz ($h \rightarrow 0$ first and then $\epsilon \rightarrow 0$).  Next, in Section \ref{sec:consequences}, we will apply our results to explain some of the artifacts discussed in Section \ref{sec:shellBased}.  We begin with some symmetry assumptions that will hold from now on.

\subsection{Symmetry assumptions} \label{sec:symmetry}

We will assume throughout that the inpainting domain $D$ is the unit square \\$(0,1] \times (0,1]$ while the image domain is $\Omega=(0,1] \times (-\delta, 1]$ equipped with Dirichlet or periodic boundary conditions at $x=0$ and $x=1$, and no condition at $y=1$.  We denote the undamaged portion of the image by 
\begin{equation} \label{eqn:undamaged}
\mathcal U := (0,1] \times (-\delta,0] \qquad \mathcal U_h := \mathcal U \cap \field{Z}^2_h.
\end{equation}
We discretize $D = (0,1]^2$ as an $N \times N$ array of pixels $D_h=D \cap (h \cdot \field{Z}^2)$ with pixel width $h:=1/N$.  In order to ensure that the update formula \eqref{eqn:update} is well defined, we need $\epsilon+2h < \delta$, which we achieve by assuming  $h < \frac{\delta}{r+2}$ (this follows from the inclusion \eqref{eqn:ballContain} ).  We assume that the default onion shell ordering is used, so that 
$$\partial D^{(k)}_h = \{ (jh,kh) \}_{j=1}^N.$$
We also assume that the sets $A_{\epsilon,h}({\bf x})$ are translations of one another, and the weights $w_{\epsilon}({\bf x},{\bf y})$ depend only on $\frac{{\bf y}-{\bf x}}{\epsilon}$, that is
$$w_{\epsilon}({\bf x},{\bf y})=\hat{w}\left(\frac{{\bf y}-{\bf x}}{\epsilon}\right)$$
For coherence transport and Guidefill this means that the guidance direction ${\bf g}$ is a constant.

\begin{remark}
We make the above assumptions not because we believe they are necessary, but because they enable us to make our analysis as simple as possible while still capturing the phenomena we would like to capture.  In particular, the above two assumptions on the weights $w_{\epsilon}$ and neighborhood $A_{\epsilon,h}({\bf x})$ ensure that the matrix $\mathcal L$ given by \eqref{eqn:Linearboundary} is either Toeplitz or circulant (depending on the boundary conditions), and also ensures that the random walk we connect Algorithm 1 to in Section \ref{sec:continuumLimit1} has i.i.d. (independent identically distributed) increments.  Without these simplifications, our already lengthy analysis would become even more technical.  Numerical experiments in Section \ref{sec:numerics} suggest that these assumptions can be weakened, but proving this is beyond the scope of the present work.
\end{remark}

These assumptions give us a high level of symmetry with which we may rewrite the update formula \eqref{eqn:update} of Algorithm 1 in the generic form
\begin{equation} \label{eqn:generic}
u_h({\bf x}) = \frac{\sum_{{\bf y} \in a^*_r} w_r({\bf 0},{\bf y}) u_h({\bf x}+h{\bf y})}{\sum_{{\bf y} \in a^*_r} w_r({\bf 0},{\bf y})}
\end{equation}
where
$$a^*_r = \left( \frac{1}{h}A_{\epsilon,h}({\bf 0}) \backslash \{ {\bf 0} \}\right) \cap \{ y \leq \delta\},$$
and $\delta = -1$ for the direct method, while $\delta = 0$ for the semi-implicit extension.  In particular, for coherence transport we have $a^*_r=b^-_r$ for the direct method and $a^*_r = b^0_r$ for the extension, where 
\begin{eqnarray} 
b^0_r &:=& \{ (n,m) \in \field{Z}^2 :  0<n^2+m^2 \leq r^2, m \leq 0 \}. \label{eqn:b0r}\\
b^-_r &:=& \{ (n,m) \in \field{Z}^2 : n^2+m^2 \leq r^2,  m \leq -1 \}. \label{eqn:bmr}
\end{eqnarray}
Similarly, for Guidefill, we have $a^*_r=\tilde{b}^-_r$ for the direct method and $a^*_r=\tilde{b}^0_r$ for the semi-implicit extension, where
\begin{eqnarray*} 
\tilde{b}^0_r &:=& \{ n\hat{\bf g}+m\hat{\bf g}^{\perp} : (n,m) \in \field{Z}^2,  0<n^2+m^2 \leq r^2,  n\hat{\bf g}\cdot e_2+m\hat{\bf g}^{\perp} \cdot e_2 \leq 0 \}  \label{eqn:bt0r}\\
\tilde{b}^-_r &:=& \{ n\hat{\bf g}+m\hat{\bf g}^{\perp} : (n,m) \in \field{Z}^2,  n^2+m^2 \leq r^2,  n\hat{\bf g}\cdot e_2+m\hat{\bf g}^{\perp} \cdot e_2 \leq -1 \}, \label{eqn:btmr}
\end{eqnarray*}
and $\hat{\bf g} := {\bf g}/\|{\bf g}\|$ (if ${\bf g} = {\bf 0}$ we set $\tilde{b}^-_r = b^-_r$).  The sets $b^-_r$, $b^0_r$, $\tilde{b}^-_r$, $\tilde{b}^0_r$ may be visualized by looking at the portion of Figure \ref{fig:ballRotate}(a)-(b) on or below the lines $y=0$ and $y=-1$ respectively.  Also important are the dilated sets $\bar{b}^0_r=D(b^0_r) \cap \{ y \leq 0\}$ and $\bar{b}^-_r=D(b^-_r) \cap \{ y \leq -1\}$.  These sets are given explicitly by
\begin{eqnarray} 
\bar{b}^0_r &:=& \{ (n+\Delta n,m+\Delta m) \nonumber \\
& : &  (n,m)\in b^0_r, (\Delta n,\Delta m) \in \{-1,0,1\} \times \{-1,0,1\}, m+\Delta m \leq 0\} \label{eqn:bb0r} \\
\bar{b}^-_r &:=& \{ (n+\Delta n,m+\Delta m) \nonumber \\
& : & (n,m) \in b^-_r, (\Delta n,\Delta m) \in \{-1,0,1\} \times \{-1,0,1\}, m+\Delta m \leq -1\}.  \label{eqn:bbmr}
\end{eqnarray}
This is because the universal support property \eqref{eqn:universalSupport} gives us the inclusion
\begin{equation} \label{eqn:universalContainment}
\mbox{Supp}(a^*_r)  \subseteq \begin{cases} \bar{b}^-_r \subseteq b^-_{r+2}  & \mbox{ if we use the direct form of Algorithm 1.} \\ \bar{b}^0_r \subseteq b^0_{r+2}  & \mbox{ if we the semi-implicit extension,} \end{cases}
\end{equation}
which will be critical later.  The sets $\bar{b}^-_r$ and $\bar{b}^0_r$ are illustrated in Figure \ref{fig:barSets}.

\begin{figure}
\centering
\begin{tabular}{cc}
\subfloat[Illustration of $\bar{b}^-_r$ for $r=3$.]{\includegraphics[width=.45\linewidth]{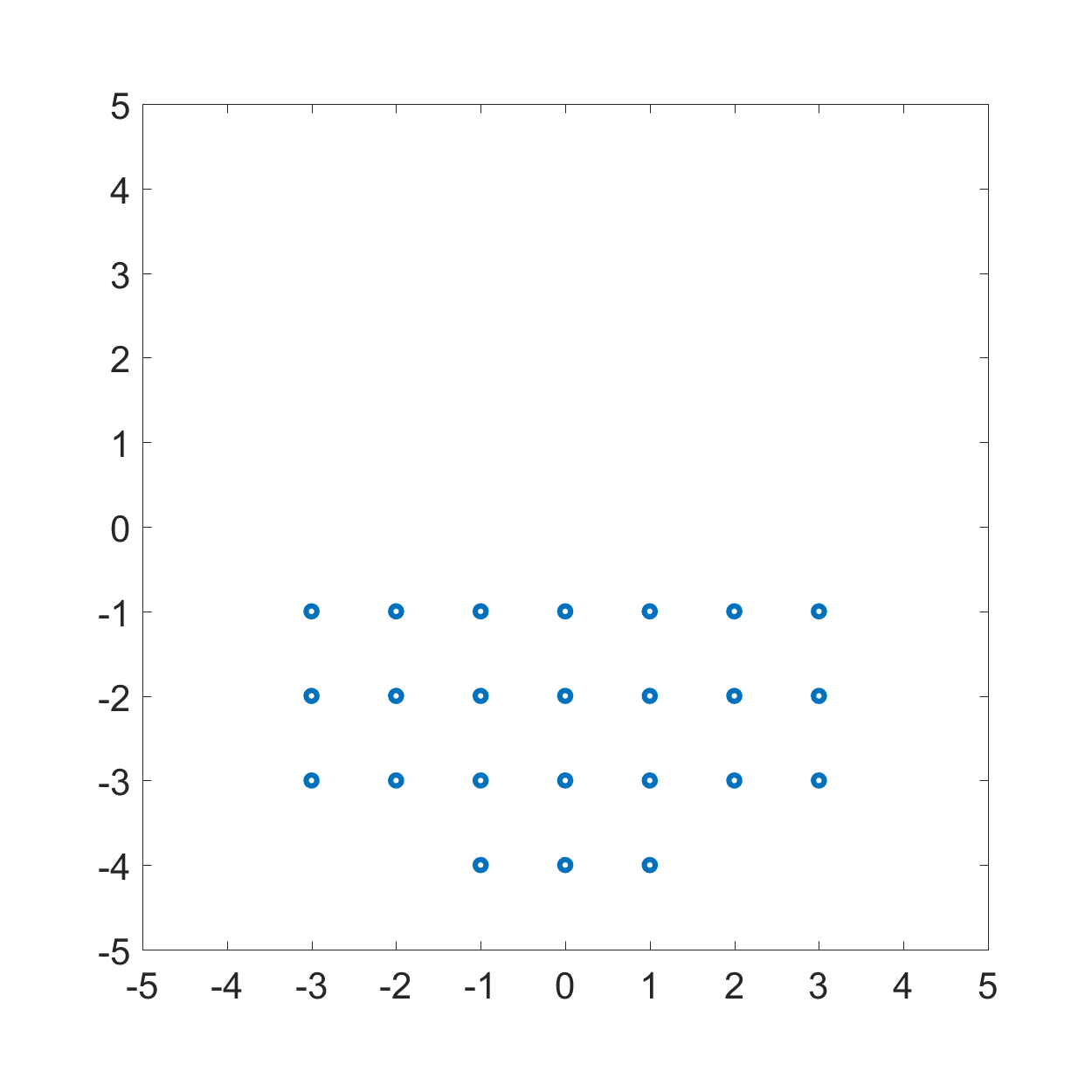}} & 
\subfloat[Illustration of $\bar{b}^0_r$ for $r=3$.]{\includegraphics[width=.45\linewidth]{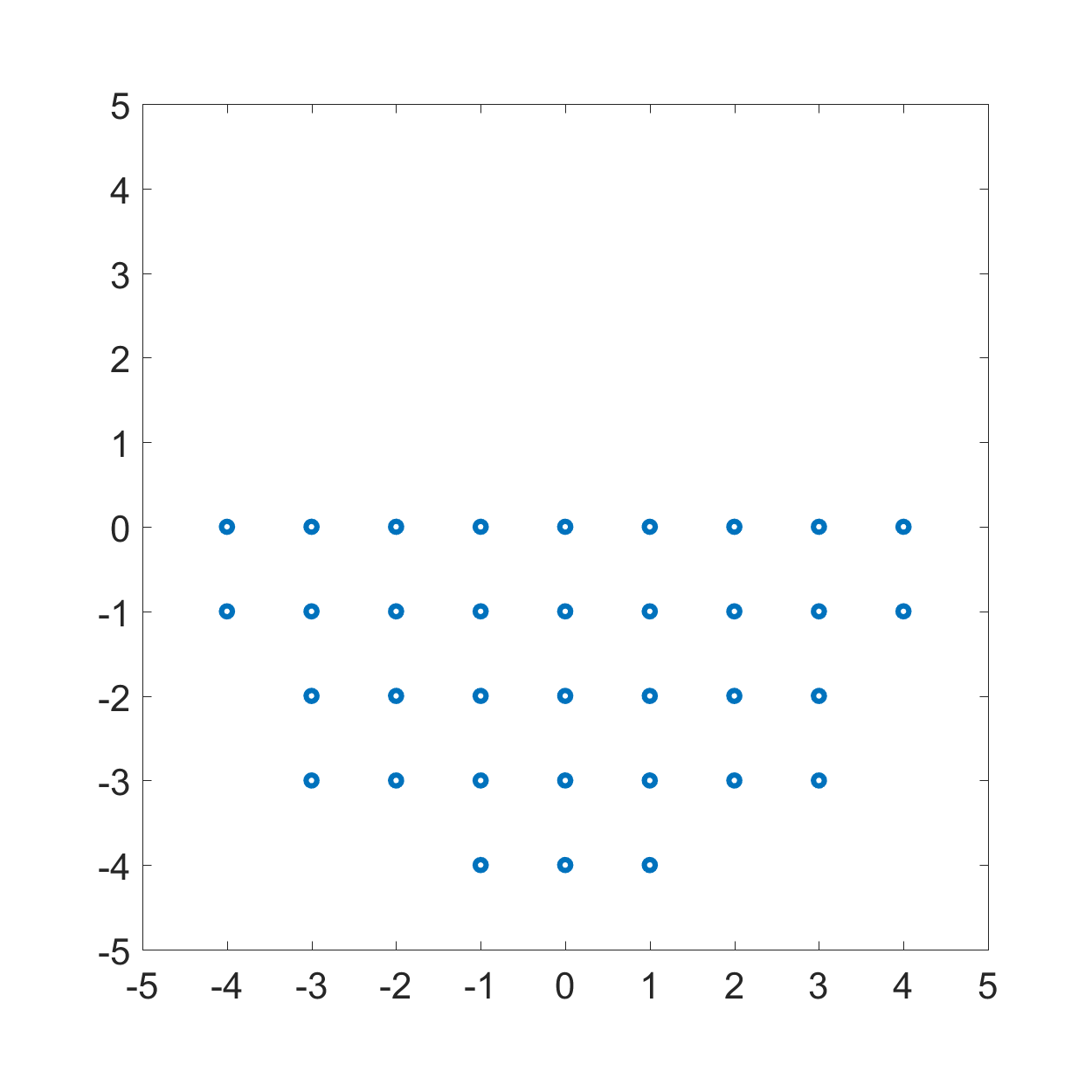}} \\
\end{tabular}
\caption{{\bf Visualization $\bar{b}^-_r$ and $\bar{b}^0_r$:}  Here we illustrate the sets \\$\bar{b}^-_r:=D(b^-_r) \cap \{ y \leq -1\}$ and $\bar{b}^0_r:=D(b^0_r) \cap \{ y \leq 0\}$ in the case $r=3$.  These sets are defined explicitly in \eqref{eqn:bmr} and \eqref{eqn:bb0r}, and $D$ is the dilation operator defined in the notation section.  These sets are important because depending on whether we use the direct form of Algorithm 1 or its semi-implicit extension, $\mbox{Supp}(a^*_r)$ is always contained in one or the other.  See \eqref{eqn:universalContainment} in the text.}
\label{fig:barSets}
\end{figure}

\vskip 2mm

\begin{definition} \label{def:stencil}
We call the set $a^*_r$ and the weights $\{ w_r({\bf 0},{\bf y}) : {\bf y} \in a^*_r\}$ the {\em stencil} and {\em stencil weights} of a method.  The {\em center of mass} of $a^*_r$ is defined in the following two equivalent ways:
$$\mbox{C.M.}=\frac{\sum_{{\bf y} \in a^*_r} w_r(0,{\bf y}){\bf y}}{\sum_{{\bf y} \in a^*_r} w_r(0,{\bf y})}=\frac{\sum_{{\bf y} \in \mbox{Supp}(a^*_r)} \tilde{w}_r(0,{\bf y}){\bf y}}{\sum_{{\bf y} \in \mbox{Supp}(a^*_r)} \tilde{w}_r(0,{\bf y})}.$$
Here $\tilde{w}_r$ denote the equivalent weights from Section \ref{sec:fiction}, and the above identity follows from \eqref{eqn:totalMass} and \eqref{eqn:centerMass}.
The center of mass of $a^*_r$ will play a critical role both in the continuum limit of Algorithm 1 (where it is the transport direction of the resulting transport PDE) and in the connection to random walks (where, after multiplication by $h$, it is the mean of the increments of the walk).
\end{definition}

\vskip 2mm

Under these assumptions, the matrix $\mathcal L$ from \eqref{eqn:Linearboundary} is independent of $k$ (that is, we solve the same linear system for every shell), and moreover $\mathcal L$ becomes a Toeplitz matrix (Dirichlet boundary conditions) or circulant matrix (periodic boundary conditions).   For a given pixel ${\bf x}$ at least $r+2$ pixels away from the boundary at $x=0$ and $x=1$, that is ${\bf x} \in \{ (jh,kh) : r+2 \leq j \leq N-r-2\}$, it takes on the form
\begin{equation} \label{eqn:Lsimplified}
(\mathcal L{\bf u})({\bf x}) =\left(1-\frac{\tilde{w}_r({\bf 0},{\bf 0})}{W}\right)u_h({\bf x}) - \sum_{{\bf y} \in s_r \backslash \{ {\bf 0} \}}\frac{\tilde{w}_{r}({\bf 0},{\bf y})}{W}u_h({\bf x}+h{\bf y}),
\end{equation}
where by \eqref{eqn:universalContainment} we have
$$s_r = \mbox{Supp}(a^*_r) \cap ( \field{Z} \times \{ 0 \}  ) \subseteq \{-(r+2)e_1,-(r+1)e_1,\ldots,(r+1)e_1,(r+2)e_1\}.$$
If ${\bf x}$  is {\em not} at least $r+2$ pixels away from the boundaries, then the formula changes in the usual way for Toeplitz and circulant matrices - we assume the reader is familiar with this and no further discussion is needed.  Under the same assumptions the vector ${\bf f}$ becomes
$${\bf f} = \sum_{{\bf y} \in \mbox{Supp}(a^*_r) \backslash ( \field{Z} \times \{ 0\} ) }\frac{\tilde{w}_{r}({\bf 0},{\bf y})}{W}u_h({\bf x}+h{\bf y})$$
where $\mbox{Supp}(a^*_r) \backslash ( \field{Z} \times \{ {\bf 0}\} ) \subseteq b^-_{r+2}$.  We also define
\begin{equation} \label{eqn:leakedMass}
\tilde{W} := \sum_{j=-r-2}^{r+2} \tilde{w}_r({\bf 0},je_1) \quad \mbox{ and } \quad \tilde{w}_{0,0}:=\tilde{w}_r({\bf 0},{\bf 0}).
\end{equation}
For a given point ${\bf x} \in \partial D^{(k)}_h$, (again, assuming ${\bf x}$ is far enough from the boundary) the ratio $\frac{\tilde{W}}{W}$ gives the fraction of the mass of the stencil (Definition \ref{def:stencil}) centered at ${\bf x}$ that gets ``leaked'' to the unknown pixels in $\partial D^{(k)}_h$, while $\frac{\tilde{w}_{0,0}}{W}$ gives the fraction that gets leaked to ${\bf x}$.  Together, these give a measure of the diagonal dominance of $\mathcal L$, as we have
$$\frac{\sum_{j \neq i} |\mathcal L_{ij}|}{|\mathcal L_{ii}|} = \frac{\tilde{W}-\tilde{w}_{0,0}}{W-\tilde{w}_{0,0}}.$$
The smaller this ratio is, the stronger the diagonal dominance of $\mathcal L$, and the faster damped Jacobi and SOR can be expected to converge - see Proposition \ref{prop:rates} for explicit formulas.

For semi-implicit Guidefill with guidance direction ${\bf g}=(\cos\theta,\sin\theta)$, it can be shown that $\mathcal L$ becomes a lower triangular matrix in the limit $\mu \rightarrow \infty$, provided we order unknowns left to right if $\cos\theta > 0$ and right to left otherwise (see Appendix \ref{app:rates}).  This gives us a hint that Gauss-Seidel and SOR might be very effective for the solution of \eqref{eqn:Linearboundary} in this case, and indeed Proposition \ref{prop:rates} confirms this.

\subsection{Convergence rates of damped Jacobi and SOR for semi-implicit Guidefill} \label{sec:guidefillGS}

Here we derive tight bounds on the convergence rates of damped Jacobi and SOR for solving \eqref{eqn:Linearboundary} in the semi-implicit extension of Guidefill described in Section \ref{sec:semiImplicit}, under the symmetry assumptions discussed above, and in the limit $\mu \rightarrow \infty$ (recall that $\mu$ is the parameter from the weights \eqref{eqn:weight} controlling the extent to which weights are biased in favor of pixels in the directions $\pm {\bf g}$).  We will prove that in each case the parameter value $\omega = 1$ is optimal, but also pay special attention to the case $\omega = \omega^*$ given by \eqref{eqn:specialomega}, since this is the value of $\omega$ that our proposed implementation of the semi-implicit extension in Algorithm 1 uses.  We consider general $\omega$ mainly in order to demonstrate that the choice $\omega = \omega^*$, while not optimal, is close enough to optimal not to matter in practice.

We will assume that $D=(0,1]^2$ with Dirichlet boundary conditions, as this simplifies our analysis of SOR - for damped Jacobi, we could just as easily have assumed periodic boundary conditions.  We will measure convergence rates with respect to the induced infinity norm, which obeys the identity
\begin{equation} \label{eqn:infinityNormDefinition}
\|A\|_{\infty} = \max_{i=1}^N \sum_{j=1}^N |a_{ij}|
\end{equation}
for any $N \times N$ matrix $A$.  Note that the iterates of the error ${\bf e}^{(0)}$, ${\bf e}^{(1)}$, $\ldots$ associated with any stationary iterative method with iteration matrix $M$ obey the bounds
\begin{equation} \label{eqn:convergenceBound}
\|{\bf e}^{(n)}\| \leq \|M\|^n\|{\bf e}^{(0)}\| \quad \mbox{ and } \quad R({\bf e}) := \sqrt[\leftroot{-3}\uproot{3}n]{\frac{\|{\bf e}^{(n)}\|}{\|{\bf e}^{(0)}\|}} \leq \|M\|
\end{equation}
for {\em any} vector norm $\|\cdot\|$ and induced matrix norm.  We will be interested in these identities in the particular case that the vector norm is $\|\cdot \|_{\infty}$, and the stationary iterative method is damped Jacobi or SOR.  Here 
$${\bf e}^{(n)} := u_h - u^{(n)}_h$$ 
denotes the difference between the exact solution to \eqref{eqn:Linearboundary}, found by first solving \eqref{eqn:Linearboundary} to machine precision, with the approximate solution at iteration $n$.

\vskip 2mm

\begin{proposition} \label{prop:rates}
Suppose semi-implicit Guidefill with guidance direction \\${\bf g}=(\cos\theta,\sin\theta)$ is used to inpaint the square $[0,1)^2$ under the assumptions above, using either damped Jacobi or SOR to solve \eqref{eqn:Linearboundary}.  Suppose that in the case of SOR, $\partial D^{(k)}_h= \{ kh\}_{k \in \field{Z}}$ is ordered from left to right if $\cos\theta \geq 0$ and from right to left otherwise.  Let $\mathcal L$ be as in \eqref{eqn:Matrix} and define $\mathcal L = D-L-U$, where $D$, $-L$, and $-U$ are the diagonal, strictly lower triangular, and strictly upper triangular parts of $\mathcal L$ respectively.  Let $J_{\omega}$ and $G_{\omega}$ denote the iteration matrices of damped Jacobi and SOR respectively with relaxation parameter $\omega$, that is
$$J_{\omega} = I - \omega D^{-1} \mathcal L \qquad G_{\omega} = (I-\omega D^{-1}L)^{-1}((1-\omega)I+D^{-1}U).$$
Let $r=\epsilon/h$ and define $\theta_c = \arcsin(1/r)$.  Define, for $\theta_c \leq \theta \leq \pi - \theta_c$, \\$j^* = \lfloor \frac{1}{\sin\theta} \rfloor \leq r$.  Let $W$ be the total weight \eqref{eqn:2ndway}, let $\tilde{W}$ and $\tilde{w}_{0,0}$ be as in \eqref{eqn:leakedMass}.
Then, in the limit as $\mu \rightarrow \infty$, we have 
\begin{eqnarray*}
W &=& \sum_{j=1}^r \frac{1}{j}, \qquad \tilde{w}_{0,0}=(1-\sin\theta)(1-|\cos\theta|)\\
\tilde{W} &=&
\begin{cases} \sum_{j=1}^r \frac{1}{j} - r \sin\theta & \mbox{ if } \theta \in (0,\theta_c] \cup [\pi-\theta_c,\pi) \\ \sum_{j=1}^{j^*} \frac{1}{j} - j^* \sin\theta & \mbox{ if } \theta \in (\theta_c , \pi - \theta_c) 
\end{cases}.
\end{eqnarray*}
and
\begin{eqnarray*}
\|J_{\omega}\|_{\infty} &=& |1-\omega|+\omega \left(  \frac{\tilde{W}-\tilde{w}_{0,0}}{W-\tilde{w}_{0,0}} \right) \quad \mbox{ for } \omega \in (0,2) \\
\|G_{\omega}\|_{\infty} &=& \frac{|1-\omega|}{1-\omega \frac{\tilde{W}-\tilde{w}_{0,0}}{W-\tilde{w}_{0,0}}} \quad \mbox{ for } \omega \in (0,1], \\
\end{eqnarray*}
where $\|\cdot\|_{\infty}$ is the induced infinity matrix norm \eqref{eqn:infinityNormDefinition}.  The optimal $\omega \in (0,2)$ is in both cases independent of $\theta$ and equal to $\omega = 1$, where we obtain
\begin{eqnarray*}
\|J_1\|_{\infty} &=&  \begin{cases} 1 - \frac{r\sin\theta}{\sum_{j=1}^r \frac{1}{j}-(1-\sin\theta)(1-|\cos\theta|)}& \mbox{ if } \theta \in (0, \theta_c] \cup [\pi-\theta_c,\pi) \\ 1 - \frac{\sum_{j=j^*+1}^r \frac{1}{j} + j^*\sin\theta}{\sum_{j=1}^r \frac{1}{j}-(1-\sin\theta)(1-|\cos\theta|)} & \mbox{ if } \theta \in (\theta_c, \pi - \theta_c) .\end{cases} \\
\|G_1\|_{\infty} &=& 0.\\
\end{eqnarray*}
\end{proposition}
\begin{proof}  
Appendix \ref{app:rates}.
\qed
\end{proof}

\vskip 2mm

\begin{corollary} \label{corollary:actualRates}
For the special case of
$$\omega^* = \left(1 - \frac{\tilde{w}_{0,0}}{W}\right) = \left(1-\frac{(1-\sin\theta)(1-|\cos\theta|)}{ \sum_{j=1}^r \frac{1}{j}}\right)$$
that is, the parameter value equivalent to running Algorithm 1 with ``semiImplicit'' set to true, we obtain
\begin{eqnarray*}
\|G_{\omega^*}\|_{\infty} &=& \frac{\tilde{w}_{0,0}}{W-\tilde{W}_k+\tilde{w}_{0,0}}\\
&=&\begin{cases} \frac{(1-\sin\theta)(1-|\cos\theta|)}{r\sin\theta+(1-\sin\theta)(1-|\cos\theta|)}   & \mbox{ if } \theta \in (0, \theta_c] \cup [\pi-\theta_c,\pi) \\ \frac{(1-\sin\theta)(1-|\cos\theta|)}{\sum_{j=j^*+1}^r \frac{1}{j} + j^*\sin\theta+(1-\sin\theta)(1-|\cos\theta|)} & \mbox{ if } \theta \in (\theta_c, \pi - \theta_c) .\end{cases}  \\
\|J_{\omega^*}\|_{\infty} & = &  \frac{\tilde{W}_k}{W}=\begin{cases} 1 - \frac{r\sin\theta}{\sum_{j=1}^r \frac{1}{j}}& \mbox{ if } \theta \in (0, \theta_c] \cup [\pi-\theta_c,\pi) \\ 1 - \frac{\sum_{j=j^*+1}^r \frac{1}{j} + j^*\sin\theta}{\sum_{j=1}^r \frac{1}{j}} & \mbox{ if } \theta \in (\theta_c, \pi - \theta_c) .\end{cases}
\end{eqnarray*}
See Figure \ref{fig:rates} for a plot of $\|G_{\omega^*}\|_{\infty}$ and $\|J_{\omega^*}\|_{\infty}$ as a function of $\theta$ for $r=3$.
\end{corollary}
\begin{proof}
This follows from direct substitution of $\omega^*$ given by \eqref{eqn:specialomega} into Proposition \ref{prop:rates}.
\qed
\end{proof}

\begin{figure}
\centering
\begin{tabular}{cc}
\subfloat[]{\includegraphics[width=.44\linewidth]{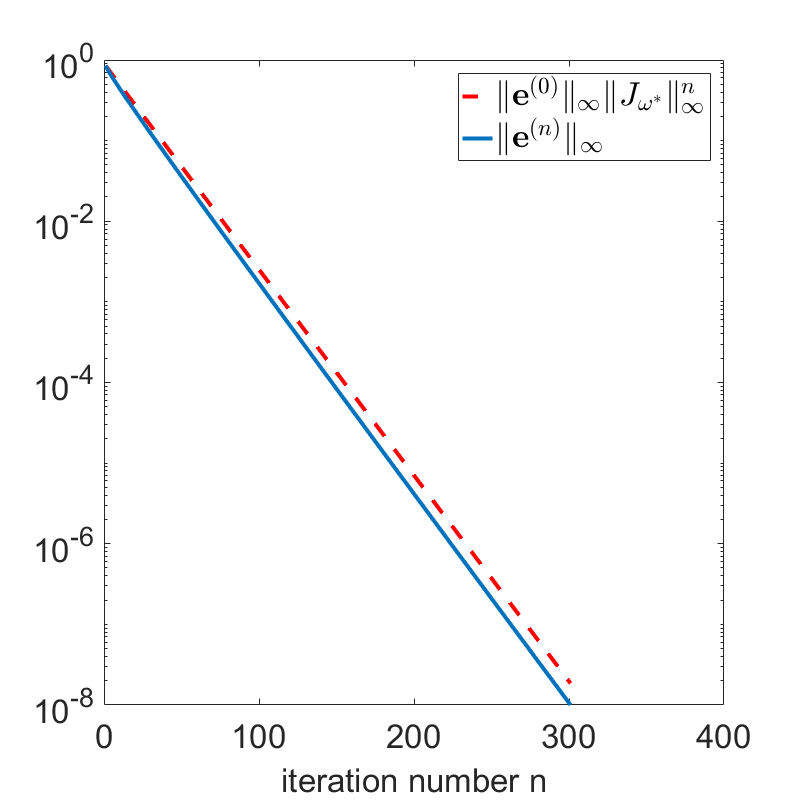}} & 
\subfloat[]{\includegraphics[width=.44\linewidth]{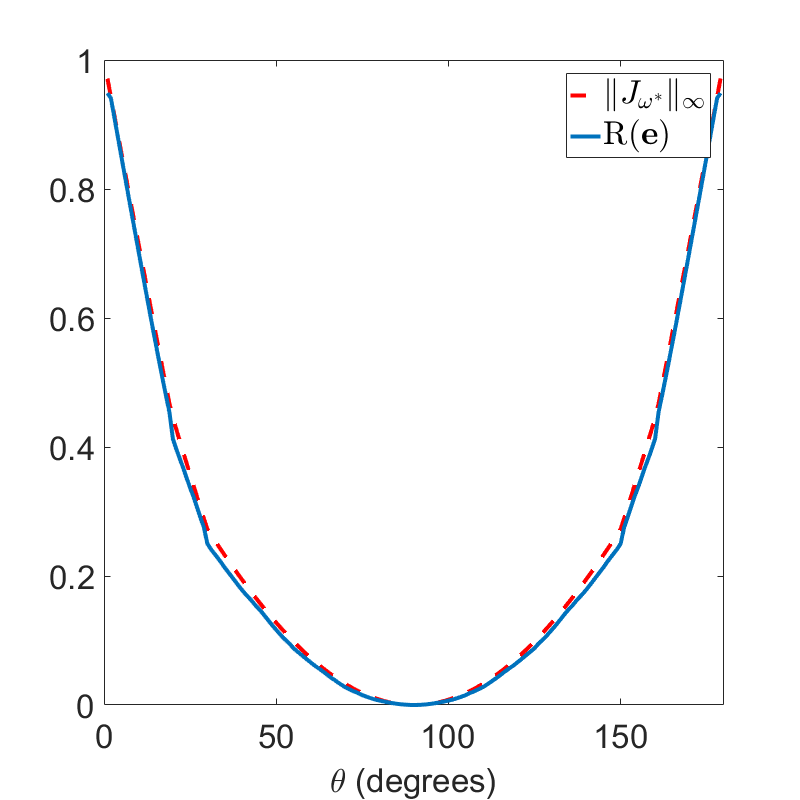}} \\
\subfloat[]{\includegraphics[width=.44\linewidth]{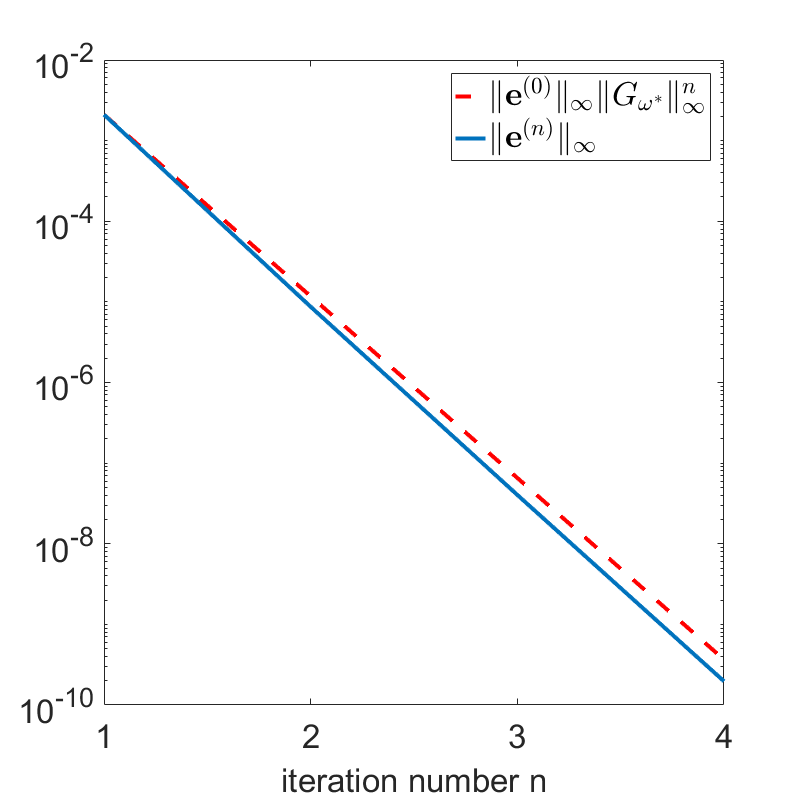}} & 
\subfloat[]{\includegraphics[width=.44\linewidth]{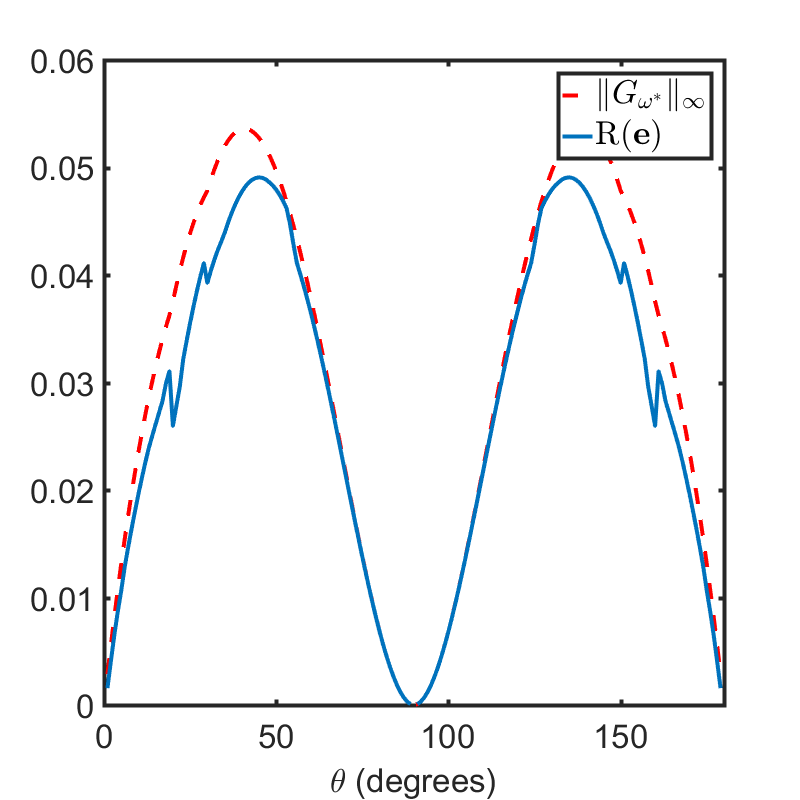}}
\end{tabular}
\caption{{\bf Convergence rates of damped Jacobi and SOR for semi-implicit Guidefill:}  As noted in Section \ref{sec:semiImplicit}, the implementation of semi-implicit Guidefill outlined in Algorithm 1 (blue text) is equivalent to solving the linear system \eqref{eqn:Linearboundary} iteratively using damped Jacobi (parallel implementation) or SOR (sequential implementation), with relaxation parameter $\omega^*$ given by \eqref{eqn:specialomega}.  Here we compare the experimentally measured convergence rates of this implementation ($r=3$, ${\bf g}=(\cos\theta,\sin\theta)$ and $\mu = 100$) with the theoretical bounds on $\|J_{\omega^*}\|_{\infty}$ and $\|G_{\omega^*}\|_{\infty}$ from Corollary \ref{corollary:actualRates}.  Specifically, (a) and (c) confirm experimentally the first bound in \eqref{eqn:convergenceBound} in the cases $M=J_{\omega^*}$ and $M=G_{\omega^*}$, that is, damped Jacobi and SOR with relaxation parameter $\omega^*$, for the case $\theta = 2^{\circ}$.  The inpainting problem in this case is the same as in Figure \ref{fig:shallowAngles}(a), and all the parameters of semi-implicit Guidefill are the same.  The ``exact'' solution $u_h$ was found by first solving \eqref{eqn:Linearboundary} to machine precision.  In each case, we measured convergence rates only within the first ``shell'' of the inpainting problem.  Next, (b)-(d) confirm experimentally the second bound in \eqref{eqn:convergenceBound}, as a function of $\theta$.  The inpainting problem is the same as the one in Figure \ref{fig:specialDir}(a), and all parameters are the same.  In this case we vary $\theta$ from $1^{\circ}$ up to $179^{\circ}$ in increments of one degree, in each case iteratively solving \eqref{eqn:Linearboundary} (again, only for the first shell), computing $R({\bf e})$, and comparing with $\|J_{\omega^*}\|_{\infty}$ and $\|G_{\omega^*}\|_{\infty}$.  Note the excellent performance of SOR in comparison with damped Jacobi.}
\label{fig:rates}
\end{figure}

\vskip 1mm

\begin{remark}  Although our choice of $\omega^*$ is non-optimal, it is convenient to implement and the difference in performance is negligible.  For example, even for the optimal value $\omega = 1$, for Jacobi we have $\|J_1\|_{\infty} \rightarrow 1$ as $\theta \rightarrow 0$ or $\theta \rightarrow \pi$.    At the same time, for SOR we have  $\|G_{\omega^*}\|_{\infty} \leq 0.06$ (Figure \ref{fig:rates}) for all $\theta$ and all $r \geq 0$ (it follows trivially from Corollary \ref {corollary:actualRates} that $\|G_{\omega^*}\|_{\infty}$ is decreasing function of $r$ for each fixed $\theta$) - while not quite as good as the optimal value $\|G_1\|_{\infty} \equiv 0$, it is more than satisfactory and moreover, we have $\|G_{\omega^*}\|_{\infty} \rightarrow 0$ as $\theta \rightarrow 0$ or $\theta \rightarrow \pi$.  As we will see in Section \ref{sec:kinkingAndConvex}, it is precisely these shallow angles where the semi-implicit extension has an advantage over the direct method.  Unfortunately, however, Guidefill was designed to be a parallel algorithm but SOR is a sequential.  While ideally we would like a method to solve the linear system \eqref{eqn:Linearboundary} that is both fast and parallel, this is beyond the scope of the present work.
\end{remark}

\subsection{Convergence of Algorithm 1 to a continuum limit} \label{sec:continuumLimit1}

Our objective in this section is to prove that the direct and semi-implicit forms of Algorithm 1 both converge to a continuum limit $u$ when we take $(h,\epsilon) \rightarrow (0,0)$ along the path $\epsilon = rh$.  Before we can do that, we need to define what that limit is, in what sense $u_h$ converges to it, and to lay out our assumptions regarding the regularity of $u_0$.  We also describe a property which, if satisfied, will allow us to connect our limit with the one studied by Bornemann and M\"arz in \cite{Marz2007} (we also show that all methods considered in this paper have this property).  When $u_0$ is smooth, convergence is straightforward to prove - we did so in \cite[Theorem 1]{Guidefill} for the direct form of Algorithm 1 using a simple argument based on Taylor series.  However, in this paper we are interested in a more general setting where $u_0$ may not be smooth and Taylor series may not be available.  Instead of Taylor series, our main tool will be a connection to stopped random walks.  We will explain this connection briefly before continuing on to the main result of this section.  Throughout this section we assume for convenience that $u_0$ and $u_h$ are both greyscale images, that is $u_0 : \Omega \backslash D \rightarrow \field{R}$, $u_h : \Omega_h \rightarrow \field{R}$.  In the color case $u_0 : \Omega \backslash D \rightarrow \field{R}^d$, $u_h : \Omega_h \rightarrow \field{R}^d$ one may easily show that our results hold channel-wise.

\vskip 2mm

\noindent {\bf Definition of the continuum limit.}  We wish to prove convergence of the direct and semi-implicit forms of Algorithm 1 to a continuum limit $u$ given by the transport equation
\begin{equation} \label{eqn:transport}
\nabla u \cdot {\bf g}_r^* = 0, \qquad u \Big |_{y = 0} = u_0 \Big |_{y = 0} \qquad u\Big |_{x = 0}=u\Big |_{x = 1}
\end{equation}
Because of the assumptions we have made in Section \ref{sec:symmetry}, ${\bf g}_r^*$ will turn out to be a constant equal to the center of mass of the stencil $a^*_r$ with respect to the stencil weights $\{ w_r({\bf 0},{\bf y}) : {\bf y} \in a^*_r \}$ (Definition \ref{def:stencil}), that is
\begin{equation} \label{eqn:transportIsCenter}
{\bf g}_r^* = \frac{\sum_{{\bf y} \in a^*_r} w_r(0,{\bf y}){\bf y}}{\sum_{{\bf y} \in a^*_r} w_r(0,{\bf y})}.
\end{equation}
As we will allow discontinuous boundary data $u_0$, the solution to \eqref{eqn:transport} must be defined in a weak sense.  However, since ${\bf g}_r^*$ is a constant, this is simple.  So long as ${\bf g}_r^* \cdot e_2 \neq 0$, we simply define the solution to the transport problem \eqref{eqn:transport} to be 
\begin{equation} \label{eqn:weak}
u({\bf x}) = u_0(\Pi_{\theta_r^*}({\bf x})), \quad \mbox{ where } \quad \Pi_{\theta_r^*}(x,y)=(x-\cot(\theta_r^*)y \mbox{ mod } 1,0).
\end{equation}
We call  $\Pi_{\theta_r^*}: D \rightarrow \partial D$ the transport operator associated with \eqref{eqn:weak}.  The $\mbox{mod}$ $1$ is due to our assumed periodic boundary conditions and
$$\theta_r^* = \theta({\bf g}_r^*) \in (0,\pi)$$
is the counterclockwise angle between the x-axis and the line $L_{{\bf g}_r^*} := \{ \lambda {\bf g}_r^* : \lambda \in \field{R} \}$.

\vskip 2mm

\noindent {\bf Modes of convergence (discrete $L^p$ norms).}  Given $f_h : D_h \rightarrow \field{R}$, we introduce the discrete $L^{\infty}$ norm
\begin{equation} \label{eqn:Linfinity}
\|f_h \|_{\infty} := \max_{{\bf x} \in D_h } |f_h({\bf x})|
\end{equation}
and the discrete $L^p$ norm
\begin{equation} \label{eqn:Lp}
\|f_h \|_p := \big(\sum_{{\bf x} \in D_h } |f_h({\bf x})|^p h^2\big)^{\frac{1}{p}}.
\end{equation}
These norms will be used to measure the convergence of $u_h$ to the continuum limit $u$.

\vskip 2mm

\noindent{{\bf Convergence of center of mass.}  There is one additional property which, if satisfied, will enable us to connect our limit with the one studied by Bornemann and M\"arz.  Specifically, if a method obeys 
\begin{equation} \label{eqn:centerOfMass}
\lim_{r \rightarrow \infty} \frac{{\bf g}^*_r}{r} \rightarrow {\bf g}^* \in \field{R}^2 \qquad \mbox{ with rate at least } O(r^{-q}) \quad \mbox{ for some } q > 0,
\end{equation}
then it is possible to make sense of a (generalized version) of Bornemann and M\"arz's limit.  Moreover, as we will see in Section \ref{sec:marzLimit}, the vector ${\bf g}^*$ {\em is} the transport direction obtained in their limit.  When $A_{\epsilon,h}({\bf x}) = B_{\epsilon,h}({\bf x})$ as in Telea's algorithm and coherence transport, or $A_{\epsilon,h}({\bf x}) = \tilde{B}_{\epsilon,h}({\bf x})$ as in Guidefill, it is easy to see that this condition is satisfied.  In fact, if $a^*_r \in \{ b^-_r, b^0_r, \tilde{b}^-_r, \tilde{b}^0_r\}$ we have
$$\lim_{r \rightarrow \infty} \frac{{\bf g}^*_r}{r} = \lim_{r \rightarrow \infty} \frac{\sum_{{\bf y} \in a^*_r} \hat{w}\left(0,\frac{\bf y}{r}\right)\frac{\bf y}{r}\frac{1}{r}}{\sum_{{\bf y} \in a^*_r} \hat{w}\left(0,\frac{\bf y}{r}\right)\frac{1}{r}}\rightarrow \frac{\int_{{\bf y} \in B^-_1({\bf 0})}\hat{w}(0,{\bf y}){\bf y} d{\bf y}}{\int_{{\bf y} \in B^-_1({\bf 0})}\hat{w}(0,{\bf y}) d{\bf y}}$$
(where $\hat{w}$ is the function we assumed existed in \eqref{eqn:scalinglaw}), because the LHS can be interpreted as a Riemann sum for the RHS.  We also know, by an elementary argument in quadrature, that the convergence rate is $O(r^{-1})$ or better, so $q=1$ in this case.  In fact, it is straightforward to show that the integral on the RHS of the above equation is the same for coherence transport, Guidefill, and semi-implicit Guidefill.  In other words, the original limit of Bornemann and M\"arz is the same for all three of these methods.  However, our limit predicts that they behave very differently - see Section \ref{sec:kink3main}.  
\vskip 2mm

\begin{figure}
\centering\includegraphics[width=1\linewidth]{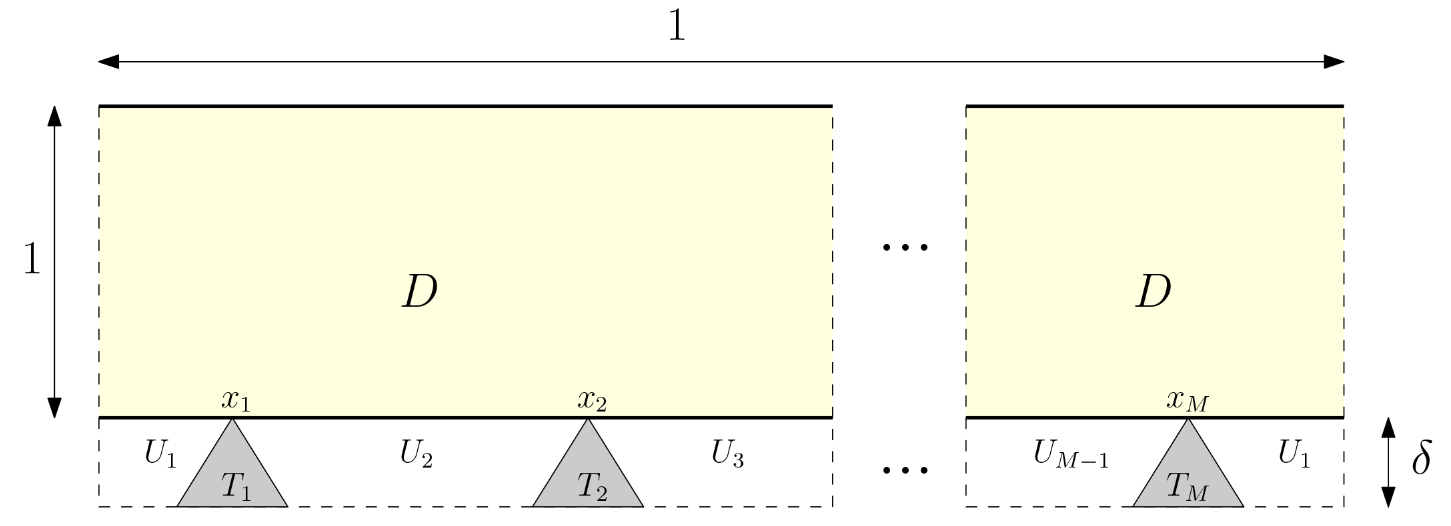}
\caption{{\bf Regularity assumptions:}  The inpainting domain $D=(0,1]^2$ together with data region $\mathcal U = (0,1] \times (-\delta,0]$.  The image data $u_0$ has high regularity when restricted to each of the sets $U_i$ but has (possibly) lower regularity on $\mathcal U$ as a whole due to problems within finitely many triangles $T_i$ (where we might have, for example, jump discontinuities in $u_0$ or its derivatives, if they exist).
In particular, we have $u_0 \in W^{s,\infty}(U_i)$ for all $i$ and for some $s > 0$ but only $u_0 \in W^{s',\infty}(\mathcal U)$ for some $0 \leq s' \leq s$, where $W^{s,\infty}(U_i)$, $W^{s',\infty}(\mathcal U)$ denote the fractional sobolev spaces described in the text.}
\label{fig:setup}
\end{figure}

\noindent {\bf Regularity of the boundary data}

\vskip 1mm

\noindent One of the major goals of this paper is to extend the analysis in \cite{Marz2007} and \cite{Guidefill}, which assume a high degree of regularity ($C^1$ and $C^2$ respectively) of the boundary data $u_0$,  to a generalized setting that is more realistic for images.  In particular, our analysis should include the case that $u_0$ is piecewise smooth.  Indeed we actually develop an even more general setting that includes piecewise smooth $u_0$ as a special case.  

Specifically, we consider $u_0$ that belongs to the fractional Sobolev space $W^{s,\infty}(U_i)$ where $0<s<\infty$, when restricted to each of a series of subsets $\{U_i\}_{i=1}^N \subseteq \mathcal U$ such that $\mathcal U \backslash \{U_i\}_{i=1}^N$ is ``small'' in some sense, while belonging to a potentially lower regularity  fractional Sobolev space  $W^{s',\infty}(\mathcal U)$ for some $0 \leq s' \leq s$ when considered on $\mathcal U$ as a whole.  Note that for $s'>0$ we have the identity $W^{s',\infty}(\mathcal U) = C^{s',s' - \lfloor s' \rfloor}(\mathcal U)$ - see also \cite{Hitchhiker} for a review of fractional Sobolev spaces.  While most authors avoid defining $W^{s',\infty}(\mathcal U)$ in the case $s'=0$, for the purposes of this paper we define 
\begin{equation} \label{eqn:sZero}
W^{0,\infty}(\mathcal U) = L^{\infty}(\mathcal U).
\end{equation}
This will enable us to seamlessly incorporate the case of $u_0$ that is piecewise continuous, which is of particular interest.  With this definition for every $s \geq 0$ the space $W^{s,\infty}(\mathcal U)$ is a Banach space with norm $\|\cdot \|_{W^{s,\infty}(\mathcal U)}$ such that any $u_0 \in W^{s,\infty}(\mathcal U)$ obeys the H\"older property
\begin{equation} \label{eqn:HolderLike}
\| \partial_{\alpha} u_0({\bf x}) - \partial_{\alpha} u_0({\bf y})\| \leq \| u_0 \|_{W^{s,\infty}(\mathcal U)}\|{\bf x} -{\bf y}\|^{s - \lfloor s \rfloor}
\end{equation}
for all multi-indices $\alpha$ s.t. $|\alpha|=\lfloor s \rfloor$.  This property will be one of the key tools of our analysis in this section.

\vskip 1mm

\noindent {\em Detailed Assumptions.}  We assume that the strip $\mathcal U = (0,1] \times (-\delta,0]$ contains $M$ closed triangles $\{ T_i \}_{i=1}^M$ with tips touching the $x$-axis at $M$ points ${\bf x}_i:=(x_i,0)$ with $1 \leq i \leq M$ as in Figure \ref{fig:setup}.  Each $T_i$ is defined by the inequality
$$T_i = \{ (x,y) \in (0,1] \times [-\delta,0] : y \leq -L|x_i - x|\}$$
for some constant $L > 0$.  We label the region between triangles $T_{i-1}$ and $T_i$ as $U_i$, where $T_{0}:=T_M$. 

We assume that $u_0 \in W^{s,\infty}(U_i)$ for each $U_i$, but has a smaller Sobolev exponent $s' < s$ on $(0,1] \times (-\delta,0]$ as a whole.  The $T_i$ can be thought of as ``bounding cones'' of a series of curves $\mathcal C_i$ intersecting $\partial \Omega$ at $\{ {\bf x}_i \}$, such that $u_0$ has lower regularity on each $\mathcal C_i$.  For example, these curves might be places where either $u_0$ or $\nabla u_0$ (if it exists) have jump discontinuities.

The assumption $\mathcal C_i \subset T_i$ ensures that each $\mathcal C_i$ intersects $\partial \Omega$ ``nicely''.  In particular, if $\mathcal C_i$ is smooth, this means that the angle between the tangent to $\mathcal C_i$ at $x_i$ and $\partial D$ is bounded away from $0$ (however, our assumption does not require that $\mathcal C_i$ be smooth).  See Figure \ref{fig:triangle}.

\begin{figure}
\centering\includegraphics[width=0.7\linewidth]{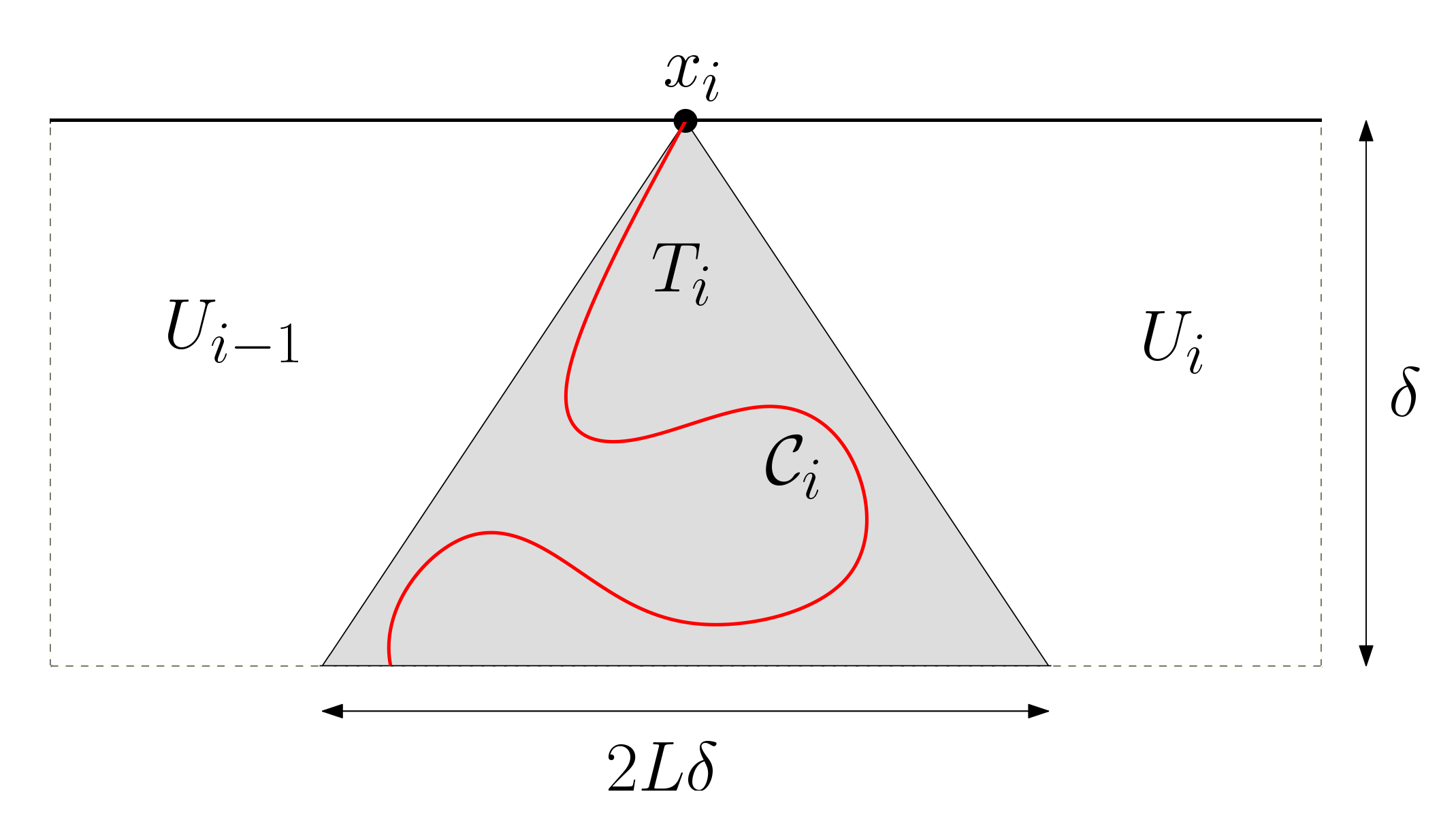}
\caption{{\bf Close up of one of the triangles $T_i$:}  The source of the problems in $T_i$ may be, for example, a curve $\mathcal C_i$ such that the regularity condition \eqref{eqn:HolderLike} with exponent $s$ fails if ${\bf x}$ and ${\bf y}$ lie on opposite sides of $\mathcal C_i$ (in which case it still holds with exponent $s' \leq s$).  In this case $T_i$ is the ``bounding cone'' of $\mathcal C_i$, ensuring that it intersects $D$ in a simple way.}
\label{fig:triangle}
\end{figure}

\vskip 2mm

\noindent {\bf Connection to stopped random walks.}  Note that the update formula \eqref{eqn:generic} gives a relationship between $u_h({\bf x})$ and its immediate neighbors in $a^*_r$, which for now we assume obeys $a^*_r \subseteq b^-_r$ or $a^*_r \subseteq b^0_r$ (if this is not the case we can apply the method of equivalent weights from Section \ref{sec:fiction}).  Now suppose we modify \eqref{eqn:generic} iteratively by repeated application of the following rule: for each ${\bf y} \in a^*_r$, if \\${\bf x}+h{\bf y} \in D_h$, replace $u_h({\bf x}+h{\bf y})$ in the RHS of \eqref{eqn:generic} with the RHS of a version of \eqref{eqn:generic} where the LHS is evaluated at ${\bf x}+h{\bf y}$ (in other words, we are substituting \eqref{eqn:generic} into itself).  Otherwise, if ${\bf x}+h{\bf y} \in \mathcal U_h$, we are already in the undamaged portion of the image, and we may replace $u_h({\bf x}+h{\bf y})$ with $u_0({\bf x}+h{\bf y})$.  Repeat this procedure until $u_h({\bf x})$ is entirely expressed as a weighted sum of $u_0 \Big|_{\mathcal U_h}$, that is
\begin{equation} \label{eqn:hope}
u_h({\bf x}) = \sum_{{\bf y} \in \mathcal U_h} \rho({\bf y}) u_0({\bf y}),
\end{equation}
for some as of yet unknown weights $\rho$.  Denoting ${\bf x}:=(nh,mh)$, then for the direct form of Algorithm 1 this procedure will terminate after $m$ steps, as in this case \eqref{eqn:generic} expresses $u_h({\bf x})$ in terms of neighbors at least $h$ units below it.  On the other hand, for the semi-implicit extension, \eqref{eqn:hope} has to be interpreted as a limit.

\begin{figure}
\centering
\begin{tabular}{cccc}
\subfloat[$k=0$]{\includegraphics[width=.22\linewidth]{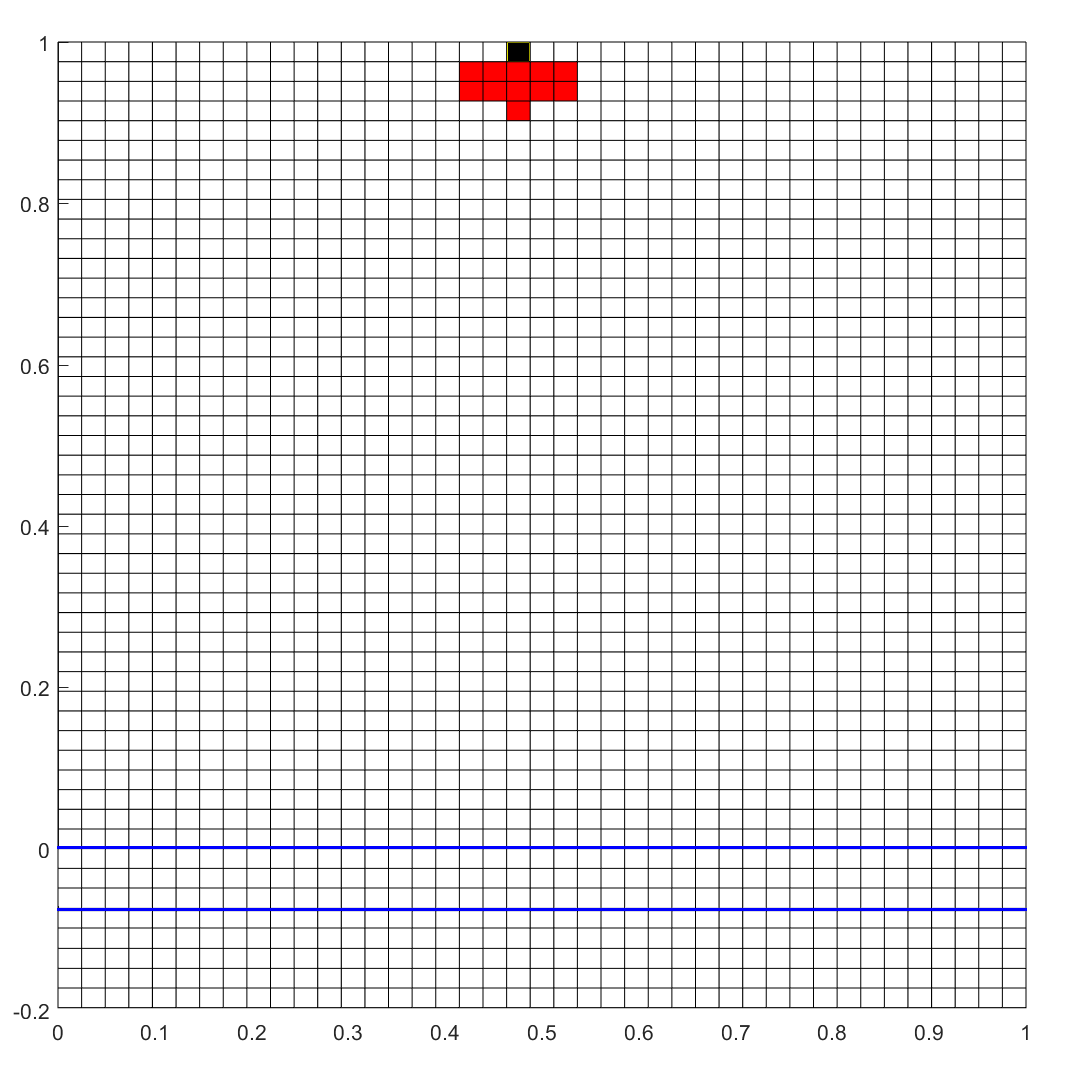}} & 
\subfloat[$k=4$]{\includegraphics[width=.22\linewidth]{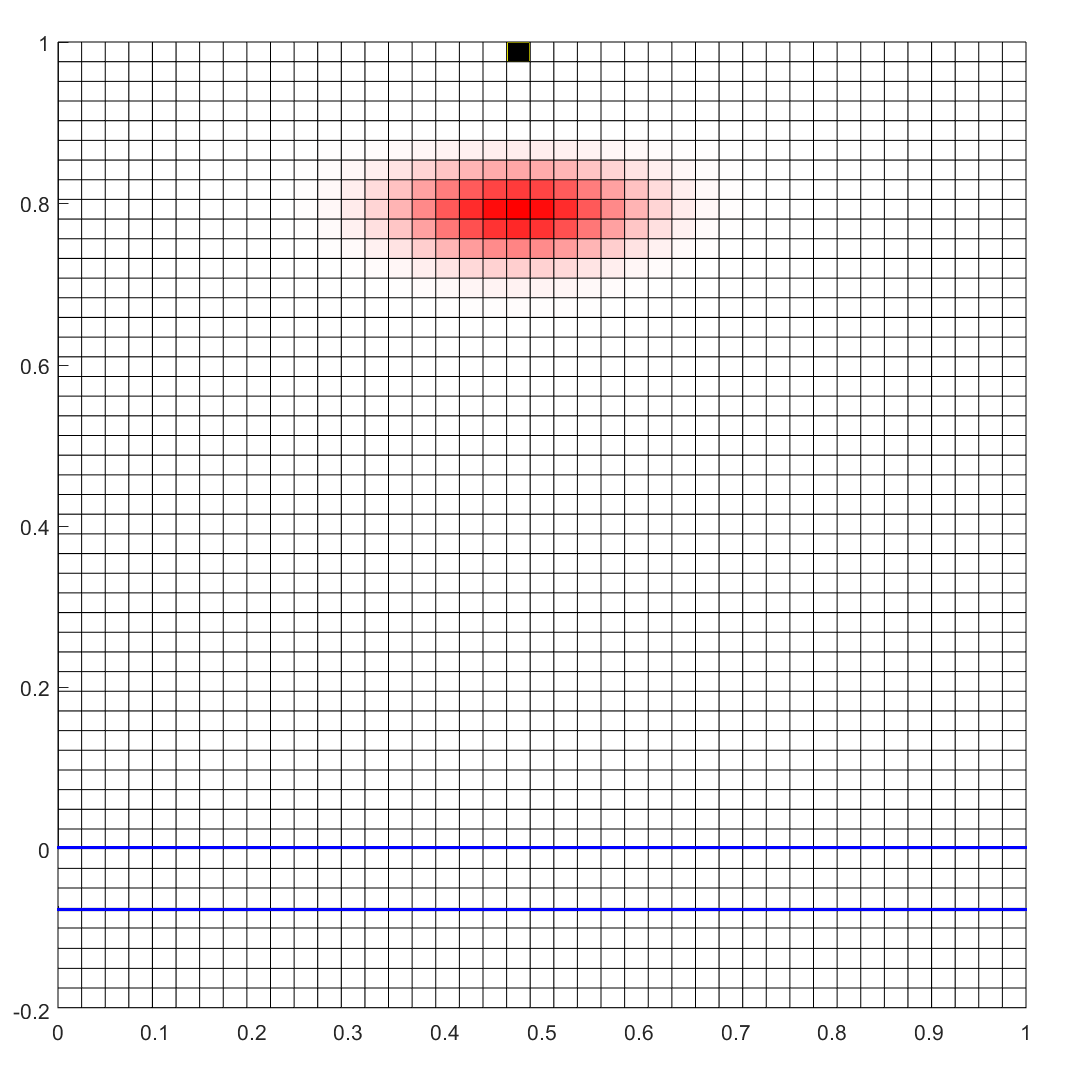}} &
\subfloat[$k=17$]{\includegraphics[width=.22\linewidth]{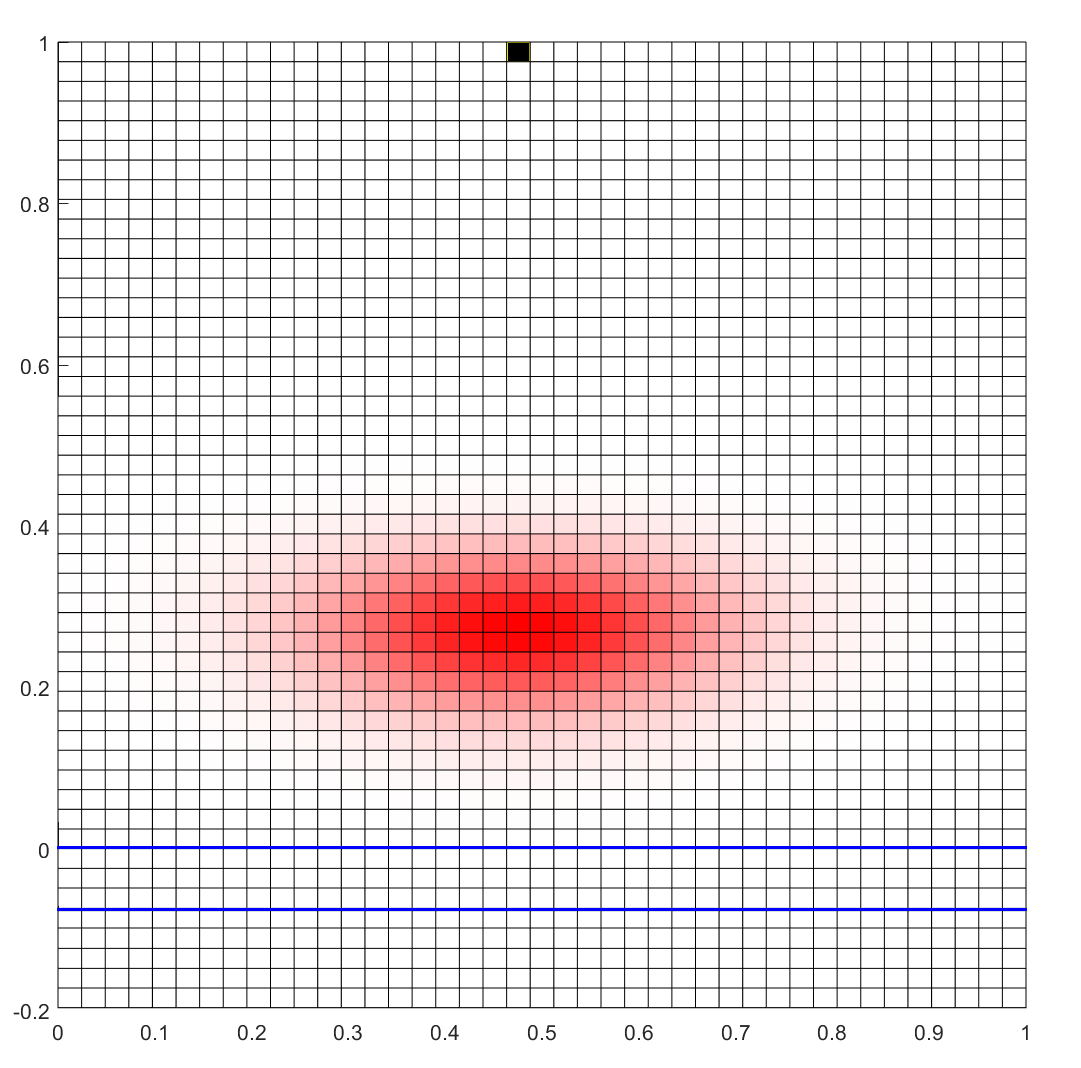}} & 
\subfloat[$k=40$]{\includegraphics[width=.22\linewidth]{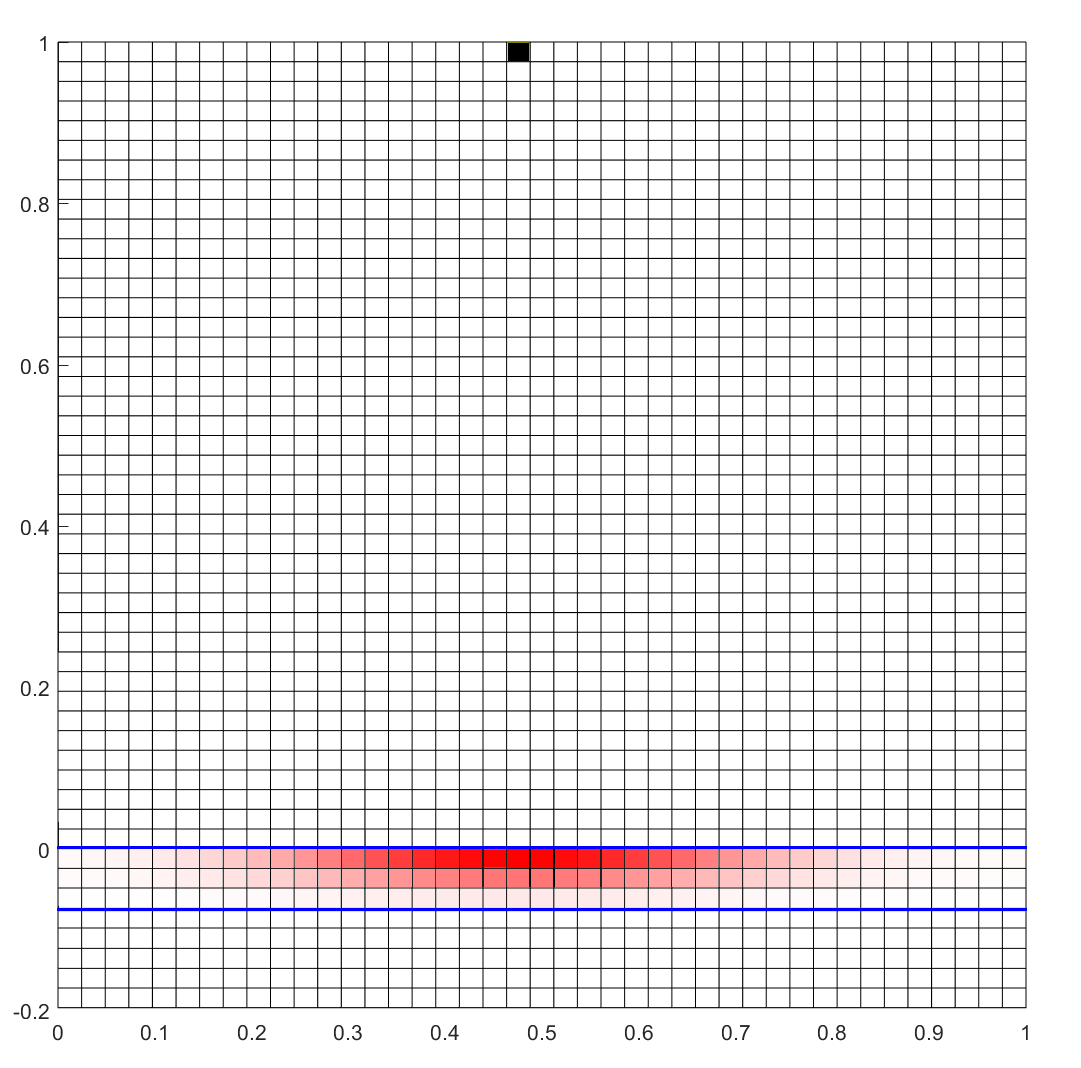}} \\
\end{tabular}
\caption{{\bf Connection to Stopped Random Walks:}  Here we illustrate the connection between the elimination procedure described in the text and stopped random walks.  In (a) the pixel ${\bf x}$ (colored black) is expressed as a weighted average \eqref{eqn:update} of its neighbors in in $b^-_r$ (colored in red).  In this case we use the direct form of Algorithm 1 with $r=3$ and uniform weights (which is why each neighbor in $b^-_r$ is the same shade of red).  In (c)-(d) we have applied $k=4$, $k=17$, and $k=40$ steps of our elimination procedure (which essentially consists of substituting \eqref{eqn:update} into itself repeatedly - details in the text), and $u_h({\bf x})$ is now expressed as weighted sum of $u_h({\bf y})$ for whichever pixels ${\bf y}$ are colored red, with darker shades of red indicating greater weight.  Since the stencil weights $\{ \frac{w_r({\bf 0},{\bf y})}{W} : {\bf y} \in b^-_r\}$ are nonegative and sum to one, this is procedure is related to a stopped random walk ${\bf X}_{\tau}:=(X_{\tau},Y_{\tau})$ started from ${\bf x}$ with stopping time $\tau=\inf\{ n : Y_n \leq 0\}$.  Specifically, after $k$ steps of elimination we have
$u_h({\bf x}) = \sum_{{\bf y} \in \Omega_h } \rho_{{\bf X}_{\tau \wedge k}}({\bf y})u_h({\bf y}),$
where $\rho_{{\bf X}_{\tau \wedge k}}$ is the density of
${\bf X}_{\tau \wedge k}:={\bf x}+h \sum_{i=1}^{\tau \wedge k} {\bf Z}_i,$
a random walk with increments ${\bf Z}_i$ i.i.d. taking values in $b^-_r$ with density $P({\bf Z}_i ={\bf y})=\frac{w_r({\bf 0},{\bf y})}{W}$.  Since $b^-_r$ is below the line $y=-1$, the density of $\rho_{{\bf X}_{\tau \wedge k}}$ necessarily moves at least one pixel downwards each iteration, and since ${\bf x}$ is only 40 pixels above $y=0$, by $k=40$ it the entire density is below $y=0$ and the walk terminates.  This process is illustrated in (b)-(d), where we see the density after $4$ steps (b), $17$ steps (c), and the final stopped density (d).  Note that as ${\bf X}_{\tau}$ has increments of size at most $rh$, the density $\rho_{{\bf X}_{\tau}}$ necessarily lies between the lines $y=-rh$ and $y=0$ (outlined in blue).  The purpose of this procedure is to express the color of a given pixel ${\bf x}$ deep inside the inpainting domain entirely in terms of the colors of known pixels below $y=0$.
}
\label{fig:Elimination}
\end{figure}

This elimination procedure has a natural interpretation in terms of stopped random walks.  Since the weights $\{ \frac{w({\bf 0},{\bf y})}{W} \}_{{\bf y} \in a^*_r}$ are non-negative and sum to $1$, we can interpret them as the density of a two dimensional random vector ${\bf Z} := (V,W)$ taking values in $b^-_r$ or $b^0_r$ .  Moreover, defining the random walk
\begin{equation} \label{eqn:randomWalk}
{\bf X}_j :=(X_j,Y_j) = (nh,mh) + h \sum_{i=1}^j {\bf Z}_i
\end{equation}
with $\{ {\bf Z}_i \}$ i.i.d. and equal to ${\bf Z}$ in distribution, and defining
\begin{equation} \label{eqn:tau}
\tau = \inf\{ j : {\bf X}_j \in \mathcal U_h  \}=\inf\{ j : Y_j \leq 0  \},
\end{equation}
then after $k$ steps of elimination, we have
$$u_h({\bf x}) = \sum_{{\bf y} \in \Omega_h} \rho_{{\bf X}_{j \wedge \tau }}({\bf y}) u_h({\bf y}),$$
where $\rho_{{\bf X}_{j \wedge \tau }}$ denotes the density of ${\bf X}_{j \wedge \tau }$.  Denoting the mean of ${\bf Z}$ by $(\mu_x,\mu_y)$ note that by \eqref{eqn:transportIsCenter} we have the equivalence
$$(\mu_x, \mu_y) = {\bf g}^*_r.$$
In other words, the mean of ${\bf Z}$ is precisely the transport direction of our limiting equation \eqref{eqn:transport}.  The condition ${\bf g}^*_r \cdot e_2 \neq 0$, which we needed for \eqref{eqn:transport} to be defined, implies $\mu_y < 0$.  In the nomenclature of random walks, this means that ${\bf X}_k$ has negative drift in the $y$ direction, while $\tau$ is the first passage time through $y=0$.  Fortunately, this type of random walk and this type of first passage time have been studied and are well understood.  See for example \cite[Chapter 4]{gut2009stopped}, \cite{GutAndJanson1983}, \cite{gut1974}, \cite{GUT1974115}.  The book \cite{gut2009stopped} also provides an good overview of stopped random walks in general.  In particular, we know immediately that $\tau$ is a stopping time, $P(\tau = \infty)=0$, and $\tau$ has finite moments of all orders \cite[Chapter 3, Theorems 1.1 and 3.1]{gut2009stopped}.  It follows that
\begin{equation} \label{eqn:reformulation}
u_h({\bf x}) = \sum_{{\bf y} \in \mathcal U_h} \rho_{{\bf X}_{\tau }}({\bf y}) u_0({\bf y}) = E[u_0({\bf X}_{\tau})].
\end{equation}  
This insight is central to our convergence argument and will also play a central role when we analyze smoothing artifacts in Section \ref{sec:blur}.  See Figure \ref{fig:Elimination} for an illustration of these ideas.

\vskip 2mm

\begin{theorem} \label{thm:convergence}
Let the continuous inpainting domain $D$ and undamaged area $\hskip 0.5mm \mathcal U$, as well as their discrete counterparts $D_h$, $\mathcal U_h$ be as described in Section \ref{sec:symmetry}.  Suppose we inpaint $D_h$ using the direct form of Algorithm 1 or the semi-implicit extension, and denote the result by $u_h : D_h \rightarrow \field{R}$.  Assume the assumptions of Section \ref{sec:symmetry} hold and let $a^*_r$ and $\left\{ w_r({\bf 0},{\bf y}) : {\bf y} \in a^*_r \right\}$ denote the stencil and stencil weights respectively (defined in Definition \ref{def:stencil}) of our inpainting method.  Let $u$ denote the weak solution to the transport equation \eqref{eqn:transport}, with transport direction equal to the center of mass of our stencil with respect to the stencil weights, that is
$${\bf g}^*_r = \sum_{{\bf y} \in a^*_r } \frac{w({\bf 0},{\bf y})}{W}{\bf y} \quad \mbox{ where } W = \sum_{{\bf y} \in a^*_r } w({\bf 0},{\bf y}).$$
Suppose the boundary data $u_0 : \mathcal U \rightarrow \field{R}$ obeys the regularity assumptions above, and let $\{ U_i\}_{i=1}^M$ be as defined above and illustrated in Figure \ref{fig:setup}.  Then for any $p \in [0,\infty]$ we have
$$\|u - u_h \|_{p} \leq K \cdot (rh)^{\left(\frac{s'}{2}+\frac{1}{2p}\right) \wedge \frac{s}{2} \wedge 1}$$
(the case $p=\infty$ is included by defining $1/\infty :=0$) where $K$ is a constant depending only on $u_0$, $\mathcal U$, $\{U_i\}_{i=1}^M$, $\frac{r}{{\bf g}^*_r \cdot e_2}$ and $\theta^*_r$.  Moreover, $K$ depends continuously on $\theta^*_r$ and $\frac{r}{{\bf g}^*_r \cdot e_2}$, is a monotonically increasing function of $\frac{r}{{\bf g}^*_r \cdot e_2}$, and $K \rightarrow \infty$ as $\theta^*_r \rightarrow 0$ or $\theta^*_r \rightarrow \pi$.
\end{theorem}

\vskip 3mm

\begin{remark}  Here we list some notable special cases of Theorem \ref{thm:convergence}, together with their associated convergence rates:

\vskip 2mm

\noindent {\em Uniform regularity:}  If $u_0 \in C^{k,\alpha}(\mathcal U)$, then 
 $$\|u - u_h \|_{p} \leq  \begin{cases}  K \cdot (rh)^{\frac{k+\alpha}{2}} & \mbox{ if } k \in \{0,1\} \\
  K \cdot (rh) & \mbox{ if } k \geq 2. \\
 \end{cases}$$
In other words, the rate of convergence depends smoothly on the regularity of $u_0$ up until the $C^2$ level, beyond which additional smoothness has no effect.  Note also that in this case the rate of convergence is independent of $p$.

\vskip 2mm

\noindent {\em Piecewise smooth:}  If $u_0 \in C^{k,\alpha}(\mathcal U \backslash \{ \mathcal C_i\}_{i=1}^M)$ where $k \geq 1$ and $\{ \mathcal C_i\}_{i=1}^M$ are a series of curves such that $\mathcal C_i \subseteq T_i$ for each $i$ as in Figure \ref{fig:triangle}, but $u_0$ is discontinuous on $\mathcal U$ as a whole, then we have
 $$\|u - u_h \|_{p} \leq K \cdot (rh)^{\frac{1}{2p}}.$$
Thus, in the piecewise smooth case our rates of convergence are independent of the regularity of $u_0$ on $\mathcal U \backslash \{ \mathcal C_i\}_{i=1}^M$ so long as it is at least $C^1$, but now depend on $p$.  In particular, the convergence rate is a monotonically decreasing function of $p$, and we converge in $L^p$ for every $1 \leq p < \infty$, but not necessarily in $L^{\infty}$.  

\vskip 2mm

\noindent {\em Piecewise H\"older continuous:}  If $u_0 \in C^{0,\alpha}(\mathcal U \backslash \{ \mathcal C_i\}_{i=1}^M)$  but discontinuous on $\mathcal U$ as a whole (where $\{ \mathcal C_i\}_{i=1}^M$ are the same as in the previous example), then we have
 $$\|u - u_h \|_{p} \leq 
 \begin{cases} K \cdot (rh)^{\frac{\alpha}{2}} & \mbox{ if } 1\leq p \leq \frac{1}{\alpha} \\
 K \cdot (rh)^{\frac{1}{2p}} & \mbox{ if } \frac{1}{\alpha} < p \leq \infty. \\
 \end{cases}$$
In this case the convergence rate is independent of $p$ for $1 \leq p \leq \frac{1}{\alpha}$, but is the same as the previous example for $p>\frac{1}{\alpha}$.

\end{remark}

\section{Proof of Theorem \ref{thm:convergence}}

First a note on notation.  Let us define for convenience $\hat{\bf x}=\Pi_{\theta^*_r}({\bf x})$, so that the solution to \eqref{eqn:transport} may be more compactly written as 
$$u({\bf x})=u_0(\hat{\bf x}).$$
For brevity, even though ${\bf X}_{\tau}$, $\hat{\bf x}$, and $\tau$ depend implicitly on ${\bf x}=(nh,mh)$, we do not write this dependence down explicitly.  We will also adopt the notation that if ${\bf X}=(X_1,X_2)$ and $X$ denote vector and scalar valued random variables respectively, then
$$\| {\bf X} \|_{L^p} := E[\|{\bf X}\|^p]^{\frac{1}{p}} \quad \mbox{ and } \quad  \| X \|_{L^p} := E[|{\bf X}|^p]^{\frac{1}{p}},$$
where $\|{\bf X}\|:=\|{\bf X}\|_2$ denotes the {\em euclidean} norm of ${\bf X}$.  The identity
\begin{equation} \label{eqn:usefulSometimes}
\|{\bf X}\|_{L^p} \leq \|X_1\|_{L^p}+\|X_2\|_{L^p}
\end{equation}
(which follows trivially from the identity $\|{\bf X}\|_2 \leq \|{\bf X}\|_1$ and the triangle inequality with respect to the norm $\|\cdot\|_{L^p}$) will occasionally be useful.  Finally, note that by the linearity of expectation, we have
\begin{equation} \label{eqn:observation}
\|u - u_h \|_p = \| E[u_0({\bf X}_{\tau}) - u_0(\hat{\bf x})]\|_p
\end{equation}

\vskip 2mm
\noindent {\bf Roadmap.}  Our proof consists of six stages, which we go through before proceeding.  Stages 1-5 assume that $a^*_r \subseteq b^-_r$ or $a^*_r \subseteq b^0_r$ for simplicity.  Stage 6 generalizes to arbitrary $a^*_r$.
\vskip 1mm
\begin{enumerate}
\item {\bf Stage 1:} Obtain bounds, as a function of $h$, on the rate at which ${\bf X}_{\tau}$ is concentrating around its mean $E[{\bf X}_{\tau}]$, as well as the rate that $E[{\bf X}_{\tau}]$ itself is converging to $\hat{\bf x}$.  For technical reasons, we will require in particular bounds on $\|{\bf X}_{\tau}-E[{\bf X}_{\tau}]\|_{L^4}$.  In effect what is happening is that 
$$\rho_{{\bf X}_{\tau}} \rightarrow \delta_{\hat{\bf x}} \qquad \mbox{ as } h \rightarrow 0,$$
where $\delta_{\hat{\bf x}}$ denotes the Dirac delta distribution centered at $\hat{\bf x}$ (see Figure \ref{fig:intuition}).  Formally speaking, we have
$$u_h({\bf x}) =  \sum_{{\bf y} \in \mathcal U_h} \rho_{{\bf X}_{\tau }}({\bf y}) u_0({\bf y}) \rightarrow \sum_{{\bf y} \in \mathcal U_h} \delta_{\hat{\bf x}}({\bf y}) u_0({\bf y})=u_0(\hat{\bf x})$$
We will not attempt to make this argument rigorous, and mention it only for the sake of building intuition.
For this part of the proof, known results from \cite[Chapter 4]{gut2009stopped} will save us some effort.
\item {\bf Stage 2:}  Use the assumed regularity of $u_0$ to obtain bounds on \\$|E[u_0({\bf X}_{\tau}) - u_0(\hat{\bf x})]|$.  In particular, we will obtain: 
\begin{enumerate}
\item A bound that holds for {\em any} starting position ${\bf x}$ of the random walk ${\bf X}_{\tau}$, and depends only on $\|{\bf X}_{\tau}-E[{\bf X}_{\tau}]\|_{L^4}$ and $\|E[{\bf X}_{\tau}]-\hat{\bf x}\|$.
\item A tighter bound that holds assuming $E[{\bf X}_{\tau}]$ and $\hat{\bf x}$ both belong to one of the well behaved sets $U_i$ (see Figure \ref{fig:setup}), and depends not only on \\$\|{\bf X}_{\tau}-E[{\bf X}_{\tau}]\|_{L^4}$ and $\|E[{\bf X}_{\tau}]-\hat{\bf x}\|$, but also on the ``rogue event'' that even though $\hat{\bf x} \in U_i$ and $E[{\bf X}_{\tau}] \in U_i$, we nonetheless have ${\bf X}_{\tau} \notin U_i$.
\end{enumerate}
\item {\bf Stage 3:} Partition $D$ into bands $\{B_i\}_{i=1}^M$ and sub-bands $\tilde{B}_i \subset B_i$ such that the area of the complement $B_i \backslash \tilde{B}_i$ goes to zero as $h \rightarrow 0$, and such that the starting position ${\bf x} \in \tilde{B}_i$ guarantees $E[{\bf X}_{\tau}] \in U_i$ and $\hat{\bf x} \in U_i$.  See Figure \ref{fig:bands} for an illustration.
\item {\bf Stage 4:}  Bound the probability of the rogue event from stage 2, under the assumption that the starting position ${\bf x} \in \tilde{B}_i$ for some $1\leq i \leq M$.  
\item {\bf Stage 5:}  Substitute the bounds from the previous four stages into \eqref{eqn:observation} to obtain a bound for $\|u-u_h\|_p$.
\item {\bf Stage 6:}  Generalize to arbitrary $a^*_r$ containing ghost pixels by exploiting the idea of equivalent weights from Section \ref{sec:fiction}.
\end{enumerate}
\vskip 1mm
Steps $1$ and $4$ will occasionally require us to use some elementary results from martingale theory, including Wald's first identity and Azuma's inequality.  For a review of martingale theory, see for example \cite{williams1991probability}.

\begin{figure}
\centering
\begin{tabular}{cccc}
\subfloat[$h=1e-2$.]{\includegraphics[width=.22\linewidth]{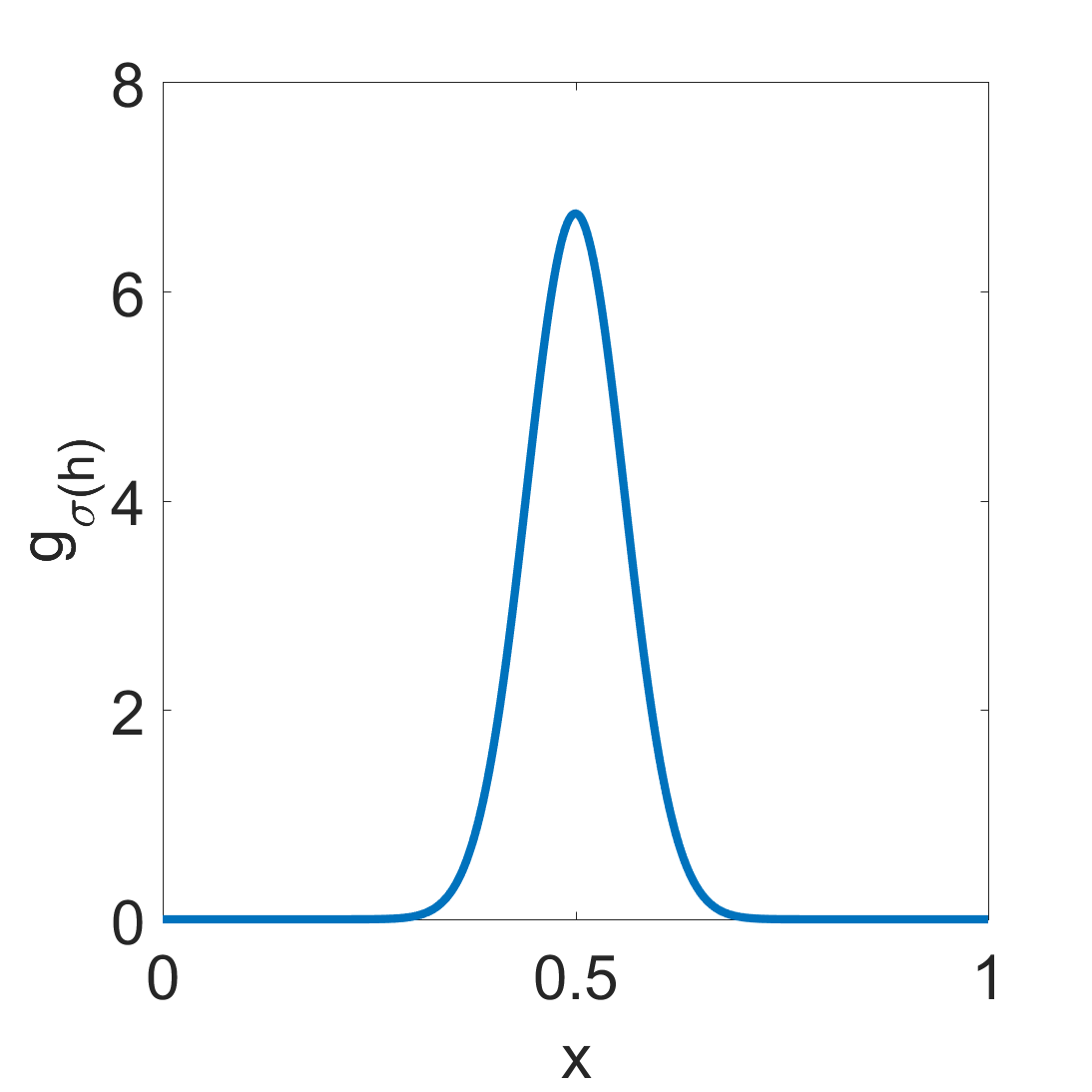}} & 
\subfloat[$h=1e-3$.]{\includegraphics[width=.22\linewidth]{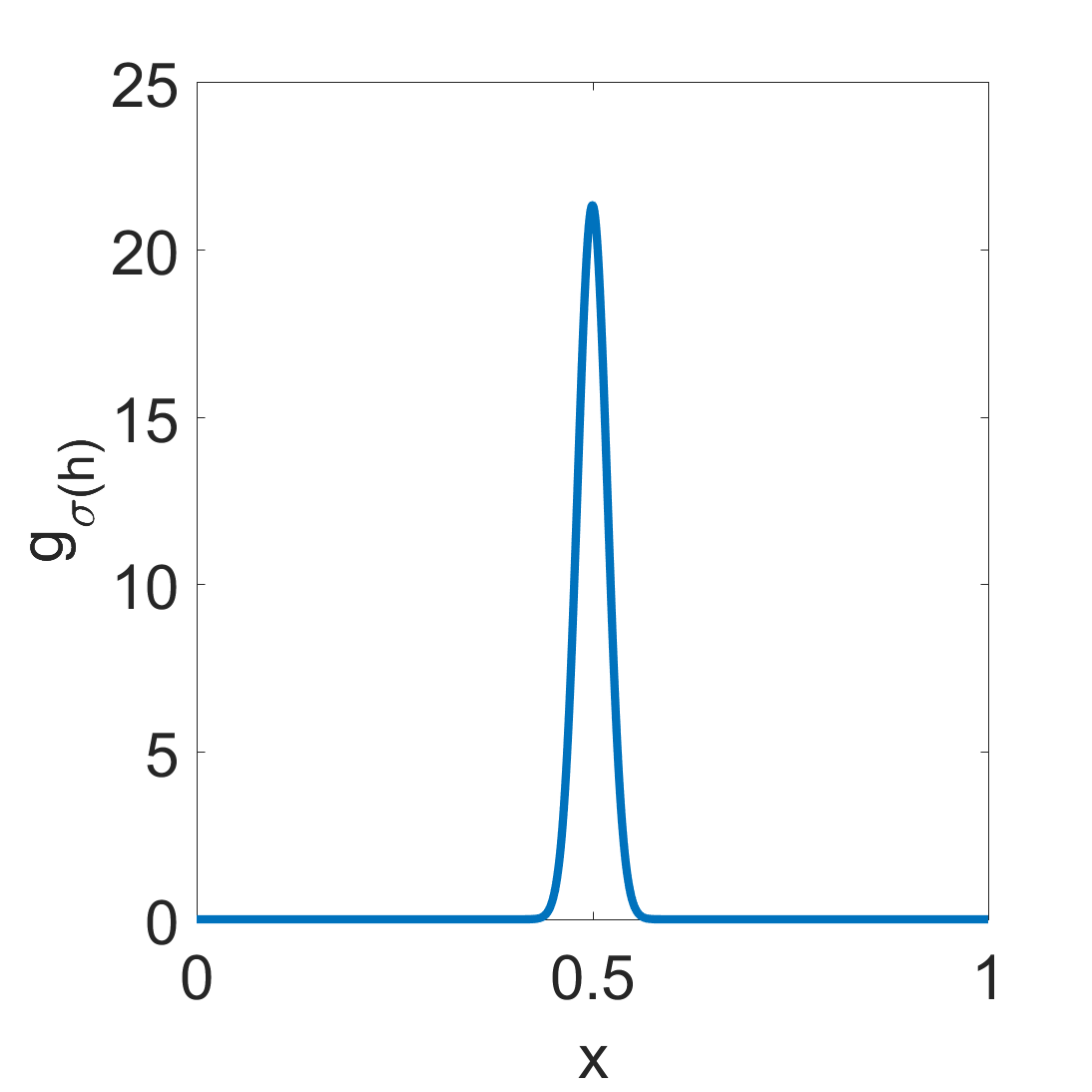}} &
\subfloat[$h=1e-4$.]{\includegraphics[width=.22\linewidth]{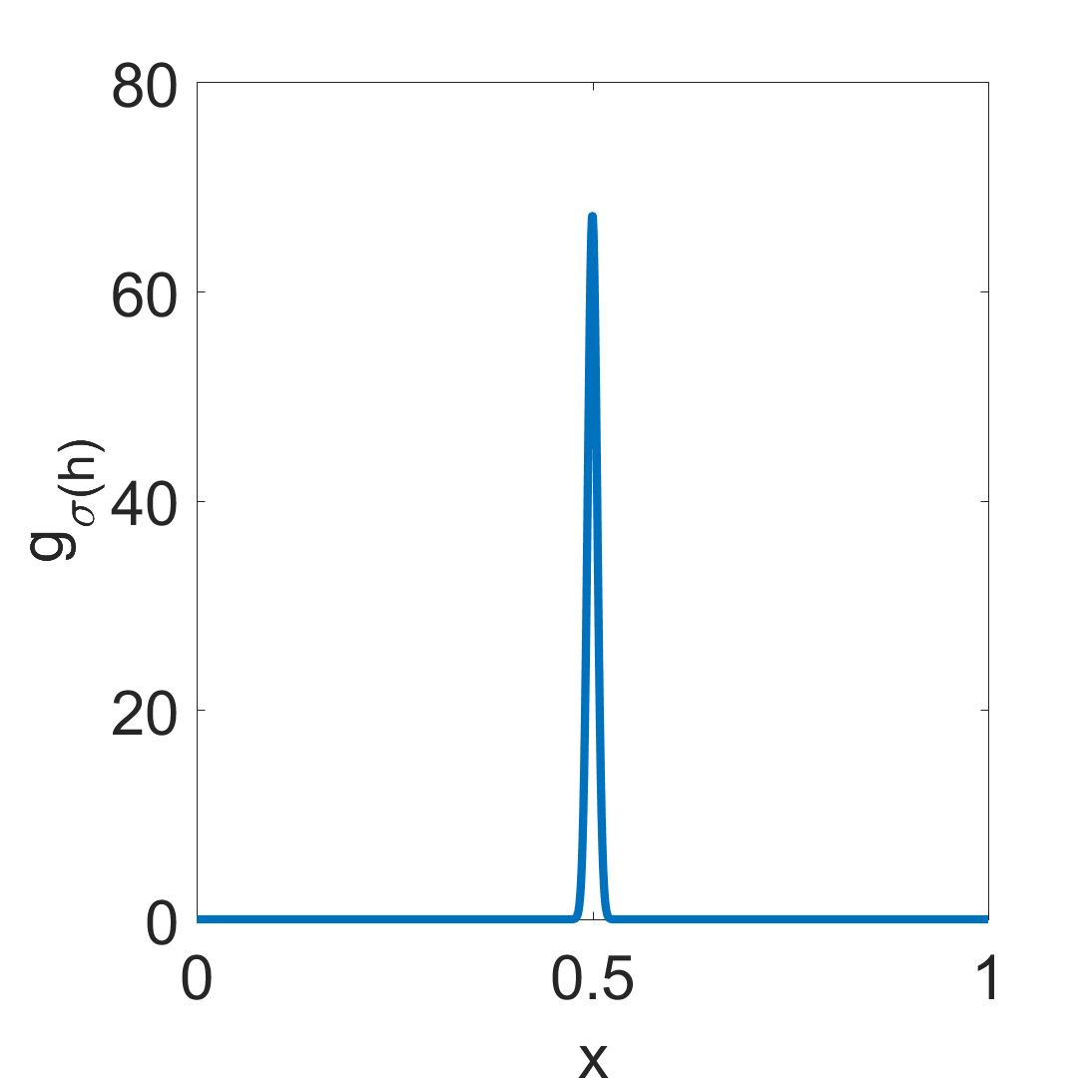}} & 
\subfloat[$h=1e-5$.]{\includegraphics[width=.22\linewidth]{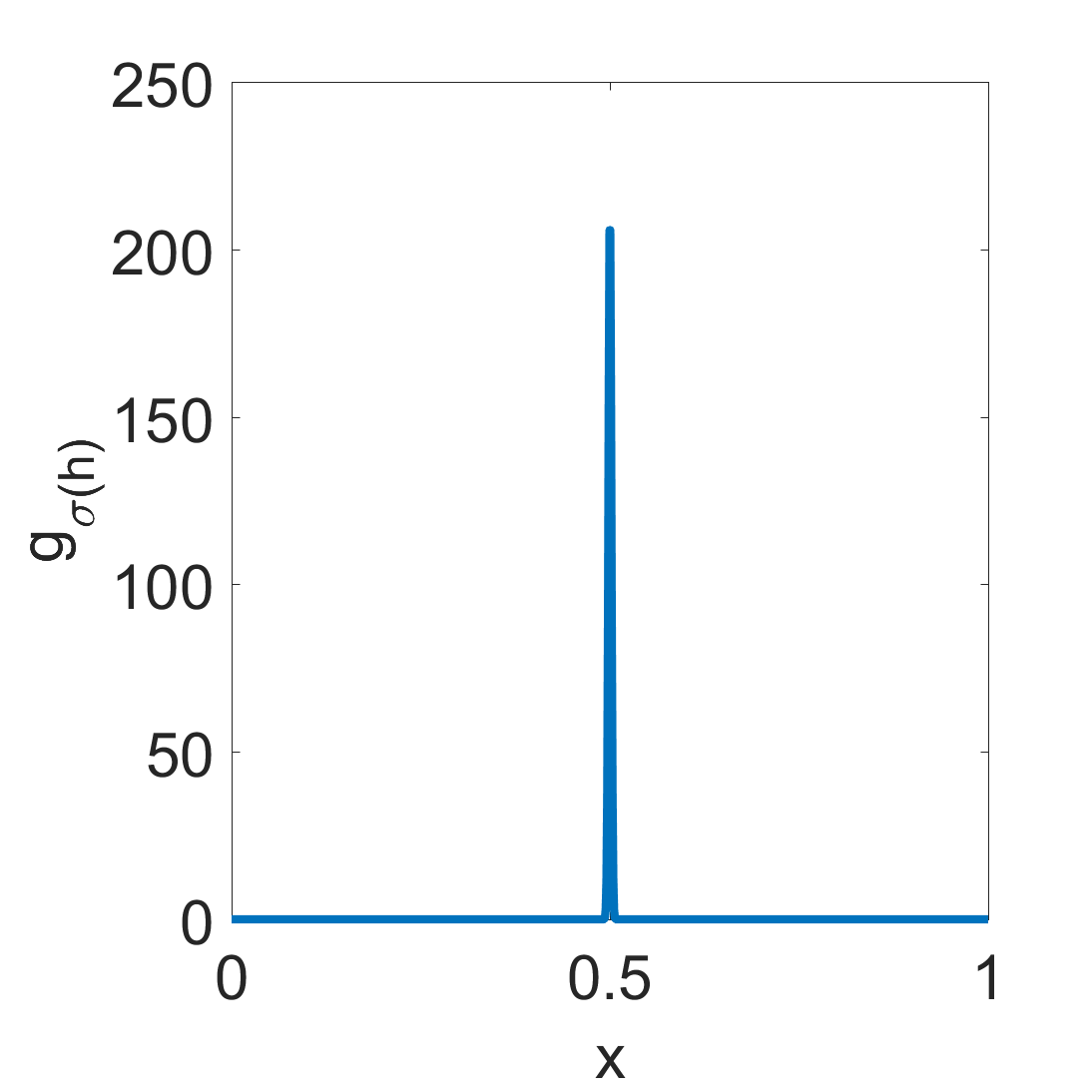}} \\
\end{tabular}
\caption{{\bf Theorem \ref{thm:convergence} through the lense of distribution theory:}  Theorem \ref{thm:convergence} says that for ${\bf x} \in D_h$, $u_h({\bf x}) \rightarrow u_0(\hat{\bf x})$ in $L^p$ as $h \rightarrow 0$ for every $p \in [1,\infty]$ if $u_0$ is continuous, but possibly not in $L^{\infty}$ if $u_0$ contains discontinuities.  Here $\hat{\bf x}=\Pi_{\theta^*_r}({\bf x})$ is the transport operator \eqref{eqn:weak} applied to ${\bf x}$.  One way of understanding this is via the identity $u_h({\bf x}) = \sum_{{\bf y} \in \mathcal U_h} \rho_{{\bf X}_{\tau }}({\bf y}) u_0({\bf y})$ \eqref{eqn:reformulation}.  As we will see in Section \ref{sec:blur}, for each fixed ${\bf x} \in D_h$, the x-coordinate $X_{\tau}$ of the stopped random walk ${\bf X}_{\tau}$ started from ${\bf x}$ converges in distribution to a Gaussian $g_{\sigma(h)}$ with mean $\hat{\bf x}$ and $h$-dependent variance $\sigma(h)^2$, such that $g_{\sigma(h)}$ itself is converging to a one dimensional Dirac delta distribution centered at $\hat{\bf x}$ as $h \rightarrow 0$.  At the same time, since ${\bf X}_{\tau}$ can only overshoot $y=0$ by distance at most $rh$, it follows that the density of ${\bf X}_{\tau}$ lies entirely between the lines $y=-rh$ and $y=0$.  These two facts together imply that ${\bf X}_{\tau}$ converges in distribution as $h \rightarrow 0$ to a two dimensional Dirac delta distribution centered at $\hat{\bf x}$.  Hence we expect $u_h({\bf x}) \rightarrow u_0(\hat{\bf x})$, at least when $\hat{\bf x}$ is a continuity point of $u_0$ (if $\hat{\bf x}$ is not a continuity point, then we need to apply the theory of distributions with discontinuous test functions - see for example \cite{derr2009} - and we do not expect $u_h({\bf x}) \rightarrow u_0(\hat{\bf x})$ in general).  We illustrate this in (a)-(d), which show plots of $g_{\sigma(h)}$ for Guidefill with $r=3$, ${\bf g}=(\cos 45^{\circ},\sin 45^{\circ})$ and $\mu \rightarrow \infty$ for various values of $h$.  We fix ${\bf x}=(0.5,1)$ so that $\hat{\bf x}=(0.5,0)$ (remember the periodic boundary conditions).}
\label{fig:intuition}
\end{figure}

\vskip 2mm
\noindent {\bf Stage 1.}  Throughout this stage we will be using the following Theorem from \cite[Chapter 1]{gut2009stopped}:
\begin{theorem} \label{thm:frombook}
Suppose $E[|X_1|^r] < \infty$ for some $0 < r < \infty$, $E[X_1]=0$, and define the stopped random walk $S_{\tau}=\sum_{i=1}^{\tau} X_i$, where $\{ X_i \}$ are i.i.d.  Then 
$$E[|S_{\tau}|^p] \leq B_p E[|X_1|^p] E[\tau^{\frac{p}{2}}].$$
where $B_p$ is a constant depending only on $p$.
\end{theorem}

\noindent  The result of this stage is the following lemma:

\begin{lemma}

Let ${\bf X}_{\tau} := (X_{\tau},Y_{\tau}) = {\bf x} + h \sum_{i=1}^{\tau} {\bf Z}_i = (nh,mh) + h \sum_{i=1}^{\tau}(V_i,W_i)$ with $\{{\bf Z}_i\}$ i.i.d. as above.  Then 

(i)  $\| E[{\bf X}_{\tau}] - \hat{\bf x} \| \leq rh. $

In addition, there is a constant $C > 0$ dependent only on $\frac{\mu_y}{r}$ and $\theta^*_r$ such that

(ii) $\|{\bf X}_{\tau} - E[{\bf X}_{\tau}]\|_{L^4} \leq C \sqrt{rh}.$

Moreover, $C \rightarrow \infty$ as $\theta^*_r \rightarrow \{0,\pi\}$ or $\frac{\mu_y}{r} \rightarrow 0$.

\end{lemma}
\begin{proof}
The idea of the proof is to combine repeated application of Theorem \ref{thm:frombook} with geometric observations of the situation at hand - specifically, the fact that $|{\bf Z}_i| \leq r$ means that ${\bf X}_{\tau}$ can overshoot $y=0$ by a distance of at most $rh$.  First, note that since $E[\tau] < \infty$, Wald's first identity
$$E[{\bf X}_{\tau}]=(nh,mh)+E[\tau](\mu_x h,\mu_y h)$$
is applicable.  This implies that the ray joining $(nh,mh)$ with $E[{\bf X}_{\tau}]$ is parallel to the ray from $(nh,mh)$ to $\hat{\bf x}$.  At the same time, since $|{\bf X}_{k+1}-{\bf X}_k| \leq rh$, it follows that $E[{\bf X}_{\tau}]$ ``overshoots'' $\hat{\bf x}$ by a distance of at most $rh$, from which we immediately have
$$\|E[{\bf X}_{\tau}]-\hat{\bf x}\| \leq rh.$$
This proves the first claim.  For the second claim, first note that by \eqref{eqn:usefulSometimes}, we have
\begin{equation} \label{eqn:overallBound}
\|{\bf X}_{\tau}-E[{\bf X}_{\tau}]\|_{L^4} \leq \|X_{\tau}-E[X_{\tau}]\|_{L^4}+\|Y_{\tau}-E[Y_{\tau}]\|_{L^4} \leq \|X_{\tau}-E[X_{\tau}]\|_{L^4}+rh,
\end{equation}
where we have used the overshooting observation again to obtain the bound \\$|Y_{\tau}-E[Y_{\tau}]| \leq rh$.  Thus it suffices to bound $\|X_{\tau}-E[X_{\tau}]\|_{L^4}$.  To do so, we are going to use Theorem \ref{thm:frombook}.  This means we are going to need an estimate for $E[\tau^2]$, and our first task is to find one.  To that end, first note that for any $1 \leq p < \infty$ we have
$$\|\mu_y \tau - \mu_y E[\tau]\|_{L^p} \leq \| \sum_{i=1}^{\tau} W_i - \tau \mu_y\|_{L^p} + \| (m+ \sum_{i=1}^{\tau} W_i ) - ( m + E[\tau]\mu_y)\|_{L^p}.$$
The second term on the RHS is exactly $\frac{1}{h}\|Y_{\tau} - E[Y_{\tau}]\|_{L^p}$, and is bounded above by $r$ since, by our overshooting observation, $Y_{\tau}$ and $E[Y_{\tau}]$ both must lie in the interval $[-rh,0]$.  For the first term, applying Theorem \ref{thm:frombook} and noting $E[|W_1|^p] \leq r^p$ gives $\| \sum_{i=1}^{\tau} W_i - \tau \mu_y\|_{L^p} \leq (B_p)^{\frac{1}{p}}r E[\tau^{\frac{p}{2}}]^{\frac{1}{p}}$, from which it follows that
\begin{equation} \label{eqn:tauLp}
\|\tau - E[\tau]\|_{L^p}  \leq  \frac{1}{|\mu_y|}\left[ (B_p)^{\frac{1}{p}}r E[\tau^{\frac{p}{2}}]^{\frac{1}{p}}+r\right] \leq \frac{1}{|\mu_y|} [(B_p)^{\frac{1}{p}}+1]r E[\tau^{\frac{p}{2}}]^{\frac{1}{p}}
\end{equation}
Since $\tau \geq 1$.  Next, we take $p=2$ and square both sides to obtain
$$E[\tau^2]-E[\tau]^2 \leq \frac{1}{|\mu_y|^2}[(B_2)^{\frac{1}{2}}+1]^2r^2 E[\tau].$$
But, since $-rh \leq E[Y_{\tau}] = mh + E[\tau]\mu_y h \leq 0$, it follows that
$$E[\tau] \leq \frac{m+1}{|\mu_y|} \leq \frac{2N}{|\mu_y|} = \frac{2}{|\mu_y|h}.$$
After some short algebra this gives
$$E[\tau^2] \leq \left( 4 + 2 \frac{[(B_2)^{\frac{1}{2}}+1]^2 r^2}{\mu_y^2}  \right)\frac{1}{|\mu_y|^2 h^2}.$$
Next, note that
\begin{eqnarray*}
\|X_{\tau}-E[X_{\tau}]\|_{L^p} &=& \| (nh + h \sum_{i=1}^{\tau} V_i - (nh + E[\tau]\mu_x)\|_{L^p} \\
&\leq& h \| \sum_{i=1}^{\tau} V_i  - \tau \mu_x\|_{L^p}+h|\mu_x|\|\tau - E[\tau]\|_{L^p}.
\end{eqnarray*}
The first term in the RHS we bound with another application of Theorem \ref{thm:frombook}, together with the observation $E[|V_1|^p] \leq r^p$.  For the second term we already have the bound \eqref{eqn:tauLp}.  Together we get
$$\|X_{\tau}-E[X_{\tau}]\|_{L^p} \leq \left( (B_p)^{\frac{1}{p}}+\frac{|\mu_x|}{|\mu_y|}(1+(B_p)^{\frac{1}{p}}) \right) rh E[\tau^{\frac{p}{2}}]^{\frac{1}{p}}.$$
Setting now $p=4$ and applying our bound on $E[\tau^2]$ gives
$$\|X_{\tau}-E[X_{\tau}]\|_{L^4} \leq \left[ (B_4)^{\frac{1}{4}}+\frac{|\mu_x|}{|\mu_y|}(1+(B_4)^{\frac{1}{4}}) \right]\left[4 + 2(1+(B_2)^{\frac{1}{2}})^2\frac{r^2}{|\mu_y|^2} \right]^\frac{1}{4}\sqrt{rh}.$$
Since $rh < 1$ it follows that $rh < \sqrt{rh}$.  Combining the above bound with \eqref{eqn:overallBound} gives
$$\|{\bf X}_{\tau} - E[{\bf X}_{\tau}]\|_{L^4} \leq C \sqrt{rh}$$
as claimed, with
$$C := 1+\left( (B_4)^{\frac{1}{4}}+\cot(\theta^*_r)(1+(B_4)^{\frac{1}{4}}) \right)\left(4 + 2(1+(B_2)^{\frac{1}{2}})^2\frac{r^2}{|\mu_y|^2} \right)^\frac{1}{4}.$$
\qed
\end{proof}

\vskip 2mm
\noindent {\bf Stage 2.}  The previous stage established quantitative bounds on the rate at which ${\bf X}_{\tau}$ is concentrating around its mean, while its mean converges to $\hat{\bf x}$.  The next step is to use the regularity of $u_0$ to express the rate at which $|E[u_0({\bf X}_{\tau})-u_0(\hat{\bf x})]|$ is tending towards zero in terms of these rates, as well as the probability of a ``rogue'' event (which we denote by $E^c_s$) that $E[{\bf X}_{\tau}]$ and $\hat{\bf x}$ each belong to one of the well behaved sets $U_i$ from Figure \ref{fig:setup}, but ${\bf X}_{\tau}$ does not.  Specifically, our bound will depend on $P(E^c_s)^{\frac{1}{2}}$, and it was in order to obtain a power of $\frac{1}{2}$ that we needed the fourth moment of ${\bf X}_{\tau}$ - if we had only the variance, we would end up with a power that depends on the regularity constants $0 \leq s' \leq s$.   This is accomplished by the following lemma, which, although we have stated it in a slightly more general setting with a general random vector ${\bf X}$ taking values in a general convex set $\Omega \subset \field{R}^2$, general convex sets $\{ U_{i}\}_{i=1}^M$ contained within $\Omega$, a general function $u : \Omega \rightarrow \field{R}$, a general norm $\|{\bf X}-E[{\bf X}]\|_{L^{p'}}$ with $p' \geq 4$, and a general point $\hat{\bf x} \in \Omega$, we will ultimately be applying this lemma only to ${\bf X}_{\tau}$, $\Omega = \mathcal U$, $p'=4$, and $\{U_i\}$ and $\hat{\bf x}$ as described above. 

\begin{lemma} \label{lemma:regularity}
Let $\Omega \subset \field{R}^2$ be convex and suppose $u \in W^{s',\infty}(\Omega)$, and in addition $u \in W^{s,\infty}(U_i)$ ($s \geq s'$) for each of a finite collection of convex sets $\{ U_{i}\}_{i=1}^M$ where each $U_i \subseteq \Omega$.  Let $\hat{\bf x} \in \Omega$ and let ${\bf X}$ be a random vector taking values in $\Omega$.  
Define
\begin{enumerate}
\item $C_{u,\Omega} := \|u\|_{W^{s',\infty}(\Omega)},$
\item $C_{u,\Omega,\{U_i\}} := \max\left\{ \|u\|_{W^{s',\infty}(\Omega)}, \|u\|_{W^{s,\infty}(U_1)},\ldots, \|u\|_{W^{s,\infty}(U_M)}\right\}$
\end{enumerate}
Then for any $p' \geq 1$ we have

1.
\begin{eqnarray*}
\qquad |E[u({\bf X})-u(\hat{\bf x})]|  \leq & C_{u,\Omega} \left\{ (\|{\bf X}-E[{\bf X}]\|_{L^{p'}})^{s' \wedge 2}+\|E[{\bf X}]-\hat{\bf x}\|^{s' \wedge 1}\right\},
\end{eqnarray*}
Next, suppose $\hat{\bf x} \in U_i$, $E[{\bf X}] \in U_i$ for some $1 \leq i \leq M$.  Denote by $E_s$ the event that ${\bf X} \in U_i$ as well.  Then for any $p' \geq 4$ such that $\|{\bf X}-E[{\bf X}]\|_{L^{p'}} \leq 1$ we have

2.
\begin{eqnarray*}
\qquad |E[u({\bf X})-u(\hat{\bf x})]| & \leq & C_{u,\Omega,\{U_i\}} \big\{ (\|{\bf X}-E[{\bf X}]\|_{L^{p'}})^{s \wedge 2}\\
&+&2(\|{\bf X}-E[{\bf X}]\|_{L^{p'}})^{s' \wedge 2}P(E^c_s)^\frac{1}{2}+\|E[{\bf X}]-\hat{\bf x}\|^{s \wedge 1} \big\},
\end{eqnarray*}
\end{lemma}
\begin{proof}
First let us define for convenience
$${\bf x}^* := E[{\bf X}],  \qquad {\bf x}_{E_s}^* := E[{\bf X} | E_s],  \qquad {\bf x}_{E^c_s}^* := E[{\bf X} | E^c_s].$$
For both statements, the first step is to divide the expectation into two pieces.
$$|E[u({\bf X}) - u(\hat{\bf x})]| \leq |E[u({\bf X}) - u({\bf x}^*)]|+|u({\bf x}^*) - u(\hat{\bf x})|:=\Pi_1+\Pi_2.$$
To prove statement one, we apply the H\"older condition \eqref{eqn:HolderLike} to find \\$\Pi_2 \leq C_{u,\Omega} \|{\bf x}^*-\hat{\bf x}\|^{s' \wedge 1}$.  Hence it suffices to prove $\Pi_1 \leq C_{u,\Omega}  E[\|{\bf X}-E[{\bf X}]\|^{p'}]^{\frac{s' \wedge 2}{p'}}.$  We proceed by cases.  If $s'<1$, then 
$$\Pi_1 \leq C_{u,\Omega}  E[\|{\bf X}-{\bf x}^*\|^{s'}]\leq C_{u,\Omega}  E[\|{\bf X}-{\bf x}^*\|^{p'}]^{\frac{s'}{p'}} = C_{u,\Omega}  E[\|{\bf X}-{\bf x}^*\|^{p'}]^{\frac{s' \wedge 2}{p'}}.$$ 
where we have used Jensen's inequality together with the concavity of $x^{\frac{s'}{p'}}$ on $[0,\infty)$ for the second inequality.  On the other hand, if $s' \geq 1$, then $\nabla u$ exists and by Taylor's theorem
$u({\bf X})-u({\bf x}_*)=\nabla u({\bf z}) \cdot ({\bf X}-{\bf x}_*)$ where ${\bf z} = (1-t){\bf x}_*+t{\bf X}$ for some $t \in [0,1]$.  Therefore
$$\Pi_1  = |E[ \nabla u({\bf z}) \cdot ({\bf X}-{\bf x}^*) ]|=|E[ (\nabla u({\bf z})-\nabla u({\bf x}^*)) \cdot ({\bf X}-{\bf x}^*)]+E[ \nabla u({\bf x}^*) \cdot ({\bf X}-{\bf x}^*) ]|.$$
But $$E[ \nabla u({\bf x}^*) \cdot ({\bf X}-{\bf x}^*) ]=\nabla u({\bf x}^*) \cdot E[{\bf X}-{\bf x}^* ]=\nabla u({\bf x}^*) \cdot ({\bf x}^*-{\bf x}^* )=0,$$ therefore
\begin{eqnarray*}
\Pi_1 &=& |E[ (\nabla u({\bf z})-\nabla u({\bf x}^*)) \cdot ({\bf X}-{\bf x}^*)]| \\
&\leq& E[\|\nabla u({\bf z})-\nabla u({\bf x}^*)\|\|{\bf X}-{\bf x}^*\| ] \\
&\leq& C_{u,\Omega}  E[\|{\bf z}-{\bf x}^*\|^{(s'-1)\wedge 1}\|{\bf X}-{\bf x}^*\| ] \\
&\leq& C_{u,\Omega}   E[\|{\bf X}-{\bf x}^*\|^{s' \wedge 2}] \\
&\leq& C_{u,\Omega}  E[\|{\bf X}-E[{\bf X}]\|^{p'}]^{\frac{s' \wedge 2}{p'}},
\end{eqnarray*}
where we have used the Cauchy-Schwarz inequality on line two, the H\"older condition \eqref{eqn:HolderLike} together with the convexity of $\Omega$ on line three, the bound \\$\|{\bf z}-{\bf x}^*\| \leq \|{\bf X}-{\bf x}^*\|$ on line four, and Jensen's inequality again on line five.

For the second statement, the idea is to split the expectation up as a sum of conditional expectations conditioning on $E_s$ and $E^c_s$, apply the corresponding regularity estimates, and express the results in terms of the conditional moments of ${\bf X}$.  Then, using the contractive property of conditional expectation, namely
$$E[\|E[{\bf X}|Y]\|^{p'}]\leq E[\|{\bf X}\|^{p'}]$$
valid for all random variables $Y$ adapted to the same sigma algebra as ${\bf X}$ and all $ p' \geq 1$, we can eliminate the conditional moments.  We do not use the contractive property itself, but rather two identities which may be easily derived from it.  In particular, let $A \in \sigma({\bf X})$ and define ${\bf x}^*_A = E[{\bf X}|A]$.  Then for any $p' \geq 1$ 
\begin{eqnarray} \label{eqn:contraction1}
\|{\bf x}^*_A-{\bf x}^*\|^{p'}P(A) &\leq& \|E[{\bf X}-{\bf x}^*|A]\|^{p'}P(A)+\|E[{\bf X}-{\bf x}^*|A^c]\|^{p'}P(A^c)  \\
&=&E[\|E[{\bf X}-{\bf x}^*| 1_A]\|^{p'}] \nonumber \\
&\leq&E[\|{\bf X}-{\bf x}^*\|^{p'}], \nonumber
\end{eqnarray}
and a similar manipulation gives 
\begin{equation} \label{eqn:contraction2}
E[\|{\bf X}-{\bf x}^*\|^{p'}|A]P(A) \leq E[\|{\bf X}-{\bf x}^*\|^{p'}].
\end{equation}
\noindent Similarly to part one, we have $\Pi_2 \leq C_{u,\Omega,\{U_i\}}\|{\bf x}^*-\hat{\bf x}\|^{s \wedge 1}$, since ${\bf x}^*,\hat{\bf x} \in U_i$.  It remains to prove
\begin{eqnarray*}
\Pi_1 \leq \Pi_1^* &:=& C_{u,\Omega,\{U_i\}} E[\|{\bf X}-E[{\bf X}]\|^{p'}]^{\frac{s \wedge 2}{p'}}\\
&+&2C_{u,\Omega,\{U_i\}} E[\|{\bf X}-E[{\bf X}]\|^{p'}]^{\frac{s' \wedge 2}{p'}}P(E^c_s)^{\frac{1}{2}}.
\end{eqnarray*}
We split $\Pi_1$ up as
\begin{eqnarray*}
\Pi_1 &=& |E[u({\bf X}) - u({\bf x}^*)]|\\ 
&= & |E[u({\bf X}) - u({\bf x}^*) | E_s]P(E_s)\\
&+&E[u({\bf X}) - u({\bf x}^*) | E^c_s]P(E^c_s)| \\
& := & |\Pi_{1,1} +\Pi_{1,2}|
\end{eqnarray*}
and proceed by cases.  First assume $0 \leq s' \leq  s < 1$.  Then, by a manipulation identical to the case $s'<1$ in part one, we have 
$$\Pi_{1,1} \leq C_{u,\Omega,\{U_i\}} (E[\|{\bf X}-{\bf x}^*\|^{p'} | E_s])^{\frac{s \wedge 2}{p'}} P(E_s)$$
and
$$\Pi_{1,2} \leq C_{u,\Omega,\{U_i\}} (E[\|{\bf X}-{\bf x}^*\|^{p'} | E^c_s])^{\frac{s' \wedge 2}{p'}} P(E^c_s).$$
We therefore have
\begin{eqnarray*}
\Pi_1 &\leq & C_{u,\Omega,\{U_i\}} (E[\|{\bf X}-{\bf x}^*\|^{p'} | E_s]P(E_s))^{\frac{s \wedge 2}{p'}} P(E_s)^{1-\frac{s \wedge 2}{p'}}\\
&+&C_{u,\Omega,\{U_i\}} (E[\|{\bf X}-{\bf x}^*\|^{p'} | E^c_s]P(E^c_s))^{\frac{s' \wedge 2}{p'}} P(E^c_s)^{1-\frac{s' \wedge 2}{p'}} \\
&\leq & C_{u,\Omega,\{U_i\}} E[\|{\bf X}-{\bf x}^*\|^{p'}]^{\frac{s \wedge 2}{p'}}+C_{u,\Omega,\{U_i\}}E[\|{\bf X}-{\bf x}^*\|^{p'}]^{\frac{s' \wedge 2}{p'}} P(E^c_s)^{\frac{1}{2}} \\
&\leq & \Pi^*_1.
\end{eqnarray*}
Where we have used \eqref{eqn:contraction2} and $p' \geq 4$ on line three.  Next suppose $0 \leq s' \leq 1$ but $s \geq 1$.  Then $\nabla u$ exists on $U_i$ and has norm bounded by $C_{u,\Omega,\{U_i\}}$, and hence whenever ${\bf X} \in U_i$ we can find a ${\bf z}=t{\bf x}^*+(1-t){\bf X} \in U_i$ for some $t \in [0,1]$ such that
$u({\bf X}) = \nabla u({\bf z}) \cdot ({\bf X}-{\bf x}^*)$, by Taylor's theorem and the convexity of $U_i$.  Therefore
\begin{eqnarray*}
\Pi_{1,1} &\leq & |E[(\nabla u({\bf z})-\nabla u({\bf x}^*)) \cdot ({\bf X}-{\bf x}^*) | E_s]P(E_s)\\
&+&\nabla u({\bf x}^*)  \cdot E[{\bf X}-{\bf x}^* | E_s]P(E_s)|,
\end{eqnarray*}
For the first term we can apply an argument that is identical to the case $s' \geq 1$ in part one.  However, unlike in part one, the second term does not vanish as
our expectation is now conditional.  We obtain
\begin{eqnarray*}
\Pi_{1,1} &\leq &  (E[\|{\bf X}-{\bf x}^*\|^{p'} | E_s]P(E_s))^{\frac{s \wedge 2}{p'}}P(E_s)^{1-\frac{s \wedge 2}{p'}}+\|\nabla u({\bf x}^*)\|\|{\bf x}^*_{E_s}-{\bf x}^*\|P(E_s) \\
& \leq & C_{u,\Omega,\{U_i\}} E[\|{\bf X}-{\bf x}^*\|^{p'}]^{\frac{s \wedge 2}{p'}}+\|\nabla u({\bf x}^*)\|\|{\bf x}^*_{E^c_s}-{\bf x}^*\|P(E^c_s) \\
& \leq & C_{u,\Omega,\{U_i\}} E[\|{\bf X}-{\bf x}^*\|^{p'}]^{\frac{s \wedge 2}{p'}}+C_{u,\Omega,\{U_i\}}E[\|{\bf X}-{\bf x}^*\|^{p'}]^{\frac{1}{p'}}P(E^c_s)^{1-\frac{1}{p'}}.
\end{eqnarray*}
Where we have used the manipulation 
$${\bf x}^*_{E_s}P(E_s)+{\bf x}^*_{E^c_s}P(E^c_s)={\bf x}^*={\bf x}^*P(E_s)+{\bf x}^*P(E^c_s)$$
\begin{eqnarray} \label{eqn:manipulation}
&\Rightarrow& ({\bf x}^*-{\bf x}^*_{E_s})P(E_s)=({\bf x}^*_{E^c_s}-{\bf x}^*)P(E_s^c).
\end{eqnarray}
together with \eqref{eqn:contraction2} on line two, and \eqref{eqn:contraction1} on line three.  The final trick is to note that since $E[\|{\bf X}-{\bf x}^*\|^{p'}] \leq 1$ and $s' \leq 1$, we have $E[\|{\bf X}-{\bf x}^*\|^{p'}]^{\frac{1}{p'}} \leq E[\|{\bf X}-{\bf x}^*\|^{p'}]^{\frac{s' \wedge 2}{p'}}$.  Hence, bounding $\Pi_{1,2}$ in exactly the same way as in the previous case, we have
$$\Pi_1 \leq C_{u,\Omega,\{U_i\}} E[\|{\bf X}-{\bf x}^*\|^{p'}]^{\frac{s \wedge 2}{p'}}+2C_{u,\Omega,\{U_i\}}E[\|{\bf X}-{\bf x}^*\|^{p'}]^{\frac{s' \wedge 2}{p'}}P(E^c_s)^{\frac{1}{2}} \leq \Pi^*_1,$$
where we have used $p \geq 4$ again to write $P(E^c_s)^{\frac{1}{p'}} \leq P(E^c_s)^{\frac{1}{2}}$.  Finally, we consider $1 \leq s' \leq  s$.  In this case $\nabla u$ exists everywhere and applying Taylor's theorem twice we have
\begin{eqnarray*}
\Pi_1 & = &  | E[(\nabla u({\bf z}_1)-\nabla u({\bf x}^*)) \cdot ({\bf X}-{\bf x}^*) | E_s]P(E_s)\\
&+&E[(\nabla u({\bf z}_2)-\nabla u({\bf x}^*)) \cdot ({\bf X}-{\bf x}^*) | E^c_s]P(E^c_s)  \\
& + & \nabla u({\bf x}^*) \cdot \underbrace{ \left\{  ({\bf x}_{E_s}^*-{\bf x}^*)P(E_s)+({\bf x}_{E^c_s}^*-{\bf x}^*)P(E^c_s) \right\}|, }_{\mbox{ =0 by \eqref{eqn:manipulation}}}
\end{eqnarray*}
where ${\bf z}_1=t_1{\bf x}^*+(1-t_1){\bf X} \in U_i$, ${\bf z}_2=t_2{\bf x}^*+(1-t_2){\bf X} \in U_i$ for some $t_1, t_2 \in [0,1]$ (again, by the convexity of $U_i$).  Applying an argument almost identical to the previous step, we then obtain
\begin{eqnarray*}
\Pi_1 & \leq &  C_{u,\Omega,\{U_i\}} (E[\|{\bf X}-{\bf x}^*\|^{p'} | E_s]P(E_s))^{\frac{s \wedge 2}{p'}}P(E_s)^{1-{\frac{s \wedge 2}{p'}}}\\
&+&C_{u,\Omega,\{U_i\}} (E[\|{\bf X}-{\bf x}^*\|^{p'} | E_s]P(E_s^c))^{\frac{s' \wedge 2}{p'}}P(E^c_s)^{1-\frac{s' \wedge 2}{p'} } \\
& \leq & C_{u,\Omega,\{U_i\}} E[\|{\bf X}-{\bf x}^*\|^{p'} ]^{\frac{s \wedge 2}{p'}}+C_{u,\Omega,\{U_i\}} E[\|{\bf X}-{\bf x}^*\|^{p'}]^{\frac{s' \wedge 2}{p'}}P(E^c_s)^{\frac{1}{2}} \\
& \leq & \Pi_1^*,
\end{eqnarray*}
where we have used \eqref{eqn:contraction2} and $p' \geq 4$ on line three.
\qed
\end{proof}

\begin{figure}
\centering\includegraphics[width=1.0\linewidth]{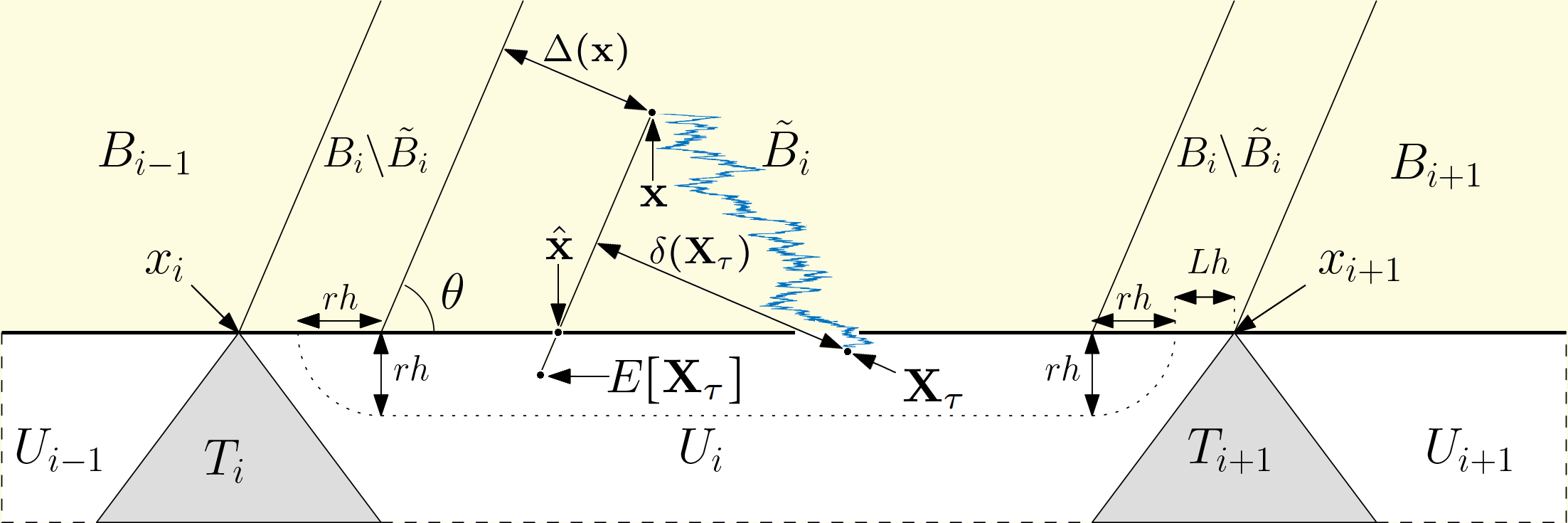}
\caption{{\bf Illustration of the stopped random walk ${\bf X}_{\tau}$ as well as the subdivision of $D$ into bands:}  The inpainting domain $D=(0,1]^2$ is partitioned into bands $\{B_i\}_{i=1}^M$, such that each band $B_i:=\Pi_{\theta^*_r}^{-1}((x_i,x_{i+1}] \times \{ 0 \})$, where $\Pi_{\theta^*_r}$ is the transport map \eqref{eqn:weak}.  With in each band we identify a subband $\tilde{B}_i$, of width dependent on $h$, such that if ${\bf X}_{\tau}$ starts in $\tilde{B}_i$, then $E[{\bf X}_{\tau}] \in U_i$ and $\hat{\bf x} \in U_i$.  As $h \rightarrow 0$, the area of $B_i \backslash \tilde{B}_i$ goes to zero.  At the same time we illustrate an example stopped random walk ${\bf X}_{\tau}$ started from ${\bf x}$ (blue curve) as well the relationship between ${\bf x}$, $E[{\bf X}_{\tau}]$ and $\hat{\bf x} = \Pi_{\theta^*_r}({\bf x})$.  The orthogonal distances $\Delta({\bf x})$ from ${\bf x}$ to $\partial \tilde{B}_i$ and $\delta({\bf X}_{\tau})$ from ${\bf X}_{\tau}$ to the ray from ${\bf x}$ to $\hat{\bf x}$ are also illustrated.  Note that $\hat{\bf x}$ is always on the line $y=0$, whereas ${\bf X}_{\tau}$ and $E[{\bf X}_{\tau}]$ will typically overshoot it.  However, they can only overshoot by distance at most $rh$.  That is, ${\bf X}_{\tau}$ and $E[{\bf X}_{\tau}]$ must always sit between the lines $y=-rh$ and $y=0$.  Note that Pythagorean theorem implies $\Delta({\bf x}) \leq 1$ regardless of where ${\bf x}$ is placed in $D$ and regardless of the number and position of the bands.}
\label{fig:bands}
\end{figure}

\vskip 2mm
\noindent {\bf Stage 3.}  In this stage our objective is to partition $D$ into a series of bands $\{B_i\}_{i=1}^M$ and sub-bands $\tilde{B}_i \subset B_i$ such that the area of the complement $B_i \backslash \tilde{B}_i$ goes to zero as $h \rightarrow 0$, and such that the starting position ${\bf x} \in \tilde{B}_i$ guarantees $E[{\bf X}_{\tau}] \in U_i$ and $\hat{\bf x} \in U_i$.  Each band $B_i$ is associated with the interval $(x_i,x_{i+1}] \times \{0\}$ (see Figure \ref{fig:setup} for a reminder of what $x_i$ is) and is equal to its inverse image under the transport map $ \Pi_{\theta^*_r} : D \rightarrow (0,1] \times \{0\}$ given by \eqref{eqn:weak}.  Within each band $B_i$ we define a sub-band $\tilde{B}_i$ given by the inverse image of $(x^+_i,x^-_{i+1}] \times \{0\}$ under $\Pi_{\theta_r}$, where $x^+_i:=x_i+(L+r)h$ and $x^-_{i+1}:=x_{i+1}-(L+r)h$.  Note that the width of $\tilde{B}_i$ is dependent on $h$, even though our choice of notation does not make this obvious.  As usual, we denote by $B_{i,h}$, $\tilde{B}_{i,h}$ the intersection of these bands with $D_h$.  The set $\tilde{B}_{i,h}$ is significant because we know $\hat{\bf x}, E[{\bf X}_{\tau}] \in U_i$ if ${\bf x} \in \tilde{B}_{i,h}$.  This situation is illustrated in Figure \ref{fig:bands}.

\vskip 2mm
\noindent {\bf Stage 4.}  In stage two we gave a bound for $|E[u_0({\bf X}_{\tau})-u_0(\hat{\bf x})]|$, valid when $E[{\bf X}_{\tau}]$ and $\hat{\bf x}$ each belong to one of the well behaved sets $U_i$, that depends on $\|E[{\bf X}_{\tau}]-\hat{\bf x}\|$ and $\|{\bf X}_{\tau}-E[{\bf X}_{\tau}]\|_{L^{p'}}$ for $p' \geq 4$, as well as the probability of the event $E^c_s$, a ``rogue'' event where even though $E[{\bf X}_{\tau}]$ and $\hat{\bf x}$ each belong to $U_i$, the stopped random walk ${\bf X}_{\tau}$ nevertheless lands outside $U_i$.  In stage three we broke $D$ into a series of bands $\{B_i\}_{i=1}^M$ and sub-bands $\tilde{B}_i \subset B_i$ such that the starting position ${\bf x} \in \tilde{B}_i$ guarantees $E[{\bf X}_{\tau}] \in U_i$ and $\hat{\bf x} \in U_i$.  We already have bounds on $\|E[{\bf X}_{\tau}]-\hat{\bf x}\|$ and $\|{\bf X}_{\tau}-E[{\bf X}_{\tau}]\|_{L^{4}}$ from stage one.  What remains is to bound $P(E^c_s)$.  We will accomplish this in two steps using elementary arguments based on martingales and Azuma's inequality.  In step one, we show that $\tau$ is bounded above by a constant (dependent on $h$) with a high probability.  In step two, we show that so long as $\tau$ is smaller than this constant, then ${\bf X}_{\tau}$ is unlikely to drift very far from $E[{\bf X}_{\tau}]$.  We then put these facts to together to bound $P(E^c_s)$.

\vskip 1mm

\noindent {\em Step 1.}  $P\left(\tau > \lceil \frac{2}{|\mu_y|h} \rceil \right) \leq \exp\left(-\frac{1}{4\frac{r}{|\mu_y|}rh}\right).$
\begin{proof}
We define
$$M_k := Y_{k \wedge \tau} - (\tau \wedge k) \mu_yh -mh$$
which the reader may verify is a zero mean martingale with bounded increments
$$|M_{k+1}-M_k| \leq rh.$$
Next we note that for any $k$ the following events are equal:
$$\{ \tau \geq k \} = \{Y_{\tau \wedge k} \geq 0\} = \{ M_k \geq -(k \wedge \tau) \mu_y h - mh \} = \{ M_k \geq -k\mu_y h - mh \}.$$
Therefore
$$\left \{ \tau \geq \left \lceil \frac{2}{|\mu_y|h} \right \rceil \right \} = \left \{ M_{ \left \lceil \frac{2}{|\mu_y|h} \right \rceil} \geq - \left \lceil \frac{2}{|\mu_y|h} \right \rceil\mu_y h - mh \right \} \subseteq \left \{ M_{ \left \lceil \frac{2}{|\mu_y|h} \right \rceil} \geq 1 \right \}$$
where we have used the inequality
$$-\left \lceil \frac{2}{|\mu_y|h} \right \rceil\mu_y h - mh = \left \lceil \frac{2}{|\mu_y|h} \right \rceil|\mu_y| h - mh \geq 2-mh \geq 1.$$
Noting that $M_0 = 0$ we apply Azuma's inequality to find
\begin{eqnarray*}
P\left(\tau > \left \lceil \frac{2}{|\mu_y|h} \right \rceil \right) &\leq& P \left(M_{ \left \lceil \frac{2}{|\mu_y|h} \right \rceil} \geq 1 \right)\\
& \leq & \exp\left( - \frac{1}{2\left \lceil \frac{2}{|\mu_y|h} \right \rceil r^2h^2} \right) \leq \exp\left(-\frac{1}{4\frac{r}{|\mu_y|}rh}\right).
\end{eqnarray*}
\qed
\end{proof}
\vskip 1mm

\noindent {\em Step 2.}  Let $\Delta({\bf x})$ denote the orthogonal distance from ${\bf x}$ to $\partial \tilde{B}_i$ (see Figure \ref{fig:bands}).  Then 
$$P(E^c_s) \leq 3\exp\left(-\frac{\Delta({\bf x})^2}{4\frac{r}{|\mu_y|}rh}\right).$$
\begin{proof}
This time we define
$$M_k := ({\bf X}_{\tau \wedge k} - {\bf x}) \cdot (-\sin\theta^*_r,\cos\theta^*_r).$$
Once again, $M_k$ is a zero-mean martingale with bounded increments \\$|M_{k+1}-M_k| \leq rh$ obeying $M_0 = 0$.  Moreover, if $\tau \leq k$, then $|M_k|$ has a geometric interpretation as the orthogonal distance $\delta({\bf X}_{\tau})$ from ${\bf X}_{\tau}$ to the line passing through the points ${\bf x}$ and $\hat{\bf x}$.  In addition, we clearly have ${\bf X}_{\tau} \notin U_i$ only if $\delta({\bf X}_{\tau}) > \Delta({\bf x})$, where $\Delta({\bf x})$ denotes the orthogonal distance from ${\bf x}$ to $\partial \tilde{B}_i$.  See Figure \ref{fig:bands} for an illustration.  Hence we have the containment
$$E^c_s \cap \{ \tau \leq k\} \subseteq \{|M_k| > \Delta({\bf x})\}.$$
Applying Azuma's inequality again gives, for any $k \in \field{N}$,
$$P(|M_k| > \Delta({\bf x})) \leq 2\exp\left( -\frac{\Delta({\bf x})^2}{2kr^2h^2}\right).$$
At the same time, for any integer $k \in \field{N}$ we have 
$$P(E^c_s)  \leq  P(E^c_s \cap \{ \tau \leq k \})+P(E^c_s \cap \{ \tau > k \}) \leq P(|M_k| > \Delta({\bf x})) + P(\tau > k).$$
Taking $k = \lceil \frac{2}{|\mu_y|h} \rceil$, substituting the above bound as well as the one from Step 1 gives
$$P(E^c_s)  \leq 2\exp\left(-\frac{\Delta({\bf x})^2}{4\frac{r}{|\mu_y|}rh}\right)+\exp\left(-\frac{1}{4\frac{r}{|\mu_y|}rh}\right) \leq 3\exp\left(-\frac{\Delta({\bf x})^2}{4\frac{r}{|\mu_y|}rh}\right),$$
since $\Delta({\bf x}) \leq 1$ (this follows trivially from the Pythagorean theorem - see Figure \ref{fig:bands}).
\qed
\end{proof}

\vskip 2mm
\noindent {\bf Stage 5.}  We are now ready to prove our main result, which we do assuming $p < \infty$ (the case $p = \infty$ is treated as a limit).  The idea is to split $\|u-u_h\|_p$, which is defined as a sum over $D_h$, into separate sums over the bands $\{ B_{i,h} \}$ defined in stage three.  Within in each band we split further into a sum over $\tilde{B}_{i,h}$ and $B_{i,h} \backslash \tilde{B}_{i,h}$, that is
\begin{eqnarray} \label{eqn:splitting}
\|u-u_h\|_p &=& \|E[u({\bf X}_{\tau})-u(\hat{\bf x})]\|_p \nonumber\\
&\leq & \sum_{i=1}^M \Big\{  \|E[u({\bf X}_{\tau})-u(\hat{\bf x})]1_{B_{i,h} \backslash \tilde{B}_{i,h}}\|_p \nonumber\\
&+& \|E[u({\bf X}_{\tau})-u(\hat{\bf x})]1_{\tilde{B}_{i,h}}\|_p   \Big\}.
\end{eqnarray}
This stage is itself divided into three steps, where first to derive a bound for \\$\|E[u({\bf X}_{\tau})-u(\hat{\bf x})]1_{B_{i,h} \backslash \tilde{B}_{i,h}}\|_p$, then derive one for $\|E[u({\bf X}_{\tau})-u(\hat{\bf x})]1_{\tilde{B}_{i,h}}\|_p$, and then put the bounds together.  Step one is by far the hardest.
\vskip 1mm
\noindent{\em Step 1:  A bound for $\|E[u({\bf X}_{\tau})-u(\hat{\bf x})]1_{B_{i,h} \backslash \tilde{B}_{i,h}}\|_p$.}  
\vskip 1mm
\noindent On each $\tilde{B}_{i,h}$, we can apply our estimate from statement two of Lemma \ref{lemma:regularity} from Stage 2.
\begin{eqnarray*}
\|E[u({\bf X}_{\tau})-u(\hat{\bf x})]1_{\tilde{B}_{i,h}}\|_p &\leq&  C_{u_0,\mathcal U,\{U_i\}}\Big\| \Big(\left\{\|{\bf X}_{\tau}-E[{\bf X}_{\tau}]\|_{L^{4}}\right\}^{s \wedge 2} \\
&+& 2\left\{\|{\bf X}_{\tau}-E[{\bf X}_{\tau}]\|_{L^{4}}\right\}^{\frac{s' \wedge 2}{4}}P(E^c_s)^\frac{1}{2}\\
 &+&\|E[{\bf X}_{\tau}]-\hat{\bf x}\|^{s \wedge 1}\Big) 1_{\tilde{B}_{i,h}} \Big \|_p
\end{eqnarray*}
Substituting in our estimates of $\|{\bf X}_{\tau}-E[{\bf X}_{\tau}]\|_{L^{4}}$ and $\|E[{\bf X}_{\tau}]-\hat{\bf x}\|$ from Stage 1 gives
$$\|E[u({\bf X}_{\tau})-u(\hat{\bf x})]1_{\tilde{B}_{i,h}}\|_p \leq C_{u_0,\mathcal U,\{U_i\}}(C+1)(rh)^{\frac{s}{2} \wedge 1} + 2(rh)^{\frac{s'}{2} \wedge 1} \| P(E^c_s) 1_{\tilde{B}_{i,h}}\|_p,$$
where we have used $rh < 1$ to bound $(rh)^{s \wedge 1} \leq (rh)^{\frac{s}{2} \wedge 1}$.  Our next job is to bound $\| P(E^c_s) 1_{\tilde{B}_{i,h}}\|_p$, which we do by applying our bound from stage four and then making an argument based on numerical quadrature to replace the sum in $\|\cdot\|_p$ with an integral that is easier to evaluate.  Specifically, let us define
$$f({\bf x}) = \exp\left(-\frac{p\Delta({\bf x})^2}{4\frac{r}{|\mu_y|}rh}\right)$$
for convenience ($f$ is one third of the bound on $P(E^c_s)$ we derived in stage 4).  Then
\begin{eqnarray*}
\| P(E^c_s) 1_{\tilde{B}_{i,h}}\|_p &\leq& 3\|f({\bf x})1_{\tilde{B}_{i,h}}\|_p \\
&\leq& 3\left( \int_{\tilde{B}_i} f({\bf x})d{\bf x} + \max_{{\bf x} \in \tilde{B}_i} \|\nabla f({\bf x})\|h + \max_{{\bf x} \in \tilde{B}_i} |f({\bf x})|h \right)^{\frac{1}{p}}.
\end{eqnarray*}
This is because $\|f({\bf x})1_{\tilde{B}_{i,h}}\|^p_p$ can be interpreted as a Riemann sum for $\int_{\tilde{B}_i} f({\bf x})d{\bf x}$.  The two error terms on the left come respectively from the error in the rectangle rule and error in the approximation of our non-rectangular domain by squares.  One may readily show $|f({\bf x})| \leq 1$ while 
$$\|\nabla f({\bf x})\|h \leq \|\nabla f({\bf x})\|rh \leq \frac{e^{-\frac{1}{2}}p^{\frac{1}{2}}}{\sqrt{2\frac{r}{|\mu_y|}}}\sqrt{rh}.$$
At the same time, changing coordinates to $(u,v)$ with $u$ parallel to $\partial \tilde{B}_i$ and $v$ perpendicular one easily obtains
$$\int_{\tilde{B}_i} f({\bf x})d{\bf x} \leq |\csc({\theta_r^*})| \int_{-\infty}^{\infty} 2\exp\left(-\frac{pv^2}{4\frac{r}{|\mu_y|}rh}\right)dv = 4|\csc({\theta_r^*})|\sqrt{\frac{\pi r}{p|\mu_y|}}(rh)^{\frac{1}{2}},$$
where the factor of $|\csc({\theta_r^*})|$ has appeared because this is the length of each side of $\partial \tilde{B}_i$ (see Figure \ref{fig:bands}).  Putting it together gives
$$\| P(E^c_s) 1_{\tilde{B}_{i,h}}\|_p \leq A_p (rh)^{\frac{1}{2p}},$$
where $A_p$ is given by
\begin{eqnarray*}
A_p &=& 3\left(1+\frac{e^{-\frac{1}{2}}p^{\frac{1}{2}}}{\sqrt{2\frac{r}{|\mu_y|}}}+2|\csc(\theta^*_r)|\sqrt{\frac{\pi r}{p|\mu_y|}}\right)^{\frac{1}{p}} \\
&\leq& 3\left(1+\frac{e^{-\frac{1}{2}}}{\sqrt{2\frac{r}{|\mu_y|}}}+2|\csc(\theta^*_r)|\sqrt{\frac{\pi r}{|\mu_y|}}\right)e^{\frac{1}{2e}}:=A,
\end{eqnarray*}
and where we have used $1 \leq p$ to pull out a common factor of $p^{\frac{1}{2}}$, applied the bound $p^{\frac{1}{2p}} \leq e^{\frac{1}{2e}}$, and noticed that what is left over is maximized at $p=1$.  Thus we have 
$$\| P(E^c_s) 1_{\tilde{B}_{i,h}}\|_p \leq A (rh)^{\frac{1}{2p}},$$
where $A$ is a constant depending only on $\frac{r}{|\mu_y|}$ and $\theta^*_r$.  Finally, we obtain the bound
\begin{eqnarray} \label{eqn:bound1}
\|E[u({\bf X}_{\tau})-u(\hat{\bf x})]1_{\tilde{B}_{i,h}}\|_p &\leq& C_{u_0,\mathcal U,\{U_i\}}(C+1)(rh)^{\frac{s}{2} \wedge 1} + 2A (rh)^{\left(\frac{s'}{2}+\frac{1}{2p}\right) \wedge \left(1+\frac{1}{2p}\right)} \nonumber \\
&\leq& \max(C_{u_0,\mathcal U,\{U_i\}}(C+1),2A) (rh)^{ \frac{s}{2} \wedge \left(\frac{s'}{2}+\frac{1}{2p}\right) \wedge 1}.
\end{eqnarray}
\vskip 1mm
\noindent {\em Step 2:  A bound for $\|E[u({\bf X}_{\tau})-u(\hat{\bf x})]1_{\tilde{B}_{i,h}}\|_p$.}  
\vskip 1mm
\noindent This step is easier.  Applying  statement one of Lemma \ref{lemma:regularity} from Stage 2 this time, and substituting in our results from stage one as before gives
\begin{eqnarray*}
\|E[u({\bf X}_{\tau})-u(\hat{\bf x})]1_{B_{i,h} \backslash \tilde{B}_{i,h}}\|_p &\leq & C_{u_0,\mathcal U} \Big\| \Big( \left\{ \|{\bf X}_{\tau}-E[{\bf X}_{\tau}]\|_{L^{4}}\right\}^{s' \wedge 2} \\
&+&\|E[{\bf X}_{\tau}]-\hat{\bf x}\|^{s' \wedge 1}\Big) 1_{B_{i,h} \backslash \tilde{B}_{i,h}} \Big \|_p \\
& \leq & C_{u_0,\mathcal U}(C+1)(rh)^{\frac{s'}{2} \wedge 1} \| 1_{B_{i,h} \backslash \tilde{B}_{i,h}} \|_p
\end{eqnarray*}
It remains to bound $\| 1_{B_{i,h} \backslash \tilde{B}_{i,h}} \|_p$.  But 
$$\|1_{B_{i,h} \backslash \tilde{B}_{i,h}} \|_p=(|B_{i,h} \backslash \tilde{B}_{i,h}|h^2)^{\frac{1}{p}} =(2(L+r)Nh^2)^{\frac{1}{p}} \leq 2(L+1)(rh)^{\frac{1}{p}} \leq 2(L+1)(rh)^{\frac{1}{2p}},$$ 
hence
\begin{equation} \label{eqn:bound2}
\|E[u({\bf X}_{\tau})-u(\hat{\bf x})]1_{B_{i,h} \backslash \tilde{B}_{i,h}}\|_p \leq 2C_{u_0,\mathcal U}(C+1)(L+1)(rh)^{\frac{s'}{2}+\frac{1}{2p}}.
\end{equation}
\vskip 1mm
\noindent{\em Step 3:  Putting the bounds together.} 
\vskip 1mm
\noindent Combining \eqref{eqn:splitting}, \eqref{eqn:bound1}, and \eqref{eqn:bound2}, we at last obtain
\begin{equation} \label{eqn:finalBound}
\|u-u_h\|_p \leq K^* (rh)^{ \frac{s}{2} \wedge \left(\frac{s'}{2}+\frac{1}{2p}\right) \wedge 1},
\end{equation}
where $K^*$ is the constant 
$$K^* = \max(C_{u_0,\mathcal U,\{U_i\}}(C+1),2A,2C_{u_0,\mathcal U}(C+1)(L+1))$$
which depends only on $u_0$, $\mathcal U$, $\{U_i\}$, $r/|\mu_y|$, and $\theta^*_r$ as claimed.  Moreover, $K$ is a monotonically increasing function of $\frac{r}{|\mu_y|}$ and $K^* \rightarrow \infty$ as $\theta^*_r \rightarrow 0$ or $|\mu_y| \rightarrow 0$ as claimed.

\vskip 2mm

\noindent {\bf Stage 6.}  The previous five stages all assumed that $a^*_r \subseteq b^-_r$ or $a^*_r \subseteq b^0_r$ so that we could exploit the connection between Algorithm 1 and stopped random walks.  Now all that remains is to generalize to arbitrary stencils $a^*_r$ possibly containing ghost pixels.  This follows easily from the idea of equivalent weights from Section \ref{sec:fiction}.

Now assume $a^*_r$ is arbitrary.  By \eqref{eqn:universalContainment} we have $\mbox{Supp}(a^*_r) \subseteq b^0_{r+2}$, we know that \eqref{eqn:finalBound} holds with $r$ replaced by $r+2$, where $u$ is the solution to the transport equation \eqref{eqn:transport} with transport direction ${\bf g}^*_r$ given in terms of $\mbox{Supp}(a^*_r)$ and the equivalent weights $\tilde{w}_r$.  But we have already noted in Definition \ref{def:stencil} that it doesn't matter whether the center of mass is calculated in terms of the original weights $w_r$ and stencil $a^*_r$, or the equivalent weights $\tilde{w}_r$ and modified stencil $\mbox{Supp}(a^*_r)$.  Hence we have
$${\bf g}^*_r = \frac{\sum_{{\bf y} \in \mbox{Supp}(a^*_r)} \tilde{w}_r({\bf 0},{\bf y}){\bf y}}{\sum_{{\bf y} \in \mbox{Supp}(a^*_r)} \tilde{w}_r({\bf 0},{\bf y})} = \frac{\sum_{{\bf y} \in a^*_r} w_r({\bf 0},{\bf y}){\bf y}}{\sum_{{\bf y} \in a^*_r} w_r({\bf 0},{\bf y})}$$
as claimed.  

We are now almost done.  All that remains is some tidying up of the constant in \eqref{eqn:finalBound}, which we are forced to do since we had to replace $r$ with $r+2$.  Specifically, we know
$$\|u-u_h\|_p \leq K\left(\frac{r+2}{|\mu_y|}\right)^* ((r+2)h)^{ \frac{s}{2} \wedge \left(\frac{s'}{2}+\frac{1}{2p}\right) \wedge 1}$$
where we have written down the dependence of $K^*$ on $\frac{r}{|\mu_y|}$ explicitly.  Since $K^*$ increases monotonically with $\frac{r}{|\mu_y|}$, and since $(\cdot)^{\left(\frac{s'}{2}+\frac{1}{2p}\right) \wedge 1}$ is itself a monotonically increasing function, we can replace all occurrences of $r+2$ with $3r$.  This gives
$$\|u-u_h\|_p \leq 3K\left(3\frac{r}{|\mu_y|}\right)^* (rh)^{ \frac{s}{2} \wedge \left(\frac{s'}{2}+\frac{1}{2p}\right) \wedge 1}.$$
We certainly have $K\left(\frac{r}{|\mu_y|}\right)^* \leq 3K\left(3\frac{r}{|\mu_y|}\right)^*$, so we now define
$$K = 3K\left(3\frac{r}{|\mu_y|}\right)^*$$
to obtain our final constant that works for both ghost pixels and real pixels.
\qed

\subsection{Convergence to M\"arz's limit} \label{sec:marzLimit}

In the previous section we proved convergence of the inpainted function $u_h$ to the fixed ratio continuum limit $u$ given by the weak solution \eqref{eqn:weak} to the transport problem \eqref{eqn:transport}, as $h \rightarrow 0$ with $r = \epsilon/h$ fixed.  In \cite{Marz2007}, M\"arz and Bornemann also considered convergence of Algorithm 1 (which they called ``the generic single pass algorithm'') to a continuum limit, under a high resolution and vanishing viscosity double limit where first $h \rightarrow 0$ and then $\epsilon=rh \rightarrow 0$.  Their limit, which we refer to here as $u_{\mbox{m\"arz}}$, is also expressed as the solution to a transport equation, but with different coefficients.  As already stated in Section \ref{sec:continuumLimit1}, if the additional condition \eqref{eqn:centerOfMass} is satisfied, then we can draw a connection between these limits.  We have also already stated that coherence transport, Guidefill, and semi-implicit Guidefill all satisfy \eqref{eqn:centerOfMass} and converge to the same limit $u_{\mbox{m\"arz}}$, in this case with transport direction given by
$${\bf g}^* = \frac{\int_{B^-_1({\bf 0})} w_1({\bf 0},{\bf y}) {\bf y}d{\bf y}}{\int_{B^-_1({\bf 0})} w_1({\bf 0},{\bf y})d{\bf y}}$$
where as before $B^-_1({\bf 0})$ is the unit disk intersected with the lower half plane and $w_1({\bf 0},{\bf y})$ are the weights \eqref{eqn:weight} used by all three methods with $\epsilon =1$.  Note that ${\bf g}^*$ depends implicitly on the parameter $\mu$ in the weights \eqref{eqn:weight}.  Let us make this explicit for a moment by writing ${\bf g}^*_{\mu}$ in place of ${\bf g}^*$.  Then, if ${\bf g}=(\cos\theta,\sin\theta)$ with $\theta \in [0,\pi)$ denotes the guidance direction, we have from \cite[pg. 14, equation (14)]{Marz2007}
\begin{equation} \label{eqn:marzNoKink}
\lim_{\mu \rightarrow \infty} \theta({\bf g}^*_{\mu}) = 
\begin{cases} 
\theta & \mbox{ if } \theta \neq 0 \\
\frac{\pi}{2} & \mbox{ if } \theta = 0
\end{cases}
\end{equation}
where $\theta({\bf g}^*_{\mu})$ denotes as usual the counterclockwise angle that the line \\$L_{{\bf g}^*_{\mu}}:=\{ \lambda{\bf g}^*_{\mu} : \lambda \in \field{R}\} $ makes with the $x$-axis.  In other words, the continuum limit from \cite{Marz2007} predicts no kinking unless ${\bf g}$ is exactly parallel to $\partial D_h$, and moreover, predicts that the three methods listed above exhibit exactly the same behaviour as $\mu \rightarrow \infty$.  This is not what is observed in practice and we will draw very different conclusions in Section \ref{sec:kink3main}.
We have three objectives:
\begin{enumerate}
\item To rigorously establish convergence of $u_h$ to $u_{\mbox{m\"arz}}$, which was treated as a ``formal limit'' in \cite{Marz2007}, at least under the simplifying assumptions of Section \ref{sec:symmetry}.  In fact, although $u_{\mbox{m\"arz}}$ was treated as an iterated double limit in \cite{Marz2007}, we will see that $u_h \rightarrow u_{\mbox{m\"arz}}$ in the single {\em simultaneous} limit where $h \rightarrow 0$ and $\epsilon \rightarrow 0$ at the same time, but with a nuance.  It is not in general enough that $h$ and $\epsilon$ separately vanish - in general they do so in such a way that their ratio $r = \epsilon/h$ diverges.
\item To show that for $r \gg 1$, the two limits $u$ and $u_{\mbox{m\"arz}}$ are similar, i.e. $u \rightarrow u_{\mbox{m\"arz}}$.
\item To argue that when $r \in \field{N}$ is small (say close to $r=5$, the recommended default value in \cite{Marz2007}), our limit $u$ typically does a better job of capturing the behaviour of $u_h$ than $u_{\mbox{m\"arz}}$ does, i.e. $\|u_h-u\|_p \ll \|u_h - u_{\mbox{m\"arz}}\|_p$.
\end{enumerate}
The following Theorem accomplishes objectives 1-2.  Objective 3 is discussed in Remark \ref{rem:marzLimit} and supported numerically in Section \ref{sec:numerics}.

\vskip 1mm

\begin{theorem} \label{thm:convergence2}  Suppose ${\bf g}^* = \lim_{r \rightarrow \infty} \frac{{\bf g}^*_r}{r}$ exists and converges with rate at least $O(r^{-q})$ as in \eqref{eqn:centerOfMass}, for some $q>0$.  Let $u_{\mbox{m\"arz}}({\bf x}) := u_0(\Pi_{\theta^*}({\bf x}))$ denote the weak solution to the transport equation \eqref{eqn:transport} with ${\bf g}^*_r$ replaced by ${\bf g}^*$, and $\Pi_{\theta^*}$ the transport operator defined as in \eqref{eqn:weak} but with $\theta^*_r = \theta({\bf g}^*_r)$ replaced by $\theta^* = \theta({\bf g}^*)$.  Assume $h \leq r^{-q}$.  Then under the same conditions as in Theorem \ref{thm:convergence} we have the bound
\begin{equation} \label{eqn:marzbound}
\|u_h-u_{\mbox{m\"arz}}\|_p  \leq  K_1 \cdot (rh)^{(\frac{s'}{2}+\frac{1}{2p}) \wedge \frac{s}{2} \wedge 1}+K_2 r^{-q\left\{s \wedge \left( s' +\frac{1}{p} \right) \wedge 1\right\}}
\end{equation}
where $K_1 >0$, $K_2 >0$ are constants depending only on $\theta^* = \theta({\bf g}^*)$, ${\bf g}^* \cdot e_2$, $u_0$, $\mathcal U$, and $\{U_i\}_{i=1}^M$ (where $\{U_i\}_{i=1}^M$ are the sets illustrated in Figure \ref{fig:setup}).  Moreover, The dependence on $\theta^*$ and ${\bf g}^* \cdot e_2$ is continuous but $K_1$ and $K_2$ diverge as $\theta^* \rightarrow 0$ or $\theta^* \rightarrow \pi$.  Additionally, we have
\begin{equation} \label{eqn:chabuduo}
\|u-u_{\mbox{m\"arz}}\|_p \rightarrow 0 \qquad \mbox{ as } r \rightarrow \infty
\end{equation}
for all $1\leq p < \infty$ if $s'=0$ and for $p = \infty$ as well if $s'>0$.
\end{theorem}
\begin{proof}
Appendix \ref{app:marz}.
\qed
\end{proof}

\vskip 1mm

\begin{remark} \label{rem:marzLimit}  Theorem \ref{thm:convergence2} implies our claim in objective one, because the bound \eqref{eqn:marzbound} implies that if $\epsilon = rh \rightarrow 0$ and $r \rightarrow \infty$, then $u_h$ converges to $u_{\mbox{m\"arz}}$ in $L^p$.  This does not, however, mean that $r \rightarrow \infty$ is {\em required}.  In Section \ref{sec:kink3main}, we will see that there are many ways for \eqref{eqn:chabuduo} to be satisfied.  We will see examples where:
\begin{itemize}
\item $\|u-u_{\mbox{m\"arz}}\|_p \rightarrow 0$ as $r \rightarrow \infty$, but remains strictly positive for all $r$.
\item $\|u-u_{\mbox{m\"arz}}\|_p > 0$ for all $r<R$ and then $u = u_{\mbox{m\"arz}}$ for all $r \geq R$.
\item $u = u_{\mbox{m\"arz}}$ independent of $r$.
\end{itemize}
Hence the bound \eqref{eqn:marzbound} is not in general tight and $r \rightarrow \infty$ is a sufficient but not necessary condition.  This means that our third objective, which is to argue that when $r$ is a small integer then
$$\|u_h - u\|_p \ll \|u_h - u_{\mbox{m\"arz}}\|_p$$
is not in general true.  However, numerical experiments in Section \ref{sec:MarzIsFar} provide strong numerical evidence that it does hold in many cases.  Finally, objective two is clearly implied by \eqref{eqn:chabuduo}.
\end{remark}

\section{Consequences} \label{sec:consequences}

In this section we apply our analysis, in particular Theorem \ref{thm:convergence}, in order to explain:
\begin{enumerate}
\item The artifacts listed in Section \ref{sec:shellBased} and illustrated in Figures \ref{fig:specialDir}, \ref{fig:taleOfTwo}, and \ref{fig:artifacts2}.
\item The differences between the direct form of Algorithm 1 and its semi-implicit extension in terms of coping with these artifacts.
\end{enumerate}
We begin by considering kinking artifacts, and then go on to blur.  Shocks and cut off isophotes are not considered as they have already been adequately explained elsewhere \cite{Marz2011,Marz2015}.

\subsection{Kinking} \label{sec:kinking}

We begin by exploring kinking artifacts.  First we prove a fundamental distinction between the direct form of Algorithm 1 and the semi-implicit extension.  Then we go on to look at the limiting transport directions associated with three specific methods:  coherence transport, Guidefill, and semi-implicit Guidefill.

\begin{figure}
\centering
\begin{tabular}{cc}
\subfloat[Repeat of the experiment in Figure \ref{fig:specialDir} for Guidefill with $r=3$, with $\mbox{Conv}(-\bar{b}^-_r)$ superimposed.]{\includegraphics[width=.45\linewidth]{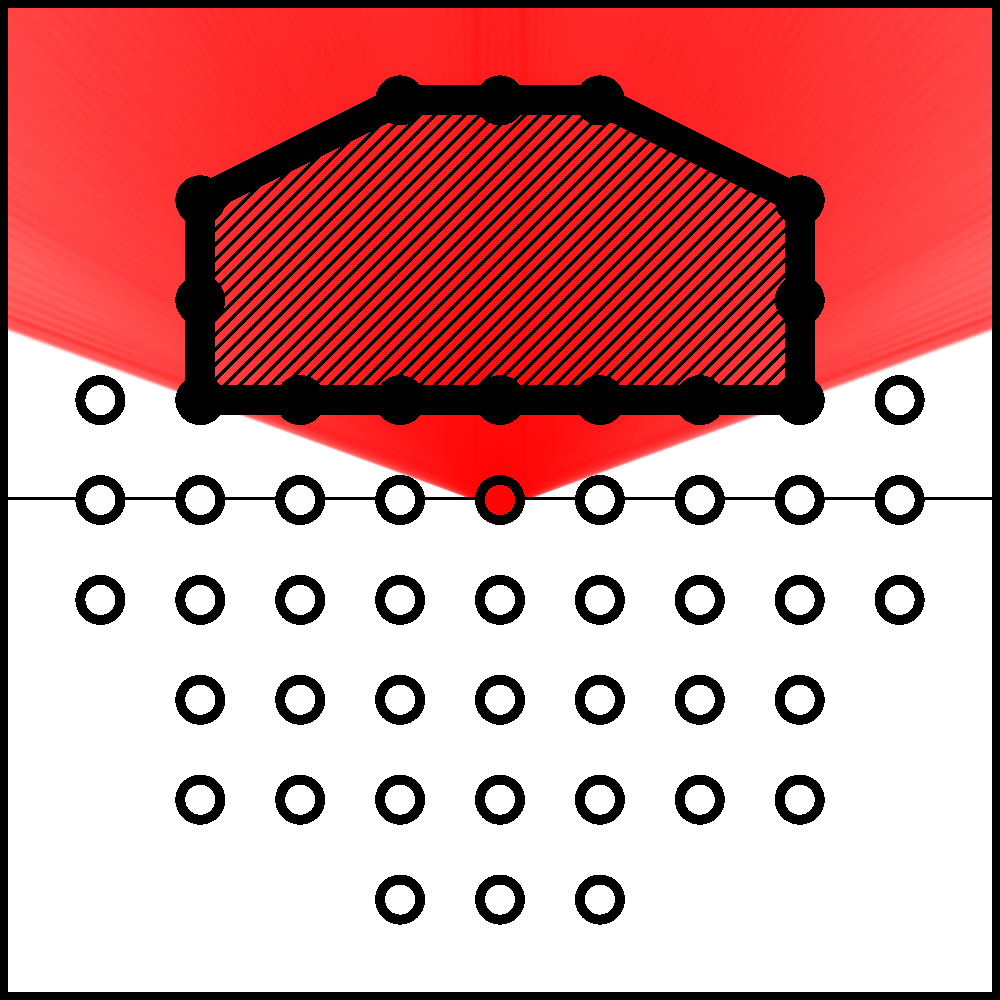}} & 
\subfloat[Repeat of the experiment in Figure \ref{fig:specialDir} for semi-implicit Guidefill with $r=3$, with $\mbox{Conv}(-\bar{b}^0_r)$ superimposed.]{\includegraphics[width=.45\linewidth]{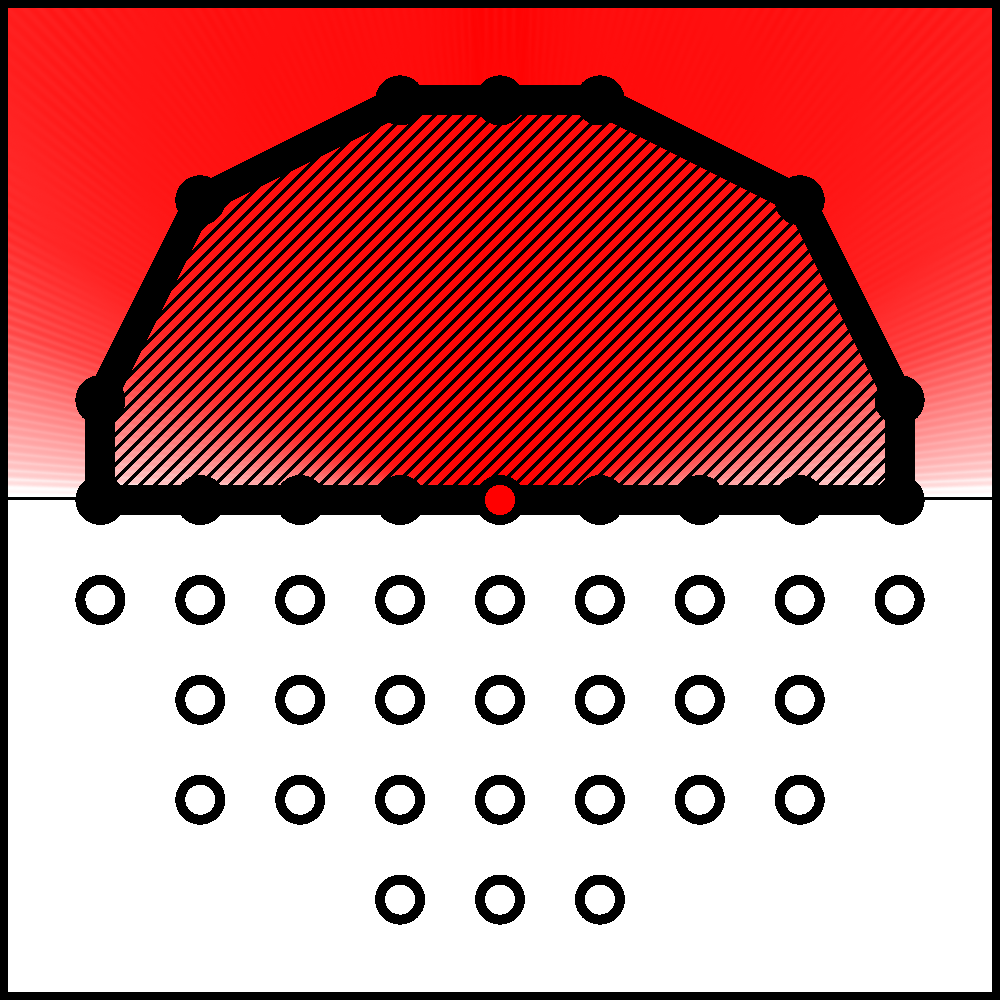}} \\
\end{tabular}
\caption{{\bf Convex hulls and transport directions:}  An immediate corollary of Theorem \ref{thm:convergence} is that the limiting transport direction ${\bf g}^*_r$ lies in the convex hull of $\bar{b}^-_r$ for the direct form of Algorithm 1, and the convex hull of $\bar{b}^0_r$ for the semi-implicit extension.  Since $\mbox{Conv}(\bar{b}^-_r)$ contains only a cone of directions while $\mbox{Conv}(\bar{b}^0_r)$ contains the full arc from $0$ to $\pi$, this explains why Guidefill (and more generally {\em all} direct methods of the general form given by Algorithm 1) fails for shallow angles, while this is not true of the semi-implicit extension.  To illustrate this, we have repeated the experiment from Figure \ref{fig:specialDir} using Guidefill (a) and semi-implicit Guidefill (b), while superimposing the sets $\mbox{Conv}(-\bar{b}^-_r)$ and $\mbox{Conv}(-\bar{b}^0_r)$ respectively (we have negated the sets for convenience which we can do since ${\bf g}^*_r$ and $-{\bf g}^*_r$ define the same transport equation). Note that in the case of semi-implicit Guidefill, the lines appear to be getting fainter as $\theta({\bf g}^*_r) \rightarrow 0$ and $\theta({\bf g}^*_r) \rightarrow \pi$.  This likely has two causes.  For one, the angular footprint of the red dot in Figure \ref{fig:specialDir}(a), i.e. the width of the dot multiplied by $\sin\theta^*_r$ (which tells you how much of the dot is ``visible'' from direction $\theta^*_r$) is going to zero.  Secondly, we will see in Section \ref{sec:blur} (see in particular Figure \ref{fig:blurangle}) that semi-implicit Guidefill has blur artifacts that become arbitrarily bad as $\theta({\bf g}^*_r) \rightarrow 0$ or $\theta({\bf g}^*_r) \rightarrow \pi$.  Nevertheless, Figure \ref{fig:shallowAngles} demonstrates that semi-implict Guidefill can successfully extrapolate lines with $\theta$ very close to $0$ without serious issues of faintness.  All parameters are the same as in Figure \ref{fig:specialDir}.}
\label{fig:convexHull}
\end{figure}

\vskip 2mm

\subsubsection{The direct form Algorithm 1 kinks, the semi-implicit extension need not}  \label{sec:kinkingAndConvex}
An immediate corollary of Theorem \ref{thm:convergence} is that the direct form of Algorithm 1 is limited in terms of what directions it can extrapolate along, while this is not true of the semi-implicit extension.  This follows from the geometric interpretation of the limiting transport direction ${\bf g}^*_r$ as the center of mass of the stencil $a^*_r$ with the respect to the stencil weights $\{ w_r({\bf 0},{\bf y})/W : {\bf y} \in a^*_r \}$.  Specifically, since the original weights $\{ \frac{w_r({\bf 0},{\bf y})}{W} \}_{{\bf y} \in a^*_r}$ are non-negative and sum to one, by the preservation of total mass \eqref{eqn:totalMass} and inheritance of non-negativity \eqref{eqn:nonNegativeInheritance} properties of equivalent weights, the same is true of the equivalent weights $\{ \frac{\tilde{w}_r({\bf 0},{\bf y})}{W} \}_{{\bf y} \in \mbox{Supp}(a^*_r)}$.  Hence
$${\bf g}^*_r = \sum_{{\bf y} \in \mbox{Supp}(a^*_r)} \frac{\tilde{w}_r({\bf 0},{\bf y})}{W} {\bf y} \in \mbox{Conv}(\mbox{Supp}(a^*_r)),$$
where $\mbox{Conv}(\mbox{Supp}(a^*_r))$ denotes the convex hull of $\mbox{Supp}(a^*_r)$.  But by \eqref{eqn:universalContainment} we have
\begin{equation}
\mbox{Supp}(a^*_r)  \subseteq \begin{cases} \bar{b}^-_r & \mbox{ if we use the direct form of Algorithm 1} \\ 
\bar{b}^0_r & \mbox{ if we use the semi-implicit extension} \end{cases}
\end{equation}
where $\bar{b}^0_r$ and $\bar{b}^-_r$ are the dilated sets defined by \eqref{eqn:bb0r} and \eqref{eqn:bbmr} respectively.  It follows that ${\bf g}^*_r$ has to lie in the convex hull of $\bar{b}^-_r$ if the direct form of Algorithm 1 is used, or the convex hull of $\bar{b}^0_r$ if the semi-implicit form is used instead.  That is
\begin{equation} \label{eqn:convHull}
{\bf g}^*_r   \in \begin{cases} \mbox{Conv}(\bar{b}^-_r) & \mbox{ if we use the direct form of Algorithm 1.} \\ 
\mbox{Conv}(\bar{b}^0_r) & \mbox{ if we the semi-implicit extension.} \end{cases}
\end{equation}
This in turn immediately implies that if we use the direct form of Algorithm 1, then $\theta({\bf g}^*_r)$ is restricted to the cone
\begin{equation} \label{eqn:cone}
\theta_c \leq \theta({\bf g}^*_r) \leq \pi - \theta_c
\end{equation}
where 
\begin{equation} \label{eqn:thetaC}
\theta_c = \arcsin\left(\frac{1}{r}\right)
\end{equation}
 is the smallest possible angle one can make given the restriction \eqref{eqn:convHull}, while for the semi-implict method we have
\begin{equation} \label{eqn:unrestricted}
0 <  \theta({\bf g}^*_r) < \pi
\end{equation}
where the angles $\theta = 0$ and $\theta = \pi$ are omitted not because of the restriction \eqref{eqn:convHull}, but because Theorem \ref{thm:convergence} does not apply in either case (indeed, the continuum limit $u$ given by \eqref{eqn:transport} is not even defined).  The cone \eqref{eqn:cone} is exactly what we saw in Figure \ref{fig:specialDir}(c).  At the same time, the lack of restrictions implied by \eqref{eqn:unrestricted} is consistent with our experience in Figures \ref{fig:shallowAngles} and \ref{fig:rates}, where semi-implicit Guidefill was able to successfully extrapolate lines making an angle as small as $1^{\circ}$ with the inpainting domain boundary for $r=3$, well under the critical angle $\theta_c= \arcsin(\frac{1}{3}) \approx 19.5^{\circ}$ where standard Guidefill breaks down.  Figure \ref{fig:convexHull} illustrates this result.

\subsubsection{ Limiting Transport Directions for Coherence Transport, Guidefill, and semi-implicit Guidefill} \label{sec:kink3main}

\begin{figure}
\centering
\begin{tabular}{cc}
\subfloat[Coherence transport, $r=3$.]{\includegraphics[width=.43\linewidth]{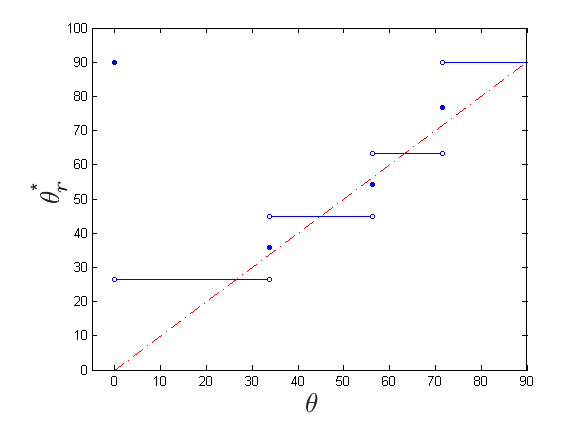}} & 
\subfloat[Coherence transport, $r=5$.]{\includegraphics[width=.43\linewidth]{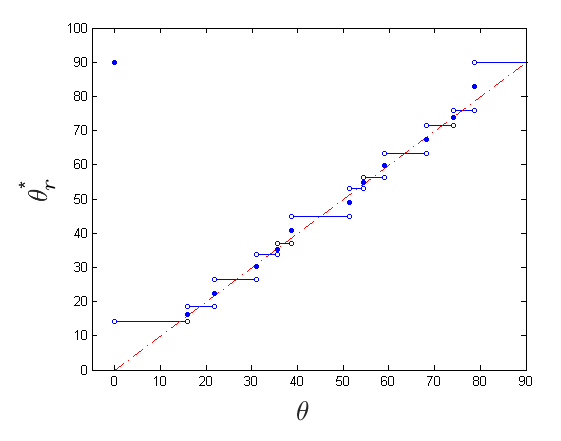}} \\
\subfloat[Guidefill, $r=3$.]{\includegraphics[width=.43\linewidth]{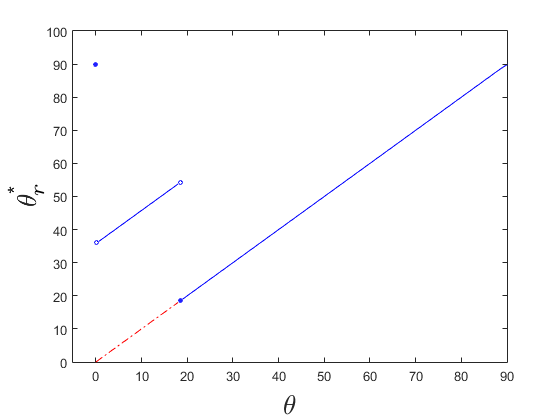}} & 
\subfloat[Guidefill, $r=5$.]{\includegraphics[width=.43\linewidth]{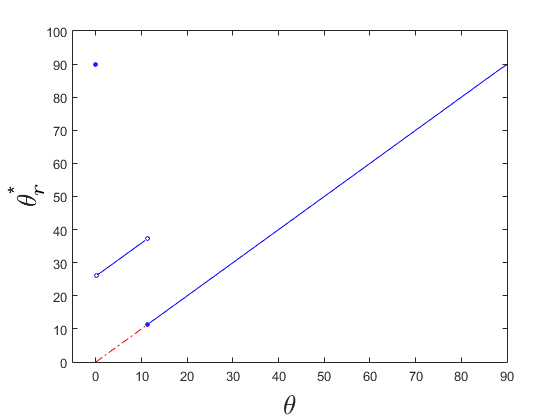}} \\
\subfloat[Semi-implicit Guidefill $r=3$.]{\includegraphics[width=.43\linewidth]{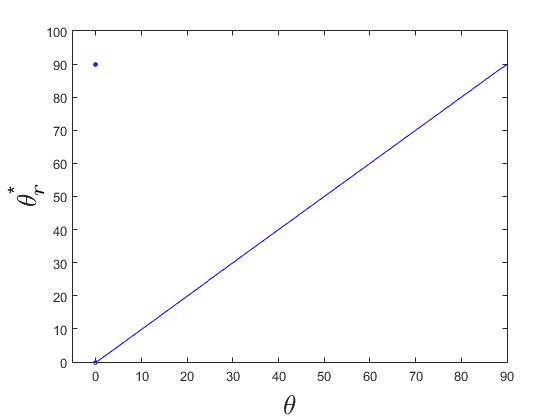}} & 
\subfloat[Semi-implicit Guidefill $r=5$.]{\includegraphics[width=.43\linewidth]{semiImplictTheory3b2.png}} \\
\end{tabular}
\caption{{\bf Limiting transport direction $\theta^*_r=\theta({\bf g}^*_r)$ as a function of the guideance direction $\theta = \theta({\bf g})$ for the three main methods:}  The theoretical limiting curves $\theta^*_r =F(\theta)$ obtained for coherence transport (a)-(b), Guidefill (c)-(d), and semi-implicit Guidefill (e)-(f).  The desired curve $F(\theta)=\theta$ is highlighted in red.  Coherence transport exhibits a staircase pattern with $F(\theta) \neq \theta$ for all but finitely many $\theta$, Guidefill obeys $F(\theta)=\theta$ for all $\theta \in (\theta_c, \pi - \theta_c)$ where $\theta_c$ is the critical angle \eqref{eqn:thetaC}, and semi-implicit Guidefill obeys $F(\theta) = \theta$ for all $\theta \neq 0$.}
\label{fig:limitingCurves}
\end{figure}

Here we derive formulas for the limiting transport directions of coherence transport, Guidefill, and semi-implicit Guidefill in the limit as $\mu \rightarrow \infty$.  For convenience, in this section we rescale the limiting transport direction ${\bf g}^*_r$ by a factor of $W = \sum_{{\bf y} \in a^*_r} w_r({\bf 0},{\bf y})$, giving
\begin{equation} \label{eqn:rescaled}
{\bf g}^*_r =  \sum_{{\bf y} \in a^*_r} w_r({\bf 0},{\bf y}){\bf y}.
\end{equation}
This is more convenient to work with and defines the same transport equation.  In fact, the ability to rescale ${\bf g}^*_r$ by any $\lambda \neq 0$ without changing the underlying transport equation is a tool that we will use repeatedly in our arguments throughout this section.  To make the dependence of the transport direction on $\mu$ explicit, within this section we write this direction as ${\bf g}^*_{r,\mu}$, and define
$${\bf g}^*_r := \lim_{\mu \rightarrow \infty} {\bf g}^*_{r,\mu}.$$
In order to resolve the ambiguity that ${\bf g}^*_r$ and $-{\bf g}^*_r$ define the same transport equation, we will frequently be taking angles modulo $\pi$.  This is reflected in the definition we have selected for $\theta({\bf v})$ - recall that this refers to the angle that the line $L_{{\bf v}}=\{ \lambda {\bf v} : \lambda \in \field{R} \}$ makes with the $x$-axis, and hence always lies in the range $[0,\pi)$.  We define ${\bf g}=(\cos\theta,\sin\theta)$, $\theta^*_{r,\mu} :=\theta({\bf g}^*_{r,\mu})$, $\theta^*_r :=\theta({\bf g}^*_r)$ and consider the function $\theta^*_r = F(\theta)$.  Our main concern is to determine when $\theta^*_r = \theta$ (no kinking) and when $\theta^*_r \neq \theta$ (kinking).  Results are illustrated in Figure \ref{fig:limitingCurves}.

\vskip 2mm

\noindent { \bf Coherence Transport.}  We have already stated in Section \ref{sec:shellBased} and illustrated in Figure \ref{fig:specialDir} that coherence transport in the limit $\mu \rightarrow \infty$ kinks unless ${\bf g}=\lambda {\bf v}$ for some ${\bf v} \in B_{\epsilon,h}({\bf 0})$ and some $\lambda \in \field{R}$.  Under the assumptions of Section \ref{sec:symmetry}, a more precise statement is that coherence transport  in the limit $\mu \rightarrow \infty$ fails to kink if and only if ${\bf g}=\lambda {\bf v}$ for some ${\bf v} \in b^-_r$ and $\lambda \in \field{R}$.  Before we prove this, let us make some definitions.  We define the {\em angular spectrum} of $b^-_r$ as
\begin{equation} \label{eqn:spectrumbr}
\Theta(b^-_r) := \{ \theta_1,\theta_2,\ldots,\theta_n : \mbox{ each } \theta_i = \theta({\bf y}) \mbox{ for some } {\bf y} \in b^-_r\}.
\end{equation}
In other words, $\Theta(b^-_r)$ is the collection of angles modulo $\pi$ that are representable using members of $b^-_r$ (or which elements of the projective space $\field{RP}^1$ are representable, to be more mathematically precise), see Figure \ref{fig:angularSpectrum} for an illustration.  The angular spectrum may be similarly defined in the obvious way for more general sets, and we do so in Appendix \ref{app:angSpect}.  We will show that for coherence transport, when $0 < \theta < \pi$, we either have 
$$\theta^*_r \in \Theta(b^-_r) \quad \mbox{ or } \quad  \theta^*_r = \frac{\theta_i+\theta_{i+1}}{2} \quad \mbox{ for two consecutive }\theta_i, \theta_{i+1} \in \Theta(b^-_r).$$
To begin, note that in this case \eqref{eqn:rescaled} becomes
$${\bf g}^*_{r,\mu} := \sum_{{\bf y} \in b^-_r} e^{-\frac{\mu^2}{2r^2}({\bf y} \cdot {\bf g}^{\perp})^2}\frac{\bf y}{\|{\bf y}\|},$$
where $b^-_r$ is the discrete half ball \eqref{eqn:bmr}.
Denote by $\Psi$ the set of minimizers of $|{\bf y} \cdot {\bf g}^{\perp}|$ for ${\bf y} \in b^-_r$, meaning that $|{\bf y} \cdot {\bf g}^{\perp}|:=\Delta \geq 0$ for all ${\bf y} \in \Psi$ and $|{\bf y} \cdot {\bf g}^{\perp}|>\Delta$ for all ${\bf y} \in b^-_r \backslash \Psi$.  After rescaling by $e^{\frac{\mu^2}{2r^2}\Delta^2}$, the transport direction ${\bf g}^*_{r,\mu}$ becomes
\begin{eqnarray} \label{eqn:coherenceLimitNotExplicit}
{\bf g}^*_{r,\mu} &=& \sum_{{\bf y} \in \Psi} \frac{\bf y}{\|{\bf y}\|}+\sum_{{\bf j} \in b^-_r \backslash \Psi} e^{-\frac{\mu^2}{2r^2}\left\{({\bf y} \cdot {\bf g}^{\perp})^2 - \Delta^2 \right\}}\frac{\bf y}{\|{\bf y}\|} \nonumber \\
& \rightarrow & \sum_{{\bf y} \in \Psi} \frac{\bf y}{\|{\bf y}\|} \qquad \mbox{ as } \mu \rightarrow \infty.
\end{eqnarray}
Note that $|{\bf y} \cdot {\bf g}^{\perp}|$ represents the distance from the point ${\bf y}$ to the line through the origin $L_{\bf g}$.  Thus computing the set $\Psi$ is equivalent to finding the set of points in $b^-_r$ closest to a given line through the origin.  In Appendix \ref{app:angSpect} we prove that as $\theta$ sweeps out an arc from $0$ to $\pi$, for all but finitely many $\theta$ the set $\Psi$ is a singleton, containing a sequence of lone minimizers that we enumerate (in order of occurrence) as ${\bf y}_1, {\bf y}_2,\ldots {\bf y}_{n'}$ (for some finite $n'$).  Now, it turns out that $n'=n$ and moreover for every $\theta_i \in \Theta(b^-_r)$ we have
$$\theta_i = \theta({\bf y}_i)$$
\begin{figure}
\centering \includegraphics[width=.7\linewidth]{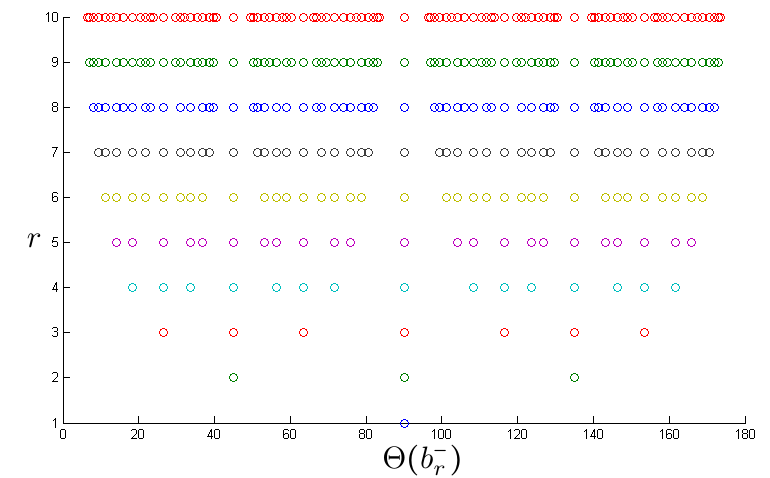}
\caption{{\bf Illustration of the angular spectrum:} The angular spectrum $\Theta(b^-_r)$ tells us which angles (modulo $\pi$), are representable using elements of $b^-_r$.   Here we have illustrated $\Theta(b^-_r)$, measured in degrees, for $r=1,2,\ldots$.  }
\label{fig:angularSpectrum}
\end{figure}
\noindent (Appendix \ref{app:angSpect}, Proposition \ref{prop:correspondence}).  In other words, we have a 1-1 correspondence between (singleton) minimizers of $|{\bf y} \cdot {\bf g}^{\perp}|$ and the angular spectrum $\Theta(b^-_r)$.  Moreover, it can be shown that if $\theta_i < \theta < \theta_{i+1}$ for some $\theta_i, \theta_{i+1} \in \Theta(b^-_r)$, then either $\Psi = \{ {\bf y}_i \}$ (for $\theta$ close to $\theta_i$) or $\Psi = \{ {\bf y}_{i+1} \}$ (for $\theta$ close to $\theta_{i+1}$) or $\Psi = \{{\bf y}_i,  {\bf y}_{i+1} \}$ if $\theta=\theta_{i,i+1}$, where $\theta_{i,i+1} \in (\theta_i,\theta_{i+1})$ is a critical angle given by
$$\theta_{i,i+1} = \theta({\bf y}_i + {\bf y}_{i+1}) \quad \mbox{ for } 1 \leq i \leq n-1.$$
For convenience, we also define $\theta_{0,1}=0$ and $\theta_{n,n+1}=\pi$.  Since one can also prove $\Psi = \{ {\bf y}_1 \}$ for $0:=\theta_{0,1}<\theta<\theta_1$ and $\Psi = \{{\bf y}_n\}$ for $\theta_n < \theta < \theta_{n,n+1}:=\pi$, with this notation we have the general result
$$\Psi=\begin{cases}
\{{\bf y}_i\} & \mbox{ if } \theta_i < \theta < \theta_{i,i+1} \quad \mbox{ for some } i=1,\ldots,n \\
\{{\bf y}_i,{\bf y}_{i+1}\} & \mbox{ if } \theta=\theta_{i,i+1} \quad \mbox{ for some } i=1,\ldots,n-1 \\
\{{\bf y}_{i+1}\} & \mbox{ if } \theta_{i,i+1} < \theta < \theta_{i+1} \quad \mbox{ for some } i=0,\ldots,n-1 \\
\end{cases}$$

\begin{figure}
\centering
\begin{tabular}{cc}
\subfloat[]{\includegraphics[width=.5\linewidth]{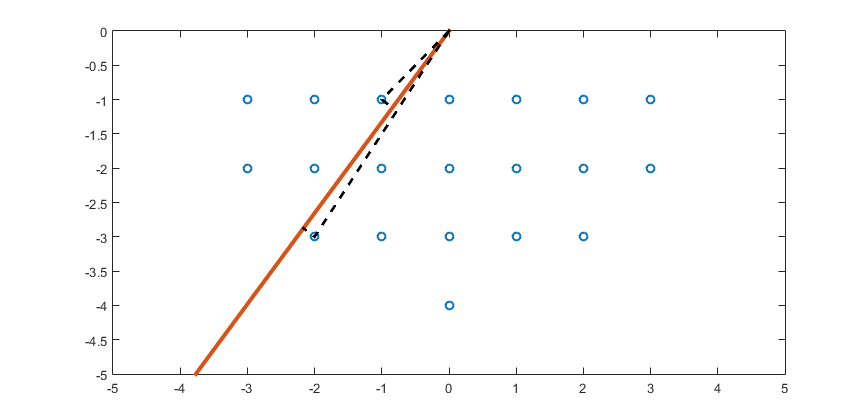}} & 
\subfloat[]{\includegraphics[width=.35\linewidth]{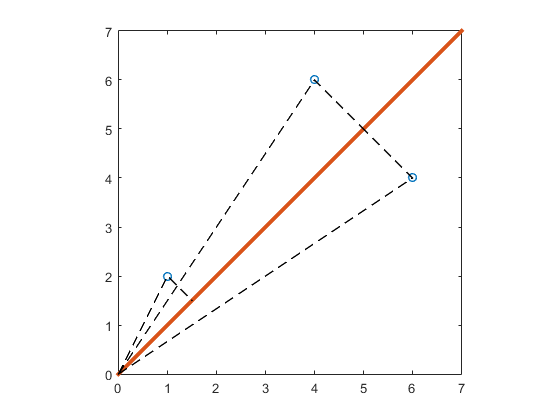}} \\
\end{tabular}
\caption{{\bf Closest points and shallowest angles:}  We claim that the closest point in the set $b^-_r \subset \field{Z}^2$ to a given line $L_{{\bf g}}$ (${\bf g}=(\cos\theta,\sin\theta)$) is always one of the two points casting the shallowest angle with $L_{\bf g}$ on either side, as illustrated in (a) for $\theta = 53^{\circ}$, $r=4$.  This statement does {\em not} hold if $b^-_r$ is replaced with a generic $A \subset \field{Z}^2$, as demonstrated by the counterexample $A=\{(1,2),(4,6),(6,4)\}$ and $\theta = 45^{\circ}$ in (b).}
\label{fig:genericVsNot}
\end{figure}

\noindent (Appendix \ref{app:angSpect}, Proposition \ref{prop:eitherOr} - note that we have carefully excluded $\theta_{0,1}:=0$ and $\theta_{n,n+1}=\pi$ from the middle case).  In words, this means that the element(s) of $b^-_r$ closest to the line $L_{\bf g}$ are also the member(s) of $b^-_r$ that cast the shallowest angles with $L_{\bf g}$ from above and below.  This statement is not true if $b^-_r$ is replaced with a generic subset of $\field{Z}^2$ - see Figure \ref{fig:genericVsNot}.  The remaining cases are $\theta=\theta_i \in \Theta(b^-_r)$ and $\theta=0$.  We deal with the former first - in the first case we have
$$\Psi = \{ {\bf y} \in b^-_r : \theta(y) = \theta_i \},$$
a set containing up to $r$ members, all of which are parallel to each other and to ${\bf g}$.  In order to make these ideas more concrete, Example \ref{example:Psi} gives explicit expressions for $\Theta(b^-_r)$ and $\Psi$ in the case $r=3$.  Also, in Appendix \ref{app:angSpect} Remark \ref{rem:practicalAlg}, we give an algorithm for computing $\Theta(b^-_r)$ and $Y(b^-_r):=\{{\bf y}_1,{\bf y}_2,\ldots,{\bf y}_n\}$ for any $r$.

\vskip 1mm

\begin{example} \label{example:Psi}
When $r=3$ we have
\begin{eqnarray*}
\Theta(b^-_3)&:=&\{\theta_1,\theta_2,\theta_3,\theta_4,\theta_5,\theta_6,\theta_7\}\\
&=&\{ \arctan(1/2), \frac{\pi}{4}, \arctan(2), \frac{\pi}{2}, \frac{\pi}{2}+\arctan(1/2),\frac{3\pi}{4},\frac{\pi}{2}+\arctan(2)\}.
\end{eqnarray*}
For $0 < \theta \leq \frac{\pi}{2}$, (we omit $\frac{\pi}{2} < \theta < \pi$ for brevity) the set of minimizers $\Psi$ is given by
$$\Psi = \begin{cases} 
\{(-2,-1)\}:=\{{\bf y}_1\} & \mbox{ if } 0 < \theta < \theta_{1,2}. \\ 
\{(-2,-1),(-1,-1)\}:=\{{\bf y}_1,{\bf y}_2\} & \mbox{ if } \theta =\theta_{1,2}. \\ 
\{(-1,-1)\}:=\{{\bf y}_2\} & \mbox{ if } \theta_{1,2}< \theta < \theta_2. \\ 
\{(-1,-1),(-2,-2)\}:=\{{\bf y}_2,2{\bf y}_2\} & \mbox{ if } \theta = \theta_2. \\ 
\{(-1,-1)\}:=\{{\bf y}_2\} & \mbox{ if } \theta_2 < \theta < \theta_{2,3}. \\ 
\{(-1,-1),(-1,-2)\}:=\{{\bf y}_2,{\bf y}_3\} & \mbox{ if } \theta = \theta_{2,3}. \\ 
\{(-1,-2)\}=\{ {\bf y}_3\} & \mbox{ if } \theta_{2,3}< \theta < \theta_{3,4}. \\ 
\{(-1,-2),(0,-1)\}=\{{\bf y}_3,{\bf y}_4 \} & \mbox{ if } \theta =\theta_{3,4}. \\ 
\{(0,-1)\}:=\{{\bf y}_4\} & \mbox{ if } \theta_{3,4}< \theta < \theta_4. \\
\{(0,-1),(0,-2),(0,-3)\}:=\{{\bf y}_4,2{\bf y}_4,3{\bf y}_4\}  & \mbox{ if } \theta = \theta_4.
\end{cases}$$
where $\theta_{0,1}=0$, $\theta_{1,2}=\arctan(2/3)$, $\theta_{2,3}=\arctan(3/2)$, $\theta_{3,4}=\arctan(3)$.
\end{example}

\vskip 1mm

\noindent When $\Psi$ is a singleton set, that is $\Psi = \{ {\bf y}_i \}$ for some $1 \leq i \leq n$, \eqref{eqn:coherenceLimitNotExplicit} becomes ${\bf g}^*_r = \frac{{\bf y}_i}{\|{\bf y}_i\|}$ and we have
$$\theta({\bf g}^*_r) = \theta\left(\frac{{\bf y}_i}{\|{\bf y}_i\|}\right) = \theta_i$$
On the other hand, if $\theta = \theta_i \in \Theta(b^-_r)$, ${\bf g}^*_r$ is a sum of vectors all parallel to one another and to ${\bf g}$, and we get $\theta^*_i = \theta_i$ again.  This is the lone case in which coherence transport doesn't kink.  Next, at the transition angles $\theta_{i,i+1}$ where $\Psi = \{ {\bf y}_i, {\bf y}_{i+1} \}$, we have ${\bf g}^*_r = \frac{{\bf y}_i}{\|{\bf y}_i\|}+\frac{{\bf y}_{i+1}}{\|{\bf y}_{i+1}\|}$, so that
$$\theta({\bf g}^*_r) = \theta\left(\frac{{\bf y}_i}{\|{\bf y}_i\|}+\frac{{\bf y}_{i+1}}{\|{\bf y}_{i+1}\|}\right) = \frac{\theta_i+\theta_{i+1}}{2},$$
where we have used the observation (proved in Appendix \ref{app:angSpect}, Observation \ref{obs:unitVec}) that $\theta({\bf v}+{\bf w}) = \frac{\theta({\bf v}) + \theta({\bf w})}{2}$ holds for all unit vectors ${\bf v}, {\bf w} \in S^1$.  Finally, suppose $\theta = 0$.  Here we have $\Psi = \{ (i,-1) : -r+1 \leq i \leq r-1 \}$, giving
$${\bf g}^*_r = e_2 \qquad \mbox{for } \theta = 0$$
after rescaling.  In summary, for $\theta \in [0,\pi)$ we then have
\begin{equation} \label{eqn:stepFunc}
\theta^*_r = \begin{cases} \frac{\pi}{2} & \mbox{ if } \theta = 0 \\ \theta_i & \mbox{ if } \theta_{i-1,i} < \theta < \theta_{i,i+1} \mbox{ for some } i = 1,\ldots n \\ \frac{\theta_i+\theta_{i+1}}{2} & \mbox{ if } \theta = \theta_{i,i+1} \mbox{ for some } i = 1,\ldots n \end{cases}
\end{equation}
See Figure \ref{fig:limitingCurves}(a)-(b) for an illustration of \eqref{eqn:stepFunc} for $r=3$ and $r=5$.  In Appendix \ref{app:angSpect}, Corollary \ref{corollary:generalizedCT} we prove a generalization of this result that applies if, for example, the discrete ball used by coherence transport is replaced with a discrete square.

\begin{remark} \label{rem:irrationalAngles}  Since $\theta^*_r = \theta$ if and only if $\theta \in \Theta(b^-_r)$, and since the latter is generated by vectors with integer coordinates, it follows that for fixed $\theta = \arctan(x)$, we have $\theta^*_r = \theta$ for all $r$ sufficiently large if $x$ is rational and $\theta^*_r \neq \theta$ for all $r$ if $x$ is irrational.  In light of \eqref{eqn:marzNoKink} and Theorem \ref{thm:convergence2}, this means that in the former case we have $u=u_{\mbox{m\"arz}}$ for all $r$ sufficiently large, but in the latter case $\|u -u_{\mbox{m\"arz}}\|_p \rightarrow 0$ for all $1\leq p < \infty$, but we always have $u \neq u_{\mbox{m\"arz}}$ for generic boundary data $u_0$.

\end{remark}

\vskip 2mm

\noindent { \bf Guidefill.}  In this case \eqref{eqn:rescaled} becomes
$${\bf g}^*_{r,\mu} := \sum_{{\bf y} \in \tilde{b}^-_r} e^{-\frac{\mu^2}{2r^2}({\bf y} \cdot {\bf g}^{\perp})^2}\frac{\bf y}{\|{\bf y}\|},$$
where $\tilde{b}^-_r$ is given by \eqref{eqn:bmr}.  It is useful to patition $\tilde{b}^-_r$ into a disjoint union of sets $\ell^-_k$ such at each $\ell^-_k$ is the collection of points in $\tilde{b}^-_r$ distance $k$ from the line $L_{\bf g} = \{ \lambda {\bf g} : \lambda \in \field{R}\}$, that is
$$\ell^-_k =\{ n{\bf g}+m{\bf g}^{\perp} \in \tilde{b}^-_r : m = \pm k\}.$$
Since the weights \eqref{eqn:weight} exponentially decay with distance from $L_{\bf g}$, then so long as $\ell^-_{0}$ is non-empty, we expect the contribution of the other $\ell^-_k$ to vanish.  For Guidefill, $\ell^-_0$ can be explicitly parametrized as 
\begin{eqnarray*}
\ell^-_0 &:=& \{ n{\bf g} \in \tilde{b}^-_r : n{\bf g} \cdot e_2 \leq -1\} \\
&=&\{ n{\bf g} : n=-r,\ldots,-\lceil \csc\theta\rceil \}.
\end{eqnarray*}
For $\ell^-_0$ to be non-empty, we need $\lceil \csc\theta\rceil \leq r$, which occurs only if $\theta_c \leq \theta \leq \pi -\theta_c$, where $\theta_c$ is the same critical angle \eqref{eqn:thetaC} from Section \ref{sec:kinkingAndConvex}.  If $\ell^-_0 \neq \emptyset$, we have  
\begin{eqnarray*}
{\bf g}^*_{r,\mu} &=& \sum_{{\bf y} \in \ell^-_0} \frac{\bf y}{\|{\bf y}\|}+\sum_{k=1}^r \sum_{{\bf y} \in \ell^-_k} e^{-\frac{\mu^2}{2r^2}k^2 \|{\bf g}\|^2 }\frac{\bf y}{\|{\bf y}\|} \\
& \rightarrow & \sum_{{\bf y} \in \ell^-_0} \frac{\bf y}{\|{\bf y}\|} \qquad \mbox{ as } \mu \rightarrow \infty. \\
& = & \sum_{n=-r}^{-\lceil \csc\theta\rceil } n {\bf g} \\
& = & {\bf g} \qquad \mbox{ after rescaling. }
\end{eqnarray*}
Hence we transport in the correct direction in this case.

One the other hand, if $\ell^-_0 = \emptyset$ but $\theta \neq 0$, then the weights \eqref{eqn:weight} concentrate all their mass into $\ell^-_1$, with all other $\ell^-_k$ vanishing.  Unfortunately, unlike $\ell^0$, the set $\ell^-_1$ is not parallel to ${\bf g}$ so we expect kinking in this case.  In general $\ell^-_1$ can consist of two parallel components on either side of $\ell^-_0$.  But in this case, since $\ell^-_0$ lies entirely above the line $y=-1$, we know $\ell^-_1$ consists of just one component and we can write it explicitly as
$$\ell^-_1 =\{ n{\bf g}-\mbox{sgn}(\cos\theta){\bf g}^{\perp} : n=-r+1,\ldots,-1 \}$$
(remember that ${\bf g}^{\perp}$ denotes the {\em counterclockwise} rotation of ${\bf g}$ by $90^{\circ}$).  After rescaling by $e^{\frac{\mu^2}{2r^2}\|{\bf g}\|^2}$ we obtain
\begin{eqnarray*}
{\bf g}^*_{r,\mu} &=& \sum_{{\bf y} \in \ell^-_1} \frac{\bf y}{\|{\bf y}\|}+\sum_{k=2}^r \sum_{{\bf y} \in \ell^-_k} e^{-\frac{\mu^2}{2r^2}(k^2-1) \|{\bf g}\|^2 }\frac{\bf y}{\|{\bf y}\|} \\
& \rightarrow & \sum_{{\bf y} \in \ell^-_1} \frac{\bf y}{\|{\bf y}\|} \qquad \mbox{ as } \mu \rightarrow \infty. \\
& = & \left( \sum_{n=-r+1}^{-1} \frac{n}{\sqrt{1+n^2}} \right){\bf g} -\mbox{sgn}(\cos\theta) \left( \sum_{n=-r+1}^{-1} \frac{1}{\sqrt{1+n^2}} \right){\bf g}^{\perp} \\
& = & {\bf g}+\mbox{sgn} (\cos\theta)\alpha_r {\bf g}^{\perp} \qquad \mbox{ after rescaling. }
\end{eqnarray*}
where
$$\alpha_r = \frac{\sum_{n=1}^{r-1} \frac{1}{\sqrt{1+n^2}}}{\sum_{n=1}^{r-1} \frac{n}{\sqrt{1+n^2}}}.$$
Finally, if $\theta=0$ we have $\tilde{b}^-_r = b^-_r$ and we obtain ${\bf g}^*_r = e_2$ as for coherence transport.  Defining $\Delta\theta_r = \arctan(\alpha_r)$, for $\theta \in [0,\pi)$ we obtain
\begin{equation} \label{eqn:guidefill}
\theta^*_r = \begin{cases} \frac{\pi}{2} & \mbox{ if } \theta = 0 \\ 
\theta+\Delta\theta_r & \mbox{ if } 0 < \theta < \theta_c \\ 
\theta & \mbox{ if } \theta_c \leq \theta \leq \pi-\theta_c \\  
\theta-\Delta\theta_r & \mbox{ if } \pi- \theta_c < \theta < \pi. \end{cases}
\end{equation}
In other words, aside from exceptional case $\theta=0$, we have $\theta_r^*=\theta$ for the ``well behaved'' range of values $\theta_c \leq \theta \leq \pi-\theta_c$, but $\theta_r^*$ jumps by a constant angle $\Delta\theta_r$ near $\theta = 0$ and $\theta=\pi$.  The first few values of $\Delta\theta_r$ are $\Delta\theta_3 \approx 35.8^{\circ}$, $\Delta\theta_4 \approx 30.0^{\circ}$, $\Delta\theta_5 \approx 25.9^{\circ}$.  See Figure \ref{fig:limitingCurves}(c)-(d) for an illustration of \eqref{eqn:guidefill} for $r=3$ and $r=5$.  

\vskip 1mm

\begin{remark} \label{rem:GuidefillAllDifferent}
Since $\theta_c = \arcsin(1/r) \rightarrow 0$ as $r \rightarrow \infty$, it follows that for each fixed $\theta \in (0,\pi)$ we have $\theta^*_r = \theta$ and hence $u=u_{\mbox{m\"arz}}$ for all $r$ sufficiently large.  On the other hand, since $\theta_c > 0$, $\alpha_r>0$ for all $r$, we never have $\theta^*_r = \theta$ independent of $\theta$ for any fixed $r$.  There are always angles we can't reach.
\end{remark}

\vskip 2mm

\noindent { \bf Semi-implicit Guidefill.}  The analysis of semi-implicit Guidefill is the same as for Guidefill except that the set $\tilde{b}^-_r$ is replaced by $\tilde{b}^0_r$.  Defining $\ell^0_k$  as the collection of points in $\tilde{b}^0_r$ distance $k$ from the line $L_{\bf g} = \{ \lambda {\bf g} : \lambda \in \field{R}\}$ as before, we find that in this case
$$\ell^0_0 := \{ n{\bf g} \in \tilde{b}_r : n {\bf g} \cdot e_2 \leq 0\}$$
is {\em never} empty.  In fact, for $0 \leq \theta < \pi$ we have
$$\ell^0_0 = \begin{cases} \{ n{\bf g} : -r \leq n \leq -1\} & \mbox{ if } 0<\theta<\pi \\ \{ n{\bf g} : -r \leq n \leq r, n \neq 0\} & \mbox{ if } \theta = 0. \end{cases}$$
If $\theta > 0$, we proceed as before and find
\begin{eqnarray*}
{\bf g}^*_{r,\mu} &=& \sum_{{\bf y} \in \ell^0_0} \frac{\bf y}{\|{\bf y}\|}+\sum_{k=1}^r \sum_{{\bf y} \in \ell^0_k} e^{-\frac{\mu^2}{2r^2}k^2 \|{\bf g}\|^2 }\frac{\bf y}{\|{\bf y}\|} \\
& \rightarrow & \sum_{{\bf y} \in \ell^0_0} \frac{\bf y}{\|{\bf y}\|} \qquad \mbox{ as } \mu \rightarrow \infty. \\
& = & \sum_{n=-r}^{-1} n {\bf g} \\
& = & {\bf g} \qquad \mbox{ after rescaling. }
\end{eqnarray*}
However, if $\theta=0$, this argument doesn't work because the elements of $\ell^0_0$ all cancel each other out.  In fact, in this case we have
$${\bf g}^*_{r,\mu} = \underbrace{\sum_{{\bf y} \in \ell^0_0} \frac{\bf y}{\|{\bf y}\|}}_{=0}+\underbrace{\sum_{{\bf y} \in b^-_r} e^{-\frac{\mu^2}{2r^2}({\bf y} \cdot {\bf g}^{\perp})^2}\frac{\bf y}{\|{\bf y}\|}}_{\text{the ${\bf g}^*_{r,\mu}$ from coherence transport.}}.$$
Hence, in this case ${\bf g}^*_r = e_2$ yet again, just like for Guidefill and coherence transport.  In general, for $0 \leq \theta < \pi$, we have 
\begin{equation} \label{eqn:semiImplicitGuidefill}
\theta^*_r = \begin{cases} \frac{\pi}{2} & \mbox{ if } \theta = 0 \\ \theta & \mbox{ if } 0 < \theta < \pi \end{cases}
\end{equation}
In other words, semi-implicit Guidefill kinks only if ${\bf g}$ is exactly parallel to boundary of the inpainting domain.  See Figure \ref{fig:limitingCurves}(e)-(f) for an illustration of \eqref{eqn:semiImplicitGuidefill} for $r=3$ and $r=5$ (the curves are of course the same, since \eqref{eqn:semiImplicitGuidefill} is independent of $r$).

\vskip 1mm

\begin{remark} \label{rem:semiImplicitGuidefillAllSame}
In contrast to Guidefill and coherence transport, \eqref{eqn:semiImplicitGuidefill} tells us that for semi-implicit Guidefill in the limit $\mu \rightarrow \infty$ we have $\theta^*_r = \theta$ for all $\theta \in (0,\pi)$, independent of $r$.  This is in fact exactly the same prediction \eqref{eqn:marzNoKink} (albeit under stronger simplifying assumptions) that M\"arz and Bornemann obtained for coherence transport in their own continuum limit $u_{\mbox{m\"arz}}$ as $\mu \rightarrow \infty$ \eqref{eqn:marzNoKink}.  So in this case we have $u=u_{\mbox{m\"arz}}$ independent of $r$.  We have in some sense come full circle - the original predictions of \cite{Marz2007} for coherence transport under their high resolution and vanishing viscosity limit are the same as ours for semi-implicit Guidefill under our fixed ratio limit.
\end{remark}

\subsection{Asymptotic limit and blur} \label{sec:blur}

In this section we utilize the connection between Algorithm 1 and stopped random walks to explore the origins of blur artifacts.  The main idea is that the continuum limit studied in Theorem \ref{thm:convergence} - while useful for studying kinking artifacts - is inadequate for studying blur.  For the latter, one instead needs to consider an {\em asymptotic limit} where $h$ is very small but still non-zero.  This is accomplished by leveraging a central limit theorem for the type of stopped random walk relevant to us, which states that our stopped random walk (or more precisely, its x-coordinate) is converging in distribution to a Gaussian as $h \rightarrow 0$.  Since we have the identity
\begin{equation} \label{eqn:blurKernel}
u_h({\bf x}) = \sum_{{\bf y} \in \mathcal U_h} \rho_{{\bf X}_{\tau }}({\bf y}) u_0({\bf y}),
\end{equation}
this suggests that $u_h$ can be related to a mollified version of $u_0$, with a Gaussian mollifier.  This idea is made precise in Conjecture \ref{conj:blur}, and is supported by numerical experiments (Figures \ref{fig:wholeStory} and \ref{fig:blurwin}).  However, because the central limit theorem we utilize gives no rates of convergence, Conjecture \ref{conj:blur} remains a conjecture - its rigorous proof will be the subject of future work.  Before elaborating on these ideas, we first provide some motivation with a closer look at our $L^p$ convergence rates from Theorem \ref{thm:convergence}, and a case in which they can be tightened - what we call {\em degenerate} stencils.  If Conjecture \ref{conj:blur} is true, then degenerate stencils give a sufficient condition under which blur artifacts can be avoided.  However, they also tend to increase kinking artifacts, so there is an apparent trade-off (recall from Figure \ref{fig:taleOfTwo} that coherence transport kinks but doesn't blur, while Guidefill blurs but doesn't kink).

\begin{figure}
\centering
\begin{tabular}{ccc}
\subfloat[{Original inpainting problem.  $D=[0,1]^2$ is $200 \times 200$px and shown in yellow.}]{\includegraphics[width=.3\linewidth]{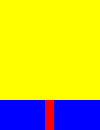}} &
\subfloat[Inpainting with Guidefill, $\mu=100$, $r=3$, ${\bf g}=e_2$.]{\includegraphics[width=.3\linewidth]{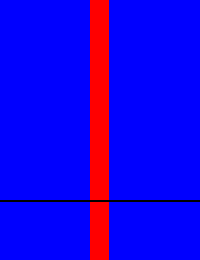}} & 
\subfloat[Inpainting with Guidefill, $\mu=100$, $r=3$, ${\bf g}=e_1$.]{\includegraphics[width=.3\linewidth]{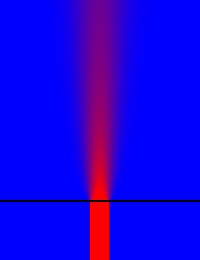}} \\
\end{tabular}
\begin{tabular}{cc}
\subfloat[Slice of the result from (c) at $y=0.1$, compared with the theoretical curve from Conjecture \ref{conj:blur}.]{\includegraphics[width=.45\linewidth]{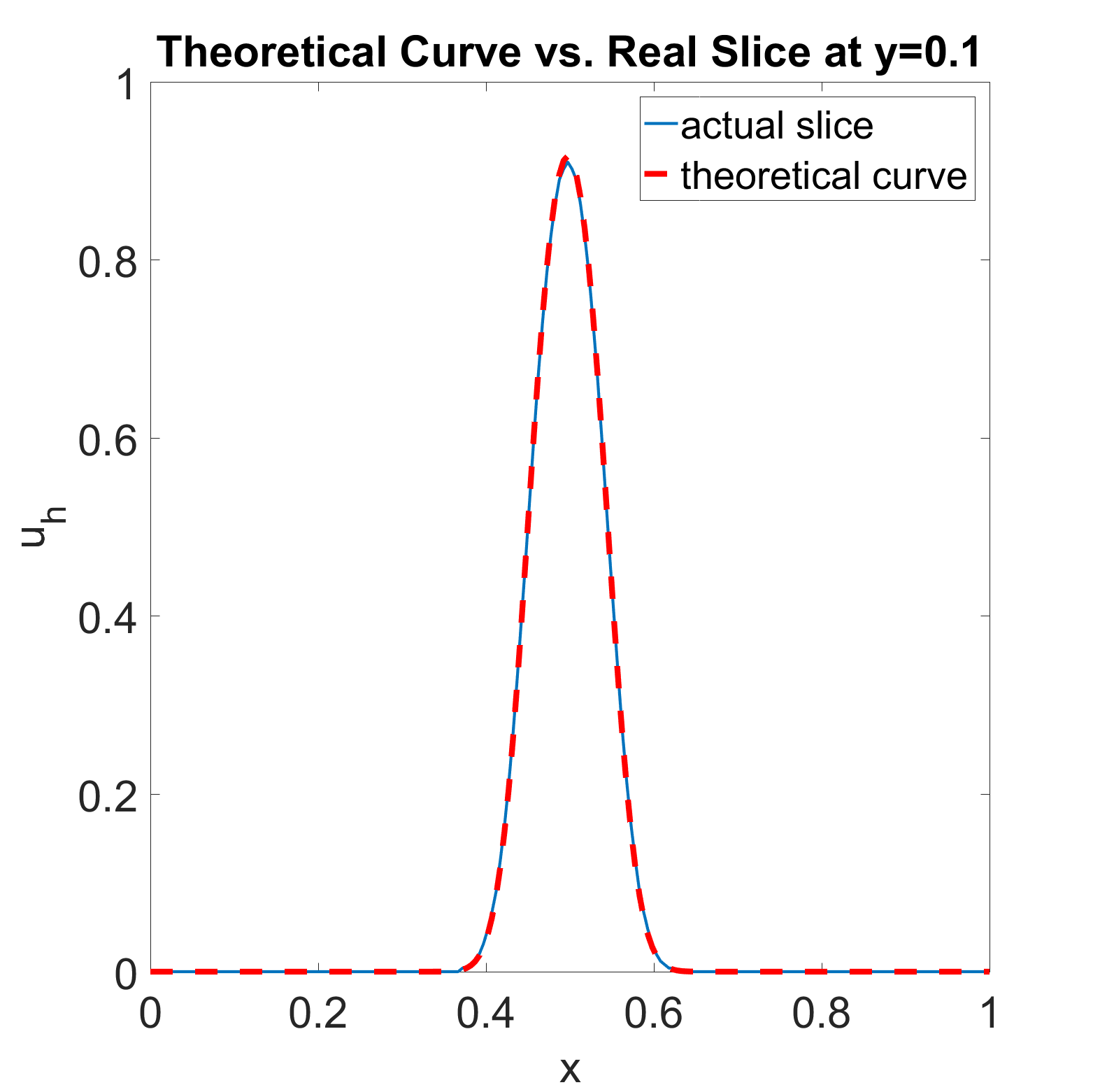}} & 
\subfloat[Slice of the result from (c) at $y=1.0$, compared with the theoretical curve from Conjecture \ref{conj:blur}.]{\includegraphics[width=.45\linewidth]{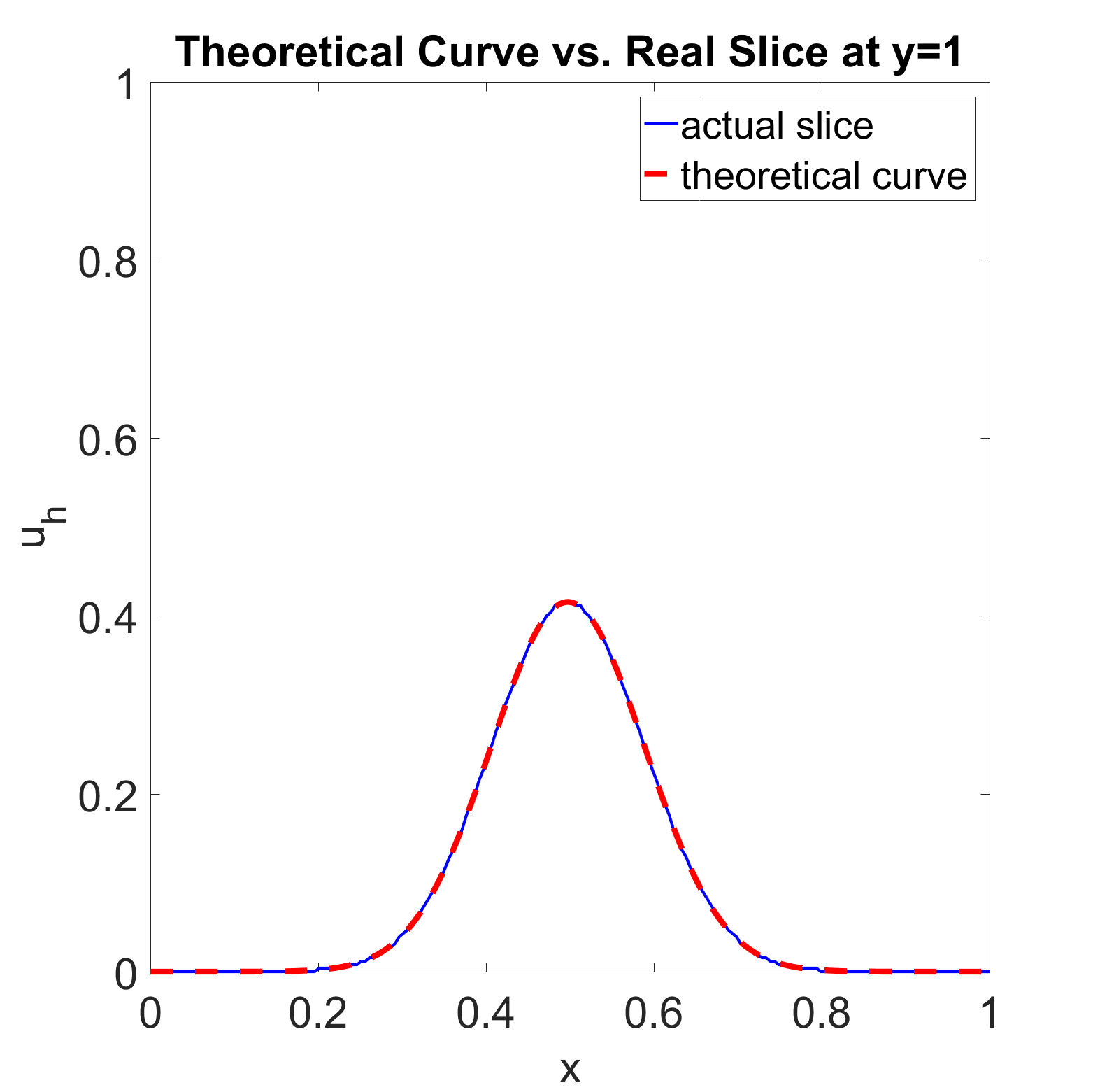}} \\
\end{tabular}
\caption{{\bf Transport is not the whole story:}  In this experiment, the problem shown in (a) of inpainting $D=[0,1]^2$ ($200 \times 200$px) given data on $[0,1] \times [-0.3,0)$ is solved using Guidefill with $\mu = 100$, $r=3$ and ${\bf g}=(\cos\theta,\sin\theta)$ for $\theta = \frac{\pi}{2}$ (b) and $\theta = 0$ (c).  Theorem \ref{thm:convergence} and the subsequent analysis of the limiting transport direction for Guidefill in Section \ref{sec:kink3main} (Figure \ref{fig:limitingCurves} and \eqref{eqn:guidefill}) suggests that these two situations are in some sense the same, as they are both converging to the same transport equation \eqref{eqn:transport} with the same transport direction ${\bf g}^*_r = e_2$.  However, one case yields a clean extrapolation while the other suffers from heavy blur.  This means the continuum limit of Algorithm 1 presented in Theorem \ref{thm:convergence}, while useful, is inadequate for studying blur artifacts.  Instead of a continuum limit, Conjecture \ref{conj:blur} proposes an {\em asymptotic} limit where $h$ is very small but non-zero.  As we illustrate here, this limit is able to make quantitative predictions that are in excellent agreement with blur artifacts measured in practice.  In (c)-(d) we compare horizontal slices of (c) at $y=0.1$ and $y=1$ respectively with the predictions of Conjecture \ref{conj:blur}.  In this case the predictions are accurate to within an error of $1/255$, the minimum increment of an image on our machine.}
\label{fig:wholeStory}
\end{figure}

\vskip 2mm

\subsubsection{Degerate stencils and $L^p$ convergence rates}  The first clue that blur artifacts might be present in $u_h$ comes from Theorem \ref{thm:convergence}, where we see convergence to $u$ in $L^p$ for every $p \in [1,\infty)$ but not necessarily in $L^{\infty}$.  This is consistent with blur that is present for every $h > 0$ but becomes less and less pronounced as $ h \rightarrow 0$.  A clue that these artifacts might be in some cases be avoidable comes from the fact that if the stencil weights put all of their mass into a single ${\bf y} \in a^*_r$ (we call such weights {\em degenerate}), as occurs in coherence transport with ${\bf g}=(\cos\theta,\sin\theta)$ for all but finitely many $\theta$ (Section \ref{sec:kink3main}), then the bounds in Theorem \ref{thm:convergence} can be tightened.  This is because the random walk ${\bf X}_{\tau}$ has become deterministic.  All terms related to variances or probabilities are killed off, and the bound immediately changes to 
\begin{equation} \label{eqn:degenerateBound}
\|u-u_h\|_p \leq K (rh)^{s \wedge \left( s' + \frac{1}{p}\right) \wedge 1}.
\end{equation}
In some cases, dependent on the boundary data $u_0 : \mathcal U_h \rightarrow \field{R}^d$, it is even possible to prove convergence in $L^{\infty}$ for $u_0$ with jump discontinuities.

\begin{figure}
\centering
\begin{tabular}{cc}
\subfloat[]{\includegraphics[width=.43\linewidth]{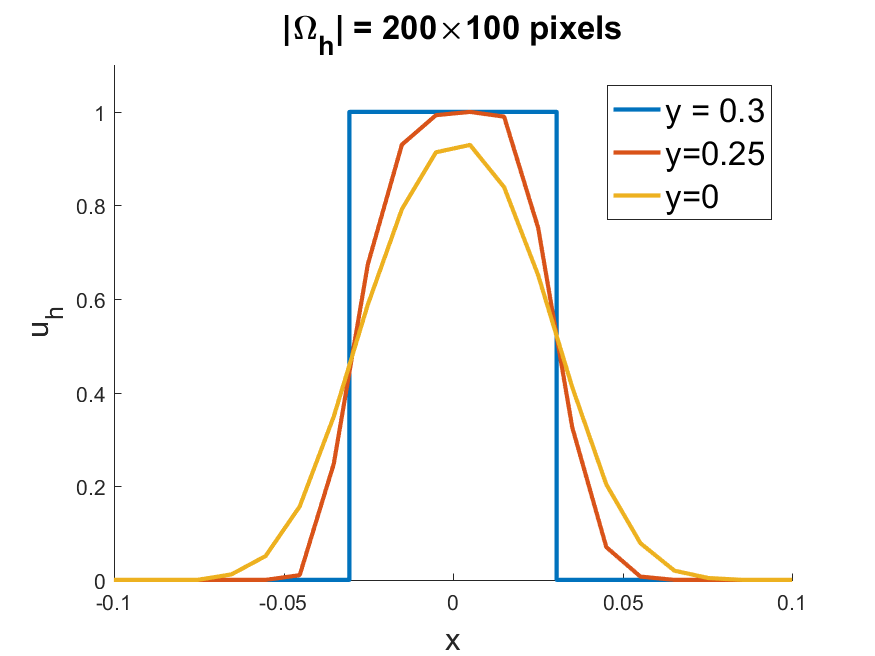}} & 
\subfloat[]{\includegraphics[width=.43\linewidth]{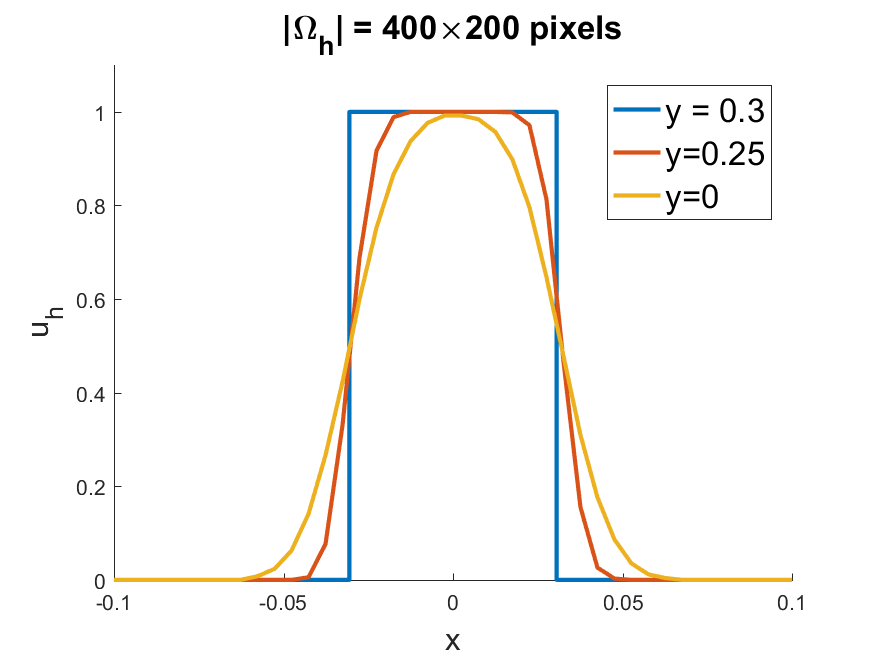}} \\
\subfloat[]{\includegraphics[width=.43\linewidth]{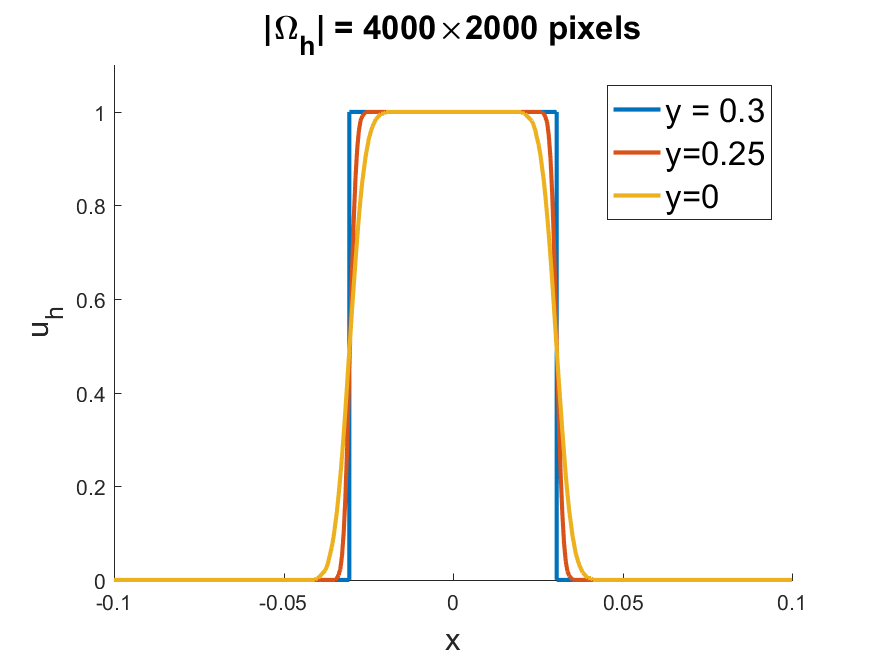}} & 
\subfloat[]{\includegraphics[width=.43\linewidth]{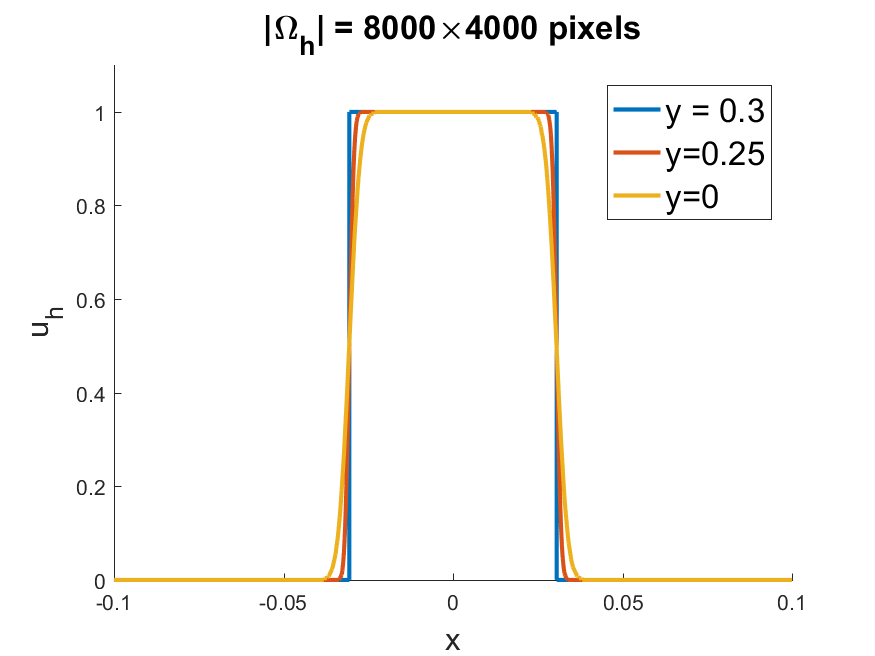}} \\
\end{tabular}
\caption{{\bf Signal degradation and $L^p$ convergence of Guidefill:}  The continuum problem of inpainting the line $\tan(73^{\circ}) - 0.1 \leq y \leq \tan(73^{\circ}) + 0.1$ with image domain $\Omega=[-1,1] \times [-0.5,0.5]$ and inpainting domain $D = [-0.8,0.8] \times [-0.3,0.3]$ is rendered at a variety of resolutions and inpainted each time using Guidefill.  Examining cross-sections of $u_h$ at $y=0.3$ (on the boundary of $D_h$), $y=0.25$ (just inside), and $y=0$ (in the middle of $D_h$) we notice a gradual deterioration of the initially sharp signal.  
If Conjecture \ref{conj:blur} is correct, then this deterioration is to be expected, as $u_h({\bf x})$ is related to a mollified version of $u_0$ with a Gaussian mollifier $g_{\sigma(h)}$.  This does not contradict our $L^p$ convergence results in Theorem \ref{thm:convergence}, since $\sigma(h) \rightarrow 0$ as $h \rightarrow 0$ (it does, however, shed some light on why Theorem \ref{thm:convergence} can establish convergence in $L^p$ for every $p<\infty$ but not in $L^{\infty}$ in general).  Here we see directly how decreasing $h$ leads to less and less loss of signal.  Another perspective is that since Guidefill is based on iterated bilinear interpolation, we should expect this effect as iterated bilinear interpolation is known to lead to signal degradation \cite[Sec. 5]{VariationalGradient}.  However, despite this we have less degradation in higher resolution images, even though we have applied more bilinear interpolation operations.}
\label{fig:degradation}
\end{figure}

\vskip 2mm

\subsubsection{An asymptotic limit for Algorithm 1}  The ideas in this section rely on the book \cite{gut2009stopped}, which covers the asymptotic distribution functions of certain stopped random walks.  Recall that we have the identity \eqref{eqn:blurKernel}, where $\rho_{{\bf X}_{\tau }}$ is the probability density function of the stopped random walk ${\bf X}_{\tau} ={\bf x} + h \sum_{i=1}^j {\bf Z}_i$.  The steps $\{{\bf Z}_i\} = \{(V_i,W_i)\}$ are i.i.d. and take values in $\mbox{Supp}(a^*_r) \subseteq \bar{b}^-_r$ (if we use the direct version of Algorithm 1) or $\mbox{Supp}(a^*_r) \subseteq \bar{b}^0_r$ (if we use the semi-implicit extension) with probability density
\begin{equation} \label{eqn:density}
P({\bf Z}_i = {\bf y}) = \frac{\tilde{w}_r({\bf 0},{\bf y})}{W},
\end{equation}
where $\tilde{w}_r$ are the equivalent weights defined in Section \ref{sec:fiction}.  Denote the mean of ${\bf Z}_i := (V_i,W_i)$ by $(\mu_x,\mu_y)$ and let $\tau$ given by \eqref{eqn:tau} be the first passage time through $y=0$.  The bound $\mu_y < 0$ is guaranteed (otherwise, the transport equation  \eqref{eqn:transport} is not well posed), and hence ${\bf X}_{\tau}$ has negative drift in the $y$-direction.  As we have already stated in Section \ref{sec:continuumLimit1}, this type of random walk is well understood \cite[Chapter 4]{gut2009stopped}, \cite{GutAndJanson1983}, \cite{gut1974}, \cite{GUT1974115}.  The way to understand blur artifacts is to recognize that $\rho_{{\bf X}_{\tau }}$ is a mollifier of sorts and moreover, because these types of stopped random walks obey a central limit theorem \cite[Chapter 4, Theorem 2.3]{gut2009stopped}, $\rho_{{\bf X}_{\tau }}$ is converging in distribution to a Gaussian mollifier.  In fact, if Conjecture \ref{conj:blur} is correct, then the precise form of mollification is identical to the definition of discrete mollifaction presented in applications such as \cite[Section 3]{DiscM}, in the context of stabilizing noisy data for the purposes of numerical differentiation, and in \cite{MollBook} for more general applications.  The mollifier is also equivalent to the one proposed in \cite[Section 2]{DiscM}.

\begin{figure}
\centering
\begin{tabular}{cc}
\subfloat[Original inpainting problem ($D_h$ in yellow).]{\includegraphics[width=.43\linewidth]{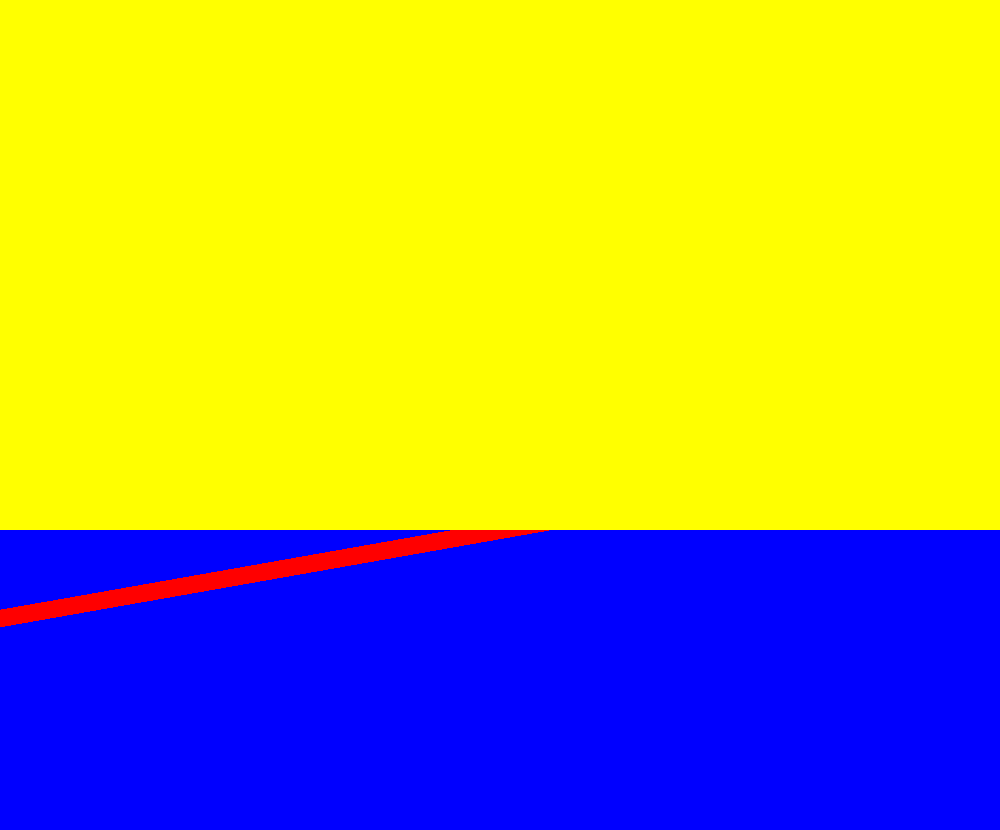}} & 
\subfloat[Inpainted using semi-implicit Guidefill with periodic boundary conditions.]{\includegraphics[width=.43\linewidth]{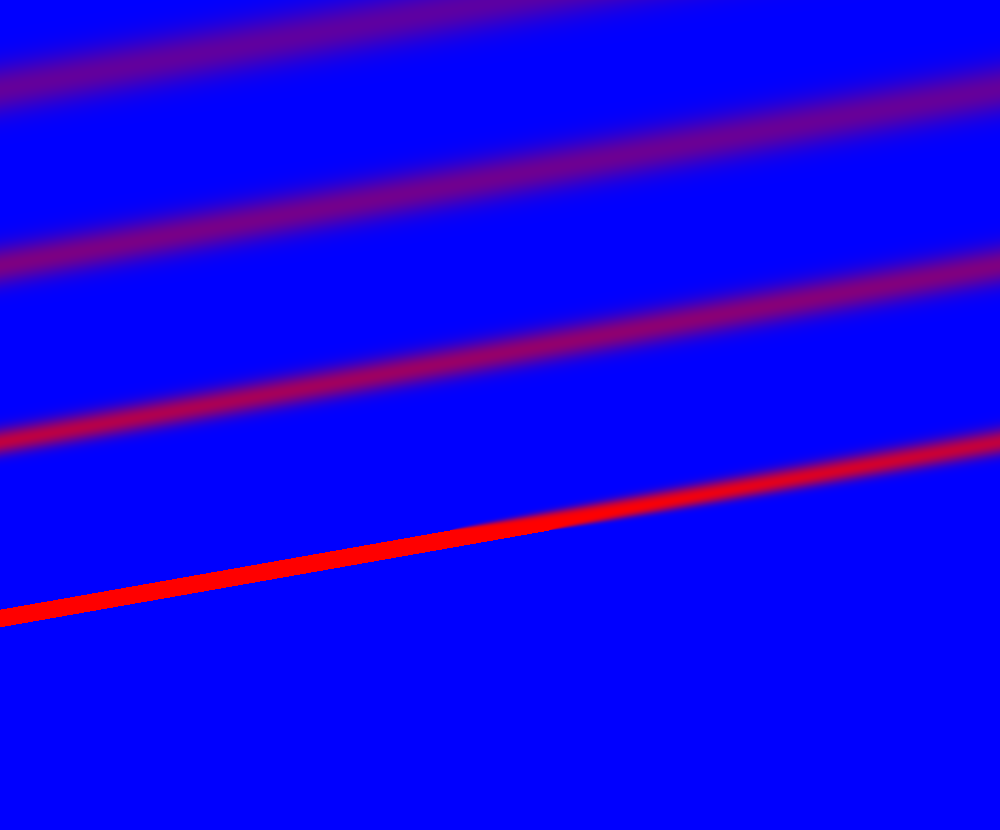}} \\
\subfloat[Boundary data at $y=0$.]{\includegraphics[width=.48\linewidth]{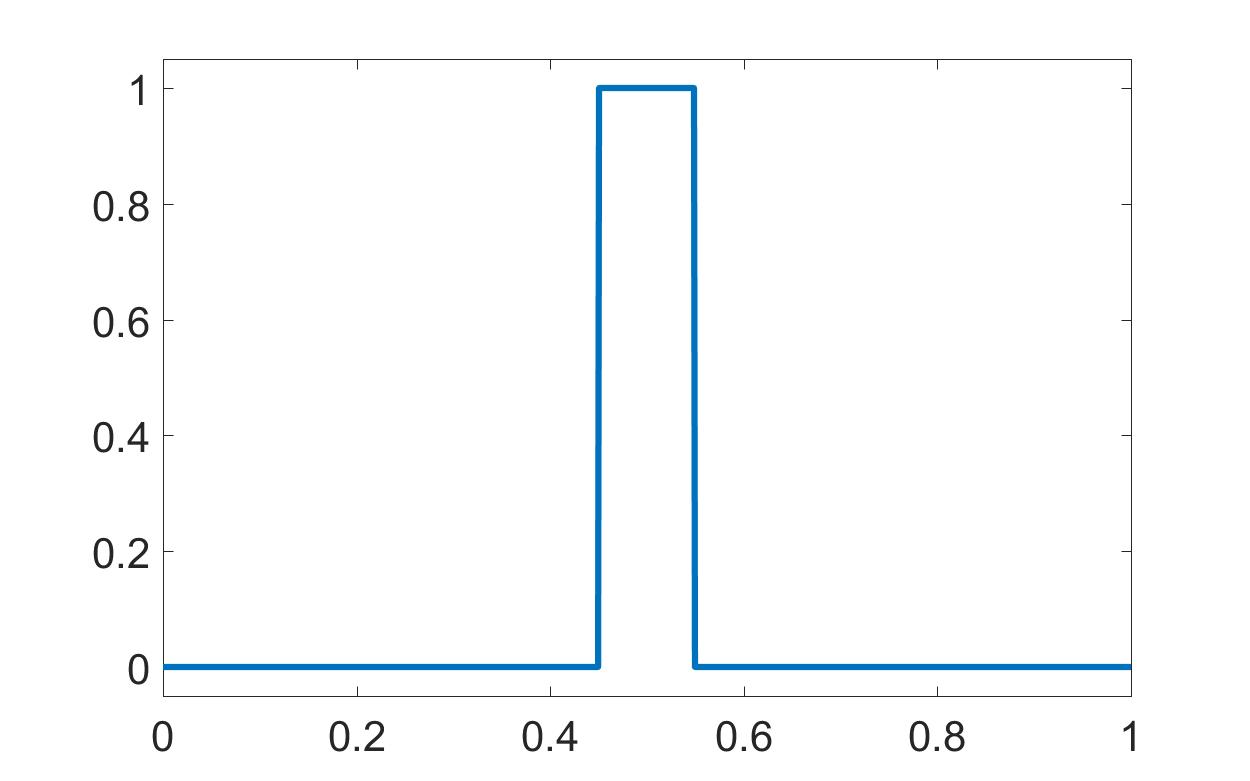}} & 
\subfloat[$y = \tan(10^{\circ}) \approx 0.18$.]{\includegraphics[width=.48\linewidth]{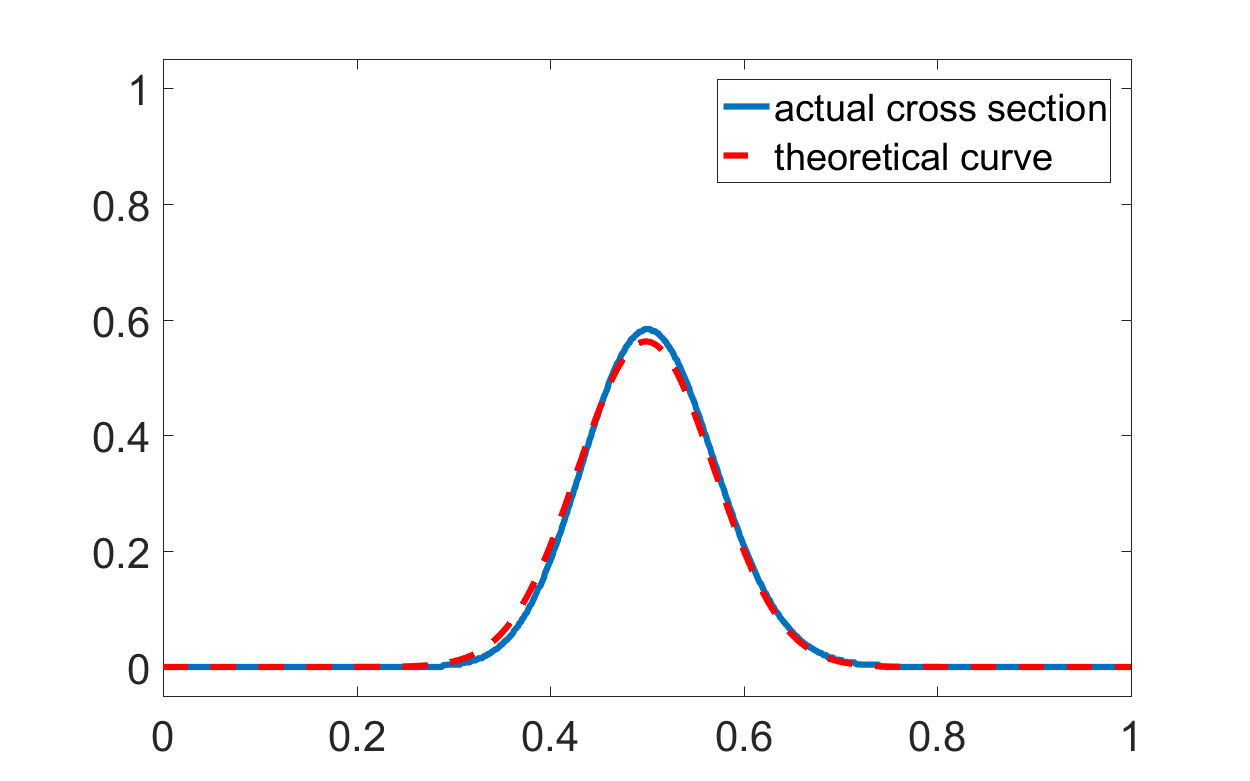}} \\
\subfloat[$y = 2\tan(10^{\circ}) \approx 0.35$.]{\includegraphics[width=.48\linewidth]{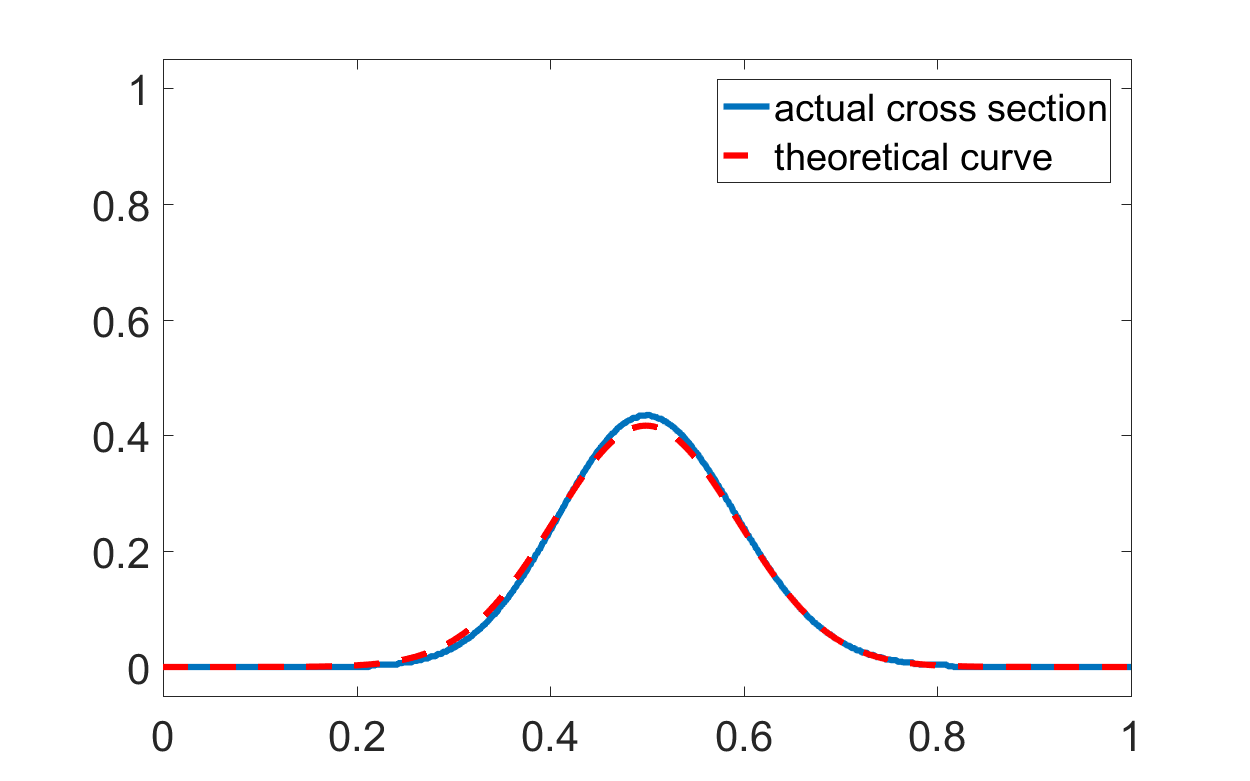}} & 
\subfloat[$y=3\tan(10^{\circ}) \approx 0.53$.]{\includegraphics[width=.48\linewidth]{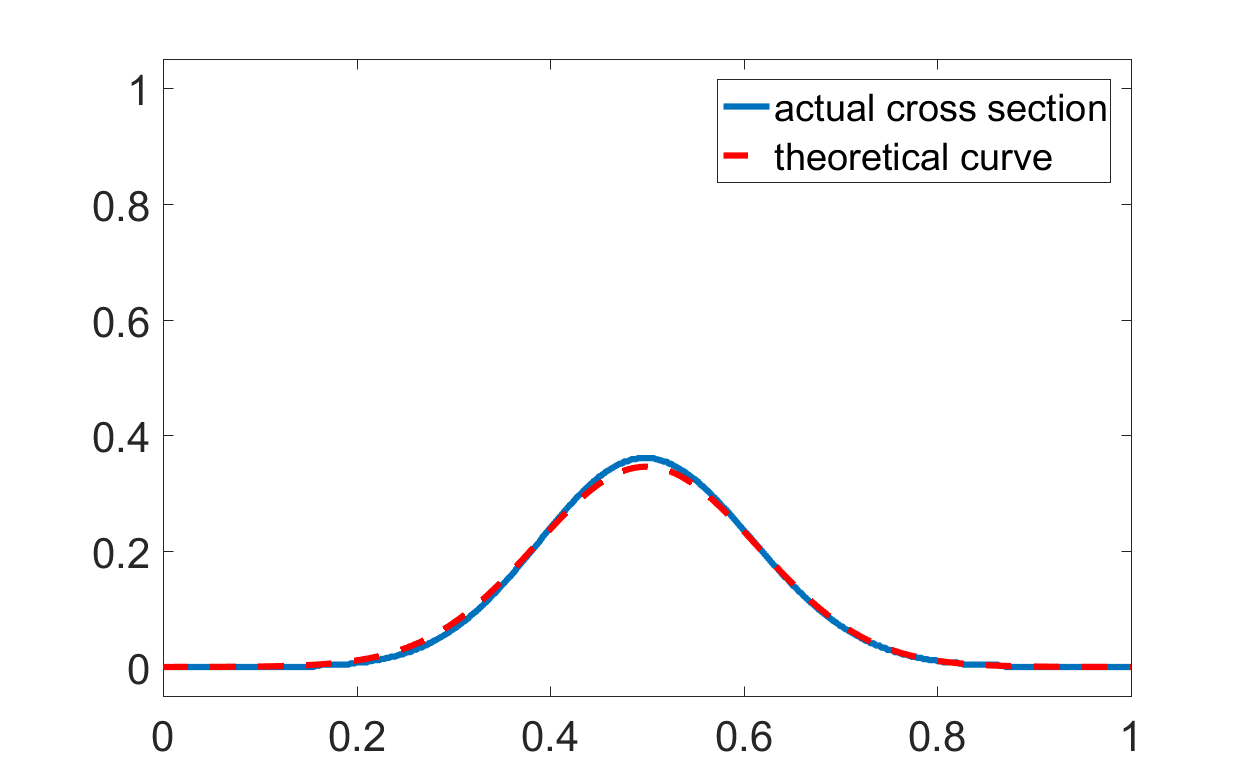}} \\
\end{tabular}
\caption{{\bf Numerical evidence for Conjecture \ref{conj:blur}:} Similarly to Figure \ref{fig:wholeStory}, here we consider again the problem of inpainting $D=[0,1]^2$ given data on $[0,1] \times [-0.3,0)$.  This time $u_0$ consists of a line making a an angle of $10^{\circ}$ with the horizontal, but the slice $u_0(x,0)$ is the same step as in Figure \ref{fig:wholeStory}.   This time $D_h$ is $1000 \times 1000$px.  Inpainting is done using semi-implicit Guidefill ($r=3$, ${\bf g}=(\cos10^{\circ},\sin10^{\circ})$  $\mu = 100$).  In (c) we show the initially sharp signal at $y=0$, while (d)-(f) compare horizontal slices at $y = \tan(10^{\circ}) \approx 0.18$, $y = 2\tan(10^{\circ} \approx 0.35$ and $y=3\tan(10^{\circ}) \approx 0.53$ with the predictions of Conjecture \ref{conj:blur}.  Even though in this case $u_0$ is not independent of $y$, we ignore this and use \eqref{eqn:molification} for our predictions, once again obtaining a very good prediction.  Compared with Figure \ref{fig:wholeStory}, notice that despite the fact that $h$ has decreased by an order of magnitude, our loss of signal is much more rapid.  This is consistent with the divergence $\sigma(h) \rightarrow \infty$ as $\theta \rightarrow 0$ we will encounter later in Figure \ref{fig:blurangle}.} 
\label{fig:blurwin}
\end{figure}

\vskip 2mm

\begin{conjecture} \label{conj:blur}
Let the inpainting domain $D$ and undamaged area $\mathcal U$, as well as their discrete counterparts $D_h$, $\mathcal U_h$ be as described in Section \ref{sec:symmetry}.  Let $u_0 : \mathcal U \rightarrow \field{R}^d$ denote as usual the undamaged portion of the image, and assume $u_0$ is non-negative and bounded, that is, there is an $M > 0$ such that $0 \leq u_0 \leq M$ on $\mathcal U$.  Suppose we inpaint $D_h$ using Algorithm 1 or its semi-implicit extension, and denote the result by $u_h : D_h \rightarrow \field{R}^d$.  Assume the assumptions of Section \ref{sec:symmetry} hold and let $a^*_r$ denote the stencil  of our inpainting method.  Let ${Z_i} = (V_i,W_i)$ taking values in $\mbox{Supp}(a^*_r)$ with mean $(\mu_x,\mu_y)$ denote the increments of the random walk described above, with probability density given by \eqref{eqn:density}.  
Let $\Pi_{\theta^*_r} : D \rightarrow \partial D$ denote the transport operator defined in \eqref{eqn:weak} (recall that that the transport direction ${\bf g}^*_r$ obeys ${\bf g}^*_r=(\mu_x,\mu_y)$).  Then, if $u_0 : \mathcal U \rightarrow \field{R}^d$ is independent of its $y$-coordinate, we have
\begin{equation} \label{eqn:molification}
u_h(x,y) \rightarrow (u_0 \Big | _{y=0}* g_{\sigma(h)})(\Pi_{\theta^*_r}(x,y)) \qquad \mbox{ asymptotically as } h \rightarrow 0
\end{equation}
at a rate that is dependent on $y$ but independent of $x$, where $u_0 \Big | _{y=0}* g_{\sigma(h)}$ denotes the discrete mollification of $u_0 \Big | _{y=0}$ with mollifier $g_{\sigma(h)}$ defined for any $x \in (0,1]$ by
$$\left(u_0 \Big | _{y=0}* g_{\sigma(h)} \right)(x):=\sum_{i=1}^N \left[ \int_{(i-1)h}^{ih} g_{\sigma(h)}(x-t)dt\right] u_0 \Big | _{y=0}(ih).$$
and where $g_{\sigma(h)}$ is a one dimensional Gaussian kernel with $h$-dependent variance given by
\begin{equation} \label{eqn:varianceH}
\sigma(h)^2 = \frac{\gamma^2 y h}{|\mu_y|^3} \quad \mbox{ where } \gamma^2 = \operatorname{Var}(\mu_x W_1 - \mu_y V_1)
\end{equation}
Moreover, for general $u_0$ we have
\begin{equation} \label{eqn:molification2}
u_h(x,y) \rightarrow (\tilde{u}_0* g_{\sigma(h)})(\Pi_{\theta^*_r}(x,y)) \qquad \mbox{ asymptotically as } h \rightarrow 0,
\end{equation}
(again, at a rate dependent on $y$ but independent of $x$) where $\tilde{u}_0 : \partial D_h \rightarrow \field{R}^d$ obeys $\tilde{u}_0(ih) = \sum_{j=-r-1}^0 \alpha_j(x) u_0(ih,jh)$ where for each $(ih,0) \in \partial D_h$ we have $0 \leq \alpha_j(ih) \leq 1$ and $\sum_{j=-r-1}^0 \alpha_j(ih) = 1$ (in other words, for each $ih$, $\tilde{u}_0(ih)$ is a convex combination of $u_0(ih,0)$ and the $r+1$ pixels directly below it). 
\end{conjecture}

\noindent{\it Proof Idea:}
Although proving this conjecture is beyond the scope of our current work, we briefly sketch how a proof might go and where the technical challenges arise.  First, note that it suffices to prove claim two, as the first claim is then a special case.  Our first job is to define $\tilde{u}_0$.  To that end, note that
\begin{eqnarray*}
u_h(x,y) & = & \sum_{i=1}^N \sum_{j=-r-1}^0 \rho_{{\bf X}_{\tau}}(ih,jh)u_0(ih,jh) \\
               & = & \sum_{i=1}^N \rho_{X_{\tau}}(ih) \underbrace{\sum_{j=-r-1}^0 \rho_{Y_{\tau} | X_{\tau}}(jh | ih) u_0(ih,jh)}_{:=\tilde{u}_0(ih)}.
\end{eqnarray*}
This definition of $\tilde{u}_0$ obeys the claimed properties since $\alpha_j(ih) := \rho_{Y_{\tau} | X_{\tau}}(jh | ih)$ obeys $0 \leq \alpha_j(ih) \leq 1$ and $\sum_{j=-r-2}^0 \alpha_j(ih)=1$ for all $i=1,\ldots N$ as claimed.  Next, define $\hat{x}(x,y) := \Pi_{\theta^*_r}(x,y)$ for convenience.  Fix $y > 0$ and let $\epsilon > 0$.  Our goal is to show that for $h$ sufficiently small we have
$$|u_h(x,y) - (\tilde{u}_0 * g_{\sigma(h)})(\hat{x}(x,y))| < \epsilon$$
independent of $x$.  To that end, note that
\begin{equation*}
|u_h(x,y) - (\tilde{u}_0 * g_{\sigma(h)})(\hat{x}(x,y))| =  \left|\sum_{i=1}^N \tilde{u}_0(ih) \left[ \rho_{X_{\tau}}(ih) -  \int_{(i-1)h}^{ih} g_{\sigma(h)}(\hat{x}-t)dt \right]\right|.\\
\end{equation*}
It follows from \cite[Chapter 4, Theorem 2.3]{gut2009stopped} that 
$$\rho_{X_{\tau}}(ih) \rightarrow \int_{(i-1)h}^{ih} g_{\sigma(h)}(\hat{x}-t)dt$$
as $h \rightarrow 0$.  However, since the number of terms in the above sum is also diverging to infinity as $h \rightarrow 0$, a convergence rate is needed to quantify the balance between these competing effects.  Unfortunately, none are given.  In the future, we hope to derive rates ourselves, so that we can upgrade Conjecture \ref{conj:blur} to a theorem.

\vskip 2mm

\begin{remark}
Conjecture \ref{conj:blur} says that $u_h(x,y)$, which we already know from Theorem \ref{thm:convergence} convergences as $h \rightarrow 0$ to the solution of the transport equation \eqref{eqn:transport}, can also be viewed asymptotically for small but non-zero $h$ as a solution to the same transport equation but where the boundary data has been mollified by a Gaussian kernel $g_{\sigma(h)}$.  Moreover, the degree of mollification \eqref{eqn:varianceH} increases as we move further into the inpainting domain  (that is, $\sigma(h)$ increases) and decreases as $h \rightarrow 0$ ($\sigma(h)$ decreases).  In fact, we have already observed in Figure \ref{fig:intuition} that $g_{\sigma(h)}$ is converging as $h \rightarrow 0$ to a Dirac delta distribution.
\end{remark}

\vskip 2mm

\subsubsection{Angular dependence of blur artifacts}  Conjecture \ref{conj:blur} proposes an asymptotic relationship between $u_h(x,y)$ and the convolution of $u_0(x,0)$ with a Gaussian kernel of $h$-dependent variance $\sigma(h)^2$ given by \eqref{eqn:varianceH}.  Although currently only a conjecture and valid only valid asymptotically as $h \rightarrow 0$, Figures \ref {fig:wholeStory} and \ref{fig:blurwin} suggest that the conjecture is not only valid but gives an extremely good approximation even for very small images (The example in Figure \ref{fig:wholeStory} is only $200 \times 200$px).  Conjecture \ref{conj:blur} has three takeaway messages:
\begin{itemize}
\item Blur gets {\em worse} as one moves further into the inpainting domain and (e.g. $y$ increases) - Figures \ref{fig:wholeStory} and \ref{fig:blurwin}.
\item Blur gets {\em better} as $h \rightarrow 0$ - Figure \ref{fig:degradation}.
\item Blur is {\em non-existent} if the stencil weights are {\em degenerate} - that is, put all of their mass into a single {\em real} pixel ${\bf y}$ (since in this case ${\bf Z}_i$ is deterministic and all variances are $0$).
\end{itemize}
All three of these observations are consistent with experience.  The third point in particular explains why coherence transport, which has a degenerate stencil in the limit $\mu \rightarrow \infty$ for all but finitely many $\theta=\theta({\bf g})$ (Section \ref{sec:kink3main}), does not appear to suffer from blur artifacts at all, e.g. Figure \ref{fig:taleOfTwo}.  However, it also predicts that for those few special angles where coherence transport puts its mass into more than one ${\bf y} \in b^-_r$, it will blur just like everything else (experiments, which we omit for the sake of space, confirm this).

\begin{figure}[H]
\centering
\includegraphics[width=1.0\linewidth]{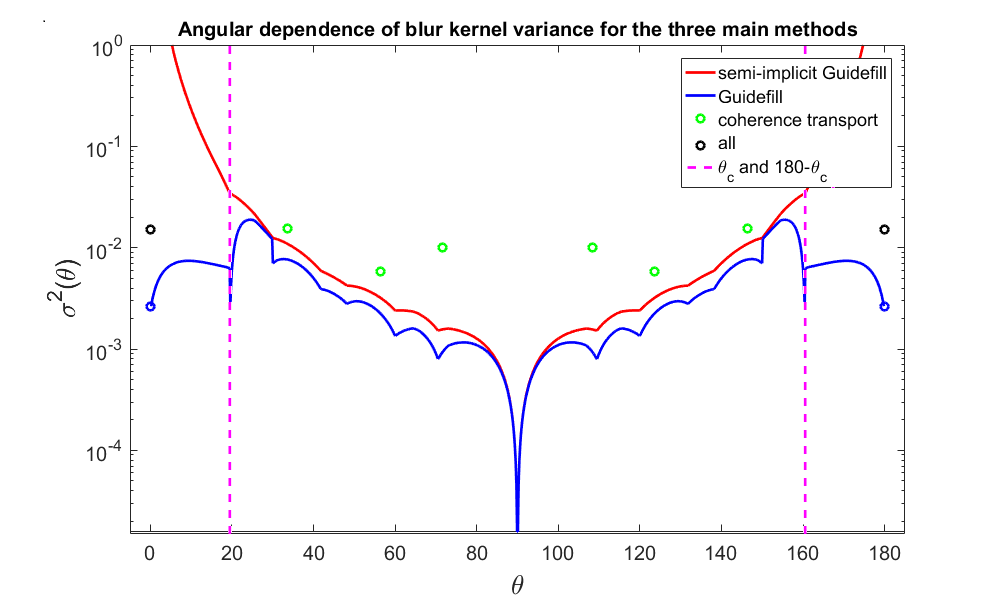}
\caption{{\bf Angular variation of blurring artifacts:}   Here we plot the angular dependence of the variance $\sigma^2(h)$ from the asymptotic blur kernel $g_{\sigma(h)}$ in Conjecture \ref{conj:blur}, for the three methods coherence transport, Guidefill, and semi-implicit Guidefill.  We fix with $r=3$, take $\mu \rightarrow \infty$, vary ${\bf g}=(\cos\theta,\sin\theta)$, and plot $\sigma^2(h)$ as a function of $\theta$.  We fix $y=1$ and $h=0.01$.  Note that for coherence transport, $\sigma^2(h) = 0$ for all but finitely many angles, explaining the methods apparent lack of blur artifacts (Figure \ref{fig:taleOfTwo}).  These special angles correspond precisely to the jumps in Figure \ref{fig:limitingCurves}(a), where coherence transport puts its mass into more than one ${\bf y} \in b^-_r$ (Section \ref{sec:kink3main}).  Note also that while $\sigma(h)^2$ remains bounded for Guidefill, for the semi-implicit extension it grows arbitrarily large as $\theta \rightarrow 0$ or $\theta \rightarrow \pi$.  Note also that all three methods agree at $\theta = 0$ and $\theta=\pi$, just like they did in Section \ref{sec:kink3main}.  Note the log scale.}
\label{fig:blurangle}
\end{figure}

What is less clear from  \eqref{eqn:varianceH} is how $\sigma(h)^2$ depends on ${\bf g}=(\cos\theta,\sin\theta)$, as $\frac{\gamma^2}{|\mu_y^3|}$ can be quite complex in general (certainly for semi-implicit Guidefill at least, where $|\mu_y| = g^*_r \cdot e_2$ can get arbitrarily close to $0$, we expect severe blur artifacts for shallow angles).  Figure \ref{fig:blurangle} illustrates the angular dependence of $\sigma(h)$ as a function of $\theta \in [0,\pi]$ with $y=1$ and $h=0.01$ fixed, for the three main methods of interest - coherence transport, Guidefill, and semi-implicit Guidefill (note the log scale).  In every case we have $r=3$ and $\mu = 100$.

\vskip 2mm

\begin{remark} \label{rem:blurTradeoff} Figure \ref{fig:blurangle} suggests a trade-off of sorts between kinking artifacts and blur artifacts.  If Conjecture \ref{conj:blur} is true, then semi-implicit Guidefill, the only method considered so far capable of avoiding kinking artifacts unless ${\bf g}=(\cos\theta,\sin\theta)$ is exactly parallel to the inpainting domain, that is ${\bf g}^*_r = {\bf g}$ unless $\theta \in \{0,\pi\}$, pays a price for this ability with blurring artifacts that become arbitrarily bad as $\theta \rightarrow 0$ or $\theta \rightarrow \pi$.  At the opposite extreme, coherence transport suffers from no blur at all for all but finitely many angles, but also kinks (that is ${\bf g}^*_r \neq {\bf g}$) for all but finitely many angles (Figure \ref{fig:specialDir}).  Guidefill is in the middle: ${\bf g}^*_r ={\bf g}$ so long as $\theta$ is not too shallow, and blur exists but remains bounded.  It remains to be seen whether a method be completely free of both kinking and blur artifacts, but it seems unlikely.  A more modest goal would be to design a method which like semi-implicit Guidefill suffers from no kinking artifacts, but for which $\sigma(h)$ defined by \eqref{eqn:varianceH} remains bounded for all $\theta$, like Guidefill.  Whether or not this is possible also remains to be seen, and is something we would like to investigate in the future.
\end{remark}

\section{Numerical Experiments} \label{sec:numerics}

In this section we have three objectives:
\begin{enumerate}
\item To compare the limiting transport directions derived in Section \ref{sec:kink3main} for coherence transport, Guidefill, and semi-implicit Guidefill as $\mu \rightarrow \infty$ with the orientation of extrapolated isophotes when the algorithm is run in practice.
\item To check experimentally our claim in Section \ref{sec:marzLimit} that when $r$ is a small integer, our limit $u$ typically does a better job of capturing the behaviour of $u_h$ than $u_{\mbox{m\"arz}}$ does, that is $\|u_h-u\|_p \ll \|u_h - u_{\mbox{m\"arz}}\|_p$.
\item To verify our bounds in Theorem \ref{thm:convergence} and check that they are tight.  Further, we wish to see whether or not our results continue to hold if some of our simplifying assumptions are relaxed (specifically, Guidefill with a constant guidance direction ${\bf g}$ replaced by a smoothly varying transport field ${\bf g}({\bf x})$).
\end{enumerate}
\noindent We will perform one experiment for each objective, but for the sake of space, Experiment III is deferred to Appendix \ref{app:exp3}.  For Experiments II and III, we will experiment using a variety of boundary data $u_0$ satisfying the hypotheses of Section \ref{sec:continuumLimit1}, as well as a few different inpainting methods that are each special cases of Algorithm 1.  These are presented in detail in Appendix \ref{app:exampleList}.

\begin{figure}
\centering
\begin{tabular}{cc}
\subfloat[Original Inpainting problem, inpainting domain shown in yellow.]{\includegraphics[width=.35\linewidth]{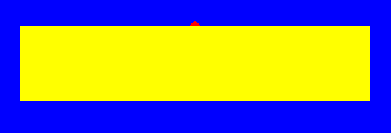}} & 
\subfloat[Inpainting using coherence transport with ${\bf g}=(\cos45^{\circ},\sin45^{\circ})$, $r=3$, $\mu=40$.]{\includegraphics[width=.35\linewidth]{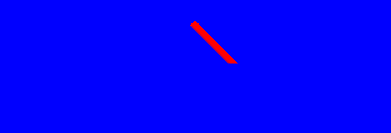}} \\
\end{tabular}
\caption{{\bf Stretching a dot into a line:}  In (a) we have an inpainting problem consisting of a red dot on a blue background, with the inpainting domain in yellow.  In (b), we see the result of inpainting using coherence transport with $\mu=40$, $r=3$, and ${\bf g}=(\cos45^{\circ},\sin45^{\circ})$.  The dot is now stretched into a line, the orientation of which may be measured to deduce ${\bf g}^*_r$.}
\label{fig:dot}
\end{figure}

\subsection{Experiment I:  Validation of limiting transport directions for coherence transport, Guidefill, and semi-implicit Guidefill.}

\begin{figure}
\centering
\begin{tabular}{ccc}
\subfloat[Coherence transport, theoretical curve, $r=3$.]{\includegraphics[width=.30\linewidth]{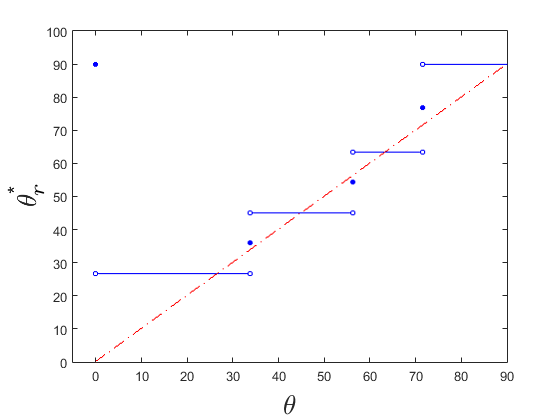}} & 
\subfloat[Guidefill, theoretical curve, $r=3$.]{\includegraphics[width=.30\linewidth]{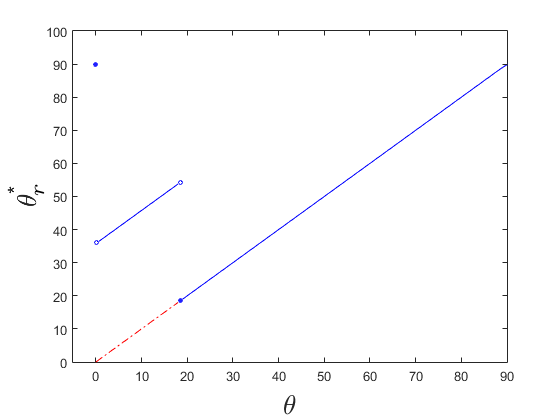}} &
\subfloat[Semi-implicit Guidefill, theoretical curve, $r=3$.]{\includegraphics[width=.30\linewidth]{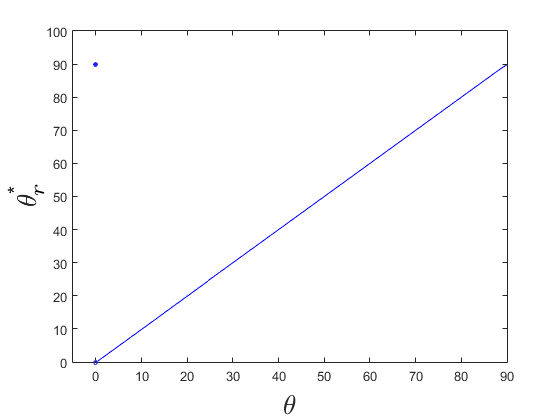}} \\ 
\subfloat[Coherence transport, actual result, $r=3$.]{\includegraphics[width=.30\linewidth]{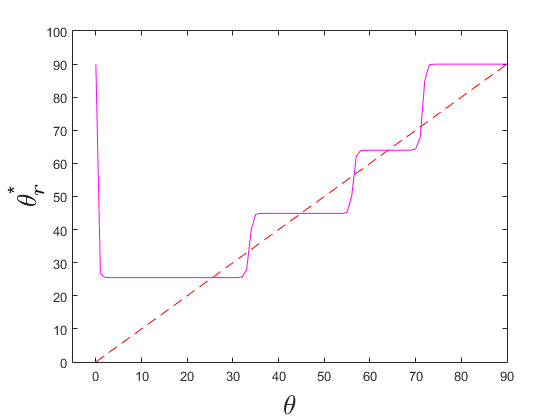}} & 
\subfloat[Guidefill, actual result, $r=3$.]{\includegraphics[width=.30\linewidth]{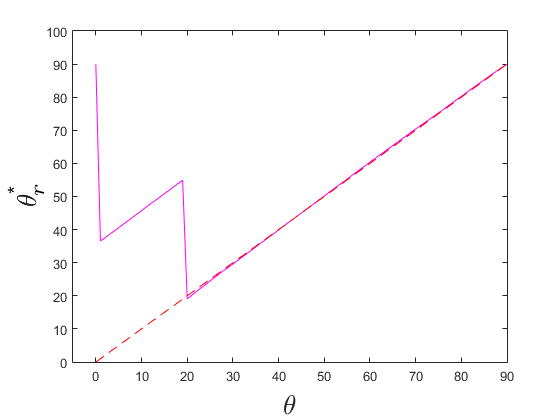}} &
\subfloat[Semi-implicit Guidefill, actual result, $r=3$.]{\includegraphics[width=.30\linewidth]{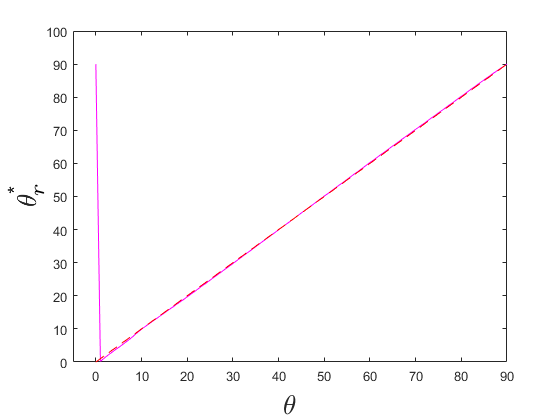}} \\ 
\subfloat[Coherence transport, theoretical curve, $r=5$.]{\includegraphics[width=.30\linewidth]{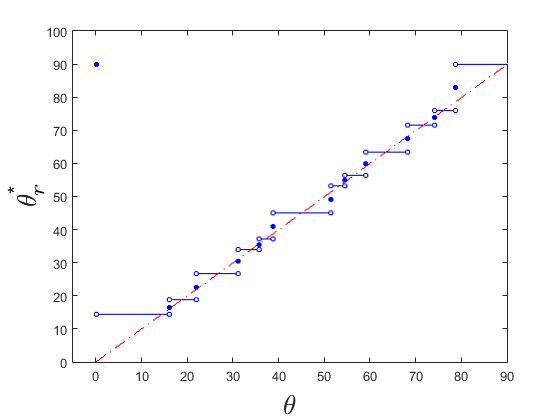}} & 
\subfloat[Guidefill, theoretical curve, $r=5$.]{\includegraphics[width=.30\linewidth]{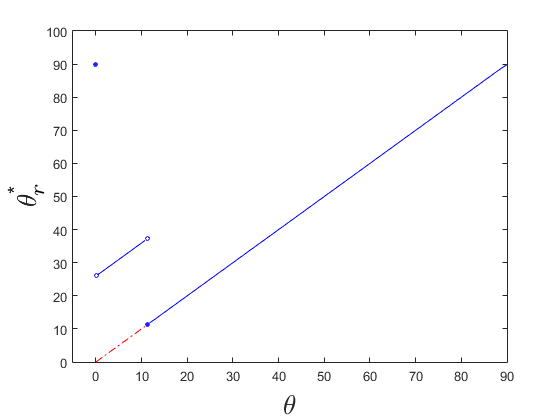}} &
\subfloat[Semi-implicit Guidefill, theoretical curve, $r=5$.]{\includegraphics[width=.30\linewidth]{semiImplictTheory3b.png}} \\ 
\subfloat[Coherence transport, actual result, $r=5$.]{\includegraphics[width=.30\linewidth]{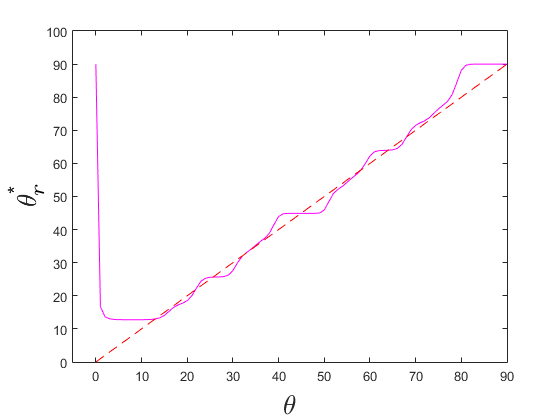}} & 
\subfloat[Guidefill, actual result, $r=5$.]{\includegraphics[width=.30\linewidth]{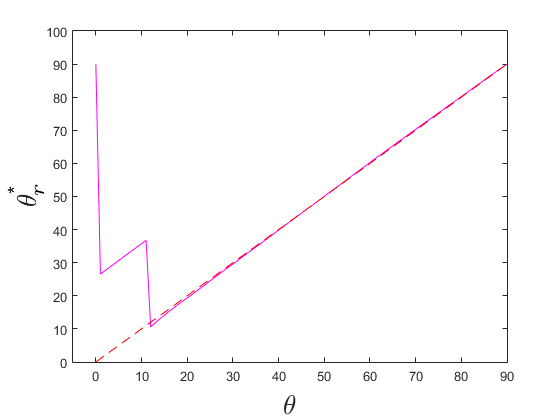}} &
\subfloat[Semi-implicit Guidefill, actual result, $r=5$.]{\includegraphics[width=.30\linewidth]{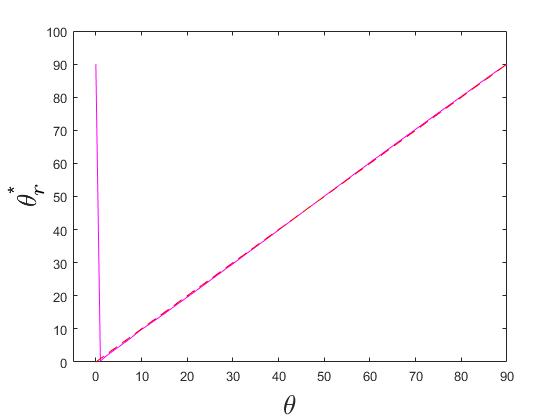}} \\ 
\end{tabular}
\caption{{\bf Validation of limiting transport directions for coherence transport, Guidefill, and semi-implicit Guidefill:}  Here we compare the limiting transport directions $\theta^*_r=\theta({\bf g}^*_r)$ as a function of $\theta = \theta({\bf g})$ derived in Section \ref{sec:kink3main} for coherence transport, Guidefill, and semi-implicit Guidefill as $\mu \rightarrow \infty$ with the orientation of extrapolated isophotes in the inpainting problem shown in Figure \ref{fig:dot}(a) (where $\mu = 40$).  The top row (a)-(c) gives the theoretical curves for $r=3$, with the actual measured results in the row underneath (d)-(f).  The third row (g)-(i) gives the theoretical curves for $r=5$, while the final row (j)-(l) gives corresponding real measurements.}
\label{fig:Refraction}
\end{figure}

In this experiment we compare the limiting transport directions derived in Section \ref{sec:kink3main} for coherence transport, Guidefill, and semi-implicit Guidefill as $\mu \rightarrow \infty$ with the orientation of extrapolated isophotes in an actual inpainted image $u_h$ obtained in practice with finite $\mu$.  In each case we choose as our boundary data the image shown in Figure \ref{fig:dot}(a), consisting of a red dot on a blue background, with the inpainting domain shown in yellow.  We run each algorithm with
$${\bf g}=(\cos (k^{\circ}),\sin (k^{\circ}))$$
for $k=0,1,\ldots,90$, with $\mu=40$ fixed and for various values of $r$.  The dot is then stretched into a line as in Figure \ref{fig:dot}(b), the orientation of which gives ${\bf g}^*_{\mu,r}$ and which can be measured numerically.  

Results are shown in Figure \ref{fig:Refraction} for $r=3$ and $r=5$.  The top row (a)-(c) gives the theoretical curves for $r=3$, with the actual measured results in the row underneath (d)-(f).  The third row (g)-(i) gives the theoretical curves for $r=5$, while the final row (j)-(l) gives corresponding real measurements.  While we see some smoothing out of the jump discontinuities in the case of coherence transport, this is expected as one can easily show that the convergence to $\theta^*=\lim_{\mu \rightarrow \infty} \theta^*_{r,\mu}$ is pointwise but not uniform, becoming arbitrarily slow in the vicinity of the jumps.  On the other hand, for Guidefill and semi-implicit Guidefill we see excellent agreement with theory even in the vicinity of jump discontinuities.  This is to be expected as well, since one can easily show that the relevant limits are uniform in this case.

\begin{figure}
\centering
\begin{tabular}{cc}
\subfloat[Example 1, $h=2^{-10}$, $u_0$ given by \eqref{eqn:u0} with $s=2$, $s'=0$, $L^1$ norm.]{\includegraphics[width=.46\linewidth]{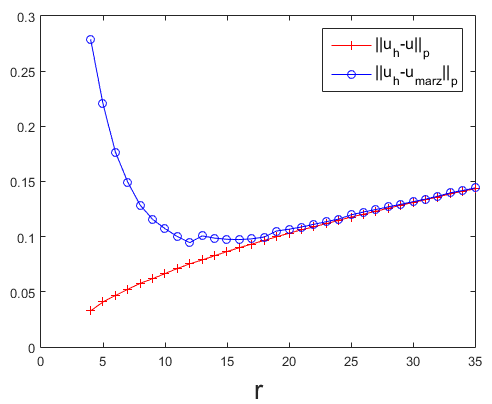}} &
\subfloat[Example 2, $h=2^{-10}$, $u_0$ given by \eqref{eqn:u0} with $s=s'=0.5$, $L^2$ norm.]{\includegraphics[width=.46\linewidth]{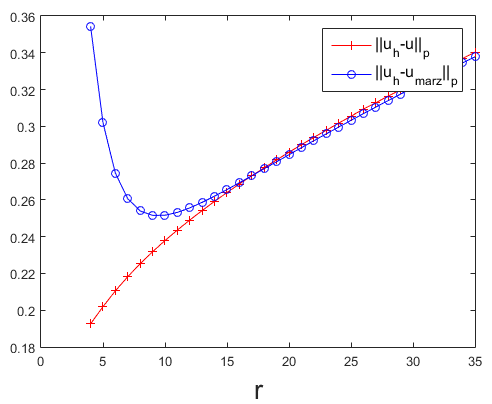}} \\
\subfloat[Example 1, $h=2^{-13}$, $u_0$ given by \eqref{eqn:u0} with $s=2$, $s'=0$, $L^1$ norm.]{\includegraphics[width=.46\linewidth]{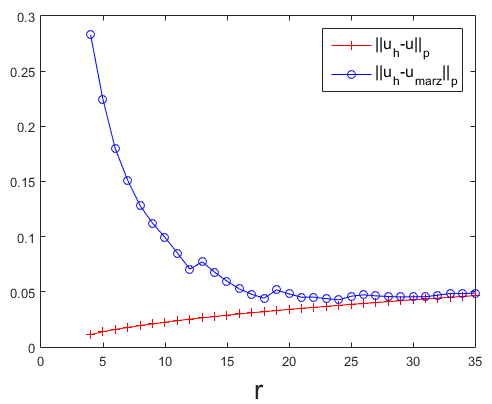}} &
\subfloat[Example 2, $h=2^{-13}$, $u_0$ given by \eqref{eqn:u0} with $s=s'=0.5$, $L^2$ norm.]{\includegraphics[width=.46\linewidth]{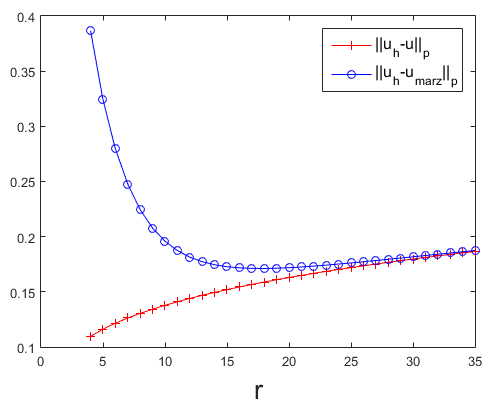}} \\
\end{tabular}
\caption{{\bf Quantitative comparison of the $L^p$ distance between $u_h$ and its two continuum limits:}  Plots of $\|u_h - u\|_p$ (starred line) and $\|u_h-u_{\mbox{m\"arz}}\|_p$ (dotted line) as a function of $r$ for $h=2^{-10}$ (top row) and $h=2^{-13}$ (bottom row) fixed. In each case $u_0$ is given by \eqref{eqn:u0} for different values of $s'$ and $s$. In the the left column we have $s=2$, $s'=0$, $w(\cdot,\cdot)$ is given by M\"arz's weights \eqref{eqn:weight} with $\mu=10$, ${\bf g}=(\cos 20^{\circ},\sin 20^{\circ})$ as in Appendix \ref{app:exampleList}, Example 1, and error is measured using the $L^1$ norm. In the right column we have $s=s'=0.5$, $w(\cdot,\cdot)$ is given by the offset Gaussian \eqref{eqn:offsetG} as in Appendix \ref{app:exampleList}, Example 2, and error is measured using the $L^2$ norm.  In every case we have $\|u_h - u\|_p \ll \|u_h-u_{\mbox{m\"arz}}\|_p$ when $r$ is small.  This sheds some light on why our continuum limit $u$ appears to much more closely match the actual inpainted solution $u_h$ than the alternative limit $u_{\mbox{m\"arz}}$ does.  Results for other values of $s$ and $s'$ as well as different choices of norm are similar but omitted. }
\label{fig:Errors}
\end{figure}

\subsubsection{Experiment II: Comparison of convergence rates to both continuum limits} \label{sec:MarzIsFar}

Our second experiment explores the relationship between $\|u_h-u\|_p$ and \\$\|u_h-u_{\mbox{m\"arz}}\|_p$, for $u_h$ such that the limit $u_{\mbox{m\"arz}}$ exists, and under the assumption that we are not in the trivial case $u = u_{\mbox{m\"arz}}$. From Section \ref{sec:marzLimit} we have the bound
\begin{equation} \label{eqn:boundAgain}
\|u-u_{\mbox{m\"arz}}\|_p  \leq  K_1 \cdot (rh)^{(\frac{s'}{2}+\frac{1}{2p}) \wedge \frac{s}{2} \wedge 1}+K_2 r^{-q\left\{s \wedge \left( s' +\frac{1}{p} \right) \wedge 1\right\}}.
\end{equation}
If this bound is tight, then when $r$ is a small integer and $h$ is small, we expect $\|u_h-u\|_p \ll \|u_h - u_{\mbox{m\"arz}}\|_p$.  At the same time, Theorem \ref{thm:convergence2} shows that $u \rightarrow u_{\mbox{m\"arz}}$ as $ r \rightarrow \infty$, so we expect $\|u_h-u\|_p \approx \|u_h - u_{\mbox{m\"arz}}\|_p$ when $r$ is large.  Figure \ref{fig:Errors} tests this by plotting $\|u_h-u\|_p$ and $\|u_h - u_{\mbox{m\"arz}}\|_p$ as a function of $r$ with $h$ fixed and for various choices of $p$, and for a few choices of boundary data $u_0$ and inpainting methods from Appendix \ref{app:exampleList}.  Specifically, the boundary data $u_0$ is given by \eqref{eqn:u0} for various values of $s$ and $s'$, and the weights and neighborhoods are as in Example 1 and Example 2 from Appendix \ref{app:exampleList}.  The results confirm our expectations.  For the sake of space, we only show a few illustrative examples.

\section{Conclusions and Future Work} \label{sec:conclusion}

In this paper we have presented a detailed analysis of a class of geometric inpainting algorithms based on the idea of filling the inpainting domain in successive shells from the boundary inwards, where every pixel to be filled is assigned a color equal to a weighted average of its already filled neighbors.  These methods are all special cases of a generic framework sketched in Algorithm 1.  Methods in the literature falling within this framework include 
\begin{itemize}
\item Telea's Algorithm \cite{Telea2004}.
\item Coherence Transport \cite{Marz2007,Marz2011}.
\item Guidefill \cite{Guidefill}.
\end{itemize}
A subtle but important point about these methods is that pixels in the current inpainting boundary are filled independently.  Noting this, we have proposed a semi-implicit extension of these methods in which pixels in the current shell are instead filled simultaneously by solving a linear system.  We have also sketched in Algorithm 1 a straightforward extension of the original framework that is equivalent to solving this linear system using damped Jacobi or successive over-relaxation (SOR).  A theoretical convergence analysis of these methods is presented for the semi-implicit extension of Guidefill, where we show that SOR is extremely effective.  This analysis is backed up by numerical experiments.  We also present in some detail additional features of semi-implicit Guidefill relating to how the method decides on a good order to fill pixels, but this is not the main focus of our work.

As all of the algorithms listed above (with the exception of semi-implicit Guidefill, which is presented for the first time in this paper) are known to create some disturbing visual artifacts, the main objective of our analysis is to understand why these occur and whether or not they are inevitable.  We focus specifically on kinking artifacts and blur artifacts.  Other artifacts including the formation of shocks and extrapolated isophotes that end abruptly are discussed but not analyzed, as they have already been studied elsewhere \cite{Marz2007,Marz2015} and are well understood.  Our analysis is based on two key ideas:
\begin{itemize}
\item A continuum limit, which we use to explain kinking artifacts.
\item A connection to the theory of stopped random walks, which we use both to prove convergence to our continuum limit, and to explore blur artifacts.
\end{itemize}
Similarly to the earlier work of Bornemann and M\"arz \cite{Marz2007}, our continuum limit is a transport equation.  However, our limit process is different and so are the coefficients of the resulting transport equation.  Moreover, numerical experiments show that our transport equation is a better reflection of the behaviour of Algorithm 1 (the direct form and our proposed semi-implicit extension) in practice, capable of accurately predicting kinking phenomena that is not captured by the alternative continuum limit proposed in \cite{Marz2007}.  The second core idea of our analysis, which is to relate Algorithm 1 and its extension to stopped random walks, is critical for two reasons.  Firstly, it allows us to prove convergence to our continuum limit even for boundary data with low regularity, such as (finitely many) jump discontinuities.   By contrast, the analysis in \cite{Marz2007} assumes smooth boundary data, which is an unrealistic assumption for images.  Secondly, this connection is central to our analysis of blur artifacts, which we analyze based not on a continuum limit where $h \rightarrow 0$, but rather an asymptotic limit where $h$ very small but nonzero.  While we have not (yet) been able to prove convergence to this asymptotic limit, it allows us to make quantitative predictions that are in excellent agreement with numerical experiments, even for relatively low resolution images (e.g. Figure \ref{fig:wholeStory} which is only $200\times200$px).  Our analysis operates in a highly idealized setting (Section \ref{sec:symmetry}), but our conclusions are far reaching.  In particular, we prove the following:
\begin{enumerate}
\item The rate of convergence to our continuum limit depends on the regularity of the boundary data, and convergence is slower for boundary data with lower regularity.  Thus, our continuum limit $u$ does a better job of approximating the actual inpainted image $u_h$ when the boundary data $u_0$ is smooth (Theorem \ref{thm:convergence}).
\item The difference between our continuum limit and the one proposed in \cite{Marz2007} is most significant when the radius $r$ of the averaging neighborhood $A_{\epsilon,h}({\bf x})$ (measured in pixels) is small, and goes to zero as $r \rightarrow \infty$ (Theorem \ref{thm:convergence2}).
\item In the direct form of Algorithm 1, kinking artifacts will always be present.  That is, certain isophotes cannot be extended into the inpainting domain without bending (Section \ref{sec:kinkingAndConvex}).
\item This is not true of the semi-implicit extension of Algorithm 1.  In particular, semi-implicit Guidefill can extrapolate isophotes with any orientation\footnote{that is, unless the isophotes are exactly parallel to the boundary of the inpainting domain.  But in this case extrapolation is not defined.}, and moreover is able to do so efficiently by using SOR to solve the required linear system (Section \ref{sec:kinkingAndConvex}, Section \ref{sec:kink3main}, Corollary \ref{corollary:actualRates}).
\end{enumerate}
The following results, which we do not prove, are implied by Conjecture \ref{conj:blur} if it is true:
\begin{enumerate}
\item Blur artifacts exhibit an angular dependence, which for semi-implicit Guidefill becomes arbitrarily bad as the angle an extrapolated isophote makes with the boundary of the inpainting domain goes to zero.  Thus, semi-implicit Guidefill pays a heavy price (on top of the need to solve a linear system for every shell) for its ability to successfully extrapolate such isophotes (Conjecture \ref{conj:blur}, Figure \ref{fig:blurangle}).
\item Blur artifacts become less significant as the resolution of the image goes up, and get worse the further into the inpainting domain you extrapolate (Conjecture \ref{conj:blur}).
\item Methods that put all of their weight into a single pixel exhibit no blur, but can only extrapolate without kinking in finitely many directions (Remark \ref{rem:blurTradeoff}).  
\end{enumerate}
The last of these conclusions suggest that there is an apparent trade off between kinking artifacts and blur artifacts, however, this is something that requires additional investigation.  Our inability to prove Conjecture \ref{conj:blur} is due to a result on stopped random walks from \cite{gut2009stopped} - that our result depends on - failing to give rates of convergence.  In the future we hope to derive those rates ourselves so that we can prove or disprove this conjecture.  In fact, there are a number of questions we would like to explore in the future.  
\begin{enumerate}
\item Does there exist an algorithm within the framework we have described that avoids blurring artifacts and kinking artifacts at the same time?  If not, is it at least possible to design an algorithm that, like semi-implicit Guidefill, avoids kinking artifacts so long as the guidance direction ${\bf g}=(\cos\theta,\sin\theta)$ obyes $\theta \neq 0$, but for which the severity of blur as a function of $\theta$ remains bounded?  Could this be done, for example, by the replacing the bilinear interpolation used to define ghost pixels with a more sophisticated interpolation scheme?
\item While we have already shown that SOR (with appropriate ordering of unknowns) is particularly effective for solving the linear system arising in semi-implicit Guidefill, is this generically true regarding the semi-implicit extension of {\em any} inpainting method within the class under investigation?  Moreover, SOR is not ideal because it is sequential whereas Guidefill was designed to be a parallel algorithm.  Does another method exist which maintains the fast convergence rate of SOR, but can be implemented in parallel?  One possibility here is the scheduled Jacobi method \cite{ScheduledJacobi}.
\item Is our semi-implicit extension of Algorithm 1 the best way of avoiding kinking artifacts?  In our analysis, we have assumed that the radius of our averaging neighborhood remains fixed.  Another possibility, based on the observation that as $r$ increases, (direct) Guidefill is able to successfully extrapolate shallower and shallower angles (Section \ref{sec:kink3main}), is to use the direct form of Guidefill but dynamically resize $\tilde{B}_{\epsilon,h}({\bf x})$ as needed.  
\item What happens if the semi-implicit version of Algorithm 1 is generalized to a fully-implicit extension in which pixel colors are computed as a weighted average not only of their (known) already filled neighbors in $A_{\epsilon,h}({\bf x}) \cap \Omega \backslash D^{(k)}$ and (unknown) neighbors within the same shell $\partial D^{(k)}_h$ , but of {\em all} neighbors in $A_{\epsilon,h}({\bf x})$?  Are their additional benefits in terms of artifact reduction and, if so, can the resulting linear system be solved efficiently enough to make this worthwhile?
\end{enumerate}

\section{Acknowledgements}
The authors would like to thank Vittoria Silvestri, James Norris, and Arieh Iserles for helpful conversations.  Thanks is also due to D\"om\"ot\"or P\'alv\"olgyi, whose idea of a symmetry-based argument eventually led to Appendix \ref{app:angSpect}, Proposition \ref{prop:eitherOr}.

\bibliographystyle{siamplain}
\bibliography{SIAMBib2}

\appendix

\section{Punctured sums} \label{app:punctured}

Here we provide further justification for our exclusion of the point ${\bf x}$ from the update formula \eqref{eqn:update} in Algorithm 1 as mentioned in Remark \ref{rem:punctured}.  As we mentioned there, this makes no difference to the direct form of Algorithm 1, because the subroutine FillRow only involves sums taken over $A_{\epsilon,h}({\bf x}) \cap (\Omega \backslash D^{(k)})$, which never contains ${\bf x}$.  However, the semi-implicit extension of Algorithm 1 expresses $u_h({\bf x})$ as a sum of $u_h({\bf y})$ over a set of points that might include ${\bf x}$.  The reason we deliberately exclude ${\bf x}$ is because, as the following proposition shows, if $w_{\epsilon}({\bf x},{\bf x})  <\infty$, it doesn't matter, but if $w_{\epsilon}({\bf x},{\bf x}) =\infty$, it wreaks havoc.  Moreover, the weights \eqref{eqn:weight} which we wish to study {\em do} have the property that $w_{\epsilon}({\bf x},{\bf x}) = \infty$.
\vskip 1mm
\begin{proposition} \label{prop:zeroOrHero}
Suppose ${\bf x} \in B$ for some finite set $B \subset \field{R}^2$, and suppose there exist non negative weights $w : B \times B \rightarrow [0,\infty]$, finite everywhere except possibly at $({\bf x},{\bf x})$.  Then if $w({\bf x},{\bf x}) <\infty$, we have
$$u_h({\bf x}) = \frac{\sum_{{\bf y} \in B} w({\bf x},{\bf y})u_h({\bf y})}{\sum_{{\bf y} \in B} w({\bf x},{\bf y})} = \frac{\sum_{{\bf y} \in B \backslash \{{\bf x} \}} w({\bf x},{\bf y})u_h({\bf y})}{\sum_{{\bf y} \in B \backslash \{{\bf x} \}} w({\bf x},{\bf y})}.$$
On the other hand, if $w({\bf x},{\bf x}) =\infty$, we have an indeterminate expression
$$u_h({\bf x}) =\frac{\infty}{\infty}.$$
\end{proposition}
\begin{proof}
The proof is an exercise in cancellation and left to the reader.
\qed
\end{proof}

\section{Properties of Ghost Pixels and Equivalent Weights} \label{app:ghostPixels}
In this appendix we prove the six properties of ghost pixels and equivalent weights listed in Section \ref{sec:fiction}.  These properties all follow from the definition of $u_h({\bf y})$ when ${\bf y}$ is a ghost pixel, namely
$$u_h({\bf y}) = \sum_{{\bf z} \in \mbox{Supp}(\{{\bf y}\})}\Lambda_{{\bf z},h}({\bf y})u_h({\bf z}),$$
where $\{ \Lambda_{{\bf z},h} \}_{{\bf z} \in \field{Z}^2_h}$ denote the basis functions of bilinear interpolation associated with the lattice $\field{Z}^2_h$, and where $\mbox{Supp}(A)$ denotes the set of real pixels needed to define a set $A$ of ghost pixels using bilinear interpolation.  Note that if ${\bf y} \in A$, then 
\begin{equation} \label{eqn:outside0}
\Lambda_{{\bf z},h}({\bf y})=0 \quad \mbox{ for all } {\bf z} \notin \mbox{Supp}(A).
\end{equation}
The following explicit formula for $\mbox{Supp}(\{(x,y)\})$ (which comes from the definition of bilinear interpolation) will occasionally be useful.  
\begin{equation} \label{eqn:explicit}
\mbox{Supp}(\{(x,y)\}) = \left\{ \left( \left \lfloor \frac{x}{h} \right \rfloor h,\left \lfloor \frac{y}{h} \right \rfloor h \right),
\left( \left \lceil \frac{x}{h} \right \rceil h,\left \lfloor \frac{y}{h} \right \rfloor h \right),
\left( \left \lfloor \frac{x}{h} \right \rfloor h,\left \lceil \frac{y}{h} \right \rceil h \right),
\left( \left \lceil \frac{x}{h} \right \rceil h,\left \lceil \frac{y}{h} \right \rceil h \right)\right\}.
\end{equation}
First we prove property one.

\vskip 1mm

\noindent {\em 1.  Explicit formula:}
$$\tilde{w}({\bf x},{\bf z}) = \sum_{{\bf y} \in A({\bf x})} \Lambda_{{\bf y},h}({\bf z})w({\bf x},{\bf y})$$
\begin{proof}
This follows straightforwardly from the definition of ghost pixels, the property \eqref{eqn:outside0}, and a few exchanges of finite sums.
\begin{eqnarray*}
u_h({\bf x})  &=& \sum_{{\bf y} \in A({\bf x})}w({\bf x},{\bf y})u_h({\bf y})  \\
                      &=& \sum_{{\bf y} \in A({\bf x})}w({\bf x},{\bf y})\sum_{{\bf z} \in \mbox{Supp}(\{{\bf y}\})}\Lambda_{{\bf z},h}({\bf y})u_h({\bf z})  \\
                      &=& \sum_{{\bf y} \in A({\bf x})}w({\bf x},{\bf y})\sum_{{\bf z} \in \mbox{Supp}(A({\bf x}))}\Lambda_{{\bf z},h}({\bf y})u_h({\bf z})  \\
                      &=& \sum_{{\bf z} \in \mbox{Supp}(A({\bf x}))} \underbrace{\left\{ \sum_{{\bf y} \in A({\bf x})}\Lambda_{{\bf z},h}({\bf y}) w({\bf x},{\bf y})\right\}}_{:=\tilde{w}({\bf x},{\bf z})}u_h({\bf z})  \\
\end{eqnarray*}
\qed
\end{proof}  
Next, instead of proving properties two and three, we prove a stronger result from which they both follow.

\vskip 1mm

\noindent {\em 1.(a)  Preservation of degree 1 polynomials:}  Suppose $f({\bf y})$ is a (scalar valued) polynomial of degree at most 1, that is $f({\bf y})=a+{\bf b} \cdot {\bf y}$ for some $a \in \field{R}$ and ${\bf b} \in \field{R}^2$.  Then
$$\sum_{{\bf y} \in A({\bf x})} w({\bf x},{\bf y})f({\bf y}) = \sum_{{\bf y} \in \mbox{Supp}(A({\bf x}))} \tilde{w}({\bf x},{\bf y})f({\bf y}).$$
\begin{proof}
This follows from the fact that the bilinear interpolant of a polynomial of degree at most 1 is just the polynomial back.  That is,
$$\sum_{{\bf z} \in \mbox{Supp}(\{{\bf y}\})}\Lambda_{{\bf z},h}({\bf y})f({\bf z}) = f({\bf y}).$$
This is obvious and we do not prove it.  We will also use \eqref{eqn:outside0} once.
\begin{eqnarray*}
 \sum_{{\bf y} \in A({\bf x})} w({\bf x},{\bf y})f({\bf y}) &=&  \sum_{{\bf y} \in A({\bf x})} w({\bf x},{\bf y})\left\{\sum_{{\bf z} \in \mbox{Supp}(\{{\bf y}\})}\Lambda_{{\bf z},h}({\bf y})f({\bf z})\right\}   \\
 &=&  \sum_{{\bf y} \in A({\bf x})} w({\bf x},{\bf y})\left\{\sum_{{\bf z} \in \mbox{Supp}(A(x))}\Lambda_{{\bf z},h}({\bf y})f({\bf z})\right\}   \\
    &=&  \sum_{{\bf z} \in \mbox{Supp}(A(x))} \left\{  \sum_{{\bf y} \in A({\bf x})} \Lambda_{{\bf z},h} ({\bf y}) w({\bf x},{\bf y})\right\}f({\bf z}) \\
     &=&  \sum_{{\bf z} \in \mbox{Supp}(A(x))} \tilde{w}({\bf x},{\bf z})f({\bf z}).
\end{eqnarray*}
\qed
\end{proof}

\vskip 1mm

\noindent {\em 2. Preservation of total mass.}
$$\sum_{{\bf y} \in A({\bf x})} w({\bf x},{\bf y}) = \sum_{{\bf y} \in \mbox{Supp}(A({\bf x}))} \tilde{w}({\bf x},{\bf y}).$$
\begin{proof}
Special case of 1.(a), $p({\bf y}) \equiv 1$.
\qed
\end{proof}

\vskip 1mm

\noindent {\em 3. Preservation of center of mass (or first moment).}
$$\sum_{{\bf y} \in A({\bf x})} w({\bf x},{\bf y}){\bf y} = \sum_{{\bf y} \in \mbox{Supp}(A({\bf x}))} \tilde{w}({\bf x},{\bf y}){\bf y}.$$
\begin{proof}
Apply 1.(a) to each component of $f({\bf y}) = {\bf y}$.
\qed
\end{proof}

\vskip 1mm

\noindent {\em 4. Inheritance of non-negativity:
$$\tilde{w}_{\epsilon}({\bf x},{\bf z}) \geq  0 \quad \mbox{ for all } {\bf z} \in \mbox{Supp}(A_{\epsilon,h}({\bf x})).$$
\begin{proof}
This is immediate from the non-negativity of the original weights $\{ w_{\epsilon} \}$, the non-negativity of the basis functions $\{ \Lambda_{{\bf y},h} \}$, and the explicit formula \eqref{eqn:explicitFormula}.
\qed
\end{proof}

\noindent {\em 5.  Inheritance of non-degeneracy condition \eqref{eqn:technical}.}
$$\sum_{{\bf y} \in \mbox{Supp}(A_{\epsilon,h}({\bf x}) \cap (\Omega \backslash D^{(k)}))}\tilde{w}_{\epsilon}({\bf x},{\bf y}) > 0.$$
\begin{proof}
Apply preservation of total mass to \eqref{eqn:technical}.
\qed
\end{proof}

\vskip 1mm

\noindent {\em 6.  Universal Support.}
For any $n \in \field{Z}$, we have
$$\mbox{Supp}(A_{\epsilon,h}({\bf x}) \cap \{ y \leq nh\} ) \subseteq D_h(B_{\epsilon,h}({\bf x})) \cap \{ y \leq nh\}  \subseteq B_{\epsilon+2h,h}({\bf x}) \cap \{ y \leq nh\} .$$
where $\{y \leq nh \} := \{ (x,y) \in \field{R}^2 : y \leq nh\}$, and where $D_h$ is the dilation operator defined in our section on notation.
\begin{proof}
Let $(x,y) \in A_{\epsilon,h}({\bf x}) \cap \{ y \leq nh\}$.  Then $x^2+y^2 \leq \epsilon^2$ and $y \leq nh$.  Hence $\left( \left \lfloor \frac{x}{h} \right \rfloor h,\left \lfloor \frac{y}{h} \right \rfloor h \right) \in B_{\epsilon,h}({\bf x}),$
and by \eqref{eqn:explicit} we have 
$$\mbox{Supp}(\{(x,y)\}) \subseteq \mathcal N \left( \left( \left \lfloor \frac{x}{h} \right \rfloor h,\left \lfloor \frac{y}{h} \right \rfloor h \right)\right) \subseteq D_h(B_{\epsilon,h}({\bf x})),$$
Where $\mathcal N({\bf x})$ denotes the nine-point neighborhood of ${\bf x} \in \field{Z}^2_h$ defined in the notation section.  But since we also know $y \leq nh$, we have
$\left \lceil \frac{y}{h} \right \rceil h \leq \left \lceil \frac{nh}{h} \right \rceil h \leq nh,$
and hence applying \eqref{eqn:explicit} again we conclude
$$\mbox{Supp}(\{(x,y)\}) \subseteq \{ y \leq nh\}$$
as well.  Since $(x,y) \in A_{\epsilon,h}({\bf x}) \cap \{ y \leq nh\}$ was arbitrary, the first inclusion follows.
For the second inclusion, note that every element of $D_h(B_{\epsilon,h}({\bf x})) \cap \{ y \leq nh\}$ is of the form $(x,y) = {\bf y} + \Delta {\bf y}$ where ${\bf y} \in B_{\epsilon,h}({\bf x})$ and $ \Delta {\bf  y} \in \{-h,0,h\} \times \{-h,0,h\}$.  Hence
$$\|(x,y)\|=\|{\bf y}+ \Delta {\bf y}\| \leq \|{\bf y}\|+\|\Delta{\bf  y}\| \leq \epsilon + \sqrt{2}h < \epsilon + 2h.$$
At the same time we have $y \leq nh$, so $(x,y) \in B_{\epsilon+2h,h}({\bf x}) \cap \{ y \leq nh\}$ as claimed.
\qed
\end{proof}

\section{Proof that our proposed extension of Algorithm 1 is equivalent to damped Jacobi or SOR} \label{app:secret}

Here we supply the proof, promised in Section \ref{sec:semiImplicit}, that the blue text in Algorithm 1 is actually an implementation of either damped Jacobi or SOR, depending on whether the ``FillBoundary'' subroutine is executed sequentially or in parallel.

\vskip 2mm

\begin{proposition} \label{prop:secretlyDampedJacobi}
Changing the boolean ``semiImplicit'' to true in Algorithm 1 is equivalent to solving \eqref{eqn:Linearboundary} with damped Jacobi if ``FillBoundary'' is executed in parallel, and to SOR if it is executed sequentially.  In either case, the relaxation parameter is given by
$$\omega^* = \left( 1- \frac{\tilde{w}_{\epsilon}({\bf x},{\bf x})}{W}\right).$$
\end{proposition}
\begin{proof}  First, note that the Jacobi iteration for solving the linear system \eqref{eqn:Linearboundary} may be written as 
$$\tilde{u}^{(n+1)}_h({\bf x}) = \frac{1}{1-\frac{\tilde{w}_{\epsilon}({\bf x},{\bf x})}{W}} \left( \sum_{{\bf y} \in S^{(k)}_{\epsilon,h}({\bf x}) \backslash \{ {\bf x} \}}\frac{\tilde{w}_{\epsilon}({\bf x},{\bf y})}{W}u^{(n)}_h({\bf y}) + {\bf f} \right)$$
with ${\bf f}$ defined as in \eqref{eqn:f}.  By comparison, repeated (parallel) executation of \\$\mbox{FillBoundary}(D^{(k+1)}_h, \partial D^{(k)}_h)$ is equivalent (after applying the definition of equivalent weights) to 
\begin{eqnarray*}
u^{(n+1)}_h({\bf x}) &=& \sum_{{\bf y} \in S^{(k)}_{\epsilon,h}({\bf x})} \frac{\tilde{w}_{\epsilon}({\bf x},{\bf y})}{W}u^{(n)}_h({\bf y}) + {\bf f} \\
 &=& \frac{\tilde{w}({\bf x},{\bf x})}{W}u^{(n)}_h({\bf x})+\sum_{{\bf y} \in S^{(k)}_{\epsilon,h}({\bf x}) \backslash \{ {\bf x} \}} \frac{\tilde{w}_{\epsilon}({\bf x},{\bf y})}{W}u^{(n)}_h({\bf y}) + {\bf f} \\
 &=& (1-\omega^*)u^{(n)}_h({\bf x})+\omega^* \tilde{u}^{(n+1)}_h({\bf x}),
\end{eqnarray*}
which is a definition of damped Jacobi.  The proof for SOR is analogous.
\qed
\end{proof}

\section{Proof of Proposition \ref{prop:rates}} \label{app:rates}

First we fix some notation.  Suppose we are on iteration $k$ of semi-implicit Guidefill and let ${\bf x} := x^{(k)}_0$ denote a fixed but arbitrary member of $\partial D_h^{(k)}$.  The pixel $x^{(k)}_0$ is coupled by \eqref{eqn:Linearboundary} to its immediate neighbors $x_j^{(k)}$ for $-r-2 \leq j \leq r+2$, and also depends on the pixels $x^{(k-\delta )}_{j}:= {\bf x}+ h (j,\delta) \in \partial D_h^{(k-\delta)}$ for $(j,\delta) \in b^-_{r+2}$ which appear in the right hand side of \eqref{eqn:Linearboundary} within the vector ${\bf f}$.

\begin{figure}
\centering
\includegraphics[width=.7\linewidth]{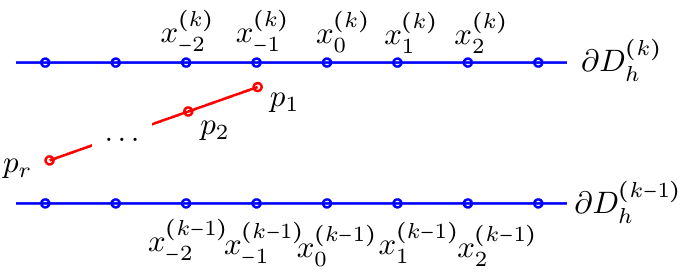}
\caption{{\bf Illustration of the position of the line $\ell^-_r$ relative to the current shell $\partial D^{(k)}_h$ and previous shell $\partial D^{(k-1)}_h$:}  Here we visualize the line $\ell^-_r := \{ -j {\bf g} \}_{j=1}^r \subseteq \tilde{b}^-_r$ when ${\bf g}=(\cos\theta,\sin\theta)$ with $0 < \theta < \theta_c = \arcsin(1/r)$.  In this case $\ell_r^-$ fits entirely into the space between $\partial D^{(k)}_h$ and $\partial D^{(k-1)}_h$.  For convenience, we enumerate $\ell^-_r$ as $\ell^-_r:=\{ p_j \}_{j=1}^r$ where $p_j = -j {\bf g}$.  We write the current pixel of interest ${\bf x}$ as $x_0^{(k)}$ for convenience, and its unknown neighbors in $\partial D^{(k)}_h$ as $x_j^{(k)}$ for $-r-2 \leq j \leq r+2$, while its already filled neighbors (we only show the ones in $\partial D^{(k-1)}_h$) are denoted by $x^{(k-\delta )}_{j}:= {\bf x}+ h (j,\delta) \in \partial D_h^{(k-\delta)}$ for $(j,\delta) \in b^-_{r+2}$.}
\label{fig:convergence}
\end{figure}

Next, note that since $\mu \rightarrow \infty$, the weights $w_r$ have vanished on all of $\tilde{b}^0_r$ except for the line of ghost pixels
$$\ell^-_r := \{ -j {\bf g} \}_{j=1}^r.$$
For convenience, we enumerate $\ell^-_r$ as $\ell^-_r:=\{ p_j \}_{j=1}^r$ where $p_j = -j {\bf g}$.  Each $p_j$ receives weight
$$w_j := \frac{1}{j}.$$
This situation is illustrated in Figure \ref{fig:convergence} for the case $0 < \theta < \theta_c$, where $\ell^-_r$ fits entirely into the space between $\partial D^{(k)}_h$ and $\partial D^{(k-1)}_h$.

To compute the entries of $\mathcal L$, we follow the idea of Section \ref{sec:fiction} and consider how the weight $w_j$ of each ghost pixel $p_j$ gets distributed to its real pixel neighbors.    For example, in Figure \ref{fig:convergence}, the weight $w_1$ of $p_1$ gets redistributed amongst the four pixels $x^{(k)}_0$, $x^{(k)}_{-1}$, $x^{(k-1)}_{0}$, and $x^{(k-1)}_{-1}$.

How exactly this weight gets redistributed is for the most part not something we need to know precisely.  For example, it is already clear from Figure \ref{fig:convergence} that if $0< \theta \leq \frac{\pi}{2}$, then none of the weight of any of the $p_j$ make it into any of $x^{(k)}_1$, $x^{(k)}_2$, $x^{(k)}_3\ldots$.  Similarly, if $\frac{\pi}{2} \leq \theta < \pi$, no weight makes it to any of $x^{(k)}_{-1}$, $x^{(k)}_{-2}$, $x^{(k)}_{-3}\ldots$.  This means that, given our assumed ordering of pixels within each layer $\partial D^{(k)}_h$, we already know that $\mathcal L$ is a lower triangular matrix.  Hence $L = \mathcal L$, $U = O$.  Therefore, $G_{\omega}$ takes on the simplified form 
$$G_{\omega} = (1-\omega)(I-\omega D^{-1}L)^{-1}.$$
We begin with $\|G_{\omega}\|_{\infty}$, the harder case.  In this case, defining
$$A := I-\omega D^{-1}L,$$
we have 
\begin{equation} \label{eqn:GwInTermsOfA}
\|G_{\omega}\|_{\infty} = |1-\omega|\|A^{-1}\|_{\infty}
\end{equation}
We know $\mathcal L = D-L$ is strictly diagonally dominant, so the following computation shows that $A$ is as well, provided $0<\omega \leq 1$:
$$\sum_{j \neq i} |A_{ij}| = \frac{|\omega|}{|D_{ii}|}\sum_{j \neq i} |L_{ij}| < |\omega| \leq 1 = |A_{ii}|.$$
Hence, the following classical bound due to Jim Varah \cite[Theorem 1]{Varah} applies:
$$\|A^{-1}\|_{\infty} \leq \frac{1}{\min_{i=1}^N \Delta_i(A)} \qquad \mbox{ where } \qquad \Delta_i(A):=\big||A_{ii}|-\sum_{j \neq i} | A_{ij}|\big|.$$
Since $A$ is a Toeplitz matrix with band width $r+2$ and at the same time a lower triangular matrix , we know that $\Delta_i(A)$ is the same for all $i \geq r+3$, but increases somewhat for $i \leq r+2$ as there are fewer off diagonal terms (due to our assumed Dirichlet boundary conditions).  In particular, the first row has {\em no} off diagonal terms, so we have $\Delta_1(A) = A_{11}=1$.  It follows that
$$\Delta_1(A) \geq \Delta_2(A) \geq \ldots \geq \Delta_{r+3}(A) = \Delta_{r+4}(A) = \ldots = \Delta_N(A).$$
Choosing row $N$ as a representative row for convenience gives
$$\|A^{-1}\| \leq \frac{1}{\Delta_{N}(A)}.$$
However, the identity
$$\|A^{-1}\|^{-1} = \inf_{{\bf x}} \frac{\|A{\bf x}\|}{\|{\bf x}\|}$$
(valid for any induced norm) means that in particular, for the vector $e$ containing $r+2$ zeros followed by $N-r-2$ ones, that is
$$ e = (\underbrace{0,\ldots,0}_{r+2},\underbrace{1,\ldots,1}_{N-r-2}),$$
we have 
$$\|A^{-1}\|^{-1}_{\infty} \leq \frac{\|Ae\|_{\infty}}{\|e\|_{\infty}}=\max_{i=1}^N \left|\sum_{j=1}^N A_{ij}e_j\right| =\max_{i=r+3}^N \left| |A_{ii}| - \sum_{j \neq i} |A_{ij}| \right|=\Delta_N(A).$$
where we have used the fact that for all $i$ we have $A_{ii} >0$ and $A_{ij} \leq 0$ for $j \neq i$.  The vector $e$ was chosen deliberately in order to avoid the first $r+2$ rows of $A$, which we have already said are different due to boundary conditions.  Hence
$$\|A^{-1}\|_{\infty} \geq \frac{1}{\Delta_{N}(A)}$$
as well, and having proven the bound in both directions we conclude
\begin{equation} \label{eqn:VarahEquality}
\|A^{-1}\|_{\infty} = \frac{1}{\Delta_{N}(A)}.
\end{equation}
\vskip 2mm
\begin{remark}
It appears that Varah's bound \cite[Theorem 1]{Varah}  should generalize to equality not only in our case, but to general strictly diagonally dominant Toeplitz matrices obeying $A_{ii}>0$ for all $i$ and $A_{ij} \leq 0$ whenever $j \neq i$, using a very similar argument.  However, this generalization does not appear in \cite{Varah} and we have been unable to find it in the literature.
\end{remark}
\vskip 2mm
The next step is to compute $\Delta_N(A)$.  To that end, note that by definition $A = I - \omega D^{-1}L$ obeys $A_{ii}=1$ for all $i$ and
$$A_{ij}=-\omega\frac{\left(\frac{\tilde{w}_r({\bf 0},(j-i){\bf e}_1)}{W}\right)}{\left(1-\frac{\tilde{w}_r({\bf 0},{\bf 0})}{W}\right)}=-\omega\frac{\tilde{w}_r({\bf 0},(j-i){\bf e}_1)}{W-\tilde{w}_r({\bf 0},{\bf 0})} \quad \mbox{ for } \quad \max(i-r-2, 1) \leq j < i.$$
by \eqref{eqn:Lsimplified}.  Here $W$ are the total weight and equivalent weights $\tilde{w}_r$ defined in Section \ref{sec:symmetry}.  Hence
$$\Delta_N(A) = \left| 1 - \omega \frac{\sum_{j=-r-2}^{-1} \tilde{w}_r({\bf 0},j{\bf e}_1)}{W-\tilde{w}_r({\bf 0},{\bf 0})} \right|=\left| 1 - \omega \left(\frac{\tilde{W}-\tilde{w}_{0,0}}{W-\tilde{w}_{0,0}}\right)\right|,$$
where $\tilde{W}$ and $\tilde{w}_{0,0}$ are defined as in \eqref{eqn:leakedMass}.  Combining the above with \eqref{eqn:GwInTermsOfA} and \eqref{eqn:VarahEquality} we finally obtain
$$\|G_{\omega}\|_{\infty} = \frac{|1-\omega|}{1-\omega \left(\frac{\tilde{W}-\tilde{w}_{0,0}}{W-\tilde{w}_{0,0}}\right)}$$
as claimed.  We leave deriving expressions for $W$, $\tilde{W}$, and $\tilde{w}_{0,0}$ until the end.  First we derive an expression for $\|J_{\omega}\|_{\infty}$ in terms of these three quantities.  Since $U=O$ we have
$$J_{\omega} = I - \omega D^{-1}\mathcal L = I - \omega D^{-1}(D-L) = (1-\omega)I + \omega D^{-1} L.$$
By definition we have
$$\|J_{\omega}\|_{\infty} := \max_{i=1}^N \sum_{j=1}^N | (J_{\omega})_{ij}|.$$
So long as $i \geq r+3$, this sum becomes 
$$\sum_{j=1}^N | (J_{\omega})_{ij}|=|1-\omega| +\omega \left(\frac{\tilde{W}-\tilde{w}_{0,0}}{W-\tilde{w}_{0,0}}\right).$$
If $i \leq r+2$, then this sum includes fewer terms and is potentially smaller.  Hence 
$$\|J_{\omega}\|_{\infty}=|1-\omega| +\omega \left(\frac{\tilde{W}-\tilde{w}_{0,0}}{W-\tilde{w}_{0,0}}\right).$$
Our remaining task is to derive the claimed expressions for $W$, $\tilde{W}$, and $\tilde{w}_{0,0}$.  The easiest is $W$.  By \eqref{eqn:2ndway} we have
$$W = \sum_{j=1}^r w_j = \sum_{j=1}^r \frac{1}{j}.$$
It is also not difficult to compute $\tilde{w}_{0,0}$, which represents fraction of the mass $w_1=1$ of the point $p_1$ that gets redistributed back to $x^{(k)}_0$ (see Figure \ref{fig:convergence}).  Since $p_1$ sits $h\sin\theta$ units below $\partial D^{(k)}_h$ and $h(1-\sin\theta)$ units above $\partial D^{(k-1)}_h$, and either $h\cos\theta$ units to the left $x^{(k)}_0$ and $h(1-\cos\theta)$ units right of $x^{(k)}_{-1}$ if $0 \leq \theta \leq \frac{\pi}{2}$ or $h|\cos\theta|$ units right of $x^{(k)}_0$ and $h(1-|\cos\theta|)$ units left of $x^{(k)}_{1}$ otherwise, it follows that 
$$\tilde{w}_{0,0} = (1-\sin\theta)(1-|\cos\theta|)w_1 = (1-\sin\theta)(1-|\cos\theta|).$$
For $\tilde{W}$, we split into cases.  If $0 \leq \theta \leq \theta_c$ or $\pi - \theta_c \leq \theta \leq \pi$, then $\ell^-_r$ fits entirely between $\partial D^{(k)}_h$ and $\partial D^{(k-1)}_h$, as in Figure \ref{fig:convergence}.  If $\theta_c < \theta < \pi - \theta_c$, then only $p_1$ up to $p_{j^*}$ fit (recall the definition of $j^*$ from the statement of the proposition).  As a result, in the first case every $p_j$ for $1 \leq j \leq r$ contribute mass to $\tilde{W}$.  In the second case, only the first $j^*$ contribute.  Each contributing $p_j$ is situated $h j \sin\theta$ units below $\partial D_h^{(k)}$ and $h(1-j\sin\theta)$ units above $\partial D^{(k-1)}_h$.  Hence each contributing $p_j$ contributes $(1-j\sin\theta)w_j$ towards $\tilde{W}$.  Hence, in the first case we have
$$\tilde{W} = \sum_{j=1}^r (1-j\sin\theta)\frac{1}{j}= \sum_{j=1}^r \frac{1}{j} - r\sin\theta,$$
while in the second we have
$$\tilde{W} = \sum_{j=1}^{j^*} \frac{1}{j} - j^* \sin\theta.$$
Our final claim on the expressions for $\|J_1\|_{\infty}$ and $\|G_1\|_{\infty}$ and the optimality of $\omega = 1$ is now a simple exercise and is left to the reader.
\qed

\section{Proof of Theorem \ref{thm:convergence2}} \label{app:marz}

We begin by breaking the error up into pieces as
$$\|u_h - u_{\mbox{m\"arz}}\|_p \leq \|u_h - u\|_p+\|u- u_{\mbox{m\"arz}}\|_p,$$
where $u$ is our proposed limit \eqref{eqn:weak}.  We will then separately prove
\begin{eqnarray*}
\|u_h - u\|_p & \leq &  K_1 \cdot (rh)^{(\frac{s'}{2}+\frac{1}{2p}) \wedge \frac{s}{2} \wedge 1} \\
\|u- u_{\mbox{m\"arz}}\|_p & \leq & K_2 r^{-q\left\{s \wedge \left( s' +\frac{1}{p} \right) \wedge 1\right\}}
\end{eqnarray*}
for constants $K_1>0$, $K_2>0$ with the claimed properties.  Taken together, these two inequalities prove the first claim \eqref{eqn:marzbound}.  For the second claim \eqref{eqn:chabuduo}, just the second is enough.

\vskip 1mm

\noindent { \bf Step 1: Bounding $\|u_h - u\|_p$:}  This step is straightforward.  By Theorem \ref{thm:convergence} we know 
$$\|u_h - u\|_p  \leq   K \cdot (rh)^{(\frac{s'}{2}+\frac{1}{2p}) \wedge \frac{s}{2} \wedge 1}$$
where $K(\theta^*_r,\frac{r}{{\bf g}^*_r \cdot e_2})>0$ depends on $u_0$, $\mathcal U$, and $\{U_i\}_{i=1}^M$ (which are fixed) and continuously on $\frac{r}{{\bf g}^*_r \cdot e_2}$ and $\theta^*_r$ (which are not fixed in this case).  By assumption, $\frac{{\bf g}^*_r}{r} \rightarrow {\bf g}^*$ as $r \rightarrow \infty$.  Moreover, by continuity we know there is a $\eta> 0$ s.t. $|\frac{{\bf g}^*_r \cdot e_2}{r} - {\bf g} \cdot e_2| < \eta$ and $|\theta^*_r - \theta^*| < \eta$ implies 
$$K\left(\theta^*_r,\frac{r}{{\bf g}^*_r \cdot e_2}\right) \leq K\left(\theta^*, \frac{1}{{\bf g}^* \cdot e_2}\right)+1.$$
We simply apply our assumption $\frac{{\bf g}^*_r}{r} \rightarrow {\bf g}^*$ to find an $R$ such that $r > R$ implies that the bounds involving $\eta$ are satisfied, and then define
$$K_1 = \max\left\{ \max_{r=1}^R\left\{ K\left(\theta^*_r,\frac{r}{{\bf g}^*_r \cdot e_2}\right)\right\}, K\left(\theta^*,\frac{1}{{\bf g}^* \cdot e_2}\right)+1\right\}.$$
All the claimed dependency properties of $K_1$ follow from the analogous properties of $K$.

\vskip 1mm

\noindent {\bf Step 2: Bounding $\|u- u_{\mbox{m\"arz}}\|_p$:}  This step is a little more work, but also straightforward.  We begin by relating $|u-u_{\mbox{m\"arz}}|$ to $|\cot\theta-\cot\theta_r|$, which in turn is related to $\|\frac{{\bf g}^*_r}{r} - {\bf g}^*\|$.  Indeed, it is straightforward to show
$$|\cot\theta - \cot\theta_r| \leq \frac{r}{|{\bf g}_r^* \cdot e_2 |}\left\|\frac{{\bf g}^*_r}{r} - {\bf g}^*\right\| \leq C r^{-q},$$
where
$$C = \max\left\{ \max_{r=1}^{R'} \left\{ \frac{r}{{\bf g}^*_r \cdot e_2} \right\}, \frac{1}{{\bf g}^* \cdot e_2}+1\right\}\cdot D,$$
and where $D$ is the hidden constant in our assumption of the $O(r^{-q})$ convergence of $\frac{{\bf g}^*_r}{r}$ to ${\bf g}^*$, and where $r > R'$ implies $\frac{r}{{\bf g}^*_r \cdot e_2} < \frac{1}{{\bf g}^* \cdot e_2}+1$ (similarly to before, $R'$ exists since we assume $\frac{{\bf g}^*_r}{r} \rightarrow {\bf g}^*$).  Note that $\frac{1}{{\bf g}^* \cdot e_2} \rightarrow \infty$ as $\theta^* \rightarrow 0$ or $\theta^* \rightarrow \pi$, and so $C$ does as well.  Next, first note that 
$$|u({\bf x})-u_{\mbox{m\"arz}}({\bf x})|=|u_0(\Pi_{\theta^*_r}({\bf x}))-u_0(\Pi_{\theta^*}({\bf x}))| \mbox{ for all } {\bf x} \in D.$$ 
where $\Pi_{\theta^*_r}$, $\Pi_{\theta^*}$ are the usual transport operators defined by \eqref{eqn:weak}.  Then note that we have the bound
$$|\Pi_{\theta^*_r}({\bf x})-\Pi_{\theta^*}({\bf x})| \leq |\cot\theta^*_r -\cot\theta^*| \leq Cr^{-q}.$$
This estimate will play the same role as our estimate on the moments of ${\bf X}_{\tau}$ in the proof of Theorem \ref{thm:convergence}.

Next, similarly to the proof of Theorem \ref{thm:convergence}, we divide $D$ into bands.  This time, however, we have two separate sets of bands $\{ B_i \}_{i=1}^M$ and $\{ \hat{B}_i \}_{i=1}^M$, each of which individually partition $D$.  The bands $B_i$ are the same as in the proof of theorem \ref{thm:convergence}, while the bands $\hat{B}_i$ are the same as $B_i$ except that the transport operator $\Pi_{\theta^*_r}$ has been replaced with $\Pi_{\theta^*}$.  That is, 
$$B_i := \Pi_{\theta^*_r}^{-1}((x_i,x_{i+1}]) \quad \mbox{ and } \quad \hat{B}_i := \Pi_{\theta^*}^{-1}((x_i,x_{i+1}]).$$
As usual we define $B_{i,h}:=B_i \cap \Omega_h$ and $\hat{B}_{i,h}:=\hat{B}_i \cap \Omega_h$.  It will be convenient in a moment to have an estimate of $\|\hat{B}_{i,h} \backslash B_{i,h}\|_p$ (for $\|\hat{B}_{i,h} \cap B_{i,h}\|_p$, the trivial bound $\|\hat{B}_{i,h} \cap B_{i,h}\|_p \leq 1$ suffices).  To that end, note that the set difference $\hat{B}_i \backslash B_i$ is a triangle with a base of length $|\cot\theta^* - \cot\theta^*_r|$ and a height of $1$.  At the same time $\|\hat{B}_{i,h} \backslash B_{i,h}\|^p_p = |\hat{B}_{i,h} \backslash B_{i,h}|h^2$ can be thought of as a Riemann sum approximating $\mbox{Area}(\hat{B}_i \backslash B_i)$.  It is elementary to show that this Riemann sum overestimates the area by at most $2h$.  Hence
$$\|\hat{B}_{i,h} \backslash B_{i,h}\|^p_p \leq \mbox{Area}(\hat{B}_i \backslash B_i)+2h = \frac{1}{2}|\cot\theta^* - \cot\theta^*_r|+2h \leq \left(\frac{C}{2}+2 \right)r^{-q},$$
and therefore 
$$\|\hat{B}_{i,h} \backslash B_{i,h}\|_p \leq \left(\frac{C}{2}+2 \right)^{\frac{1}{p}}r^{-\frac{q}{p}} \leq \left(\frac{C}{2}+2 \right)r^{-\frac{q}{p}},$$
since we have assumed $h \leq r^{-q}$.  Proceeding now as in the proof of Theorem \ref{thm:convergence}, we note that
\begin{eqnarray*}
\|u- u_{\mbox{m\"arz}}\|_p & \leq & \sum_{i=1}^M \Big\{ \|(u- u_{\mbox{m\"arz}})1_{\hat{B}_{i,h} \backslash B_{i,h}}\|_p+\|(u- u_{\mbox{m\"arz}})1_{\hat{B}_{i,h} \cap B_{i,h}}\|_p \Big\} \\
& \leq & C_{u_0,\mathcal U,\{U_i\}}\sum_{i=1}^M \Big\{ |\cot\theta^* - \cot\theta^*_r|^{s' \wedge 1}\|1_{\hat{B}_{i,h} \backslash B_{i,h}}\|_p\\
&+&|\cot\theta^* - \cot\theta^*_r|^{s \wedge 1}\|1_{\hat{B}_{i,h} \cap B_{i,h}}\|_p \Big\} \\
& \leq & C_{u_0,\mathcal U,\{U_i\}} \sum_{i=1}^M \Big\{ Cr^{-q (s' \wedge 1)} \cdot \|1_{\hat{B}_{i,h} \backslash B_{i,h}}\|_p+Cr^{-q(s \wedge 1)} \cdot \|1_{\hat{B}_{i,h} \cap B_{i,h}}\|_p \Big\} \\
& \leq & C_{u_0,\mathcal U,\{U_i\}} \sum_{i=1}^M \Big\{ C\left(\frac{C}{2}+1\right)r^{-q\left(s'+\frac{1}{p}\right) \wedge \left(1+\frac{1}{p}\right)}+Cr^{-q(s \wedge 1)} \cdot 1 \Big\} \\
& \leq & MC_{u_0,\mathcal U,\{U_i\}}C\left(\frac{C}{2}+1\right)\left( r^{-q\left\{{\left(s'+\frac{1}{p}\right) \wedge \left(1+\frac{1}{p}\right)}\right\}}+r^{-q(s \wedge 1)}\right)\\
& \leq &2MC_{u_0,\mathcal U,\{U_i\}}C\left(\frac{C}{2}+1\right)r^{-q\left\{{\left(s'+\frac{1}{p}\right) \wedge s \wedge 1}\right\}}\\
& := & K_2r^{-q\left\{\left(s'+\frac{1}{p}\right) \wedge s \wedge 1\right\}}.\\
\end{eqnarray*}
That $K_2$ also satisfies the claimed properties follows from its relationship to $C_{u_0,\mathcal U,\{U_i\}}$ and $C$.
\qed

\section{Additional details on coherence transport and the angular spectrum} \label{app:angSpect}

In Section \ref{sec:kink3main} we related the limiting transport direction ${\bf g}^*_r = \lim_{\mu \rightarrow \infty} {\bf g}^*_{\mu,r}$ of coherence transport to the angular spectrum $\Theta(b^-_r)$ of $b^-_r$ defined by \eqref{eqn:spectrumbr}.  More precisely, first we connected ${\bf g}^*_r$ to the set of minimizers $\Psi$ within $b^-_r$ of the orthogonal distance to the line $L_{\bf g} = \{ \lambda {\bf g} : \lambda \in \field{R} \}$, where ${\bf g}$ is the guidance direction of coherence transport.  Then we claimed that $\Psi$ and $\Theta(b^-_r)$ are related.  Now is the time to prove that claim.  We will do this in Proposition \ref{prop:correspondence} not just for $b^-_r$, but for a general finite subset $A \subseteq \field{Z}^2 \cap \{ y \leq -1\}$.  To do this, however, first we generalize the concept of angular spectrum to a general subset $A \subseteq \field{Z}^2$.

\vskip 2mm

\begin{definition}
Given $A \subseteq \field{Z}^2$ we define the angular spectrum of $A$ by
\begin{equation} \label{eqn:spectrum}
\Theta(A) = \{ \theta \in [0,\pi) : \theta=\theta({\bf y}) \mbox{ for some } {\bf y} \in A \backslash \{{\bf 0}\} \}
\end{equation}
If $A$ is finite it follows that $\Theta(A)$ is as well, and we write
$$\Theta(A)=\{0 \leq \theta_1 < \theta_2 < \ldots <\theta_n < \pi\}.$$
See Figure \ref{fig:relatedTriangles}(b) for an illustration of $\Theta(A)$ in the case $A = b^-_r$.  
\end{definition}

Once again we have defined $\Theta(A)$ modulo $\pi$ to reflect the fact that ${\bf g}^*_r$ and $-{\bf g}^*_r$ define the same transport equation.  The characterization of $\Theta(A)$ is of interest in and of itself and has been studied for $A = b_r$ by the likes of Erd\"os \cite{Erdos} and many others, see for example \cite{Cilleruelo1993198} (they, however, do not define it modulo $\pi$).

\vskip 2mm

\begin{remark} \label{rem:why}

 The point of generalizing the concept of angular spectrum and generalizing Proposition \ref{prop:correspondence} from $b^-_r$ to a general $A \subseteq \field{Z}^2 \cap \{ y \leq -1\}$ is so that we can show (Corollary \ref{corollary:generalizedCT}) that our kinking results for coherence transport from Section \ref{sec:kink3main} continue to hold, essentially unchanged, if coherence transport is modified to sum over a discrete square, for example, rather than a discrete ball.
\end{remark}

\vskip 2mm

\begin{proposition} \label{prop:correspondence}
Let $A \subseteq \field{Z}^2 \cap \{ y \leq -1\}$ obey $|A|<\infty$, and let $\Theta(A) = \{\theta_1,\theta_2,\ldots,\theta_n\}$ denote the angular spectrum of $A$, and assume $n=|\Theta(A)|\geq 2$ in order to avoid degenerate cases (that is, the elements of $A$ do not all sit on a single line through the origin).  Let ${\bf g}_{\theta}=(\cos\theta,\sin\theta)$ and denote by $\Psi_{\theta}$ the set of minimizers of $|{\bf g}_{\theta}^{\perp} \cdot {\bf y}|$ over ${\bf y} \in A$ (that is, the point(s) in $A$ minimizing the orthogonal distance to the line $L_{{\bf g}_{\theta}}$.  Given ${\bf y} \in A$, we say that ${\bf y}$ is a singleton minimizer if there is some $\theta \in [0,\pi)$ for which $\Psi_{\theta} = \{ {\bf y} \}$.  Let $Y:=\{ {\bf y}_1,{\bf y}_2,\ldots,{\bf y}_{n'}\}$ denote the set of all singleton minimizers, ordered so that $\theta({\bf y}_i) \leq \theta({\bf y}_{i+1})$ for all $i=1,\ldots n'-1$.  Then $n'=n$,
$$\Theta(A) = \{\theta({\bf y}_1),\theta({\bf y}_2),\ldots,\theta({\bf y}_n)\},$$
and moreover $\theta_i = \theta({\bf y}_i)$ for all $i=1,\ldots,n$.  Finally, each singleton minimizer ${\bf y}_i$ is the {\em shortest} vector in $A$ such that $\theta({\bf y})=\theta_i$, that is for every ${\bf y} \in A$ such that $\theta({\bf y})=\theta_i$, we have $\|{\bf y}\| \geq \|{\bf y}_i\|$.
\end{proposition}
\begin{proof}
Let $\theta_i \in \Theta(A)$.  Our main objective is to show that $\theta_i = \theta({\bf y}_i)$.  To achieve that, it suffices to show that the {\em sets} $\Theta(A)$ and $\{ \theta({\bf y}_1),\theta({\bf y_2}),\ldots,\theta({\bf y}_{n'})\}$ are equal, since from here it follows that $n'=n$, and then the desired identity follows from the ordering property $\theta({\bf y}_i) \leq \theta({\bf y}_{i+1})$ for all $i=1,\ldots,n-1$.  Our secondary objective, to show that ${\bf y}_i$ is the shortest vector in $A$ obeying $\theta({\bf y})=\theta_i$ is something that will be proved along the way.

For the first step, the inclusion $\Theta(A) \supseteq \{\theta({\bf y}_1),\theta({\bf y}_2),\ldots,\theta({\bf y}_{n'})\}$ is obvious and follows from the definition of $\Theta(A)$.  Hence it suffices to prove 
\begin{equation} \label{eqn:desiredInclusion}
\Theta(A) \subseteq \{\theta({\bf y}_1),\theta({\bf y}_2),\ldots,\theta({\bf y}_{n'})\}.
\end{equation}
Still fixing $\theta_i \in \Theta(A)$, by definition we have $\theta_i = \theta({\bf y})$ for some ${\bf y} \in A$.  In fact, we have $\theta_i = \theta({\bf y})$ for all ${\bf y} \in \Psi_{\theta_i}$, which in this case is a set of vectors that are all scalar multiples of ${\bf g}_{\theta_i}$ and hence all of which obey $|{\bf g}^{\perp}_{\theta_i} \cdot {\bf y}|=0$.  Define the functions $\Delta(\theta)$ and $\delta(\theta)$ by
$$\delta(\theta) := \max_{{\bf y} \in \Psi_{\theta}} |{\bf g}^{\perp}_{\theta} \cdot {\bf y}|$$
$$\Delta(\theta) := \begin{cases} 
\min_{{\bf y} \in A \backslash \Psi_{\theta}} |{\bf g}^{\perp}_{\theta} \cdot {\bf y}|. & \mbox{ if } A \backslash \Psi_{\theta} \neq \emptyset \\
\delta(\theta) & \mbox{ otherwise. }
\end{cases}$$
Then $\delta(\theta)$ and $\Delta(\theta)$ each depend continuously on $\theta$.  Moreover, it is straightforward to show that $\Delta(\theta_i) > \delta(\theta_i)=0$, since we have assumed $|\Theta(A)|\geq 2$ (this condition could only ever be violated if all elements of $A$ were scalar multiples of one another).  By continuity, it follows that for $|\theta-\theta_i|$ sufficiently small we have $\delta(\theta) < \Delta(\theta)$, which means that $\Psi_{\theta} \subseteq \Psi_{\theta_i}$.  But for $|\theta-\theta_i| \leq \frac{\pi}{2}$ and for ${\bf y} \in \Psi_{\theta_i}$, we have the explicit formula
$$|{\bf g}^{\perp}_{\theta} \cdot {\bf y}|=\|{\bf y}\|\sin|\theta-\theta_i|.$$
This is obviously minimized by whichever ${\bf y}^* \in \Psi_{\theta_i}$ is shortest - i.e. $\|{\bf y}^*\| \leq \|{\bf y}\|$ for all ${\bf y} \in \Psi_{\theta_i}$.  Moreover, since $A$ is contained in the lower half plane we know this minimizer is unique.  Hence $\Psi_{\theta} = \{ {\bf y}^*\}$, which makes ${\bf y}^*$ a singleton minimizer.  Since $\theta_i = \theta({\bf y}^*)$, this proves the desired inclusion \eqref{eqn:desiredInclusion}, and we have already proven our secondary claim on the length of ${\bf y}^*$ being minimal.
\qed
\end{proof}

\vskip 2mm

\begin{figure}
\centering
\includegraphics[width=.6\linewidth]{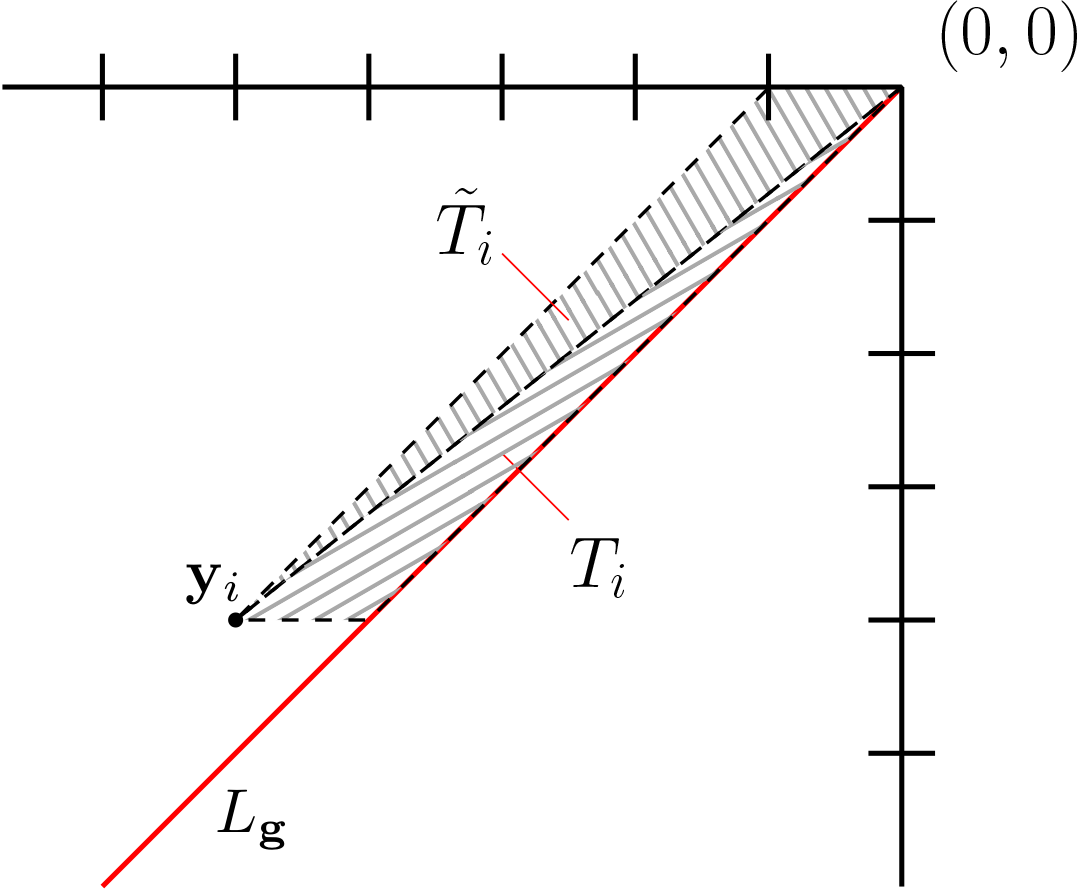}
\caption{{\bf Proving that the point casting the shallowest angle on $L_{\bf g}$ from above is also the point minimizing the orthogonal distance from above:}  Given ${\bf y}_i \in A \subseteq \field{Z}^2 \cap \{ y \leq -1\}$ and line $L_{{\bf g}}=\{\lambda {\bf g} : \lambda \in \field{R}\}$, ${\bf g}=(\cos\theta,\sin\theta)$ passing through the origin, we define the (open) triangles $T_i$, $\tilde{T}_i$ to be the unique pair of open triangles with one side parallel to $L_{{\bf g}}$, another side horizontal, and a third side equal to the ray from the origin to ${\bf y}_i$.  A symmetry-based argument in Proposition \ref{prop:eitherOr} shows that, under modest hypotheses on $A$, the triangle $T_i$ contains a lattice point (element of $\field{Z}^2$) if and only if $\tilde{T}_i$ does.}
\label{fig:relatedTriangles}
\end{figure}

Our next claim was a formula for $\Psi$ valid when $\theta_i < \theta < \theta_{i+1}$ for two consecutive members $\theta_i,\theta_{i+1} \in \Theta(A)$, when $0:=\theta_{0,1}<\theta<\theta_1$, and when $\theta_n < \theta < \theta_{n,n+1}:=\pi$.  Proposition \ref{prop:eitherOr} derives this formula, under the assumption that $A$ can be described a union of discrete rectangles on or below the line $y=-1$ and straddling the line $x=0$.  This includes the case $A = b^-_r$, but also covers a number of other cases, as mentioned in Remark \ref{rem:why}.  Credit for this proposition goes to D\"om\"ot\"or P\'alv\"olgyi, who had the critical idea of using a symmetry based argument \cite{MathOverflow}.  

\vskip 2mm

\begin{proposition} \label{prop:eitherOr}
Let $A \subseteq \field{Z}^2 \cap \{ y \leq -1 \}$ be a finite union of discrete rectangles of the form
$$A = \bigcup_{k=1}^K [a_k, b_k ] \times [c_k, d_k] \cap \field{Z}^2$$
where $-\infty < a_k \leq 0 \leq b_k < \infty$, $-\infty < c_k \leq d_k \leq -1$ for all $k$.  Let 
$$\Theta(A):=\{\theta_1,\theta_2, \ldots,\theta_n\}$$ 
denote the angular spectrum of $A$, let ${\bf g} =  (\cos\theta,\sin\theta)$, and let $Y = \{{\bf y}_1,{\bf y}_2,\ldots,{\bf y}_n\}$ be the set of singleton minimizers defined in Proposition \ref{prop:correspondence} of $|{\bf g}^{\perp} \cdot {\bf y}|$ over $A$ as $\theta$ varies from $0$ to $\pi$.  For each $1 \leq i \leq n-1$, define the transition angle $\theta_{i,i+1} \in (\theta_i,\theta_{i+1})$ by
$$\theta_{i,i+1} = \theta({\bf y}_i+{\bf y}_{i+1}).$$
Define also $\theta_{0,1}:=0$ and $\theta_{n,n+1}=\pi$ for convenience.  Then
$$\Psi=\begin{cases}
\{{\bf y}_i\} & \mbox{ if } \theta_i < \theta < \theta_{i,i+1} \quad \mbox{ for some } i=1,\ldots,n \\
\{{\bf y}_i,{\bf y}_{i+1}\} & \mbox{ if } \theta=\theta_{i,i+1} \quad \mbox{ for some } i=1,\ldots,n-1 \\
\{{\bf y}_{i+1}\} & \mbox{ if } \theta_{i,i+1} < \theta < \theta_{i+1} \quad \mbox{ for some } i=0,\ldots,n-1 \\
\end{cases}$$
\end{proposition}
\begin{proof}
The bulk of the work is to prove that if $\theta_i < \theta < \theta_{i+1}$, then 
$$\Psi_{\theta} := \operatorname{argmin}_{{\bf y} \in A} |{\bf g}^{\perp} \cdot {\bf y}| \subseteq \{ {\bf y}_i, {\bf y}_{i+1}\}.$$
Once this is established, since we evidently have $\Psi_{\theta}=\{ {\bf y}_i \}$, $\Psi_{\theta}=\{{\bf y}_{i+1}\}$ for $\theta$ sufficiently close to $\theta_i$ and $\theta_{i+1}$ respectively, it follows that that
$$\Psi_{\theta} = \begin{cases} \{ {\bf y}_i \}& \mbox{ if } \theta < \theta_c \\  \{ {\bf y}_i,{\bf y}_{i+1} \} & \mbox{ if } \theta = \theta_c \\  \{ {\bf y}_{i+1} \} & \mbox{ if } \theta > \theta_c.\end{cases}$$
where $\theta_c$ is defined by $|{\bf g}^{\perp} \cdot {\bf y}_i|=|{\bf g}^{\perp} \cdot {\bf y}_{i+1}|$.  One can readily show this is equivalent to $\theta_c = \theta({\bf y}_i+{\bf y}_{i+1})$.

To prove that $\Psi_{\theta} := \operatorname{argmin}_{{\bf y} \in A} |{\bf g}^{\perp} \cdot {\bf y}| \subseteq \{ {\bf y}_i, {\bf y}_{i+1}\}$ as claimed, 
consider the open triangles $T_i$, $\tilde{T}_i$ defined in terms of ${\bf y}_i$ as shown in Figure \ref{fig:relatedTriangles}, as well as open triangles $T_{i+1}$, $\tilde{T}_{i+1}$ defined in the same way in terms of ${\bf y}_{i+1}$.  The triangles $T_i$ and $T_{i+1}$ each have empty intersection with $A$, as ${\bf y}_i$ and ${\bf y}_{i+1}$ are the elements of $A$ that cast the shallowest angles on $L_{{\bf g}}$ from above and below.  To prove $\Psi_{\theta} \subseteq \{ {\bf y}_i, {\bf y}_{i+1}\}$, we need to show that ${\bf y}_i$ and ${\bf y}_{i+1}$ are also the two {\em closest} points in $A$ to $L_{{\bf g}}$.  This amounts to proving that the triangles $\tilde{T}_i$ and $\tilde{T}_{i+1}$ have empty intersection with $A$ as well.

To show this, first note that our assumptions on $A$ imply that the intersection of each of our four triangles with $A$ is equal to their intersection with $\field{Z}^2$ as a whole (because $A$ has no ``holes'').  Second, note that the map
$${\bf F}({\bf x}) = {\bf y}_i - {\bf x}$$
is a bijection of the plane taking $T_i$ to $\tilde{T}_i$ such that ${\bf F}(\field{Z}^2) = \field{Z}^2$.  Hence $\tilde{T}_i$ contains a lattice point if and only if $T_i$ does.  But 
$$T_i \cap \field{Z}^2 = T_i \cap A = \emptyset$$
by assumption, and so 
$$\tilde{T}_i \cap \field{Z}^2 = \tilde{T}_i \cap A = \emptyset$$
as well.  This proves the claim for $\tilde{T}_i$ and the proof for $\tilde{T}_{i+1}$ is analogous.  The remaining cases $0:=\theta_{0,1}<\theta<\theta_1$ and $\theta_n < \theta < \theta_{n,n+1}:=\pi$ are straightforward and left as an exercise.
\qed
\end{proof}

\vskip 2mm

Proposition \ref{prop:eitherOr} has a couple of straightforward corollaries.  The first is our claim from \ref{sec:kink3main} that $\Psi$ is a singleton set for all but finitely many $\theta$.  This is obvious from the statement of the proposition (which gives an expression for $\Psi$ for all but finitely many $\theta$) and requires no proof.  The second corollary generalizes our formula for $\theta^*_r$ from Section \ref{sec:kink3main}, and uses the following observation, which we also used in Section \ref{sec:kink3main} and owe a proof of.

\begin{observation} \label{obs:unitVec}
Let ${\bf v}$, ${\bf w}$ be unit vectors in $S^1$.  Then
$$\theta({\bf v}+{\bf w})=\frac{\theta({\bf v})+\theta({\bf w})}{2}.$$
\end{observation}
\begin{proof}
This is simplest if we work in complex arithmetic, that is, we write ${\bf v}=e^{i\psi}$ and ${\bf w}=e^{i\phi}$ for some $\psi,\phi \in [0,2\pi)$.  However, by symmetry we may assume ${\bf v}=1$ (otherwise rotate the plane).  Hence, it suffices to prove
$$\frac{1+e^{i\phi}}{|1+e^{i\phi}|}=e^{i\frac{\phi}{2}}.$$
But this follows from the following simple manipulation:
$$1+e^{i \phi} = e^{i\frac{\phi}{2}}(e^{-i\frac{\phi}{2}}+e^{i\frac{\phi}{2}})=2\cos\left(\frac{\phi}{2}\right)e^{i\frac{\phi}{2}}.$$
\qed
\end{proof}

\vskip 2mm

The following corollary shows that coherence transport-like algorithms, which use the same weights but replace $B_{\epsilon,h}({\bf x})$ with a different set of pixels (a finite union of discrete rectangles) exhibit similar kinking behaviour in the limit $\mu \rightarrow \infty$.  

\vskip 2mm

\begin{corollary} \label{corollary:generalizedCT}
Suppose we inpaint $D = (0,1]^2$ using Algorithm 1 with boundary data $u_0 : \mathcal U \rightarrow \field{R}^d$ and suppose the symmetry assumptions of Section \ref{sec:symmetry} hold as usual.  Assume our stencil $a^*_r$ is of the form
$$a^*_r = \bigcup_{k=1}^K [a_k, b_k ] \times [c_k, d_k] \cap \field{Z}^2$$
where $-\infty < a_k \leq 0 \leq b_k < \infty$, $c_k \leq d_k \leq -1$ for all $k$.  Let
$$\Theta(a^*_r):=\{\theta_1,\theta_2, \ldots,\theta_n\}$$ 
denote the angular spectrum of $a^*_r $, let ${\bf g} =  (\cos\theta,\sin\theta)$, assume we use as stencil weights the weights of M\"arz \eqref{eqn:weight}, and denote by ${\bf g}^*_{\mu,r}$ the limiting transport direction from Theorem \ref{thm:convergence}.  Let ${\bf g}^*_r = \lim_{\mu \rightarrow \infty} {\bf g}^*_{\mu,r}$ and define $\theta^*_r := \theta({\bf g}^*_r)$.  Then
\begin{equation} \label{eqn:stepFunc2}
\theta^*_r = \begin{cases} \frac{\pi}{2} & \mbox{ if } \theta = 0 \\ \theta_i & \mbox{ if } \theta_{i-1,i} < \theta < \theta_{i,i+1} \mbox{ for some } i = 1,\ldots n \\ \frac{\theta_i+\theta_{i+1}}{2} & \mbox{ if } \theta = \theta_{i,i+1} \mbox{ for some } i = 1,\ldots n \end{cases}
\end{equation}
\end{corollary}
\begin{proof}
This follows from Proposition \ref{prop:eitherOr} and observation \ref{obs:unitVec} in exactly the same way as when showed this for the special case $a^*_r = b^-_r$ in Section \ref{sec:kink3main}.
\qed
\end{proof}

\vskip 2mm

We conclude this appendix with a remark on a practical algorithm for computing the angular spectrum $\Theta(A)$ given $A \subseteq \field{Z}^2$ satisfying the hypotheses of Proposition \ref{prop:correspondence}.  This algorithm was used to generate the theoretical limiting curves for $\theta^*_r$ for coherence transport in Section \ref{sec:kink3main}.

\begin{remark} \label{rem:practicalAlg}
Given $A \subseteq \field{Z}^2$ satisfying the hypotheses of Proposition \ref{prop:correspondence}, a simple algorithm for computing the angular spectrum $\Theta(A)$ and singleton minimizers $Y$ is as follows:
\begin{enumerate}
\item Starting with $Y^*=\emptyset$, go through each ${\bf y} \in A$ and find the unique ${\bf y}' \in A$ such that $\theta({\bf y})=\theta({\bf y}')$ and ${\bf y}'$ is of minimal length.  If ${\bf y}'$ is not already in $Y^*$, add it.
\item For each ${\bf y} \in Y^*$, compute $\theta({\bf y})$.  Sort $Y^*$ according to $\theta({\bf y})$.  The sorted list $Y^*$ is now equal to $Y$, and the sorted list of angles is $\Theta(A)$.
\end{enumerate}

\end{remark}

\section{List of Examples} \label{app:exampleList}  In this appendix we describe in detail the examples used in some of the numerical experiments in Section \ref{sec:numerics}.

\vskip 1mm

\noindent{\bf Boundary Data.} Our first objective is to build boundary data satisfying the assumptions listed in Section \ref{sec:continuumLimit1} and illustrated in Figure \ref{fig:setup}.   To that end, we let $\mathcal U = (0,1] \times (-\delta,0]$ as usual and define the sets $\{U_i\}_{i=1}^2$ as
$$U_1 := [0,0.25) \times (-\delta,0] \cup  (0.75,1] \times (-\delta,0] \quad  U_2 := (0.25,0.75) \times (-\delta,0]$$
for some $\delta > (r+2)h$. Next, to test that our bounds in Theorem \ref{thm:convergence} are sharp, what we would like to do is build a family $\mathcal F$ of boundary data $u_0$, parameterized by $0<s$, $0\leq s' \leq s$, such that $u_0 \in W^{\nu,\infty}(U_i)$ for $i=1,2$ iff $\nu \in [0,s]$ and $u_0 \in W^{\nu',\infty}(\mathcal U)$ iff $\nu' \in [0,s']$.  We will almost, but not quite be able to do this.  Specifically, if $s \in \field{N}$, we will instead have $u_0 \in W^{\nu,\infty}(U_i)$ for $i=1,2$ iff $\nu \in [0,s)$.  However, we do not expect this detail to have a serious impact on our convergence rates (indeed, our numerical experiments in Appendix \ref{app:exp3} - with a few possible exceptions - suggest that it doesn't matter).  To accomplish this we set
\begin{equation} \label{eqn:u0}
u_0(x,y) = w_s(x)+H_{s'}(x)
\end{equation}
where $w_s$ is Fourier series
\begin{equation} \label{eqn:weirstrass}
w_s(x) = \sum_{n=1}^{\infty} 2^{-sn} \cos(2^n \pi x),
\end{equation}
which for $0<s \leq 1$ is an example of a nowhere-differentiable {\em Weierstrass function} \cite{Weirstrass}.  The other component $H_{s'}$ is the (possibly smooth) step function
$$H_{s'}(x) := 1- 1\left(x \leq \frac{1}{4}\right)\tanh\left(20\left[\frac{1}{4}-x\right]\right)^{s'} - 1\left(x \geq \frac{3}{4}\right)\tanh\left(20\left[x-\frac{3}{4}\right]\right)^{s'}.$$
See Figure \ref{fig:Kink} for an illustration of $w_s$ and $H_{s'}$ for various values of $s$ and $s'$.  The following Lemma establishes that $u_0$ given by \eqref{eqn:u0} has the claimed regularity properties.

\vskip 1mm

\begin{lemma} \label{lem:boundaryData}
Let $0<s$, $0 \leq s' \leq s$.  Let $u_0:=w_s+H_{s'}$ be defined as in \eqref{eqn:u0} and illustrated in Figure \ref{fig:Kink}.  Then $u_0 \in W^{\nu,\infty}(U_i)$ for $i=1,2$ iff $\nu \in [0,s]$ if $s \notin \field{N}$, and iff $\nu \in [0,s)$ if $s \in \field{N}$.  At the same time, $u_0 \in W^{\nu',\infty}(\mathcal U)$ iff $\nu' \in [0,s']$.
\end{lemma}
\begin{proof}
It suffices to prove that $w_s \in W^{\nu,\infty}(\mathcal U)$ for $i=1,2$ iff $\nu \in [0,s]$ if $s \notin \field{N}$, and iff $\nu \in [0,s)$ if $s \in \field{N}$, while $H_{s'} \in W^{\nu',\infty}$ iff $\nu' \in [0,s']$.  The claims involving $H_{s'}$ follow from the fact that this function is $C^{\infty}$ everywhere except $x=0.25$ and $x=0.75$, along with the identity $\tanh(x) \leq x^{\alpha}$ for all $x \geq 0$ iff $0 \leq \alpha \leq 1$.  The details are left as an exercise.

On the other hand, the regularity of $w_s$ has been studied for $0<s\leq1$ by a number of authors.  In particular, our claim for $0<s<1$ follows from \cite[Theorem 1.31]{Weirstrass}, while for $s=1$ it is an easy corollary of a statement on \cite[Pg. 311]{Weirstrass}.  The case $s>1$ is not typically considered, but it is a simple exercise to show that in this case $w_s \in W^{s,\infty}((0,1])$ whenever $s \notin \field{N}$, and $w_s \in W^{s-\epsilon,\infty}((0,1])$ for all $\epsilon > 0$ if $s \in \field{N}$.  To see this, note that if $s \notin \field{N}$ and $\lfloor s \rfloor = n \in \field{N}$, then the $n$th derivative of the partial sums in \eqref{eqn:weirstrass} converges uniformly by the Weierstrass M-test.  Thus the $n$th derivative of $w_s$ exists - moreover, it is related in a simple way to the Weirstrass function $w_{s-\lfloor s \rfloor}$ (in some cases with sines replacing the cosines, but this makes no difference), and hence $w_s^{(n)} \in W^{s-\lfloor s \rfloor,\infty}((0,1]))$.  The argument for $s \in \field{N}$ is similar.
\qed
\end{proof}

\begin{figure}
\centering
\begin{tabular}{cc}
\subfloat[$H_{s'}(x)$]{\includegraphics[width=.5\linewidth]{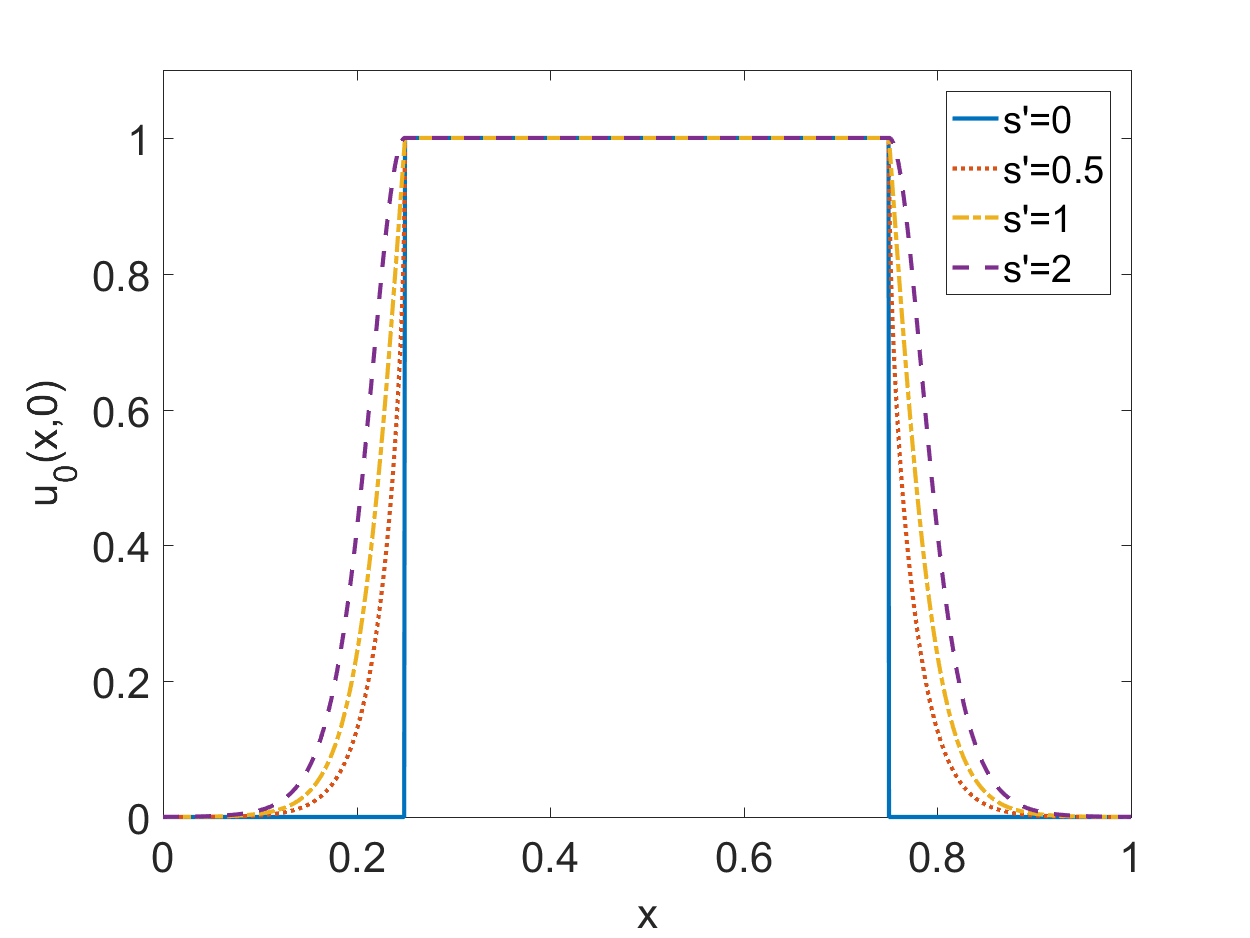}} &
\subfloat[$w_{s}(x)$]{\includegraphics[width=.5\linewidth]{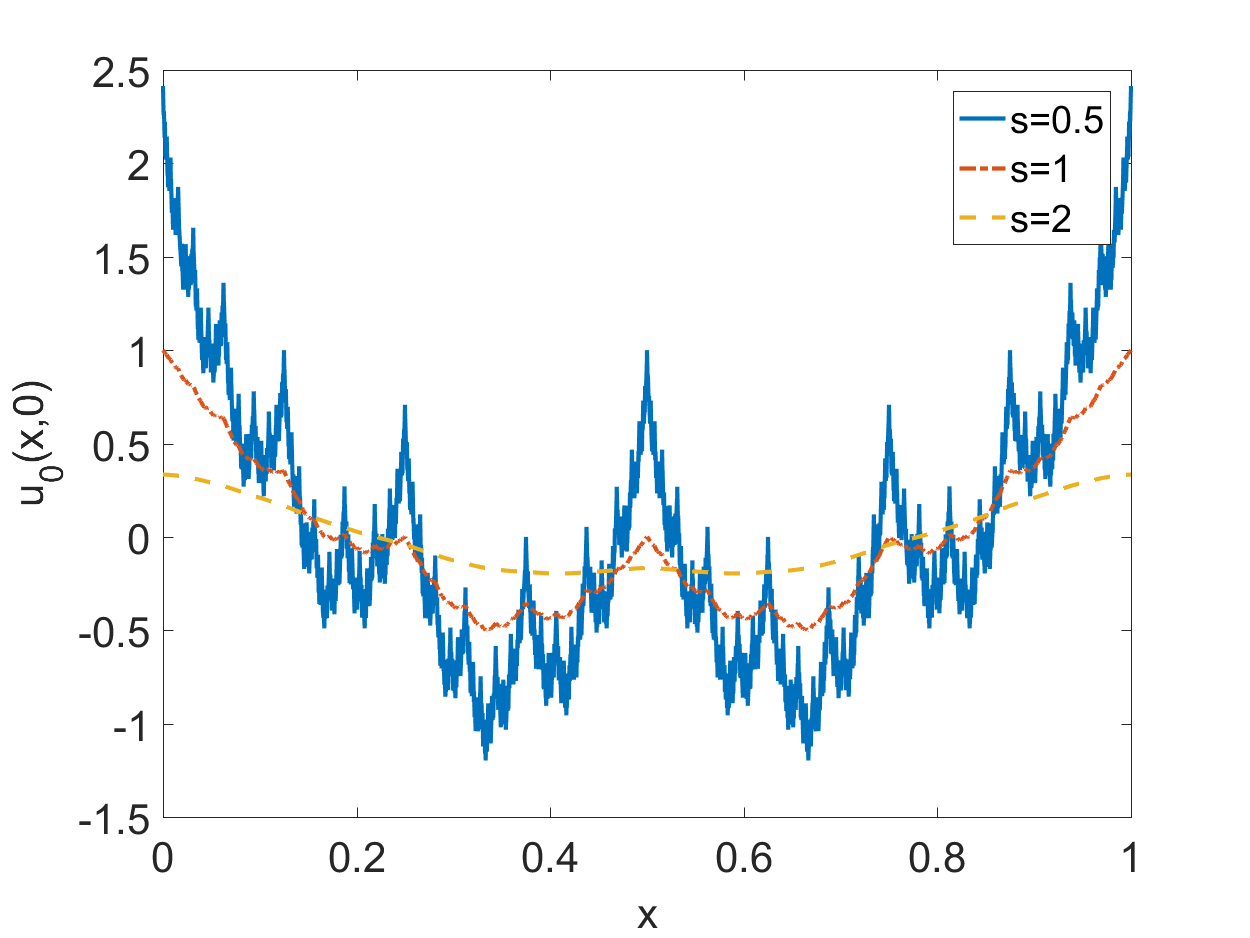}} \\
\end{tabular}
\caption{{\bf Building boundary data with desired smoothness properties:} Our boundary data $u_0(x,y)=w_s(x)+H_{s'}(x)$ is a sum of two components. The first component $H_{s'}$ is smooth everywhere with the exception of two isolated points, while the second component $w_s$ has uniform (but potentially very low) regularity. When $s'=0$, the first component $H_{s'}$ is a step function. When $0<s \leq 1$, the second component $w_s$ is an example of a nowhere differentiable Weierstrass function \cite{Weirstrass}.  See Lemma \ref{lem:boundaryData}.}
\label{fig:Kink}
\end{figure}

\vskip 1mm

\noindent{\bf Neighborhood and weights.} We consider three separate pairings of neighborhoods $A_{\epsilon,h}$ and weights $w_{\epsilon}(\cdot,\cdot)$.

\vskip 1mm

\noindent {\em Example 1: Coherence Transport with constant transport direction}

\vskip 1mm

In this case we choose as our weights $w_{\epsilon}({\bf x},{\bf y})$ M\"arz's weights \eqref{eqn:weight} and neighborhood $\mathcal V_{\epsilon,h}({\bf x})=B_{\epsilon,h}({\bf x})$, that is, we use the same neighborhood and weights as in coherence transport. We fix ${\bf g}=(\cos(20^{\circ}),\sin(20^{\circ}))$ and test two different values of $\mu$, namely $\mu=10$ and $\mu=50$.

\vskip 1mm

\noindent {\em Example 2: Gaussian weights with a box neighbourhood}

\vskip 1mm

In this example for our weights we choose the offset Gaussian
\begin{equation} \label{eqn:offsetG}
w_{\epsilon}({\bf x},{\bf y})= \exp\left(-5\left\|\frac{{\bf x}-{\bf y}}{\epsilon}+\left(\frac{1}{2},\frac{1}{2} \right) \right\|^2 \right).
\end{equation}
and as a neighborhood $A_{\epsilon,h}({\bf x})$ we use the discrete square
$$\mathcal A_{\epsilon,h}({\bf x})={\bf x}+\{-rh,(-r+1)h\ldots (r-1)h,rh)\}^2.$$
\vskip 1mm

\noindent {\em Example 3: Guidefill with a smoothly varying Transport Field}

\vskip 1mm

In this example we consider Guidefill - that is $A_{\epsilon,h}({\bf x}) = \tilde{B}_{\epsilon,h}({\bf x})$ and $w_{\epsilon}(\cdot,\cdot)$ given by \eqref{eqn:weight} - with the spatially varying transport field
\begin{equation} \label{eqn:smoothlyVarying}
{\bf g}(x,y) =\left( \frac{4xy}{1+2y^2},1\right).
\end{equation}
We fix $r=3$ and $\mu=50$. Assuming our results continue to hold in this case, from the discussion in Section \ref{sec:kink3main} for $\mu$ sufficiently large we expect
$${\bf g}^*_r(x,y) = {\bf g}(x,y)$$
to machine precision provided ${\bf g}(x,y)$ never makes an angle shallower than $\theta_c = \arctan(\frac{1}{\sqrt{10}})$ with the x-axis. Indeed, in this case we can easily check that
$$\inf_{(x,y) \in (0,1]^2}\theta({\bf g}(x,y)) = \arctan\left(\frac{5}{4\sqrt{2}}\right) > \theta_c.$$
In this case we know that the characteristics of ${\bf g}(x,y)={\bf g}^*_r(x,y)$ are the level curves of the function $f(x,y) = x/(1+2y^2)$ (See Figure \ref{fig:Characteristics}). This allows us to write the transport operator in this case explicitly as
$$\Pi(x,y) = \left( \frac{x}{1+2y^2},0\right).$$
The continuum limit is then given in terms of our boundary data $u_0$ by $u({\bf x}) = u_0(\Pi({\bf x}))$ as usual.

\begin{figure}
\centering
\includegraphics[width=.5\linewidth]{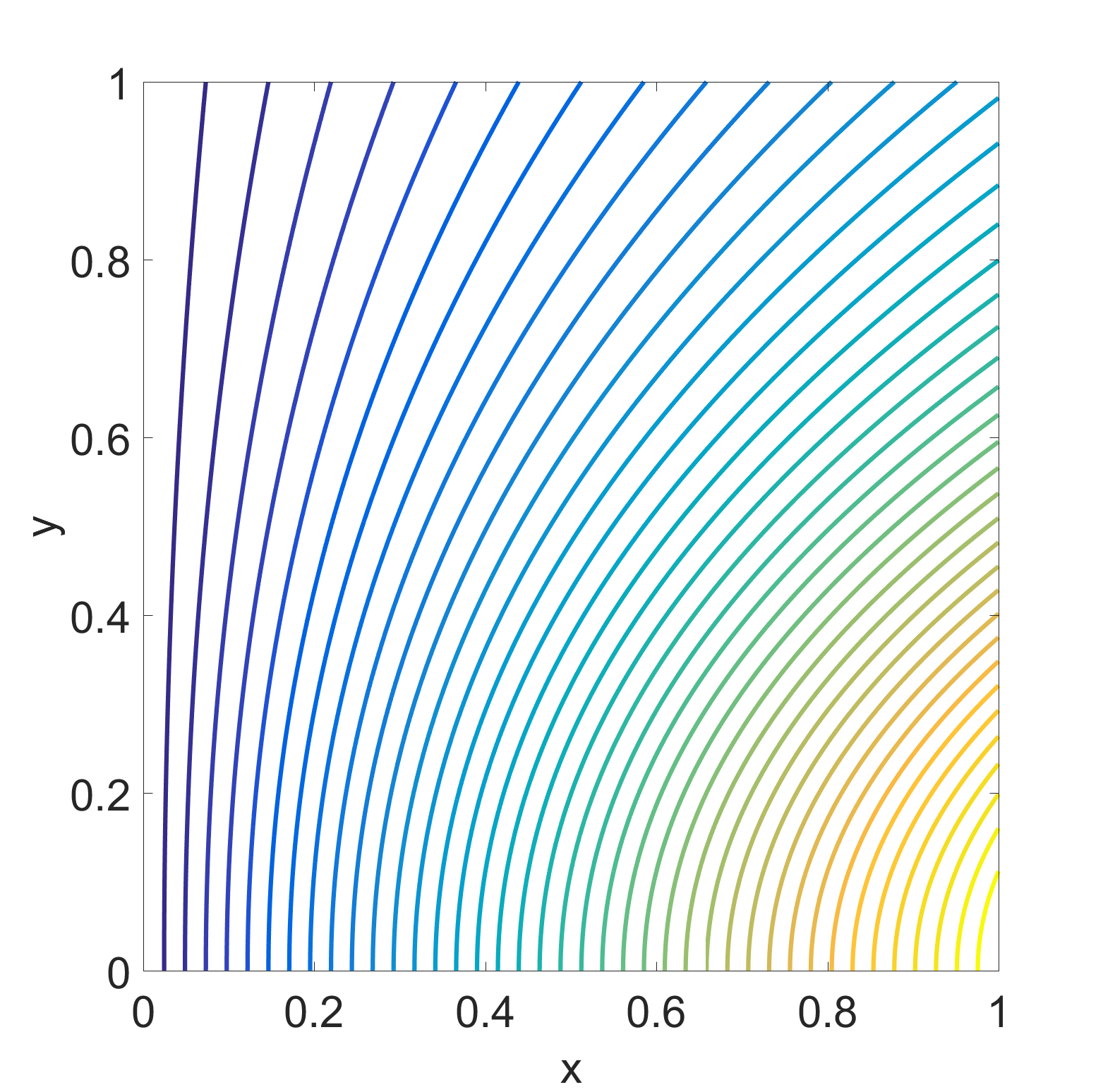}
\caption{{\bf A smoothly varying guideance field:}  Characteristic curves of the smoothly varying transport field ${\bf g}(x,y)$ given by \eqref{eqn:smoothlyVarying}. The transport operator $\Pi({\bf x})$ follows the characteristic passing through ${\bf x}$ backwards to image data located on the boundary of the inpainting domain at $(0,1] \times \{0\}$. }
\label{fig:Characteristics}
\end{figure}

\section{Experiment III : Verifying Theorem \ref{thm:convergence}} \label{app:exp3}

Our final experiment aims to verify the correctness of bounds in Theorem \ref{thm:convergence}, and to check if they are tight.

For each of the examples listed above we first derive the continuum limit $u$ predicted Theorem \ref{thm:convergence}, then compute $u_h$ for a sequence of grids with resolution $h = 2^{-n}$ with $r=3$ fixed. We then look at the ratio
\begin{equation} \label{eqn:Rh}
R_h := \frac{-1}{\log2}\log\left(\frac{\|u-u_{2h}\|_p}{\|u-u_h\|_p} \right).
\end{equation}
Assuming the bounds in Theorem \ref{thm:convergence} are asympototically tight, we expect
\begin{equation} \label{eqn:alpha}
R_h \underset{h \rightarrow 0}{\longrightarrow} \alpha \quad \mbox{ where } \alpha := \frac{s}{2} \wedge \left( \frac{s'}{2}+\frac{1}{2p}\right) \wedge 1
\end{equation}
for stencil weights that are non-degenerate (assign mass to more than one ${\bf y} \in a^*_r$), and
\begin{equation} \label{eqn:alphaDEGENERATE}
\alpha = s \wedge \left( s'+\frac{1}{p}\right) \wedge 1
\end{equation}
otherwise (recall \eqref{eqn:degenerateBound} and the discussion at the start of Section \ref{sec:blur}).

Examples 1 to 3 all use non-degerate stencil weights, with the exception of example 1 with $\mu=50$, where the weights are degenerate to machine precision. We therefore expect convergence rates given by \eqref{eqn:alphaDEGENERATE} in this case and by \eqref{eqn:alpha} in every other case. Although Theorem \ref{thm:convergence} is not technically applicable to Example 3 due to the presence of spatially varying weights, we include this example as a test to see if our bounds continue to hold in this case.  It appears that they do.

Table \ref{table:experiment3} provides a comparison of $R_h$ with the expected convergence rate $\alpha$ for a variety of choices of $s$, $s'$ and $p$, and for neighborhoods and neighborhood weights given by Examples 1 and 3. Results for Example 2 and for additional values of $s$ and $s'$ are analogous but omitted.

Table \ref{table:experiment3} suggests that our bounds are tight {\em asymptotically} as $h \rightarrow 0$, with the resolution required to enter the asymptotic regime apparently dependent on the regularity of $u_0$. In particular, when $u_0$ is very rough (e.g. $s=0.5$, $s'=0$) we need to use a grid with $|\Omega_h|$ greater than a billion before $R_h$ and $\alpha$ agree to a one decimal place. At the opposite extreme, when $u_0$ very smooth (e.g. $s=s'=2$) we enter the asymptotic regime much earlier.  In nearly every case the experimental rate $R_h$, possibly after some initial oscillations, monotically approaches the predicted value $\alpha$ (albeit rather slowly in some cases).  There are, however, two exceptions, namely the case of coherence transport with $\mu = 10$ and ${\bf g} = (\cos20^{\circ},\sin20^{\circ})$ applied to the boundary data \eqref{eqn:u0} parameterized by $s=1$ and $s'=0.5$, with $p \in \{2,\infty\}$.  In these two cases we start above the expected $\alpha$ and then shoot past it.  This may, we concede, have something to do with the exceptional nature of the case $s \in \field{N}$ pointed out in Lemma \ref{lem:boundaryData}.  Or it may be that in these cases we need $h$ to be extremely small before we enter the asymptotic regime.  In any event, the overall message appears to be that the rates from Theorem \ref{thm:convergence} correspond well to experimental rates measured in practice.  Also note that our results do not appear to depend on whether or not we use the constant weights as in Example 1 or the smoothly varying weights from Example 3. This suggests that the hypotheses of Theorem \ref{thm:convergence} can be weakened, which is something we aim to do in the future.

\begin{table}
\caption{Comparison of experimental convergence rate $R_h$ \eqref{eqn:Rh} with the theoretical rate predicted by Theorem \ref{thm:convergence}, for a variety of choices of neighborhoods $A_{\epsilon,h}({\bf x})$, weights $w_{\epsilon}(\cdot,\cdot)$, boundary data $u_0$, and norms $\|\cdot\|_p$. In each case our boundary data is the function \eqref{eqn:u0} parametrized by $s$ and $s'$. For weights we use M\"arz's weights \eqref{eqn:weight}, with either ${\bf g}=(\cos(20^{\circ}),\sin(20^{\circ}))$ fixed or ${\bf g}(x,y)$ given by the smoothly varying transport field \eqref{eqn:smoothlyVarying}. We consider $\mu=10$ and $\mu=50$ - the latter case leads to degenerate stencil weights and a bigger $\alpha$ (notice, for example, $s=s'=0.5$ for $\mu=10$ vs $\mu=50$). As a neighborhood we use either the discrete ball $B_{\epsilon,h}({\bf x})$ used in coherence transport or the rotated ball $\tilde{B}_{\epsilon,h}({\bf x})$ used by guidefill. }
\begin{center}
\begin{tabular}{| c | c | c | c | c | c | c | c | c | c | c | c | c | c |}
\hline
$A_{\epsilon,h}({\bf x})$ & $w(\cdot,\cdot)$ & $ u_0 $ & $\|\cdot \|_p$ & $R_{2^{-9}}$ & $R_{2^{-11}}$ & $R_{2^{-13}}$ & $R_{2^{-15}}$ & $R_{2^{-17}}$ & $\alpha$ \\
\hline
\multirow{24}{*}{$B_{\epsilon,h}({\bf x})$} & \multirow{12}{*}{\parbox{1.1cm}{$\mu=10$, ${\bf g}$ constant}} & \multirow{3}{*}{\parbox{1.1cm}{$s=0.5$\\ $s'=0$}} & $p=1$ & 0.408 & 0.344 & 0.309 & 0.289 & 0.278 & 0.25 \\
\cline{4-10}
& & & $p=2$ & 0.392 & 0.318 & 0.281 & 0.265 & 0.257 & 0.25 \\
\cline{4-10}
& & & $p=\infty$ & 0.126 & 0.07 & 0.036 & 0.019 & 0.009 & 0 \\
\cline{3-10}
& & \multirow{3}{*}{\parbox{1.1cm}{$s=0.5$\\ $s'=0.5$}} & $p=1$ & 0.408 & 0.336 & 0.299 & 0.28 & 0.27 & 0.25 \\
\cline{4-10}
& & & $p=2$ & 0.402 & 0.323 & 0.284 & 0.266 & 0.257 & 0.25 \\
\cline{4-10}
& & & $p=\infty$ & 0.253 & 0.181 & 0.228 & 0.235 & 0.239 & 0.25 \\
\cline{3-10}
& & \multirow{3}{*}{\parbox{1.1cm}{$s=1.0$\\ $s'=0.5$}} & $p=1$ & 0.735 & 0.606 & 0.554 & 0.53 & 0.521 & 0.5 \\
\cline{4-10}
& & & $p=2$ & 0.701 & 0.574 & 0.509 & 0.492 & 0.49 & 0.5 \\
\cline{4-10}
& & & $p=\infty$ & 0.437 & 0.308 & 0.255 & 0.237 & 0.232 & 0.25 \\
\cline{3-10}
& & \multirow{3}{*}{\parbox{1.1cm}{$s=2$\\ $s'=1$}} & $p=1$ & 0.932 & 0.95 & 0.968 & 0.98 & 0.988 & 1 \\
\cline{4-10}
& & & $p=2$ & 0.825 & 0.831 & 0.816 & 0.794 & 0.775 & 0.75 \\
\cline{4-10}
& & & $p=\infty$ & 0.57 & 0.573 & 0.551 & 0.528 & 0.514 & 0.5 \\
\cline{2-10}
& \multirow{12}{*}{\parbox{1.1cm}{$\mu=50$, ${\bf g}$ constant}} & \multirow{3}{*}{\parbox{1.1cm}{$s=0.5$\\ $s'=0.5$}} & $p=1$ & 0.497 & 0.494 & 0.497 & 0.499 & 0.5 & 0.5 \\
\cline{4-10}
& & & $p=2$ & 0.5 & 0.493 & 0.496 & 0.498 & 0.5 & 0.5 \\
\cline{4-10}
& & & $p=\infty$ & 0.477 & 0.472 & 0.494 & 0.494 & 0.499 & 0.5 \\
\cline{3-10}
& & \multirow{3}{*}{\parbox{1.1cm}{$s=1$\\ $s'=0$}} & $p=1$ & 0.92 & 0.937 & 0.947 & 0.954 & 0.96 & 1 \\
\cline{4-10}
& & & $p=2$ & 0.564 & 0.522 & 0.507 & 0.502 & 0.501 & 0.5 \\
\cline{4-10}
& & & $p=\infty$ & 0.072 & 0.021 & 0.005 & 0.001 & 0 & 0 \\
\cline{3-10}
& & \multirow{3}{*}{\parbox{1.1cm}{$s=1$\\ $s'=1$}} & $p=1$ & 0.946 & 0.949 & 0.953 & 0.958 & 0.963 & 1 \\
\cline{4-10}
& & & $p=2$ & 0.94 & 0.946 & 0.952 & 0.958 & 0.962 & 1 \\
\cline{4-10}
& & & $p=\infty$ & 0.882 & 0.864 & 0.907 & 0.916 & 0.928 & 1 \\
\cline{3-10}
& & \multirow{3}{*}{\parbox{1.1cm}{$s=2$\\ $s'=0$}} & $p=1$ & 1.003 & 1.001 & 1 & 1 & 1 & 1 \\
\cline{4-10}
& & & $p=2$ & 0.517 & 0.504 & 0.501 & 0.5 & 0.5 & 0.5 \\
\cline{4-10}
& & & $p=\infty$ & 0.019 & 0.005 & 0.001 & 0 & 0 & 0 \\
\hline
\multirow{12}{*}{$\tilde{B}_{\epsilon,h}({\bf x})$} & \multirow{12}{*}{\parbox{1.1cm}{$\mu=50$, ${\bf g}$ smooth}} & \multirow{3}{*}{\parbox{1.1cm}{$s=0.5$\\ $s'=0$}} & $p=1$ & 0.366 & 0.349 & 0.329 & 0.308 & 0.29 & 0.25 \\
\cline{4-10}
& & & $p=2$ & 0.357 & 0.323 & 0.298 & 0.277 & 0.263 & 0.25 \\
\cline{4-10}
& & & $p=\infty$ & 0.164 & 0.102 & 0.092 & 0.07 & 0.054 & 0 \\
\cline{3-10}
& & \multirow{3}{*}{\parbox{1.1cm}{$s=1$\\ $s'=0$}} & $p=1$ & 0.58 & 0.529 & 0.507 & 0.501 & 0.5 & 0.5 \\
\cline{4-10}
& & & $p=2$ & 0.345 & 0.278 & 0.258 & 0.253 & 0.251 & 0.25 \\
\cline{4-10}
& & & $p=\infty$ & 0.134 & 0.034 & 0.014 & 0.012 & 0.006 & 0 \\
\cline{3-10}
& & \multirow{3}{*}{\parbox{1.1cm}{$s=1$\\ $s'=1$}} & $p=1$ & 0.724 & 0.6 & 0.533 & 0.509 & 0.503 & 0.5 \\
\cline{4-10}
& & & $p=2$ & 0.726 & 0.586 & 0.522 & 0.503 & 0.499 & 0.5 \\
\cline{4-10}
& & & $p=\infty$ & 0.687 & 0.476 & 0.485 & 0.493 & 0.496 & 0.5 \\
\cline{3-10}
& & \multirow{3}{*}{\parbox{1.1cm}{$s=2$\\ $s'=2$}} & $p=1$ & 1.001 & 0.997 & 0.997 & 0.997 & 0.998 & 1 \\
\cline{4-10}
& & & $p=2$ & 0.964 & 0.984 & 0.993 & 0.997 & 0.998 & 1 \\
\cline{4-10}
& & & $p=\infty$ & 0.955 & 0.99 & 0.995 & 0.994 & 0.994 & 1 \\
\cline{3-10}

\hline
\end{tabular}
\end{center}
\label{table:experiment3}
\end{table}

\end{document}